\documentclass[11pt,reqno]{amsart}

\usepackage[latin1]{inputenc}
\usepackage{amsmath,amsthm,amssymb,amscd,mathrsfs,mathtools,color,esint}
\usepackage[perpage]{footmisc}
\definecolor{darkgreen}{rgb}{0.0, 0.6, 0.13}

\usepackage{hyperref}

\usepackage[framemethod=tikz]{mdframed}

\usepackage{simpler-wick}
\usepackage{bbm}
\usepackage[linguistics]{forest}
\usepackage{float}
\usepackage{caption}

\usepackage{tikz}
\usetikzlibrary{arrows.meta,arrows}
\tikzset{
fillcirc/.style={
fill,
circle,
minimum size=0.3cm,
draw=black}
}
\tikzset{
circ/.style={
circle,
minimum size=0.3cm,
draw=black}
}
\tikzset{
emptyrectangle/.style={
circle,
minimum size=0.3cm,
draw=black}
}
\tikzset{
fillstar/.style={
star,
star points=5,
star point ratio=2.5, 
inner sep=0.08cm,
draw=black,
fill}
}

\newtheorem{thm}{Theorem}[section]
 
 \newtheorem{lem}[thm]{Lemma}
 \newtheorem{prop}[thm]{Proposition}

 \theoremstyle{definition}
 \newtheorem{defn}[thm]{Definition}
 \theoremstyle{remark}
 \newtheorem{rem}[thm]{Remark}
 
 \numberwithin{equation}{section}

\newcommand{\defeq}{\mathrel{\mathop:}=}

\newcommand{\R}{\mathbb R}
\newcommand{\T}{\mathbb T}
\newcommand{\Z}{\mathbb Z}

 \newcommand{\Addresses}{{
  \bigskip
  \footnotesize

  Xiao Ma, \textsc{Department of Mathematics, Princeteon University}\par\nopagebreak
  \textit{E-mail address}: \texttt{xiaom@princeton.edu}

  }}

\begin{document}
  \author{Xiao Ma}
 \title{Almost sharp wave kinetic theory of multidimensional KdV type equations with $d\ge 3$}
 \begin{abstract}
In this work, we study the random series expansion of a multidimensional KdV type equation with a diffusion term, the so-called Zakharov-Kuznetsov (ZK) equation. We impose random initial data and periodic boundary condition with period $L$ on this equation. Using the random series expansion, we derive the $3$-wave kinetic equation on the inertial range for $t\lesssim L^{-\varepsilon}T_{\text{kin}}$. Our result reaches kinetic time scale up to $\varepsilon$ loss.

% In this work, we study the random series expansion of a multidimensional KdV type equation with a diffusion term, the Zakharov-Kuznetsov (ZK) equation. The equation is given with random initial data and periodic boundary condition with period $L$. Using this series expansion, we derive the $3$-wave kinetic equation on the inertial range for $t\lesssim L^{-\varepsilon}T_{\text{kin}}$. Our result reaches kinetic time scale up to $\varepsilon$ loss. 

%Compared to the first version of the paper of Staffilani and Tran \cite{ST}, our result works for low dimension case $d\ge 3$. 
\end{abstract}

 \maketitle
 
 \tableofcontents
 \section{Introduction}
 
In this paper, we study the wave turbulence theory for the following KdV type equation
\begin{equation}
    \partial_t\psi(t,x)+\Delta\partial_{x_1}\psi(t,x)-\nu \Delta \psi(t,x)=\lambda \partial_{x_1}(\psi^2(t,x)).
\end{equation}
as an example of three wave system. %The ZK equation is an effective equation of Euler-Poisson equation in plasma physics which characterize the dynamics of magnetized plasma \cite{LLS}.

Before talking about our main result, let us recall some basic concepts from the wave turbulence theory. For a more thorough introduction to this topic, see the book \cite{Nazarenko}.

\subsection{Wave turbulence theory.} It is well-known that in many PDEs coming from physics, energy can be transferred from low frequency Fourier modes to high frequency modes. %(like from $k$ to $k+1$)
This process can be repeated many times so that finally the energy can be injected to very high frequency. 
%(like $1\rightarrow 2\rightarrow 3\rightarrow\cdots\rightarrow 10^4$)
This is called the energy cascade phenomenon. Since it takes many steps to transfer energy to very high modes, the information in the low energy modes is lost and high energy modes become random and exhibit \textit{universality}. This is very similar to the game, \textit{Galton board}, as illustrated in Figure \ref{fig.Galtonboard}.
Each step of energy transfer is similar to a collision of a ball with a pin. To get to the bottom (the high energy modes), a ball must collide with many pins (there must be many steps of energy transfer). Finally, the distribution of balls (energy) becomes random and universal.  
\begin{figure}[H]
    \centering
    \includegraphics[scale = 0.2]{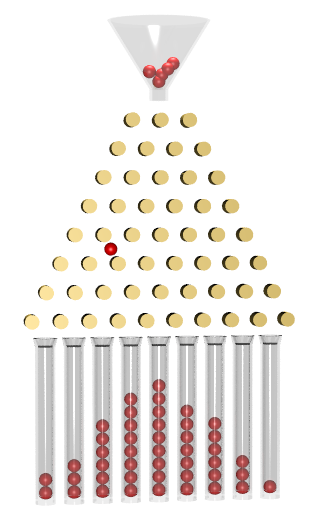}
    \caption{The Galton board}
    \label{fig.Galtonboard}
    \end{figure}

In the Galton board case, the final distribution is the Gaussian distribution. This distribution is universal in the sense that it is independent of the details of the input flow of balls. Although in many examples of energy cascade the final distribution is random and universal, they are usually not i.i.d Gaussian or even do not exhibit properties similar to i.i.d Gaussian.

The cascade spectrum of strong turbulence in the fluid is one example of highly non-Gaussian universality behavior. The structure functions satisfy power laws which are independent of the initial data. The second order structure function satisfies the famous Kolmogorov $5/3$ law while the high order structure functions exhibit a significant deviation from the Kolmogorov's prediction. This deviation is called \textit{intermittency}. 

\begin{figure}[H]
    \centering
    \includegraphics[scale = 0.6] {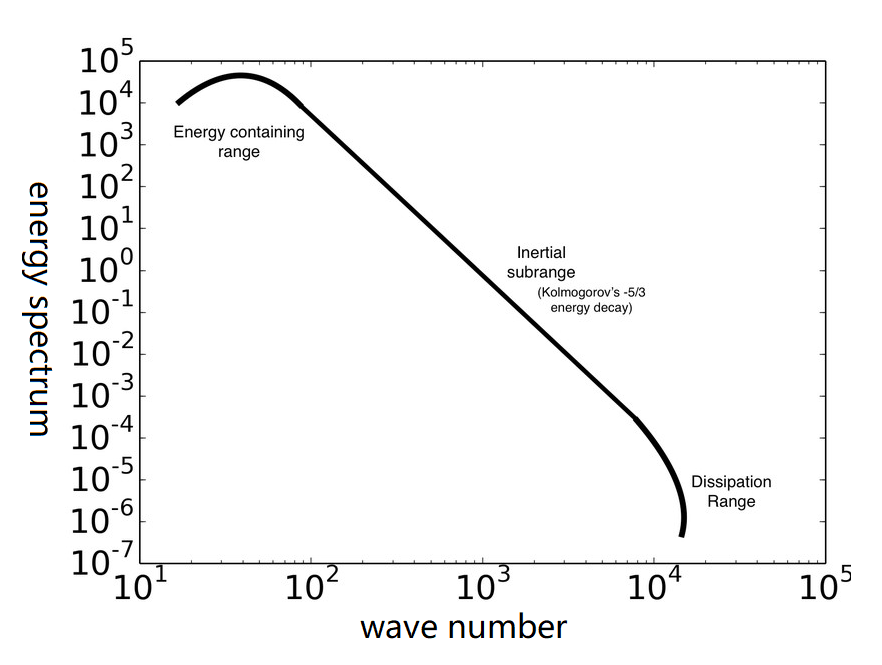}
    \caption{The $5/3$ law}
    \label{fig.5/3law}
    \end{figure}

While i.i.d Gaussian assumption does not work for turbulence in fluid, it can work for turbulence generated by dispersive waves. Unlike fluid equation, a dispersive equation can exhibit turbulent behavior even when the size of the solution is small. In this case, the distribution can be calculated by a perturbative expansion and it is close to i.i.d Gaussian variables if it is so initially. In other words, there is no intermittency in this case. In the wave turbulence, the energy distribution $n(k)$ is supposed to evolve according to the following wave kinetic equation
\[
\tag{WKE}\label{eq.WKE}
\begin{split}
\partial_t n(t, k) =&\mathcal K\left(n(t, \cdot)\right),
\\
\mathcal K(n)(k):=& |k_x|^2\int_{\substack{(k_1, k_2)\in \R^{2d}\\k_1+k_2=k}}n(k_1) n(k_2)\delta(|k_1|^2k_{1x}+|k_2|^2k_{2x}-|k|^2k_{x})\, dk_1 dk_2
\\
-& 2n(k)\int_{\mathbb{R}^d}k_x(k_x-k_{1x})n(k_1) \delta(|k_1|^2k_{1x}+|k_2|^2k_{2x}-|k|^2k_{x})\, dk_1
\end{split}
\]
\eqref{eq.WKE} admits a special solution $|k|^{-d-1}$ and the general solution of \eqref{eq.WKE} is supposed to converge to this power law. This power law is universal and independent of the detail of external force and initial data.

The \eqref{eq.WKE} is a Boltzmann type equation. Compared to the original ZK equation, \eqref{eq.WKE} has a monotonically decreasing entropy (see \cite{GI}) and exhibits irreversibility. The boltzmann equation statistically describes the collision of particles and in the collision integral there are momentum and energy momentum conservation conditions $p_1+p_2=p$ and $p^2_1+p^2_2=p^2$. \eqref{eq.WKE} is a counterpart which describes interaction of waves. The right hand side of \eqref{eq.WKE} also contains momentum and energy conservation $k_1+k_2=k$,  $|k_1|^2k_{1x}+|k_2|^2k_{2x}-|k|^2k_{x}$. These conditions come from the oscillatory phase in \eqref{eq.intmain} $e^{i s\Omega(k_1,k_2,k)}$. The energy conservation is also the time resonance surface in the space-time resonance method.

For a general dispersive PDE, there may not be exactly $3$ wave numbers $k_1,k_2,k$ in the WKE. In general, a quadratic equation has $3$ wave numbers in its WKE, while a cubic equation has $4$ wave numbers in its WKE. Null conditions can increase the number of waves in WKE. If the resonant coefficients vanish on the resonant surface, a quadratic (cubic) equation can have $4$ ($5$) wave numbers in its WKE. ZK equation is an example of $3$ wave interaction case. Nonlinear Schrodinger equations studied in \cite{DH}-\cite{DH3} are examples of $4$ wave interaction case.

The energy distribution $n(k)$ looks very different in different ranges. As in Figure \ref{fig.5/3law}, in the energy-containing range $|k|\ll 1$, the wave number $k$ is not large and the energy distribution in this range can be influenced by the external force and initial data, so there is no universality in this range. In the inertial range $1\lesssim |k|\ll l_d^{-1}$, the influence of external force and initial data starts to be lost, while the wave number is not large enough to activate the effect of dissipation. The energy distribution is universal and satisfies the \eqref{eq.WKE}.
%exhibits power law behavior
In the dissipation range $|k|\gtrsim l_d^{-1}$, the effect of dissipation is dominant and $n(k)$ decays to $0$ very fast. In this paper, we are mostly interested in the energy distribution in the inertial and dissipation range.

In this paper, we choose the ZK equation as an example of $3$ wave turbulence theory.  There are many other PDEs (mostly from plasma physics and capillary waves) whose resonant surfaces contain $3$ wave numbers. The proof in this paper can be adapted to these equations with some additional efforts. 

\subsection{The setup of this paper} We now specify rigorously the randomness and energy distribution used in this paper.

In this paper, we consider the Cauchy problem of the following equation,
\begin{equation}\tag{MKDV}\label{eq.MKDV}
\begin{cases}
\partial_t\psi(t,x)+\Delta\partial_{x_1}\psi(t,x)-\nu \Delta \psi(t,x)=\lambda \partial_{x_1}(\psi^2(t,x)),\\[.6em]
\psi(0,x) = \psi_{\textrm{in}}(x), \quad x\in \T^d_{L}.
\end{cases}    
\end{equation}
We consider the periodic boundary condition, which implies that the spatial domain is a torus $\T^d_{L}=[0,L]^d$. 

We know that the Fourier coefficients of $\psi$ lie on the lattice $\mathbb{Z}_L^d\defeq \{k=\frac{K}{L}:K\in \mathbb{Z}^d\}$. Let $n_{\textrm{in}}$ be a known function, we assume that
\begin{equation}\label{eq.wellprepared}
\psi_{\textrm{in}}(x)=\frac{1}{L^d}\sum_{k\in\mathbb{Z}^d_L}\sqrt{n_{\textrm{in}}(k)} \eta_k(\omega)\,  e^{2\pi i kx}
\end{equation}
where $\eta_k(\omega)$ are mean-zero and identically distributed complex Gaussian random variables satisfying $\mathbb E |\eta_k|^2=1$. To ensure $\psi_{\textrm{in}}$ to be a real value function, we assume that $n_{\textrm{in}}(k)=n_{\textrm{in}}(-k)$ and $\eta_k=\overline{\eta_{-k}}$. Finally, we assume that $\eta_k$ is independent of $\{\eta_{k'}\}_{k'\ne k,-k}$.

In a real turbulent wave, the low frequency Fourier modes in the energy-containing range are influenced by the external force and initial data, so they are not random. In the wave turbulence theory, we just care about high frequency part in the inertial and dissipation range. To simplify the theory, we assume that all Fourier coefficients are random.

%$L$ \textbf{high frequency}

The energy spectrum $n(t,k)$ mentioned in previous section is defined to be $\mathbb E |\widehat \psi(t, k)|^2$, where $\psi(t, k)$ are Fourier coefficients of the solution. Although the initial data is assumed to be the Gaussian random field, it is possible to develop a theory for other types of random initial data.

\subsection{Statement of the results} We define $\Lambda(k)\coloneqq k_{1}(k_1^2+\cdots k_d^2)$. Under this new notations the ZK equation becomes,
\[
\partial_t\psi(t,x)=i\Lambda(\nabla)\psi(t,x)+\nu \Delta \psi(t,x)+\lambda \partial_{x_1}(\psi^2(t,x)).
\]
Now we introduce the main theorem of this paper.

\begin{thm}\label{th.main}
Let $d\ge 3$ and $L$ be a large number. Suppose that $n_{\mathrm{in}} \in C^\infty_0(\mathbb{R}^d)$ is compactly supported in a domain whose diameter is bounded by $D$. Assume that $\psi$ is a solution of \eqref{eq.MKDV} with randomized initial data $\psi_{\mathrm{in}}$ given by \eqref{eq.wellprepared}. Set $\alpha=\lambda L^{-\frac{d}{2}}$ to be the strength of the nonlinearity and $T_{\mathrm{kin}}=\frac{1}{8\pi\alpha^2}$ to be the wave kinetic time. Fix a small constant $\varepsilon> 0$, set $T_{\text{max}} = L^{-\varepsilon} \alpha^{-2}=O(L^{-\varepsilon}T_{\mathrm{kin}})$. If $\alpha$ satisfies
\begin{equation}\label{eq.conditionalpha}
\alpha^{-1}\le L^{\frac{1}{2}}
\end{equation}
and for some small constant $c$, $\nu$ satisfies
\begin{equation}\label{eq.conditionnu}
% \nu\ge L^{\frac{1}{2}C\varepsilon}\alpha^{2}
\nu\ge c^2T^{-1}_{\text{max}}
\end{equation}
then for all $L^{\varepsilon} \leq t \leq T_{\text{max}}$, we have the following conclusions
\begin{enumerate}
    \item If $\sup_{k}n_{\mathrm{in}}(k)\le C_{\mathrm{in}}$, then $\mathbb E |\widehat \psi(t, k)|^2$ is bounded by $2C_{\mathrm{in}}$ for $t\le T_{\text{max}}$ and for any $M$, we can construct an approximation series 
    \begin{equation}\label{eq.approx1}
        \mathbb E |\widehat \psi(t, k)|^2=n_{\mathrm{in}}(k)+n^{(1)}(k)+n^{(2)}(k)+\cdots n^{(N)}(k)+O(L^{-M})
    \end{equation}
    where each terms $n^{(i)}(k)$ can be exactly calculated.
    \item Define $l_{d}=(\nu T_{\mathrm{max}})^{\frac{1}{2}}\ge c$ and  $\theta=C_1\varepsilon$ ($C_1$ just depends on dimension $d$). We assume that $D\le C_2 l_d^{-1}$. Then in the inertial range $|k|\le \epsilon_1 l_{d}^{-1}$ and dissipation range $|k|\ge C_{2}l_d^{-1}$, we have the following estimate respectively
    \begin{equation}\label{eq.n1}
        n^{(1)}(k)=\left\{
        \begin{aligned}
            &\frac{t}{T_{\mathrm{kin}}}\mathcal K(n_{\mathrm{in}})(k)+O_{\ell^\infty_k}\left(L^{-\theta}\frac{T_{\text{max}}}{T_{\mathrm {kin}}}\right)+\widetilde{O}_{\ell^\infty_k}\left(\epsilon_1\text{Err}_{D}(k_x)\frac{T_{\text{max}}}{T_{\mathrm {kin}}}\right)
            && \text{if } |k|\le \epsilon_1 l_{d}^{-1},
            \\
            &0, && \text{if } |k|\ge 2C_{2}  l_{d}^{-1}
        \end{aligned}\right.
    \end{equation}%O_{\ell^\infty_k}\left(\frac{T_{\text{max}}}{T_{\mathrm {kin}}}|k|^{-2}\right)
    and 
    \begin{equation}\label{eq.n(j)estimate}
        n^{(j)}(k)=O_{\ell^\infty_k}\left(L^{-\theta}\frac{t}{T_{\mathrm {kin}}}\right), \qquad j>1
    \end{equation}
    where $\mathcal K$ is defined in \eqref{eq.WKE}, and $O_{\ell^\infty_k}(A)$ (resp. $\widetilde{O}_{\ell^\infty_k}(A)$) is a quantity that is bounded by $A$ in $\ell^\infty_k$ by some universal constant (resp. constant just depending on $d$). The definition of universal constant can be found in section \ref{sec.notat}. The definition of $\text{Err}_D$ is 
    \begin{equation}
        \text{Err}_{D}(k_x)=\left\{\begin{aligned}
             &D^{d+1}, && \text{if } |k_x|\le D,
            \\
            &D^{d-1}(|k_x|^2+D|k_x|), && \text{if } |k_x|\ge D.
        \end{aligned}
        \right.
    \end{equation}
    
\end{enumerate}

\end{thm}

\begin{rem}
The condition $d\ge 3$ is essential. When $d=1$, the ZK equation becomes the KdV equation which is an integrable system whose long time behavior is quasi-periodic instead of turbulent \cite{JM}. When $d=2$, as mentioned in Remark \ref{rem.nottrue2d}, the desire number theory result is not true and a major revision to \eqref{eq.WKE} is required to obtain a valid wave kinetic theory.
\end{rem}

\begin{rem}
$\frac{t}{T_{\mathrm{kin}}}\mathcal K(n_{\mathrm{in}})(k)=O_{\ell^\infty_k}\left(\frac{t}{T_{\mathrm{kin}}}\right)=O_{\ell^\infty_k}\left(\frac{T_{\mathrm{max}}}{T_{\mathrm{kin}}}\right)$ in \eqref{eq.n1} is the largest term in \eqref{eq.n1} and \eqref{eq.n(j)estimate}. Compared to the second term in \eqref{eq.n1} or the first term in \eqref{eq.n(j)estimate}, $\frac{t}{T_{\mathrm{kin}}}\mathcal K(n_{\mathrm{in}})(k)$ is larger than them by a factor $L^{\theta}$.  It is also larger than the third term in \eqref{eq.n1} by a factor of $\epsilon_1$.
\end{rem}

\begin{rem}
The restriction on $\alpha^{-1}$ is not optimal. The optimal result is expected to be $\alpha^{-1}\le L$ for general torus and $\alpha^{-1}\le L^{d/2}$ for generic torus. Here the torus is general or generic in the sense of \cite{DH}. Except for those in the appendix \ref{sec.numbertheoryA}, all the arguments in this paper work under these stronger assumptions.
\end{rem}

\begin{rem}
Due to $\partial_x$ in the nonlinearity, there is a potential risk of loss of derivative. This causes a serious difficulty in controlling the high frequency part. To resolve this difficulty, some regularization to ZK equation is required. In \cite{ST} and this paper, grid discretization and viscosity are introduced respectively. Both of them serve as canonical high frequency truncations.
\end{rem}

\subsection{Ideas of the proof} The basic strategy of proving the main theorem is to construct an approximation series and use probability theory and number theory to control the size and error of this approximation.

\subsubsection{The approximate solution}\label{sec.appsol} The equation of Fourier coefficients is

\begin{equation}\label{eq.Fourierintro}
\dot{\psi}_{k} =  i\Lambda(k) \psi_k -\nu |k|^2 \psi_k
 +\frac{i\lambda}{L^{d}} \sum\limits_{\substack{(k_1,k_2) \in (\mathbb{Z}^d_L)^2 \\ k_1 + k_2 = k}} k_{x_1}\psi_{k_1} \psi_{k_2}
\end{equation}

Define a new dynamical variable $\phi= e^{-it\Lambda(\nabla)} \psi$ and integrate \eqref{eq.Fourierintro} in time. Then (\ref{eq.MKDV}) with initial data (\ref{eq.wellprepared}) becomes
\begin{equation}\label{eq.intmainintro}
\begin{split}
    \phi_k =\xi_k+\frac{i\lambda}{L^{d}} \sum\limits_{k_1 + k_2 = k}\int^{t}_0k_{x_1}\phi_{k_1} \phi_{k_2}e^{i s\Omega(k_1,k_2,k)-\nu(t-s)|k|^2} ds.  
\end{split}
\end{equation}

Here $\Omega(k_1,k_2,k) =\Lambda(k_1)+\Lambda(k_2)-\Lambda(k)$ and $\xi_k$ are the Fourier coefficients of the initial data of $\psi$ defined by $\xi_k=\sqrt{n_{\textrm{in}}(k)} \, \eta_{k}(\omega)$.

Denote the second term of right hand side by $\mathcal{T}(\psi,\psi)_k$ and the right hand side by $\mathcal{F}(\psi)_k=\xi_k+\mathcal{T}(\psi,\psi)_k$. Then the equation becomes $\psi=\mathcal{F}(\psi)_k$. We can construct the approximation by iteration: $\psi=\mathcal{F}(\psi)=\mathcal{F}(\mathcal{F}(\psi))=\mathcal{F}(\mathcal{F}(\mathcal{F}(\psi)))=\cdots$. 

Define the approximate solution by $\psi_{app}=\mathcal{F}^{N}(\xi)$. By recursively expanding  $\mathcal{F}^{N}$, we know that $\psi_{app}$ is a polynomial of $\xi$.
The expansion can be described as the following,
\begin{equation*}
\begin{split}
    \psi_{app}=&\mathcal{F}^{N}(\xi)=\xi+\mathcal{T}(\mathcal{F}^{N-1}(\xi),\mathcal{F}^{N-1}(\xi))
    \\
    =&\xi+\mathcal{T}\Big(\xi+\mathcal{T}(\mathcal{F}^{N-2}(\xi),\mathcal{F}^{N-2}(\xi)),
    \cdots\Big)=\xi+\mathcal{T}(\xi,\xi)+\cdots
    \\
    =&\xi+\mathcal{T}(\xi,\xi)+\mathcal{T}(\mathcal{T}(\xi,\xi),\xi)
    +\mathcal{T}(\xi,\mathcal{T}(\xi,\xi))+\cdots
\end{split}    
\end{equation*}
In the above iteration, we recursively replace $\mathcal{F}^{l}(\xi)$ by $\xi+\mathcal{T}(\mathcal{F}^{l-1}(\xi),\mathcal{F}^{l-1}(\xi))$.

We need a good upper bound for each terms of $\psi_{app}$. To get this we introduce tree diagrams to represent terms $\xi$, $\mathcal{T}(\xi,\xi)$, $\mathcal{T}(\mathcal{T}(\xi,\xi),\xi)$, $\cdots$. The basic notation of tree diagrams will be introduced in section \ref{sec.appFey}.

\subsubsection{The perturbative analysis}\label{sec.pert intro} To prove the main theorem, we need to bound the approximation error of $\psi_{app}$ defined by $w=\psi-\psi_{app}$. To do this, we use the follow equation of $w$ which can be derived from (\ref{eq.intmainintro}):
\begin{equation}\label{eq.eqwintro}
    w= Err(\xi)+Lw+B(w,w)
\end{equation}
Here $Err(\xi)$ is a polynomial of $\xi$ whose degree $\le N+1$ monomials vanish. $Lw$, $B(w,w)$ are linear, quadratic in $w$ respectively.

We prove the smallness of $w$ using the bootstrap method.

Define $||w||_{X^p}=\sup_{k} \langle k\rangle^{p} |w_k|$. Starting from the assumption that $\sup_t||w||_{X^p}\le CL^{-M}$ ($C,M\gg 1$), we need to prove that $\sup_t||w||_{X^p}\le (1+C/2)L^{-M}<CL^{-M}$. To prove $||w||_{X^p}\le (1+C/2)L^{-M}$, we use (\ref{eq.eqwintro}), which gives
\begin{equation}\label{eq.ineqw}
    ||w||_{X^p}\le ||Err(\xi)||_{X^p}+||Lw||_{X^p}+||B(w,w)||_{X^p}
\end{equation}

We just need to show that 
\begin{equation}
    ||Err(\xi)||_{X^p}\le L^{-M},
    \quad ||B(w,w)||_{X^p}\le C^2L^{d+O(1)-2M}.
\end{equation}
Combining with a special treatment of $Lw$, the above estimates imply that $||w||_{X^p}\le (1+C/2)L^{-M}$ which closes the bootstrap.

\subsubsection{Couple diagrams, lattice points counting and $||Err(\xi)||_{X^p}$}\label{sec.latticeintro} In this section we explain the idea of proving upper bound of $||Err(\xi)||_{X^p}$.

$(Err(\xi))_{k}$ is a sum of terms of the form
\begin{equation}
\begin{split}
    &\mathcal{J}_{k}^0(\xi)=  \xi_k, \quad \mathcal{J}_k^1(\xi)=\frac{i\lambda}{L^{d}} \sum_{k_1+k_2-k=0} H^1_{k_1k_2}  \xi_{k_1}\xi_{k_2} , \quad\cdots  \\
    &\mathcal{J}_{T,k}^l(\xi)=\left(\frac{i\lambda}{L^{d}}\right)^l\sum_{k_1+k_2+\cdots+k_{l+1}-k=0} H^l_{k_1\cdots k_{l+1}}(T)  \xi_{k_1}\xi_{k_2}\cdots\xi_{k_{l+1}}, \quad\cdots 
\end{split}
\end{equation}
According to section \ref{sec.appFey}, each terms correspond to a tree diagram and their coefficients can be calculated from these diagrams. This calculation is done in section \ref{sec.refexp}. As a corollary of tree diagram representation, we know that $H^l$ is large near a surface given by $2l$ equations $S=\{S_{\mathfrak{n}_1}(T)=0,\Omega_{\mathfrak{n}_1}(T)=0,\cdots,S_{\mathfrak{n}_{l}}(T)(T)=0,\Omega_{\mathfrak{n}_l}(T)=0\}$.

By the large deviation principle, to obtain upper bounds of Gaussian polynomials $\mathcal{J}_{T,k}^l(\xi)$, it suffices to calculate their variance. This calculation is done in section \ref{sec.coupwick} using the Wick theorem and we introduce the concept of couple diagrams to represent the final result. 

As a corollary of couple diagram representation, we know that the coefficients of the variance concentrate near a surface given by $n$ equations ($n$ is the number of nodes in the couple) $S=\{S_{\mathfrak{n}_1}(T)=0,\Omega_{\mathfrak{n}_1}(T)=0,\cdots,S_{\mathfrak{n}_{n}}(T)(T)=0,\Omega_{\mathfrak{n}_n}(T)=0\}$. Then in order to estimate the variance it suffices to upper bound the number of lattice points near this surface. This is done in section \ref{sec.numbertheory} using the edge cutting argument to reduce the size of the couple. 

The method in \cite{DH} of getting number theory estimate based on tree diagram does not work in our setting. This is because the energy conservation equation $\Lambda(k_1)+\Lambda(k_2)-\Lambda(k)=0$ of ZK equation degenerates seriously when $k_{x}$ is close to $0$. In \eqref{eq.numbertheory1}, the number of solutions of the diophantine equation $\Lambda(k_1)+\Lambda(k_2)=\Lambda(k)+\sigma+O(T^{-1})$ can only be bounded by $|k_x|^{-1}$ which goes to infinity when $k_{x}\rightarrow 0$. This difficulty is resolved by the fact that the multiplier $k_x$ in the last term of \eqref{eq.Fourierintro} vanishes when $k_{x}\rightarrow 0$. Since multipliers become very complicated in higher order tree terms, we introduce the concept of norm edges to keep track of them. 

In conclusion, combining the above arguments, we can show that, for any $M$, we can take $N$ large enough so that $||Err(\xi)||_{X^p}\le L^{-M}$.

% \begin{equation}
%     P(T)=\sum_{S_{3}(T)=0,\Omega_{3}(T)=0,\cdots,\Omega_{l+1}(T)=0,\Omega_{l+1}(T)=0} \xi_{k_1}\xi_{k_2}\cdots\xi_{k_{l+1}}.
% \end{equation}

\subsubsection{Upper bounds for $||B(w,w)||_{X^p}$} $||B(w,w)||_{X^p}$ is a sum of terms of the form
\begin{equation}
    \frac{i\lambda}{L^{d}} \int^{t}_0\sum_{k_1+k_2-k=0} B_{k_1k_2}(s)  w_{k_1}(\xi)w_{k_2}
\end{equation}

%By assumptions and proofs in this paper, we know that $t\le \alpha^{-2}\le L^{O(1)}$ ,  $|B_{k_1k_2k_3}(s)|\lesssim 1$ and $|\mathcal{J}^{l}_{k_1}(\xi)|\lesssim \langle k\rangle^{-p}$. By boostrap assumption, $\sup_{k} \langle k\rangle^{p} |w_k|\le CL^{-M}$. Therefore we have following estimate of $||B(w,w)||_{X^p}$
The upper bound of $||B(w,w)||_{X^p}$ can be obtained by a straight forward estimate
\begin{equation}
||B(w,w)||_{X^p}\le L^{O(1)} ||w||_{X^p} \le C^2 L^{O(1)-2M},
\end{equation}

Therefore, we get the desire upper bounds $||B(w,w)||_{X^p}\ll L^{-M}$ by taking $M> O(1)$.

\subsubsection{A random matrix bound and $Lw$}\label{sec.randmatintro} To obtain a good upper bound for $Lw$, we need to estimate the norm of the random matrix $L$, following the idea in \cite{DH}, \cite{DH2}.

In \cite{DH}, they consider solutions in the Bourgain space $X^{s,b}$ and use $TT^*$ method to get the upper bound for the operator norm of $L$, $||L||_{X^{s,b}\rightarrow X^{s,b}}\ll 1$. But we prefer to work in the simpler functional space $X^p$ which is not a Hilbert space. Although the standard $TT^*$ method is not useful in a non-Hilbert space, we can bypass it using a Neumann series argument.

Let us first explain how $TT^*$ method works. Here we pretend that $||\cdot||_{X^p}$ is a Hilbert norm. The key idea of $TT^*$ method is the inequality $||L||_{X^p\rightarrow X^p}=||(LL^*)^K||_{X^p\rightarrow X^p}^{\frac{1}{K}}\le (L^d\sup_{k,l} ((LL^*)^K)_{k,l})^{1/K}$. To upper bound $||L||_{X^p\rightarrow X^p}$, we just need to estimate $(LL^*)^K)_{k,l}$ which can be calculated by couple diagrams and be estimated by the large deviation inequality. By taking $K$ large, the loss $L^{d/K}$ could be made arbitrarily small.

Unfortunately, $||\cdot||_{X^p}$ is not a Hilbert norm. However, we can bypass the $TT^*$ method using a Neumann series argument. Note that from \eqref{eq.intmainintro} we have the identity
\begin{equation}
    w-Lw= Err(\xi)+B(w,w).
\end{equation}
We have good upper bounds for all of the three terms on the right hand side. By Neumann series argument we have
\begin{equation}
    w= (1-L)^{-1}(\textit{RHS}) =(1-L^K)^{-1}(1+L+\cdots+L^{K-1})(\textit{RHS}).
\end{equation}
By calculating $(L^K)_{k,l}$ we can show that $||L^K||_{X^p\rightarrow X^p}\ll 1$. This implies that $||(1-L^K)^{-1}||_{X^p\rightarrow X^p}\lesssim 1$. Therefore, a good upper bound of $\textit{RHS}$ gives us a good upper bound of $(1+L+\cdots+L^{K-1})(\textit{RHS})$. Combining the above arguments, we obtain the desire estimate of $w$. This is done in section \ref{sec.errorw} and \ref{sec.randommatrices}.

One additional difficulty is the unboundedness of $L$ due to the derivative $\partial_x$ in the nonlinearity. This is controlled by the high frequency decay coming from viscosity.

\subsubsection{Proof of the main theorem} In summary, the above arguments in section \ref{sec.appsol}-\ref{sec.randmatintro} prove that when $t\le \alpha^{-2}$, we have $||w||_{X^p}\le L^{-M}$ with high probability ($P(\textit{false})\lesssim e^{-CL^{\theta}}$).

The above inequality is equivalent to $\sup_k\, |\langle k \rangle^s w_k|\le CL^{-M}$. Remember that $w:=\psi-\psi_{app}$, so with high probability we have the following estimate $\sup_k\, \langle k \rangle^s |\psi_k-\psi_{app,k}|\le CL^{-M}$. This implies that $\mathbb E |\widehat \psi(t, k)|^2=\mathbb E |\psi_{app,k}|^2+O(L^{-M})$. This suggests that we may get the approximation of $\mathbb E |\widehat \psi(t, k)|^2$ by calculating $\mathbb E |\psi_{app,k}|^2$. $\mathbb E |\psi_{app,k}|^2$ can be exactly calculated and the theorem can be proved by extract the main term in $\mathbb E |\psi_{app,k}|^2$. This is done in section \ref{sec.proofmain}.

\subsection{Notations}\label{sec.notat} 

\underline{Universal constants:} In this paper, universal constants are constants that just depend on dimension $d$, diameter $D$ of the support of $n_{\text{in}}$ and the length of the inertial range $l^{-1}_d$. 

\underline{$O(\cdot)$, $\ll$, $\lesssim$, $\sim$:} Throughout this paper, we frequently use the notation, $O(\cdot)$, $\ll$, $\lesssim$. $A=O(B)$ or $A\lesssim B$ means that there exists $C$ such that $A\lesssim CB$. $A\ll B$ means that there exists a small constant $c$ such that $A\lesssim cB$. $A\sim B$ means that there exist two constant $c$, $C$ such that $cB\lesssim A\lesssim CB$. Here the meaning of constant depends on the context. If they appear in conditions involving $k$, $\Lambda$, $\Omega$, etc., like $|k|\lesssim 1$, $\iota_{\mathfrak{e}_1}k_{\mathfrak{e}_1}+\iota_{\mathfrak{e}_2}k_{\mathfrak{e}_2}+\iota_{\mathfrak{e}}k_{\mathfrak{e}}=0$, then they are universal constants. If these constants appear in an estimate which gives upper bound of some quantity, like $||L^K||_{X^p\rightarrow X^p}\ll 1$ or $\sup_t\sup_k  |(\mathcal{J}_T)_k|\lesssim L^{O(l(T)\theta)} \rho^{l(T)}$, then in addition to the quantities that universal constants depend, they can also depend on the quantities $\theta$, $\varepsilon$, $K$, $M$, $N$, $\epsilon_1$.

\underline{Order of constants:} Here is the order of all constants which can appear in the exponential or superscript of $L$. These constants are $\theta$, $\varepsilon$, $K$, $M$, $N$, $\epsilon_1$.%, $D$, $l_{d}$.

All the constants are small compared to $L$ in the sense they are less than $L^{\theta}$ for arbitrarily small $\theta>0$.
%$D$ is assumed to be less than $10l_{d}^{-1}$ and 

$\varepsilon$ can be an arbitrarily small constant less than $0.5$, the reader is encouraged to assume it to be $0.01$. The order of other constants can be decided by the relations $\theta\ll \varepsilon$, $K=O(\theta^{-1})$, $M\gg K$, $N\ge M/\theta$, here the constants in $\ll$, $O(\cdot)$ are universal. 

\underline{$\mathbb{Z}_L^d$:} $\mathbb{Z}_L^d\defeq \{k=\frac{K}{L}:K\in \mathbb{Z}^d\}$

\underline{$k_x$, $k_{\perp}$:} Given any vector $k$, let $k_x$ be its first component and $k_{\perp}$ be the vector formed by the rest components. 

\underline{$\Lambda(k)$, $\Lambda(\nabla)$:} $\Lambda(k)\coloneqq k_{1}(k_1^2+\cdots k_d^2)$ and $\Lambda(\nabla) = i|\nabla|^2\partial_{x_1}$

\underline{Fourier series:} The spatial Fourier series of a function $u: \T_L^d \to \mathbb C$ is defined on $\Z^d_L:=L^{-1}\Z^{d}$ by
\begin{equation}\label{fourierset}
u_k=\int_{\T^d_L} u(x) e^{-2\pi i k\cdot x},\quad \mathrm{\; so \,that \;}\quad u(x)=\frac{1}{L^d}\sum_{k \in \Z^d_L} u_k \,e^{2\pi i k\cdot x}. 
\end{equation}
Given any function $F$, let $F_k$ or $(F)_k$ be its Fourier coefficients.

\underline{Order of $L$:} In this paper, $L$ is assumed to be a constant which is much larger than all the universal constants and $\theta$, $\varepsilon$, $K$, $M$, $N$, $\epsilon_1$. 

\underline{$L$-certainty:} If some statement $S$ involving $\omega$ is true with probability $\geq 1-O_{\theta}(e^{-L^\theta})$, then we say this statement $S$ is $L$-certain.

\subsection{A short survey of previous papers} (1) \underline{Results about the ZK equation:} ZK equation was introduced in \cite{ZK} as an asymptotic model to describe the propagation of nonlinear ionic-sonic waves in a magnetized plasma. For a good reference about the physical background, see the book \cite{Dbook}. For rigorous results about wellposedness and derivation from the Euler-Poisson system see \cite{LLS} and references therein.

(2) \underline{Previous papers about wave turbulence theory:} There are numerous physics papers about the derivation of wave kinetic equation. In particular, the wave kinetic equation for the ZK equation is derived in \cite{K}. For general references, see the books \cite{ZLFBook} and \cite{Nazarenko}, and the review paper \cite{NR}. 

%Peierls \cite{Peierls1}, \cite{Peierls2}, ZaslavskiiSagdeev [110], Hasselmann [59, 60], Benney-Saffman-Newell [6, 7], 

The WKE was rigorously verified for the Gibbs measure initial data by Lukkarinen and Spohn \cite{LukSpohn}. Then the basic concepts of general wave turbulence were rigorously formulated by Buckmaster, Germain, Hani, Shatah \cite{BGHS2} and a non-trivial result that verified WKE for a short time scale was also proved by them. The WKE was proved for almost sharp time scale independently by Deng and Hani \cite{DH} and Collot and Germain \cite{CG1}, \cite{CG2} using the ideas from the study of randomly initialized PDE. The full WKE for the sharp time was proved independently by the deep works of Deng and Hani \cite{DH2} and Staffilani and Tran \cite{ST} for a four-wave problem and a three-wave problem respectively. One key contribution of \cite{DH2} and \cite{ST} was the classification of Feynman diagrams in the contexts of normal form expansion and Liouville equation respectively. WKE for the space-inhomogeneous case was derived by Ampatzoglou, Collot, and Germain \cite{ACG} for almost sharp time scale. The higher order correlation functions were studied by Deng and Hani \cite{DH3}. A linearized wave kinetic equation near the Rayleigh-Jeans spectrum was derived by Faou \cite{Faou}. The discrete wave turbulence was studied by Dymov and Kuksin \cite{DK1}-\cite{DK4}.
%But the most physically interesting solution of WKE, the Kolmogorov-Zakharov spectrum, are not Gibbs measure.  

Among the above papers, \cite{ST} is the only one working on ZK equation. \cite{ST} derived the wave kinetic equation for lattice ZK equation with random force for $t\le T_{\text{kin}}$, while our paper is for dissipative continuous ZK equation for $t\le L^{-\varepsilon}T_{\text{kin}}$.

% Compared to their lattice model case, the degeneracy of resonance surfaces in our dissipative continuous setting can be handled using the multiplier $k_{x_1}$ in the equation. 

% The main conclusion of this paper is similar to the random force free case of the first arXiv version of Staffilani-Tran \cite{ST}. Compared to this version of \cite{ST}, the assumption on dimension $d$ of this paper is better while the time scale is shorter. We also mention that a few days before the first submission of this paper, Staffilani and Tran also submitted a new arXiv version of \cite{ST} which obtained some results about the low dimensional KZ equation. Compared to our paper, their new results need to add random force into the KZ equation, while our paper does not have this problem. 

(3) \underline{Previous papers about the dynamics of WKE:} There are also many papers about the dynamics of WKE itself. For references, see \cite{GI}, \cite{GST}, \cite{SoT1}, \cite{SoT2} and the reference therein.

%(4) numerical simulation

%(4) about boltzmann equation and randomized initial data

\section{The Perturbation Expansion}
In this section, we calculate the renormalized approximation series and introduce Feynman diagrams to represent terms in this series. Then we bound the error of this approximation by the bootstrap method, assuming several propositions about the upper bounds of higher order terms. We will prove these propositions in the rest part of the paper.

\subsection{The approximation series and Feynman diagrams}\label{sec.appFey} In this section we derive the equation for Fourier coefficients and construct the approximate solution. 

\subsubsection{The Equation of Fourier coefficients}

Let $\psi_k$ be the Fourier coefficient of $\psi$. Then in term of $\psi_k$ equation (\ref{eq.MKDV}) becomes
\begin{equation}\label{eq.mainfourier}
\begin{cases}
 \dot{\psi}_{k} =  i\Lambda(k) \psi_k -\nu |k|^2 \psi_k
 +\frac{i\lambda}{L^{d}} \sum\limits_{\substack{(k_1,k_2) \in (\mathbb{Z}^d_L)^2 \\ k_1 + k_2 = k}} k_{x}\psi_{k_1} \psi_{k_2}  \\[2em]
\psi_k(0) = \xi_k = \sqrt{n_{\textrm{in}}(k)} \, \eta_{k}(\omega)
\end{cases}
\end{equation}

Define the linear profile by
\begin{equation}
\phi_k(t):= e^{-i\Lambda(k) t}  \psi_k(t)    
\end{equation}

Rewriting \eqref{eq.mainfourier} in terms of $\phi_k$ gives  

\begin{equation}\label{eq.mainlinearprofile}
\begin{split}
\dot{\phi}_{k} 
= -\nu |k|^2 \phi_k + \frac{i\lambda}{L^{d}} \sum\limits_{S(k_1,k_2,k)=0}k_{x}\phi_{k_1} \phi_{k_2}e^{i t\Omega(k_1,k_2,k)}
\end{split}
\end{equation}
where
\begin{equation}
\begin{split}
    &S(k_1,k_2,k) = k_1 + k_2 - k,
    \\
    &\Omega(k_1,k_2,k) =\Lambda(k_1)+\Lambda(k_2)-\Lambda(k).
\end{split}
\end{equation}

We will work with \eqref{eq.mainlinearprofile} in the rest part of this paper.

Integrating (\ref{eq.mainlinearprofile}) gives

\begin{equation}\label{eq.intmain}
\begin{split}
    \phi_k =\xi_k+
    \underbrace{\frac{i\lambda}{L^{d}} \sum\limits_{S(k_1,k_2,k)=0}\int^{t}_0k_{x}\phi_{k_1} \phi_{k_2}e^{i s\Omega(k_1,k_2,k)- \nu|k|^2(t-s)} ds}_{\mathcal{T}(\phi,\phi)_k}.  
\end{split}
\end{equation}

Denote the second term on the right hand side by $\mathcal{T}(\phi,\phi)_k$. Denote the right hand side by $\mathcal{F}(\phi)_k=\xi_k+\mathcal{T}(\phi,\phi)_k$. With these notations, \eqref{eq.intmain} becomes  $\phi=\mathcal{F}(\phi)_k$. 

We construct the approximation series by iteration: $\phi=\mathcal{F}(\phi)=\mathcal{F}(\mathcal{F}(\phi))=\mathcal{F}(\mathcal{F}(\mathcal{F}(\phi)))=\cdots$. To estimate this approximation series, we need a compact graphical notation to represent the huge amount of terms generated from iteration. This is done by introducing the concept of Feynman diagrams.

\subsubsection{Some basic definitions from graph theory} In this section we introduce the concept of binary trees, branching nodes, leaves, subtrees, node decoration
%sign of a node
and expanding leaves.

\begin{defn}
\begin{enumerate}
    \item \textbf{Binary trees:} A \underline{binary tree} $T$ is a tree in which each node has $2$ or $0$ children. An example of binary tree used in this paper is show in Figure \ref{fig.decsub}.
    \item \textbf{Branching nodes:} A \underline{branching node} in a binary tree is a node which has $2$ children. The number of all branching nodes in a tree $T$ is denoted by $l(T)$. In Figure \ref{fig.decsub}, $l(T)=2$.
    \item \textbf{Leaves:} A \underline{leaf} of a tree $T$ is a node which has no child. In Figure \ref{fig.decsub}, all $\star$ nodes and $\Box$ nodes are leaves.
    \item \textbf{Subtrees:} If any child of any node in a subset $T'$ of a tree $T$ is also contained in $T'$ then $T' $ also forms a tree, we call $T'$ a \underline{subtree} of $T$. If the root node of $T'$ is $\mathfrak{n}\in T$, we say $T'$ is the \underline{subtree rooted at $\mathfrak{n}$} or \underline{subtree of $\mathfrak{n}$} and denote it by $T_\mathfrak{n}$. In Figure \ref{fig.decsub}, the tree inside the box is the subtree rooted at node $\bullet$.
     \item \textbf{Node decoration:} In Figure \ref{fig.decsub}, each node is associated with a symbol in $\{\bullet,\ \star,\ \Box\}$. If a node $\mathfrak{n}$ has symbol $\bullet$ (similarly $\star,\ \Box$), we say $\mathfrak{n}$ is decorated by $\bullet$ ($\star,\ \Box$) or $\mathfrak{n}$ has decoration $\bullet$ ($\star,\ \Box$). In what follows we adopt the convention that leaves always have decoration $\star$ or $\Box$ and nodes other than leaves always have decoration $\bullet$.
     
    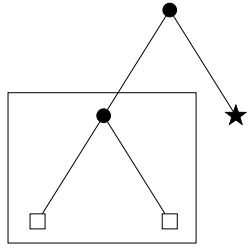
\begin{figure}[H]
    \centering
    \scalebox{0.5}{
    \begin{tikzpicture}[level distance=80pt, sibling distance=100pt]
        \draw node[fillcirc](1) {}
            child {node[fillcirc] (2) {}
                   child {node[draw, minimum size=0.4cm] (5) {}}
                   child {node[draw, minimum size=0.4cm] (6) {}}
                  }
            child {node[fillstar] (3) {}};
        \node[rectangle, draw, minimum width = 5cm, minimum height = 4cm] at (-1.8,-4.2) {}; 
    \end{tikzpicture}
    }
    \caption{Subtrees and node decoration.}
    \label{fig.decsub}
    \end{figure}

    \item \textbf{Expanding and final leaves:} Leaves denoted by $\Box$ are called \underline{expanding leaves}. Other leaves denoted by $\star$ are called \underline{final leaves}. The notion of expanding leaves is useful in the construction of trees, in which the presence of $\Box$ means that the construction is not finishing and $\Box$ denotes leaves that may be replaced by a branching node later.
    
    The concept of expanding leaves and $\Box$ is only used in section \ref{sec.connection}, so the readers can safely forget it after that section.
\end{enumerate}

    % \begin{forest}[$\circ$ [$\bullet$ [$\times$] [$\times$] [$\times$]] [$*$] [$*$] ]
    % \end{forest}

\end{defn}

\subsubsection{Connection between iteration and trees}\label{sec.connection} In this section we explain non-rigorously the connection between perturbation expansion and trees. Rigorous argument can be find in the next section. 

This iteration process can be described as the following, 
\begin{equation*}
\begin{split}
    \phi=&\mathcal{F}(\phi)=\xi+\mathcal{T}(\phi,\phi)
    \\
    =&\xi+\mathcal{T}\Big(\xi+\mathcal{T}(\phi,\phi),
    \cdots\Big)=\xi+\mathcal{T}(\xi,\xi)+\mathcal{T}\Big(\mathcal{T}(\phi,\phi),
    \xi\Big)\cdots
    \\
    =&\xi+\mathcal{T}(\xi,\xi)+\mathcal{T}(\mathcal{T}(\xi,\xi),\xi)
    +\mathcal{T}(\xi,\mathcal{T}(\xi,\xi))+\cdots
\end{split}    
\end{equation*}
In the above iteration, we recursively choose one $\phi$, replace it by $\xi+\mathcal{T}(\phi,\phi)$ and use the linearity of $\mathcal{T}$ to expand into two terms.
\begin{equation}\label{eq.termgeneration}
\begin{split}
    &\mathcal{T}\Big(\cdots,\mathcal{T}(\mathcal{T}(\xi,\underline{\phi}),\cdots)\Big)\rightarrow \mathcal{T}\Big(\cdots,\mathcal{T}(\mathcal{T}(\xi,\underline{\xi+\mathcal{T}(\phi,\phi)}),\cdots)\Big)
    \\
    =& \underbrace{\mathcal{T}\Big(\cdots,\mathcal{T}(\mathcal{T}(\xi,\underline{\xi}),\cdots)\Big)}_{I}
    +\underbrace{\mathcal{T}\Big(\cdots,\mathcal{T}(\mathcal{T}(\xi,\underline{\mathcal{T}(\phi,\phi)}),\cdots)\Big)}_{II}
\end{split}        
\end{equation}
Here $I$ and $II$ are obtained by replacing $\phi$ by $\xi$ and $\mathcal{T}(\phi,\phi)$ respectively.

In summary, all terms in the expansion can be generated by following steps

\begin{itemize}
    \item \textbf{Step $0$.} Add a term $\phi$ in the summation $\mathcal{J}$.
    \item \textbf{Step $i$ ($i\ge 1$).} Assume that \textbf{Step $i-1$} has been finished which produces a sum of terms $\mathcal{J}$, then choose a term in $\mathcal{J}$ which has least number of $\xi$ and $\phi$, remove this term from $\mathcal{J}$ and add the two terms in $\mathcal{J}$ constructed in \eqref{eq.termgeneration}.
\end{itemize}

This process is very similar to the construction of binary trees, in which we recursively replace a chosen expanding node by a leaf or branching node.

\begin{itemize}
    \item \textbf{Step $0$.} Start from a expanding root node $\Box$.
    
    \item \textbf{Step $i$ ($i\ge 1$).} Assume that we have finish the \textbf{Step $i-1$} which produces a collection of trees $\mathscr{T}$, then choose a tree in $\mathscr{T}$ which has least number of expanding leaves $\Box$ and final leaves $\star$, remove this tree from $\mathscr{T}$ and add two new trees in $\mathscr{T}$. In these two new trees, we replace a expanding leaf $\Box$ by a final leaf $\star$ or a branching node $\bullet$ with two expanding children leaves $\Box$. This construction is illustrated by Figure \ref{fig.construction}.
    \begin{figure}[H]
    \centering
    \scalebox{0.5}{
    \begin{tikzpicture}[level distance=80pt, sibling distance=100pt]
        \draw node[fillcirc](1) {} 
            child {node[draw, minimum size=0.4cm] (2) {}}
            child {node[fillstar] (3) {}};
        \node[draw, single arrow,
              minimum height=33mm, minimum width=8mm,
              single arrow head extend=2mm,
              anchor=west, rotate=0] at (4,-1.5) {};  
        \node[scale=3.0] at (16,-2.9) {,};
        \node[fillcirc](4) at (12,0) {} 
            child {node[fillstar] (5) {}}
            child {node[fillstar] (6) {}};
        \node[fillcirc](7) at (22,1.5) {} 
            child {node[fillcirc] (8) {}
                child {node[draw, minimum size=0.4cm] (5) {}}
                child {node[draw, minimum size=0.4cm] (6) {}}
                }
            child {node[fillstar] (9) {}};    
    \end{tikzpicture}
    }
    \caption{One step in the construction of binary trees}
    \label{fig.construction}
    \end{figure}
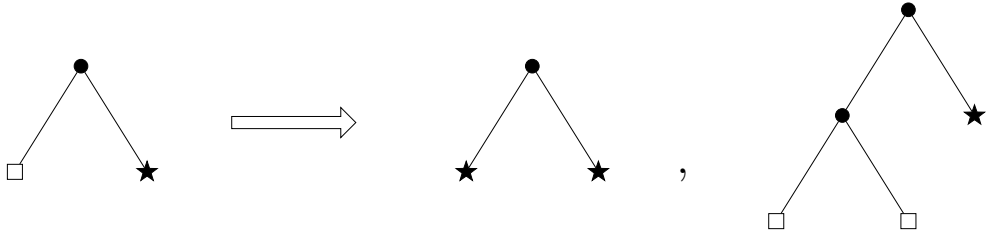
    % \begin{figure}
    %     \begin{forest}[$\bullet$ [$\times$] [$*$] [$*$] ]
    %     \end{forest}
    % \end{figure}
    % \begin{figure}
    %     \begin{forest}[$\bullet$ [$*$] [$*$] [$*$] ]
    % \end{forest}
    % \end{figure}
    % \begin{figure}
    %     \begin{forest}[$\bullet$ [$\bullet$ [$\times$] [$\times$] [$\times$]] [$*$] [$*$] ]
    % \end{forest}
    % \end{figure}

% \begin{tikzpicture}
%     \matrix (m)[matrix of math nodes]
%         {
%     & \begin{forest}[$\bullet$ [$\times$] [$\bullet$] [$\bullet$] ]
%         \end{forest} & \qquad & \begin{forest}[$\bullet$ [$\bullet$] [$\bullet$] [$\bullet$] ]
%         \end{forest} & \quad & \begin{forest}[$\bullet$ [$\bullet$ [$\bullet$] [$\bullet$] [$\bullet$]] [$\bullet$] [$\bullet$] ]
%         \end{forest}  \\
% };
% \draw[thick, ->] (m-1-2) -- (m-1-4);
%         \centering
%     \end{tikzpicture}
\end{itemize}

By comparing the above two process, we can make the connection between terms and trees more explicit. Each node $\bullet$ other than leaf in the tree $T$ corresponds to a $\mathcal{T}(\cdots,\cdots)$ in a term $\mathcal{J}_{T}$. Each final leaf $\star$ and expanding leaf $\Box$ corresponds to $\xi$ and $\phi$ respectively. The \textbf{Step} $i$ of replacing $\phi$ by $\xi$ or $\mathcal{T}(\phi,\phi)$ corresponds to replacing $\Box$ by $\star$ or a branching node with two children $\Box$.

We have following recursive formula for calculating a term $\mathcal{J}_T$ from a binary tree $T$. 

If $T$ has only one node then $\mathcal{J}_T=\xi$. Otherwise let $\bullet_1$, $\bullet_2$ be two children of the root node $\bullet$, let $T_{\bullet_1}$, $T_{\bullet_2}$ be the subtrees of $T$ rooted at the above nodes. If $\mathcal{J}_{T_{\bullet_1}}$, $\mathcal{J}_{T_{\bullet_2}}$ have been recursively calculated, then $\mathcal{J}_T$ can be calculated by
\begin{equation}\label{eq.treeterm'}
    \mathcal{J}_T=\mathcal{T}(\mathcal{J}_{T_{\bullet_1}}, \mathcal{J}_{T_{\bullet_2}}).
\end{equation}

The formal power series obtained by iterating $\phi=\mathcal{F}(\phi)$ can be calculated from trees by $\sum_{T\in \mathscr{T}} \mathcal{J}_T$.

Let $l(T)$ be the number of branching nodes in $T$, then it can be shown that $\mathcal{J}_T$ is a degree $l(T)+1$ polynomial of $\xi$. We define the approximation series to be a finite degree truncation of the formal power series which equals to $\sum_{l(T)\le N} \mathcal{J}_T$.

\subsubsection{Feynman diagrams and construction of the approximation solution} In this section we present the rigorous argument equivalent to that in the above section. 

In the construction of trees, finally all $\Box$ nodes will be replaced by $\bullet$, $\star$, so in what follows we only consider trees whose nodes are decorated by $\bullet$, $\star$.

\begin{defn}\label{def.treeterms} Given a binary tree $T$ whose nodes are decorated by $\bullet$, $\star$,
% if the sign of the root is $+$, 
we inductively define the quantity $\mathcal{J}_T$ by:
\begin{equation}\label{eq.treeterm}
    \mathcal{J}_T=
    \begin{cases}
    \xi, \qquad\qquad\quad\  \textit{ if $T$ has only one node $\star$.}
    \\
    \mathcal{T}(\mathcal{J}_{T_{\mathfrak{n}_1}}, \mathcal{J}_{T_{\mathfrak{n}_2}}), \textit{ otherwise.}
    \end{cases}
\end{equation}
Here $\mathfrak{n}_1$, $\mathfrak{n}_2$ are two children of the root node $\mathfrak{r}$ and $T_{\mathfrak{n}_1}$, $T_{\mathfrak{n}_2}$ are the subtrees of $T$ rooted at the above nodes.
\end{defn}

\begin{defn}
Given a large number $N$, define the approximate solution $\phi_{app}$ by
\begin{equation}\label{eq.approxsol}
    \phi_{app}=\sum_{l(T)\le N} \mathcal{J}_T
\end{equation}
\end{defn}

Section \ref{sec.connection} explains why the approximation series should equals to \eqref{eq.approxsol}, a sum of many tree terms, but if we know this fact, we can directly prove it, and forget all the motivations. The lemma below proves that $\phi_{app}$ defined by the above expression is an approximate solution.  

\begin{lem}\label{lem.approxerror} Define 
\begin{equation}
    Err=\mathcal{F}(\phi_{app})-\phi_{app},
\end{equation}
then we have 
\begin{equation}\label{eq.approxerror}
    Err=\sum_{T\in \mathcal{T}_{>N}^*} \mathcal{J}_T,
\end{equation}
where $\mathcal{T}_{>N}^*$ is defined by
\begin{equation}
\begin{split}
    \mathcal{T}_{>N}^*=\{&T:l(T)>N,\ l(T_{\mathfrak{n}_1})\le N,
    l(T_{\mathfrak{n}_2})\le N,
    \\
    &\textit{$T_{\mathfrak{n}_1}$, $T_{\mathfrak{n}_2}$ are the subtrees defined in Definition \ref{def.treeterms}} \}
\end{split}
\end{equation}
\end{lem}

\begin{rem}
Notice that all terms in $\sum_{T\in \mathcal{T}_{>N}^*}$ are polynomials of $\xi$ of degree $>N$. Therefore, the approximation error of $\phi_{app}$ is of very high order, which proves that $\phi_{app}$ is an appropriate approximation solution. 
\end{rem}

\begin{proof} By \eqref{eq.approxsol}, we get
\begin{equation}\label{eq.lemapproxerror}
\begin{split}
    Err=&\mathcal{F}(\phi_{app})-\phi_{app}    
    \\
    =&\xi+\mathcal{T}(\phi_{app},\phi_{app})-\phi_{app}
    \\
    =&\xi+\sum_{l(T_1),l(T_2)\le N} \mathcal{T}(\mathcal{J}_{T_1},\mathcal{J}_{T_2})-\sum_{l(T)\le N} \mathcal{J}_T
\end{split}
\end{equation}

Let $T$ be a tree constructed by connecting the root nodes $\mathfrak{n}_1$, $\mathfrak{n}_2$ of $T_1$, $T_2$ to a new node $\mathfrak{r}$. We define $\mathfrak{r}$ to be the root node of $T$.

Then by \eqref{eq.treeterm}, we have
\begin{equation}
    \mathcal{J}_T=\mathcal{T}(\mathcal{J}_{T_1}, \mathcal{J}_{T_2})
\end{equation}
and 
\begin{equation}
    \sum_{l(T_1),l(T_2)\le N} \mathcal{T}(\mathcal{J}_{T_1},\mathcal{J}_{T_2})=\sum_{\substack{l(T)\ge 1\\ l(T_1),l(T_2)\le N}} \mathcal{J}_{T}
\end{equation}

By \eqref{eq.lemapproxerror}, we get
\begin{equation}
\begin{split}
    Err=&\xi+\sum_{\substack{l(T)\ge 1\\ l(T_1),l(T_2)\le N}} \mathcal{J}_{T}-\sum_{l(T)\le N} \mathcal{J}_T
    \\
    =&\sum_{\substack{T_1,T_2\text{ are subtrees of }\mathfrak{r}\\ l(T_1),l(T_2)\le N}} \mathcal{J}_{T}-\sum_{\substack{T_1,T_2\text{ are subtrees of }\mathfrak{r}\\l(T)\le N\\ l(T_1),l(T_2)\le N}} \mathcal{J}_T.
    \\
    =&\sum_{T\in \mathcal{T}_{>N}^*} \mathcal{J}_T
\end{split}
\end{equation}
Here in the second equality, we use the fact that $\sum_{l(T)\le N}=\sum_{\substack{l(T)\le N\\ l(T_1),l(T_2)\le N}}$.

Therefore, we complete the proof of this lemma.
\end{proof}

\subsection{Estimates of the approximation solution}

\subsubsection{Estimates of tree terms} By \eqref{eq.approxerror}, in order to control the approximation error $Err$, it suffices to get upper bounds of tree terms $\mathcal{J}_T$. We state the upper bound in the proposition below and delay its proof to section \ref{sec.treetermsupperbound}.

Let us introduce a definition before state the proposition.

\begin{defn}\label{def.Lcertainly}
Given a property $A$, we say $A$ happens $L$-certainly if the probability that $A$ happens satisfies $P(A)\ge 1-Ce^{-L^\theta}$ for some $C, \theta>0$.
\end{defn}

\begin{prop}\label{prop.treetermsupperbound}
We have $L$-certainly for any $\theta$ that 
\begin{equation}\label{eq.treetermsupperbound}
    \sup_t\sup_k  |(\mathcal{J}_T)_k|\lesssim L^{O(l(T)\theta)} \rho^{l(T)}.
\end{equation}
and $(\mathcal{J}_T)_k=0$ if $|k|\gtrsim 1$. Here $(\mathcal{J}_T)_k$ is the Fourier coefficients of $\mathcal{J}_T$ and \begin{equation}
    \rho=\alpha\, T^{\frac{1}{2}}_{\text{max}}.
\end{equation}
\end{prop}

\subsubsection{Linearization around the approximation solution} Let $w=\phi-\phi_{app}$ be the deviation of $\phi_{app}$ to the true solution $\phi_{app}$. In order to estimate $w$, we consider the linearized equation of
\begin{equation}
    \phi=\mathcal{F}(\phi)=\xi+\mathcal{T}(\phi,\phi).
\end{equation}
The linearized equation is a equation of $w$ given by
\begin{equation}\label{eq.eqw}
    w= Err(\xi)+Lw+B(w,w),
\end{equation}
where $Err(\xi)$, $Lw$, $B(w,w)$ are given by
\begin{equation}
\left\{
\begin{aligned}
    &Err=\mathcal{F}(\phi_{app})-\phi_{app},
    \\
    &Lw=2\mathcal{T}(\phi_{app},w),
    \\
    &B(w,w)=\mathcal{T}(w,w).
\end{aligned}\right.
\end{equation}

As explained in section \ref{sec.randmatintro}, to control $w$, we need operator norm bound $||L^K||_{X^p}$. Notice that $\phi_{app}=\sum_{l(T)\le N} \mathcal{J}_T$, so we know that $Lw$ is a sum of terms like $\mathcal{T}(\mathcal{J}_{T},w)$.

It suffices to get following operator norm bound for $w\rightarrow \mathcal{T}(\mathcal{J}_{T},w)$ in order to upper bound $L^K$.

\begin{prop}\label{prop.operatorupperbound}
Define $\rho=\alpha\, T^{\frac{1}{2}}_{\text{max}}$ as in Proposition \ref{prop.treetermsupperbound} and $\mathcal{P}_{T}$ to be the linear operator
\begin{equation}
\begin{split}
    w\rightarrow \mathcal{T}(\mathcal{J}_{T},w).
\end{split}
\end{equation} 
Then for any sequence of trees $\{T_1,\cdots,T_K\}$, we have $L$-certainly for any $\theta$ the operator bound
\begin{equation}\label{eq.operatornorm'}
    \left|\left|\prod_{j=1}^K\mathcal{P}_{T_j}\right|\right|_{L_t^{\infty}X^p\rightarrow L_t^{\infty}X^p}\le L^{O\left(1+\theta\sum_{j=1}^K l(T_j)\right)} \rho^{\sum_{j=1}^K l(T_j)}.
\end{equation}
for any $T_j$ with $l(T_j)\le N$. 

The the above inequality implies that 
\begin{equation}\label{eq.operatornorm}
    \left|\left|L^K\right|\right|_{L_t^{\infty}X^p\rightarrow L_t^{\infty}X^p}\le L^{O(1+K\theta)} \rho^{K}.
\end{equation}
\end{prop}

The proof of this proposition can be found in section \ref{sec.randommatrices}.
\subsection{Bound the error of the approximation}\label{sec.errorw}

Define $w=v-v_{app}$ to be the approximation error. In this section we prove the following theorem that gives an upper bound of $w$, assuming Proposition \ref{prop.treetermsupperbound} and Proposition \ref{prop.operatorupperbound} in previous section.

\begin{thm}\label{th.app}
Let $w=\phi-\phi_{app}$. Given any $M\gg 1$, there exists $N$ such that if $\phi_{app}$ is the $N$-th order approximate solution, then $\sup_{t\le T_{\text{max}}}||w(t)||_{X^p}\lesssim L^{-M}$ $L$-certainly (with probability $\geq 1-Ce^{-CL^\theta}$).
\end{thm}

\begin{proof}
If for some $C$ sufficiently large we can show that

\begin{equation}\label{eq.claimw}
    \sup _{t\le T}||w(t)||_{X^p}\le CL^{-M}\textit{  implies that } \sup _{t\le T}||w(t)||_{X^p}< CL^{-M},
\end{equation}
for all $T\le T_{\text{max}}$, then we finish the proof of this theorem. 

Here is the explanation. \eqref{eq.claimw} implies that if we define the set $A=\{T: \sup _{t\le T}||w||_{X^p}\le CL^{-M}\}$ then the set equals to $\{T: \sup _{t\le T}||w||_{H^s}< CL^{-M}\}$ which is open. The original definition $A=\{T: \sup _{t\le T}||w||_{X^p}\le CL^{-M}\}$ implies that this set is also closed. It is nonempty because $||w(0)||_{X^p}=0$ implies that $0\in A$. Therefore, $A$ is open, closed and nonempty in $[0,T_{\text{max}}]$, so $A=[0,T_{\text{max}}]$ which implies that the theorem.

Now we prove (\ref{eq.claimw}). By \eqref{eq.eqw}, 
\begin{equation}
    w-Lw= Err(\xi)+B(w,w).
\end{equation}
By Neumann series we have
\begin{equation}\label{eq.maintheoremeq1}
    \begin{split}
        w=& (1-L)^{-1}(Err(\xi)+B(w,w))
        \\
        =&(1-L^K)^{-1}(1+L+\cdots+L^{K-1})(Err(\xi)+B(w,w)).
    \end{split}
\end{equation}
Assume that the constant in $O(1+K\theta)$ in \eqref{eq.operatornorm} is $C_{norm}$. Since $T_{\text{max}}\le L^{-\varepsilon} \alpha^{-2}$, we know that $\rho=\alpha\, T^{\frac{1}{2}}_{\text{max}}\lesssim L^{-\varepsilon}$. Take $\theta\le C_{norm}\varepsilon/2$ and $K\gg \frac{C_{norm}}{\varepsilon}$, then $\left|\left|L^K\right|\right|_{L^{\infty}X^p\rightarrow L^{\infty}X^p}\le L^{C_{norm}(1+K\theta)}\rho^K\lesssim L^{C_{norm}(1+K\theta)} L^{-K\varepsilon}\ll 1$, so we get $\left|\left|L^K\right|\right|_{L^{\infty}X^p\rightarrow L^{\infty}X^p}\ll 1$ and thus $\left|\left|(1-L^K)^{-1}\right|\right|_{L^{\infty}X^p\rightarrow L^{\infty}X^p}\lesssim 1$. 

By \eqref{eq.maintheoremeq1}, we get
\begin{equation}
    ||w(t)||_{X^p}\lesssim \sum_{j=1}^K||L^j(Err(\xi))||_{X^p}+ ||(1+L+\cdots+L^{K-1})(B(w,w))||_{X^p}
\end{equation}

By \eqref{eq.approxerror}, we know that $Err$ is a sum of tree terms $\sum_{T\in \mathcal{T}_{>N}^*} \mathcal{J}_T$ of order $\ge N$. Since $L=\sum_{1\le l(T)\le N} \mathcal{P}_{T}$, we know that $L^j(Err(\xi))$ is a sum of terms like $\mathcal{P}_{T_1}\circ\cdots\circ\mathcal{P}_{T_{j}}(\mathcal{J}_T)$ which by \eqref{eq.operatoreqsimpleJ_T} equals to $\mathcal{J}_{T_1\circ\cdots\circ T_{j}\circ T}$. By Proposition \ref{prop.treetermsupperbound}, we get $||\mathcal{J}_{T_1\circ\cdots\circ T_{j}\circ T}||_{X^p}\lesssim (L^{O(\theta)} \rho)^{l(T_1\circ\cdots\circ T_{j}\circ T)}\lesssim L^{O(l(T)\theta)} \rho^{l(T)}$. Since $\rho=\alpha\, T^{\frac{1}{2}}_{\text{max}}\lesssim L^{-\varepsilon}$ and $l(T)>N$ in the sum of $Err$, we get 
\begin{equation}\label{eq.maintheoremeq2}
    \sum_{j=1}^K||L^j(Err(\xi))||_{X^p}\lesssim L^{O(N\theta)} \rho^{N}\lesssim L^{O(N\theta)} L^{-N\varepsilon}\ll L^{-M}
\end{equation}
if we take $N\gg M/\varepsilon$ and $\theta\ll \varepsilon$.

Taking $K=1$ in \eqref{eq.operatornorm'}, we know that $\left|\left|L\right|\right|_{L^{\infty}X^p\rightarrow L^{\infty}X^p}\le L^{O(1)}$, so $\left|\left|L^j\right|\right|_{L^{\infty}X^p\rightarrow L^{\infty}X^p}\le L^{O(K)}$ if $j\le K$. Taking $M\gg K$, by \eqref{eq.claimw}, we have $\sup _{t\le T}||w(t)||_{X^p}\le CL^{-M}$. Therefore, we have 
\begin{equation}\label{eq.maintheoremeq3}
    ||(1+L+\cdots+L^{K-1})(B(w,w))||_{X^p}\lesssim L^{O(K)} L^d ||w||^2_{X^p}\lesssim L^{O(K)} L^{O(1)} L^{-2M}\ll L^{-M}.
\end{equation}

Combining \eqref{eq.maintheoremeq2} and \eqref{eq.maintheoremeq3}, we prove $\sup _{t\le T}||w(t)||_{X^p}\ll L^{-M} < CL^{-M}$ from the assumption that $\sup _{t\le T}||w(t)||_{X^p}\le CL^{-M}$. We thus complete the proof of Theorem \ref{th.app}.
\end{proof}

\subsection{Proof of the main theorem}\label{sec.proofmain} In this section, we prove Theorem \ref{th.main}.

\begin{proof}[Proof of Theorem \ref{th.main}] \textbf{Step 1.} ($\mathbb E |\widehat \psi(t, k)|^2$ is close to $\mathbb E |\psi_{app,k}|^2$) By Theorem \ref{th.app}, we know that when $t\le T_{\text{max}}= L^{-\varepsilon}\alpha^{-2}$, we have $||w||_{X^p}\le L^{-M}$ with $L$-certainly.

The above inequality is equivalent to $\sup_k\, |\langle k \rangle^s w_k|\le CL^{-M}$. Remember that $w:=\psi-\psi_{app}$, so $L$-certainly we have the following estimate
\begin{equation}\label{eq.psikminusxik}
    \sup_k\, \langle k \rangle^s |\psi_k-\psi_{app,k}|\le CL^{-M}
\end{equation}

Denote by $A$ the event that the above estimate is true, then $\mathbb E |\widehat \psi(t, k)|^2=\mathbb E (|\psi_k|^2 1_{A})+\mathbb E (|\psi_k|^2 1_{A^c})$. $L$-certainty implies that $\mathbb P(A^c) \lesssim e^{-CL^{\theta}}$. Since $||\psi||_{L^2}$ is conservative and $|\psi_k|^2\le L^{d/2} ||\psi||_{L^2}\le L^{d/2}$, we know that $\mathbb E (|\psi_k|^2 1_{A^c})\lesssim L^{d/2} e^{-CL^{\theta}}= O(L^{-M})$. Therefore, $\mathbb E |\widehat \psi(t, k)|^2=\mathbb E (|\psi_k|^2 1_{A})+O(L^{-M})$. Since we also have $\mathbb E |\psi_{app,k}|^2=\mathbb E (|\psi_{app,k}|^2 1_{A})+O(L^{-M})$, we conclude that
\begin{equation}
    \mathbb E |\widehat \psi(t, k)|^2=\mathbb E |\psi_{app,k}|^2+\mathbb E ((|\psi_k|^2-|\psi_{app,k}|^2)1_{A})+O(L^{-M})
\end{equation}

By (\ref{eq.psikminusxik}), $\mathbb E ((|\psi_k|^2-|\psi_{app,k}|^2)1_{A})=O(L^{-M})$. We may conclude that 
\begin{equation}
    \mathbb E |\widehat \psi(t, k)|^2=\mathbb E |\psi_{app,k}|^2+O(L^{-M}).
\end{equation}
This suggests that we may get the approximation of $\mathbb E |\widehat \psi(t, k)|^2$ by calculating $\mathbb E |\psi_{app,k}|^2$.

\textbf{Step 2.} (Expansion of $\mathbb E |\psi_{app,k}|^2$)
By \eqref{eq.approxsol}, we know that 
\begin{equation}
    \phi_{app}=\sum_{l(T)\le N} \mathcal{J}_T
\end{equation}

Define 
\begin{equation}\label{eq.n(j)}
    n^{(j)}(k)\defeq \sum_{l(T)+l(T')=j} \mathbb E \mathcal{J}_{T,k}\overline{\mathcal{J}_{T',k}}
\end{equation}
then Proposition \ref{prop.treetermsupperbound} or \ref{prop.treetermsvariance} gives upper bounds of $n^{(j)}(k)$ , which proves (1) and \eqref{eq.n(j)estimate} of Theorem \ref{th.main}.

\textbf{Step 3.} (Asymptotics of $n^{(1)}(k)$) The only thing left in Theorem \ref{th.main} is \eqref{eq.n1}. This is a corollary of Proposition \ref{prop.mainterms}.
\end{proof}

\section{Lattice points counting and convergence results}
In this section, we prove Proposition \ref{prop.treetermsupperbound} and \ref{prop.operatorupperbound} which gives upper bounds for tree terms $\mathcal{J}_{T,k}$ and the linearization operator $\mathcal{P}_T$. As explained before, these results are crucial in the proof of the main theorem. The proof of is divided into several steps.

In section \ref{sec.refexp}, we calculate the coefficients of $\mathcal{J}_{T,k}$ as polynomials of Gaussian random variables.

In section \ref{sec.uppcoef}, we obtain upper bounds for the coefficients of these Gaussian polynomials.

Large deviation theory suggests that an upper bound of an Gaussian polynomial can be derived from an upper bound of its expectation and variance.

In section \ref{sec.coupwick}, we introduce the concept of couples which is a graphical method of calculating the expectation of Gaussian polynomials.

In section \ref{sec.numbertheory}, we use couple to establish an lattice points counting result.

In section \ref{sec.treetermsupperbound} and section \ref{sec.randommatrices}, we apply the lattice points counting result to derive upper bounds for  $\mathcal{J}_{T,k}$ and $\mathcal{P}_T$ respectively. This finishes the proof of Proposition \ref{prop.treetermsupperbound} and \ref{prop.operatorupperbound} and therefore the proof of the main theorem.

\subsection{Refined expression of coefficients}\label{sec.refexp} From \eqref{eq.treeterm}, it is easy to show that $\mathcal{J}_{T,k}$ are polynomials of $\xi$. In this section, we calculate the coefficients of $\mathcal{J}_{T,k}$ using the definition \eqref{eq.treeterm} of them.

Notice that all non-leaf nodes other than the root in a tree have degree $3$. For convenience, we add a new edge, called \underline{leg} $\mathfrak{l}$, to the root node, which makes the root also of degree $3$. This process is illustrated by Figure \ref{fig.leg}.
    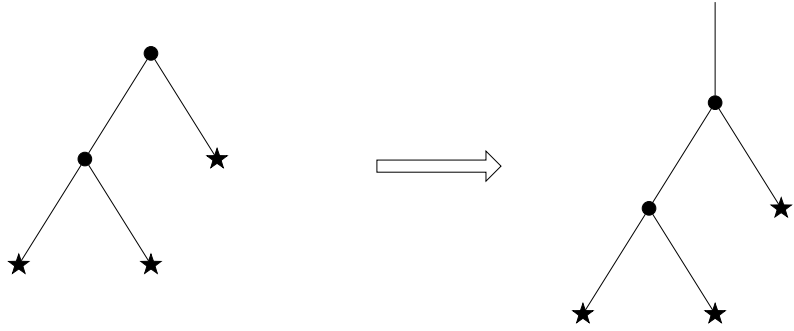
\begin{figure}[H]
    \centering
    \scalebox{0.5}{
    \begin{tikzpicture}[level distance=80pt, sibling distance=100pt]
        \node[fillcirc] {} 
            child {node[fillcirc]  {}
                child {node[fillstar] {}}
                child {node[fillstar] {}}
                }
            child {node[fillstar] {}};
        \node[draw, single arrow,
              minimum height=33mm, minimum width=8mm,
              single arrow head extend=2mm,
              anchor=west, rotate=0] at (6,-3) {};  
        \node[] at (15,1.5) {} 
            child {node[fillcirc] {} 
                child {node[fillcirc]  {}
                    child {node[fillstar] {}}
                    child {node[fillstar] {}}
                }
                child {node[fillstar] {}}
            };
    \end{tikzpicture}
    }
        \caption{Adding a leg to a tree}
        \label{fig.leg}
    \end{figure}

We also need following concepts about trees.
\begin{defn}\label{def.treemore}
\begin{enumerate}
    \item \textbf{Nodes and children of an edge:} As in Figure \ref{fig.childrenofedge}, let $\mathfrak{n}_{u}$ and $\mathfrak{n}_{l}$ be two endpoints of an edge $\mathfrak{e}$ and assume that $\mathfrak{n}_{l}$ is a children of $\mathfrak{n}_{u}$. We define $\mathfrak{n}_{u}$ (resp. $\mathfrak{n}_{l}$) to the \underline{upper node} (resp. \underline{lower node}) of $\mathfrak{e}$. Let $\mathfrak{n}_1$, $\mathfrak{n}_2$ be the two children of $\mathfrak{n}_{l}$ and let $\mathfrak{e}_1$(resp. $\mathfrak{e}_2$) be the edge between two nodes $\mathfrak{n}_{l}$ and $\mathfrak{n}_1$(resp. $\mathfrak{n}_2$). $\mathfrak{e}_1$, $\mathfrak{e}_2$ are defined to be the two \underline{children edges} of $\mathfrak{e}$. 
    \begin{figure}[H]
    \centering
    \scalebox{0.5}{
    \begin{tikzpicture}[level distance=80pt, sibling distance=100pt]
        \node[scale=2.0] at (13.8,-2.7) {$\mathfrak{e}$};
        \node[scale=2.0] at (14.3,-1.5) {$\mathfrak{n}_{u}$};
        \node[scale=2.0] at (12.6,-4) {$\mathfrak{n}_{l}$};
        \node[scale=2.0] at (12,-5.5) {$\mathfrak{e}_1$};
        \node[scale=2.0] at (14.5,-5.5) {$\mathfrak{e}_2$};
        \node[scale=2.0] at (11.5,-7.5) {$\mathfrak{n}_1$};
        \node[scale=2.0] at (15.2,-7.5) {$\mathfrak{n}_2$};
        \node[] at (15,1.5) {} 
            child {node[fillcirc] {} 
                child {node[fillcirc]  {}
                    child {node[fillstar] {}}
                    child {node[fillstar] {}}
                }
                child {node[fillstar] {}}
            };
    \end{tikzpicture}
    }
        \caption{Children $\mathfrak{e}_1$, $\mathfrak{e}_2$ of an edge $\mathfrak{e}$}
        \label{fig.childrenofedge}
    \end{figure}
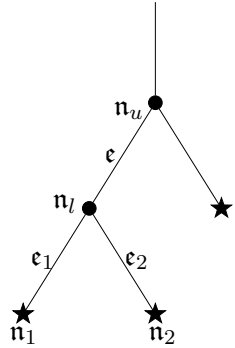
    \item \textbf{Direction of an edge:} As in Figure \ref{fig.orientation}, each edge $\mathfrak{e}$ is  assigned with a \underline{direction}. This concept is mostly used to decide the value of variables $\iota_{\mathfrak{e}}\in\{\pm\}$ that will be defined later. Although it can be shown that the final result does not depend on the choices of direction of each edge, for definiteness, we assign downward direction to each edge. The orientation in Figure \ref{fig.orientation} is one example of this choices.
    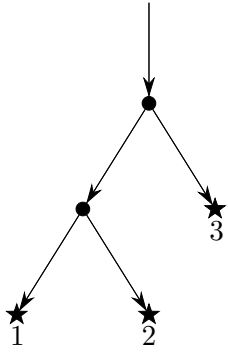
\begin{figure}[H]
    \centering
    \scalebox{0.5}{
    \begin{tikzpicture}[level distance=80pt, sibling distance=100pt]
        \node[scale=2.0] at (11.5,-7.5) {$1$};
        \node[scale=2.0] at (15.0,-7.5) {$2$};
        \node[scale=2.0] at (16.8,-4.7) {$3$};
        \node[] at (15,1.5) (1) {} 
            child {node[fillcirc] (2) {} 
                child {node[fillcirc] (3) {}
                    child {node[fillstar] (4) {}}
                    child {node[fillstar] (5) {}}
                }
                child {node[fillstar] (6) {}}
            };
        \draw[-{Stealth[length=5mm, width=3mm]}] (1) -- (2);
        \draw[-{Stealth[length=5mm, width=3mm]}] (2) -- (3);
        \draw[-{Stealth[length=5mm, width=3mm]}] (2) -- (6);
        \draw[-{Stealth[length=5mm, width=3mm]}] (3) -- (4);
        \draw[-{Stealth[length=5mm, width=3mm]}] (3) -- (5);
    \end{tikzpicture}
    }
        \caption{Children $\mathfrak{e}_1$, $\mathfrak{e}_2$ of an edge $\mathfrak{e}$}
        \label{fig.orientation}
    \end{figure}
    \item \textbf{Labelling of leaves:} As in Figure \ref{fig.orientation}, each leaf is labelled by $1$, $2$, $\cdots$, $l(T)+1$ from left to right. An edge pointing to a leaf $\mathfrak{n}$ is also labelled by $j$ if $\mathfrak{n}$ is labelled by $j$.
\end{enumerate}

\end{defn}

Now we calculate the coefficients of $\mathcal{J}_{T,k}$.

\begin{lem}\label{lem.treeterms} Given a tree $T$ of depth $l=l(T)$, denote by $T_{\text{in}}$ the tree formed by all non-leaf nodes $\mathfrak{n}$, then associate each node $\mathfrak{n}\in T_{\text{in}}$ and edge $\mathfrak{l}\in T$ with a variable $t_{\mathfrak{n}}$ and $k_{\mathfrak{l}}$ respectively. Given a labelling of all leaves by $1$, $2$, $\cdots$, $l+1$, we identify $k_{\mathfrak{e}}$ with $k_j$ if $\mathfrak{e}$ is connected to a leaf labelled by $j$. Given a node $\mathfrak{n}$, let $\mathfrak{e}_1$, $\mathfrak{e}_2$ and $\mathfrak{e}$ be the three edges from and pointing to $\mathfrak{n}$ ($\mathfrak{e}$ is the parent of $\mathfrak{e}_1$ and $\mathfrak{e}_2$), $\mathfrak{n}_1$, $\mathfrak{n}_2$ be children of $\mathfrak{n}$ and $\hat{\mathfrak{n}}$ be the parent of $\mathfrak{n}$. 

Let $\mathcal{J}_T$ be terms defined in Definition \ref{def.treeterms}, then their Fourier coefficients $\mathcal{J}_{T,k}$ are degree $l$ polynomials of $\xi$ given by the following formula

\begin{equation}\label{eq.coefterm}
\mathcal{J}_{T,k}=\left(\frac{i\lambda}{L^{d}}\right)^l\sum_{k_1,\, k_2,\, \cdots,\, k_{l+1}} H^T_{k_1\cdots k_{l+1}}  \xi_{k_1}\xi_{k_2}\cdots\xi_{k_{l+1}}   
\end{equation}
where $H^T_{k_1\cdots k_{l+1}}$ is given by
\begin{equation}\label{eq.coef}
H^T_{k_1\cdots k_{l+1}}=\int_{\cup_{\mathfrak{n}\in T_{\text{in}}} A_{\mathfrak{n}}} e^{\sum_{\mathfrak{n}\in T_{\text{in}}} it_{\mathfrak{n}}\Omega_{\mathfrak{n}}-\nu(t_{\widehat{\mathfrak{n}}}-t_{\mathfrak{n}})|k_{\mathfrak{e}}|^2} \prod_{\mathfrak{n}\in T_{\text{in}}} dt_{\mathfrak{n}} %\prod_{\mathfrak{n}\in T_{\text{in}} %\textit{with children $\mathfrak{n}_1$, $\mathfrak{n}_2$, $\mathfrak{n}_3$ }}
\ \delta_{\cap_{\mathfrak{n}\in T_{\text{in}}} \{S_{\mathfrak{n}}=0\}}\ \prod_{\mathfrak{e}\in T_{\text{in}}}\iota_{\mathfrak{e}}k_{\mathfrak{e},x} ,
\end{equation}
and $\iota$, $A_{\mathfrak{n}}$, $S_{\mathfrak{n}}$, $\Omega_{\mathfrak{n}}$ are defined by 
\begin{equation}\label{eq.iotadef}
    \iota_{\mathfrak{e}}=\begin{cases}
        +1 \qquad \textit{if $\mathfrak{e}$ pointing inwards to $\mathfrak{n}$}
        \\
        -1 \qquad  \textit{if $\mathfrak{e}$ pointing outwards from $\mathfrak{n}$}
    \end{cases}
\end{equation}
\begin{equation}
    A_{\mathfrak{n}}=
        \begin{cases}
            \{t_{\mathfrak{n}_1},\, t_{\mathfrak{n}_2},\, t_{\mathfrak{n}_3}\le t_{\mathfrak{n}}\} \qquad \textit{if $\mathfrak{n}\ne$ the root $\mathfrak{r}$}
            \\
            \{t_{\mathfrak{r}}\le t\} \qquad\qquad\qquad\  \textit{if $\mathfrak{n}= \mathfrak{r}$ }
        \end{cases}
\end{equation}
\begin{equation}\label{eq.defnS_n}
    S_{\mathfrak{n}}=\iota_{\mathfrak{e}_1}k_{\mathfrak{e}_1}+\iota_{\mathfrak{e}_2}k_{\mathfrak{e}_2}+\iota_{\mathfrak{e}}k_{\mathfrak{e}}
\end{equation}
\begin{equation}
    \Omega_{\mathfrak{n}}=\iota_{\mathfrak{e}_1}\Lambda_{k_{\mathfrak{e}_1}}+\iota_{\mathfrak{e}_2}\Lambda_{k_{\mathfrak{e}_2}}+\iota_{\mathfrak{e}}\Lambda_{k_{\mathfrak{e}}}
\end{equation}
For root node $\mathfrak{r}$, we impose the constrain that $k_{\mathfrak{r}}=k$ and $t_{\widehat{\mathfrak{r}}}=t$ (notice that $\mathfrak{r}$ does not have a parent so $\widehat{\mathfrak{r}}$ is not well defined). 
\end{lem}

\begin{proof}
We can check that $\mathcal{J}_T$ defined by \eqref{eq.coefterm} and \eqref{eq.coef} satisfies the recursive formula \eqref{eq.treeterm} by a direct substitution, so they are the unique solution of that recursive formula, and this proves Lemma \ref{lem.treeterms}.
\end{proof}

\subsection{An upper bound of coefficients in expansion series}\label{sec.uppcoef} In this section, we derive an upper bound for coefficients $H^T_{k_1\cdots k_{l+1}}$.

Notice that in \eqref{eq.coefterm}, $H^T_{k_1\cdots k_{l+1}}$ are integral of some oscillatory functions. An upper bound can be derived by the standard integration by parts arguments.

Associate each $\mathfrak{n}\in T_{\text{in}}$ with two variables $a_{\mathfrak{n}}$, $b_{\mathfrak{n}}$. Then we define
\begin{equation}\label{eq.defF_T}
F_{T}(t,\{a_{\mathfrak{n}}\}_{\mathfrak{n}\in T_{\text{in}}},\{b_{\mathfrak{n}}\}_{\mathfrak{n}\in T_{\text{in}}})=\int_{\cup_{\mathfrak{n}\in T_{\text{in}}} A_{\mathfrak{n}}} e^{\sum_{\mathfrak{n}\in T_{\text{in}}} it_{\mathfrak{n}} a_{\mathfrak{n}} - \nu(t_{\widehat{\mathfrak{n}}}-t_{\mathfrak{n}})b_{\mathfrak{n}}} \prod_{\mathfrak{n}\in T_{\text{in}}} dt_{\mathfrak{n}} 
\end{equation}

\begin{lem}\label{lem.boundcoef'}
We have the following upper bound for $F_{T}(t,\{a_{\mathfrak{n}}\}_{\mathfrak{n}},\{b_{\mathfrak{n}}\}_{\mathfrak{n}})$,
\begin{equation}\label{eq.boundcoef'}
    \sup_{\{b_{\mathfrak{n}}\}_{\mathfrak{n}}\lesssim 1} |F_{T}(t,\{a_{\mathfrak{n}}\}_{\mathfrak{n}},\{b_{\mathfrak{n}}\}_{\mathfrak{n}})|\lesssim \sum_{\{d_{\mathfrak{n}}\}_{\mathfrak{n}\in T_{\text{in}}}\in\{0,1\}^{l(T)}}\prod_{\mathfrak{n}\in T_{\text{in}}}\frac{1}{|q_{\mathfrak{n}}|+T^{-1}_{\text{max}}}.
\end{equation}
Fix a sequence $\{d_{\mathfrak{n}}\}_{\mathfrak{n}\in T_{\text{in}}}$ whose elements $d_{\mathfrak{n}}$ takes boolean values $\{0,1\}$. We define the two sequences $\{q_{\mathfrak{n}}\}_{\mathfrak{n}\in T_{\text{in}}}$, $\{r_{\mathfrak{n}}\}_{\mathfrak{n}\in T_{\text{in}}}$ by following recursive formula
\begin{equation}\label{eq.q_n'}
    q_{\mathfrak{n}}=
    \begin{cases}
    a_{\mathfrak{r}}, \qquad\qquad \textit{ if $\mathfrak{n}=$ the root $\mathfrak{r}$.}
    \\
    a_{\mathfrak{n}}+d_{\mathfrak{n}}q_{\mathfrak{n}'},\ \ \textit{ if $\mathfrak{n}\neq\mathfrak{r}$ and $\mathfrak{n}'$ is the parent of $\mathfrak{n}$.}
    \end{cases}
\end{equation}
\begin{equation}\label{eq.r_n'}
    r_{\mathfrak{n}}=
    \begin{cases}
    b_{\mathfrak{r}}, \qquad\qquad \textit{ if $\mathfrak{n}=$ the root $\mathfrak{r}$.}
    \\
    b_{\mathfrak{n}}+d_{\mathfrak{n}}q_{\mathfrak{n}'},\ \ \textit{ if $\mathfrak{n}\neq\mathfrak{r}$ and $\mathfrak{n}'$ is the parent of $\mathfrak{n}$.}
    \end{cases}
\end{equation}

% \begin{equation}
% Q=\{\{q_{\mathfrak{n}}\}_{\mathfrak{n}\in T_{\text{in}}}:q_{\mathfrak{n}}= q_{\mathfrak{n}}(),\ \forall\mathfrak{n}\in T_{\text{in}}\}    
% \end{equation}
% and the set $Q_{\mathfrak{n}}$ is defined inductively by 
% \begin{equation}\label{eq.treeterm}
%     Q_{\mathfrak{n}}=
%     \begin{cases}
%     \{a_{\mathfrak{r}}\}, \qquad\qquad\qquad\qquad \textit{ if $\mathfrak{n}=$ the root $\mathfrak{r}$.}
%     \\
%     \mathcal{T}_1(\mathcal{J}_{T_{\mathfrak{n}_1}}, \mathcal{J}_{T_{\mathfrak{n}_2}}, \mathcal{J}_{T_{\mathfrak{n}_3}}), \textit{ if $\mathfrak{n}\neq\mathfrak{r}$.}
%     \end{cases}
% \end{equation}
\end{lem}
\begin{proof} The lemma is proved by induction.

For a tree $T$ contains only one node $\mathfrak{r}$, $F_{T}=1$ and \eqref{eq.boundcoef'} is obviously true.

Assume that \eqref{eq.boundcoef'} is true for trees with $\le n-1$ nodes. We prove the $n$ nodes case. 

For general $T$, let $T_1$, $T_2$ be the two subtrees and $\mathfrak{n}_1$, $\mathfrak{n}_2$ be the two children of the root $\mathfrak{r}$, then by the definition of $F_T$\eqref{eq.defF_T}, we get
\begin{equation}\label{eq.lemboundcoef'1}
\begin{split}
    F_{T}(t)=&\int_{\cup_{\mathfrak{n}\in T_{\text{in}}} A_{\mathfrak{n}}} e^{\sum_{\mathfrak{n}\in T_{\text{in}}}it_{\mathfrak{n}} a_{\mathfrak{n}} - \nu(t_{\widehat{\mathfrak{n}}}-t_{\mathfrak{n}})b_{\mathfrak{n}}} \prod_{\mathfrak{n}\in T_{\text{in}}} dt_{\mathfrak{n}}    
    \\
    =&\int_{\cup_{\mathfrak{n}\in T_{\text{in}}} A_{\mathfrak{n}}}e^{it_{\mathfrak{r}} a_{\mathfrak{r}} - \nu(t-t_{\mathfrak{r}})b_{\mathfrak{r}}} e^{\sum_{\mathfrak{n}\in T_{\text{in},1}\cup T_{\text{in},2}} it_{\mathfrak{n}} a_{\mathfrak{n}} - \nu(t_{\widehat{\mathfrak{n}}}-t_{\mathfrak{n}})b_{\mathfrak{n}}}  \left(dt_{\mathfrak{r}}\prod_{j=1}^2\prod_{\mathfrak{n}\in T_{\text{in},j}}dt_{\mathfrak{n}}  \right)
    \\
    =&\int_{\cup_{\mathfrak{n}\in T_{\text{in}}} A_{\mathfrak{n}}}e^{it_{\mathfrak{r}}(a_{\mathfrak{r}}+T^{-1}_{\text{max}}\, \text{sgn}(a_{\mathfrak{r}}))- \nu(t-t_{\mathfrak{r}})b_{\mathfrak{r}}} e^{-iT^{-1}_{\text{max}}t_{\mathfrak{r}} \text{sgn}(a_{\mathfrak{r}})} e^{\sum_{\mathfrak{n}\in T_{\text{in},1}\cup T_{\text{in},2}} it_{\mathfrak{n}} a_{\mathfrak{n}} - \nu(t_{\widehat{\mathfrak{n}}}-t_{\mathfrak{n}})b_{\mathfrak{n}}}  \left(dt_{\mathfrak{r}}\prod_{j=1}^2\prod_{\mathfrak{n}\in T_{\text{in},j}}dt_{\mathfrak{n}}  \right)
\end{split}
\end{equation}

We do integration by parts in the above integrals using Stokes formula. Notice that for $t_{\mathfrak{r}}$, there are three inequality constrains, $t_{\mathfrak{r}}\le t$ and $t_{\mathfrak{r}}\ge t_{\mathfrak{n}_1},t_{\mathfrak{n}_2}$.

\begin{equation}\label{eq.lemboundcoefexpand}
\begin{split}
    F_{T}(t)=\frac{1}{ia_{\mathfrak{r}}+iT^{-1}_{\text{max}} \text{sgn}(a_{\mathfrak{r}})+\nu b_{\mathfrak{r}} }&\int_{\cup_{\mathfrak{n}\in T_{\text{in}}} A_{\mathfrak{n}}} \frac{d}{dt_{\mathfrak{r}}}e^{it_{\mathfrak{r}}(a_{\mathfrak{r}}+T^{-1}_{\text{max}}\, \text{sgn}(a_{\mathfrak{r}}))- \nu(t-t_{\mathfrak{r}})b_{\mathfrak{r}}}  
    \\
    &e^{-iT^{-1}_{\text{max}}t_{\mathfrak{r}} \text{sgn}(a_{\mathfrak{r}})} e^{\sum_{\mathfrak{n}\in T_{\text{in},1}\cup T_{\text{in},2}} it_{\mathfrak{n}} a_{\mathfrak{n}} - \nu(t_{\widehat{\mathfrak{n}}}-t_{\mathfrak{n}})b_{\mathfrak{n}}}  \left(dt_{\mathfrak{r}}\prod_{j=1}^2\prod_{\mathfrak{n}\in T_{\text{in},j}}dt_{\mathfrak{n}}  \right)
\end{split}
\end{equation}
\begin{flalign*}
\hspace{1.3cm}
=&\frac{1}{ia_{\mathfrak{r}}+iT^{-1}_{\text{max}} \text{sgn}(a_{\mathfrak{r}})+\nu b_{\mathfrak{r}} }\left(\int_{\cup_{\mathfrak{n}\in T_{\text{in}}} A_{\mathfrak{n}},\ t_{\mathfrak{r}}=t}-\int_{\cup_{\mathfrak{n}\in T_{\text{in}}} A_{\mathfrak{n}},\ t_{\mathfrak{r}}=t_{\mathfrak{n}_1}}-\int_{\cup_{\mathfrak{n}\in T_{\text{in}}} A_{\mathfrak{n}},\ t_{\mathfrak{r}}=t_{\mathfrak{n}_2}}\right) &&
\\
& e^{it_{\mathfrak{r}}(a_{\mathfrak{r}}+T^{-1}_{\text{max}}\, \text{sgn}(a_{\mathfrak{r}}))- \nu(t-t_{\mathfrak{r}})b_{\mathfrak{r}}} e^{-iT^{-1}_{\text{max}}t_{\mathfrak{r}} \text{sgn}(a_{\mathfrak{r}})} e^{\sum_{\mathfrak{n}\in T_{\text{in},1}\cup T_{\text{in},2}} it_{\mathfrak{n}} a_{\mathfrak{n}} - \nu(t_{\widehat{\mathfrak{n}}}-t_{\mathfrak{n}})b_{\mathfrak{n}}} \left(dt_{\mathfrak{r}}\prod_{j=1}^2\prod_{\mathfrak{n}\in T_{\text{in},j}}dt_{\mathfrak{n}}  \right) &&
\end{flalign*}
\begin{flalign*}
\hspace{1.3cm}
-&\frac{1}{ia_{\mathfrak{r}}+iT^{-1}_{\text{max}} \text{sgn}(a_{\mathfrak{r}})+\nu b_{\mathfrak{r}} }\int_{\cup_{\mathfrak{n}\in T_{\text{in}}} A_{\mathfrak{n}}}e^{it_{\mathfrak{r}}(a_{\mathfrak{r}}+T^{-1}_{\text{max}}\, \text{sgn}(a_{\mathfrak{r}}))- \nu(t-t_{\mathfrak{r}})b_{\mathfrak{r}}} &&
    \\
    &\qquad\qquad\qquad\qquad\qquad \frac{d}{dt_{\mathfrak{r}}}(e^{-iT^{-1}_{\text{max}}t_{\mathfrak{r}} \text{sgn}(a_{\mathfrak{r}})}) e^{\sum_{\mathfrak{n}\in T_{\text{in},1}\cup T_{\text{in},2}} it_{\mathfrak{n}} a_{\mathfrak{n}} - \nu(t_{\widehat{\mathfrak{n}}}-t_{\mathfrak{n}})b_{\mathfrak{n}}}  \left(dt_{\mathfrak{r}}\prod_{j=1}^2\prod_{\mathfrak{n}\in T_{\text{in},j}}dt_{\mathfrak{n}}  \right) &&
\end{flalign*}
\begin{flalign*}
\hspace{1.3cm}
= \frac{1}{ia_{\mathfrak{r}}+iT^{-1}_{\text{max}} \text{sgn}(a_{\mathfrak{r}})+\nu b_{\mathfrak{r}} }(F_{I}-F_{T^{(1)}}-F_{T^{(2)}}-F_{II}) &&
\end{flalign*}

Here $T^{(j)}$, $j=1,2$ are trees that is obtained by deleting the root $\mathfrak{r}$, adding edges connecting $\mathfrak{n}_j$ with another node and defining $\mathfrak{n}_j$ to be the new root. For $T^{(j)}$, we can define the term $F_{T^{(j)}}$ by \eqref{eq.defF_T}. It can be shown that $F_{T^{(j)}}$ defined in this way is the same as the $\int_{\cup_{\mathfrak{n}\in T_{\text{in}}} A_{\mathfrak{n}},\ t_{\mathfrak{r}}=t_{\mathfrak{n}_j}}$ term in the second equality of \eqref{eq.lemboundcoefexpand}, so the last equality of \eqref{eq.lemboundcoefexpand} is true. $F_{I}$ is the $\int_{\cup_{\mathfrak{n}\in T_{\text{in}}} A_{\mathfrak{n}},\ t_{\mathfrak{r}}=t}$ term and $F_{II}$ is the last term containing $\frac{d}{dt_{\mathfrak{r}}}$.

We can apply the induction assumption to $F_{T^{(j)}}$ and show that $\frac{1}{ia_{\mathfrak{r}}+iT^{-1}_{\text{max}} \text{sgn}(a_{\mathfrak{r}})+\nu b_{\mathfrak{r}} } F_{T^{(j)}}$ can be bounded by the right hand side of \eqref{eq.boundcoef'}.

A direct calculation gives that 
\begin{equation}
    F_{I}(t)=e^{it a_{\mathfrak{r}} } F_{T_1}(t)F_{T_2}(t).
\end{equation}
Then the induction assumption implies that $\frac{1}{ia_{\mathfrak{r}}+iT^{-1}_{\text{max}} \text{sgn}(a_{\mathfrak{r}})+\nu b_{\mathfrak{r}} } F_{I}$ can be bounded by the right hand side of \eqref{eq.boundcoef'}.

Another direct calculation gives that 
\begin{equation}
    F_{II}(t)=\int^t_0  e^{it_{\mathfrak{r}}(a_{\mathfrak{r}}+T^{-1}_{\text{max}}\, \text{sgn}(a_{\mathfrak{r}}))- \nu(t-t_{\mathfrak{r}})b_{\mathfrak{r}}} \frac{d}{dt_{\mathfrak{r}}}(e^{-iT^{-1}_{\text{max}}t_{\mathfrak{r}} \text{sgn}(a_{\mathfrak{r}})})  F_{T_1}(t_{\mathfrak{r}})F_{T_2}(t_{\mathfrak{r}}) dt_{\mathfrak{r}}.
\end{equation}
Apply the induction assumption
\begin{equation}
\begin{split}
    &\left| \frac{1}{ia_{\mathfrak{r}}+iT^{-1}_{\text{max}} \text{sgn}(a_{\mathfrak{r}})+\nu b_{\mathfrak{r}} } F_{II}(t)\right|
    \\
    \le& \frac{1}{|q_{\mathfrak{r}}|+T^{-1}_{\text{max}}}\prod_{j=1}^2\left(\sum_{\{d_{\mathfrak{n}}\}_{\mathfrak{n}\in T_{\text{in},j}}\in\{0,1\}^{l(T_j)}}\prod_{\mathfrak{n}\in T_{\text{in},j}}\frac{1}{|q_{\mathfrak{n}}|+T^{-1}_{\text{max}}}\right)
    \\
    \le& \sum_{\{d_{\mathfrak{n}}\}_{\mathfrak{n}\in T_{\text{in}}}\in\{0,1\}^{l(T)}}\prod_{\mathfrak{n}\in T_{\text{in}}}\frac{1}{|q_{\mathfrak{n}}|+T^{-1}_{\text{max}}}.
\end{split}
\end{equation}

Combining the bounds of $F_{I}$, $F_{T^{(1)}}$, $F_{T^{(2)}}$, $F_{II}$, we conclude that $F_T$ can be bounded by the right hand side of \eqref{eq.boundcoef'} and thus complete the proof of Lemma \ref{lem.boundcoef'}.
\end{proof}
 
% For general $T$, let $T_1$, $T_2$, $T_3$ be the three subtrees of the root, then we have the following recursive formula for $F_{T}$
% \begin{equation}\label{eq.lemboundcoefrecur}
%     F_{T}(t)=\int^t_0 e^{-i(t_{\mathfrak{r}}a_{\mathfrak{r}}+\alpha c(t_{\mathfrak{r}})b_{\mathfrak{r}})} F_{T_1}(t_{\mathfrak{r}})F_{T_2}(t_{\mathfrak{r}})F_{T_3}(t_{\mathfrak{r}}) dt_{\mathfrak{r}}.
% \end{equation}

A straight forward application of the above lemma gives following upper bound of the coefficients $H^T_{k_1\cdots k_{l+1}}$. 

\begin{lem}\label{lem.boundcoef}
We have the following upper bound for $H^T_{k_1\cdots k_{l+1}}$,
\begin{equation}\label{eq.boundcoef}
    |H^T_{k_1\cdots k_{l+1}}|\lesssim \sum_{\{d_{\mathfrak{n}}\}_{\mathfrak{n}\in T_{\text{in}}}\in\{0,1\}^{l(T)}}\prod_{\mathfrak{n}\in T_{\text{in}}}\frac{1}{|q_{\mathfrak{n}}|+T^{-1}_{\text{max}}}\ \prod_{\mathfrak{e}\in T_{\text{in}}}|k_{\mathfrak{e},x}|\ \delta_{\cap_{\mathfrak{n}\in T_{\text{in}}} \{S_{\mathfrak{n}}=0\}}.
\end{equation}
Fix a sequence $\{d_{\mathfrak{n}}\}_{\mathfrak{n}\in T_{\text{in}}}$ whose elements $d_{\mathfrak{n}}$ takes boolean values $\{0,1\}$. We define the two sequences $\{q_{\mathfrak{n}}\}_{\mathfrak{n}\in T_{\text{in}}}$, $\{r_{\mathfrak{n}}\}_{\mathfrak{n}\in T_{\text{in}}}$ by following recursive formula
\begin{equation}\label{eq.q_n}
    q_{\mathfrak{n}}=
    \begin{cases}
    \Omega_{\mathfrak{r}}, \qquad\qquad \textit{ if $\mathfrak{n}=$ the root $\mathfrak{r}$.}
    \\
    \Omega_{\mathfrak{n}}+d_{\mathfrak{n}}q_{\mathfrak{n}'},\ \ \textit{ if $\mathfrak{n}\neq\mathfrak{r}$ and $\mathfrak{n}'$ is the parent of $\mathfrak{n}$.}
    \end{cases}
\end{equation}
\begin{equation}\label{eq.r_n}
    r_{\mathfrak{n}}=
    \begin{cases}
    |k_{\mathfrak{r}}|^2, \qquad\qquad \textit{ if $\mathfrak{n}=$ the root $\mathfrak{r}$.}
    \\
    |k_{\mathfrak{n}}|^2+d_{\mathfrak{n}}q_{\mathfrak{n}'},\ \ \textit{ if $\mathfrak{n}\neq\mathfrak{r}$ and $\mathfrak{n}'$ is the parent of $\mathfrak{n}$.}
    \end{cases}
\end{equation}

\end{lem}
\begin{proof}
This is a direct corollary of Lemma \ref{eq.boundcoef'} if we take $a_{\mathfrak{n}}=\Omega_{\mathfrak{n}}$, $b_{\mathfrak{n}}=|k_{\mathfrak{e}}|^2$. 
\end{proof}

Lemma \ref{lem.boundcoef} suggests that the coefficients are small when $|q_{\mathfrak{n}}|\gg T^{-1}_{\text{max}}$. Therefore, in order to bound $\mathcal{J}_{T,k}$, we should count the lattice points on $|q_{\mathfrak{n}}|\lesssim T^{-1}_{\text{max}}$
\begin{equation}\label{eq.diophantineeq''}
    \{k_{\mathfrak{e}}\in \mathbb{Z}^d_L,\ |k_{\mathfrak{e}}|\lesssim 1,\ \forall \mathfrak{e}\in T: |q_{\mathfrak{n}}|\lesssim T^{-1}_{\text{max}},\ S_{\mathfrak{n}}=0,\ \forall \mathfrak{n}\in T.\ k_{\mathfrak{l}}=k\}
\end{equation}

By solving \eqref{eq.q_n}, we know that $\Omega_{\mathfrak{n}}$ is a linear combination of $q_{\mathfrak{n}}$, so there exist constants $c_{\mathfrak{n},\mathfrak{n}'}$ such that $\Omega_{\mathfrak{n}}=\sum_{\mathfrak{n}'}c_{\mathfrak{n},\mathfrak{n}'}q_{\mathfrak{n}'}$. Therefore, $|q_{\mathfrak{n}}|\lesssim T^{-1}_{\text{max}}$ implies that $|\Omega_{\mathfrak{n}}|\le\sum_{\mathfrak{n}'}|c_{\mathfrak{n},\mathfrak{n}'}q_{\mathfrak{n}'}|\lesssim T^{-1}_{\text{max}}$.

%For the sake of concreteness let's take $d_{\mathfrak{n}}=0$ for all nodes ${\mathfrak{n}}$, then $q_{\mathfrak{n}}=\Omega_{\mathfrak{n}}$ for all ${\mathfrak{n}}$ and
$|\Omega_{\mathfrak{n}}|\lesssim T^{-1}_{\text{max}}$ implies that \eqref{eq.diophantineeq''} is a subset of
\begin{equation}\label{eq.diophantineeq'}
    \{k_{\mathfrak{e}}\in \mathbb{Z}^d_L,\ |k_{\mathfrak{e}}|\lesssim 1,\ \forall \mathfrak{e}\in T: |\Omega_{\mathfrak{n}}|\lesssim T^{-1}_{\text{max}},\ S_{\mathfrak{n}}=0,\ \forall \mathfrak{n}\in T. \ k_{\mathfrak{l}}=k\}.
\end{equation}
To bound the number of elements of \eqref{eq.diophantineeq''}, we just need to do the same thing for \eqref{eq.diophantineeq'}.

\eqref{eq.diophantineeq'} can be read from the tree diagrams $T$. As in Figure \ref{fig.equations}, each edge corresponds to a variable $k_{\mathfrak{e}}$. The leg $\mathfrak{l}$ corresponds to equation $k_{\mathfrak{l}}=k$. Each node $\mathfrak{n}$ is connected with three edges $\mathfrak{e}_1$, $\mathfrak{e}_2$, $\mathfrak{e}$ whose corresponding variables $k_{\mathfrak{e}_1}$, $k_{\mathfrak{e}_2}$, $k_{\mathfrak{e}}$ satisfy the momentum conservation equation
\begin{equation}
\iota_{\mathfrak{e}_1}k_{\mathfrak{e}_1}+\iota_{\mathfrak{e}_2}k_{\mathfrak{e}_2}+\iota_{\mathfrak{e}}k_{\mathfrak{e}}=0
\end{equation}
and the energy conservation equation (if the node is decorated by $\bullet$)
\begin{equation}
    \iota_{\mathfrak{e}_1}\Lambda_{k_{\mathfrak{e}_1}}+\iota_{\mathfrak{e}_2}\Lambda_{k_{\mathfrak{e}_2}}+\iota_{\mathfrak{e}}\Lambda_{k_{\mathfrak{e}}} = O(T^{-1}_{\text{max}}).
\end{equation}

\begin{figure}[H]
    \centering
    \scalebox{0.5}{
    \begin{tikzpicture}[level distance=80pt, sibling distance=150pt]
        \node[scale=2.0] at (14.3,0) {$k_{\mathfrak{e}}$};
        \node[scale=2.0] at (14.5,-1.3) {$\mathfrak{n}$};
        \node[scale=2.0] at (13,-2.4) {$k_{\mathfrak{e}_1}$};
        \node[scale=2.0] at (17,-2.4) {$k_{\mathfrak{e}_2}$};
        \node[] at (15,1.5) (1) {} 
            child {node[fillcirc] (2) {} 
                child {node[fillstar] (3) {}
                }
                child {node[fillstar] (6) {}}
            };
        \draw[-{Stealth[length=5mm, width=3mm]}] (1) -- (2);
        \draw[-{Stealth[length=5mm, width=3mm]}] (2) -- (3);
        \draw[-{Stealth[length=5mm, width=3mm]}] (2) -- (6);
        \node[scale=2.0] at (15,-5.5) {$\iota_{\mathfrak{e}_1}=\iota_{\mathfrak{e}_2}=1, \iota_{\mathfrak{e}}=-1$};
        \node[scale=2.0] at (15,-6.5) {$k_{\mathfrak{e}_1}+k_{\mathfrak{e}_1}-k_{\mathfrak{e}}=0$};
        \node[scale=2.0] at (15,-7.5) {$\Lambda_{k_{\mathfrak{e}_1}}+\Lambda_{k_{\mathfrak{e}_2}}-\Lambda_{k_{\mathfrak{e}}}=O(T^{-1}_{\text{max}})$};
    \end{tikzpicture}
    }
        \caption{Equations of a node $\mathfrak{n}$}
        \label{fig.equations}
    \end{figure}
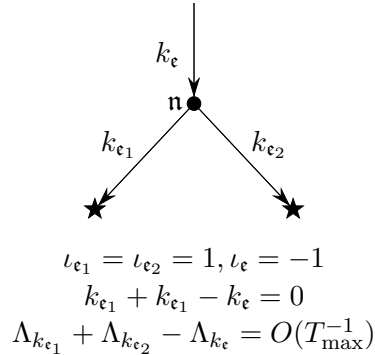

The goal of the next two sections is to count the number of solutions of a modified version of the above equation \eqref{eq.diophantineeq'}. 

\subsection{Couples and Wick theorem} In this section, we calculate $\mathbb{E}|\mathcal{J}_{T,k}|^2$ using Wick theorem. We also introduce another type of diagrams, the couple diagrams, to represent the result.

\label{sec.coupwick} By the upper bound in the last section, the coefficients $H^T_{k_1\cdots k_{l+1}}$ concentrate near the surface $q_{\mathfrak{n}}=0$, $\forall \mathfrak{n}$. But to get an upper bound of $\mathcal{J}_{T,k}$, we need upper bound of their variance $\mathbb{E}|\mathcal{J}_{T,k}|^2$. The coefficients of $\mathbb{E}|\mathcal{J}_{T,k}|^2$ also concentrate near a surface whose expression is similar to \eqref{eq.diophantineeq''}. 

Let's derive the expression of the coefficients of $\mathbb{E}|\mathcal{J}_{T,k}|^2$ and its concentration surface.

%$\mathbb{E}(\mathcal{J}_{T,k}\overline{\mathcal{J}_{T',k'}})$

%In this section, we derive the formula \eqref{eq.couples} for calculating expectations $\mathbb{E}(\mathcal{J}_{T,k}\overline{\mathcal{J}_{T',k'}})$ in terms of couples. 

By Lemma \ref{lem.treeterms}, we know that $\mathcal{J}_{T,k}$ is a polynomial of $\xi$ which are proportional to i.i.d Gaussians. Therefore, 
\begin{equation}\label{eq.termexp1}
\begin{split}
    \mathbb{E}|\mathcal{J}_{T,k}|^2=&\mathbb{E}(\mathcal{J}_{T,k}\overline{\mathcal{J}_{T,k}})=\left(\frac{\lambda}{L^{d}}\right)^{2l(T)}
    \sum_{k_1,\, k_2,\, \cdots,\, k_{l(T)+1}}\sum_{k'_1,\, k'_2,\, \cdots,\, k'_{l(T)+1}}
    \\[0.5em]
    & H^T_{k_1\cdots k_{l(T)+1}} \overline{H^{T}_{k'_1\cdots k'_{l(T)+1}}}  \mathbb{E}\Big(\xi_{k_1}\xi_{k_2}\cdots\xi_{k_{l(T)+1}}\xi_{k'_1}\xi_{k'_2}\cdots\xi_{k'_{l(T)+1}}\Big)
\end{split}
\end{equation}
% \begin{equation}\label{eq.termexp1}
% \begin{split}
%     \mathbb{E}(\mathcal{J}_{T,k}\overline{\mathcal{J}_{T',k'}})&=\left(\frac{-i\lambda^2}{L^{2d}}\right)^{l(T)+l(T')}
%     \sum_{k_1,\, k_2,\, \cdots,\, k_{l(T)+1}}\sum_{k'_1,\, k'_2,\, \cdots,\, k'_{2l(T')+1}}
%     \\[0.5em]
%     & H^T_{k_1\cdots k_{l(T)+1}} H^{T'}_{k'_1\cdots k'_{2l(T')+1}}   \mathbb{E}\Big([\xi_{k_1}\xi_{k_2}\cdots\xi_{k_{l(T)+1}}]_{R(T)}
%     [\xi_{k'_1}\xi_{k'_2}\cdots\xi_{k'_{2l(T')+1}}]_{R(T')}\Big)
% \end{split}
% \end{equation}

We just need to calculate 
\begin{equation}\label{eq.expectation''}
    \mathbb{E}\Big(\xi_{k_1}\xi_{k_2}\cdots\xi_{k_{l(T)+1}}
    \xi_{k'_1}\xi_{k'_2}\cdots\xi_{k'_{l(T)+1}}\Big).
\end{equation}
Notice that $\xi_k=\sqrt{n_{\textrm{in}}(k)} \, \eta_{k}(\omega)$ and $\eta_{k}$ are i.i.d Gaussians. We can apply the Wick theorem to calculate the above expectations. 

To introduce the Wick theorem, we need the following definition.

\begin{defn}
\begin{enumerate}
    \item \textbf{Pairing:} Suppose that we have a set $A=\{a_1,\cdots,a_{2m}\}$. A \underline{pairing} is a partition of $A=\{a_{i_1},a_{i_2}\}\cup\cdots\cup \{a_{i_{2m-1}},a_{i_{2m}}\}$ into $m$ disjoint subsets which have exactly two elements. Given a pairing $p$, elements $a_{i_{k}}$, $a_{i_{k'}}$ in the same subset of $p$ are called \underline{paired} with each other, which is denoted by $a_{i_{k}}\sim_{p} a_{i_{k'}}$.
    \item \textbf{$\mathcal{P}(A)$:} Denote by $\mathcal{P}(A)$ the set of all pairings of $A$.
\end{enumerate}

% Pairing: $\wick{\c i_1 \c i_{2} \c i_3 \c i_{4} \cdots \c i_{2m-1} \c i_{2m}}$ (representation of a pairing is unique if $i_1<i_3<\cdots<i_{2m-1}$).

% components of a pairing: $p_{k}=i_k$

\end{defn}

\begin{lem}[Wick theorem]\label{th.wick}
Let $\{\eta_k\}_{k\in\mathbb{Z}^d_L}$ be i.i.d complex Gaussian random variables with reflection symmetry (i.e. $\eta_{k}=\bar{\eta}_{-k}$). Let $\mathcal{P}$ be the set of all pairings of $\{k_1,k_2,\cdots,k_{2m}\}$, then
\begin{equation}
    %\mathbb{E}(\xi_{k_1}\cdots \xi_{k_{2m}})=\sum_{p\in \mathcal{P}} \prod_{i=1}^{m} \delta_{k_{p_{2i-1}}=k_{p_{2i}}}
    \mathbb{E}(\eta_{k_1}\cdots \eta_{k_{2m}})=\sum_{p\in \mathcal{P}}  \delta_{p}(k_1,\cdots,k_{2m}), 
\end{equation}
where 
\begin{equation}\label{eq.deltapairing}
\delta_{p}=\begin{cases}
1\qquad \textit{if $k_{i}=-k_{j}$ for all $k_{i}\sim_{p}k_{j}$,}
\\
0\qquad \textit{otherwise.}
\end{cases}
\end{equation}
\end{lem}
\begin{proof}
By Isserlis' theorem, for $X_1$, $X_2$, $\cdots$, $X_n$ zero-mean i.i.d Gaussian, we have 
\begin{equation}
    \mathbb{E} [X_1 X_2 \cdots X_n] = \sum_{p\in\mathcal{P}} \prod_{i\sim_{p}j} \mathbb{E} [X_i X_j]
\end{equation}
Here $\mathcal{P}$ is the set of pairings of $\{1,2,\cdots,n\}$.

Since $\mathbb{E} [\eta_{k_i}\eta_{k_j}]=\delta_{k_i=-k_j}$, take $X_1=\eta_{k_1}$, $\cdots$, $X_{2m}=\eta_{k_{2m}}$, then we can check that $\prod_{i\sim_{p}j} \mathbb{E} [X_i X_j]=\delta_{p}(k_1,\cdots,k_{2m})$. This finishes the proof of the Wick theorem.
\end{proof}

Applying Wick theorem to \eqref{eq.termexp1}, we get
\begin{equation}\label{eq.termexp'}
\begin{split}
    &\mathbb{E}|\mathcal{J}_{T,k}|^2=\left(\frac{\lambda}{L^{d}}\right)^{2l(T)}
    \sum_{p\in \mathcal{P}(\{k_1,\cdots, k_{l(T)+1}, k'_1,\cdots, k'_{l(T)+1}\})}
    \\[0.5em]
    & \underbrace{\sum_{k_1,\, k_2,\, \cdots,\, k_{l(T)+1}}\sum_{k'_1,\, k'_2,\, \cdots,\, k'_{l(T)+1}} H^T_{k_1\cdots k_{l(T)+1}} \overline{H^{T}_{k'_1\cdots k'_{l(T)+1}}} \delta_{p}(k_1,\cdots, k_{l(T)+1}, k'_1,\cdots, k'_{l(T)+1})\sqrt{n_{\textrm{in}}(k_1)}\cdots}_{Term(T, p)_k}.
\end{split}
\end{equation}
% \begin{equation}\label{eq.termexp1}
% \begin{split}
%     &\mathbb{E}(\mathcal{J}_{T,k}\overline{\mathcal{J}_{T',k'}})=\left(\frac{-i\lambda^2}{L^{2d}}\right)^{l(T)+l(T')}
%     \sum_{p\in \mathcal{P}(\{k_1,\cdots, k_{l(T)+1}, k'_1,\cdots, k'_{2l(T')+1}\})}
%     \\[0.5em]
%     & \underbrace{\sum_{k_1,\, k_2,\, \cdots,\, k_{l(T)+1}}\sum_{k'_1,\, k'_2,\, \cdots,\, k'_{2l(T')+1}} H^T_{k_1\cdots k_{l(T)+1}} H^{T'}_{k'_1\cdots k'_{2l(T')+1}} \delta_{p}(k_1,\cdots, k_{l(T)+1}, k'_1,\cdots, k'_{2l(T')+1})\sqrt{n_{\textrm{in}}(k_1)}\cdots}_{Term(T, T', p)}.
% \end{split}
% \end{equation}

We see that the correlation of two tree terms is a sum of smaller expressions $Term(T, p)$. By \eqref{eq.diophantineeq'}, the coefficients $H^T_{k_1\cdots k_{l(T)+1}} H^{T}_{k'_1\cdots k'_{l(T)+1}}$ of $Term(T, p)$ concentrate near the subset 

\begin{equation}\label{eq.diophantineequnpaired}
\{k_{\mathfrak{e}}, k_{\mathfrak{e}'}\in \mathbb{Z}^d_L,\ |k_{\mathfrak{e}}|, |k_{\mathfrak{e}'}|\lesssim 1,\ \forall \mathfrak{e},\mathfrak{e}'\in T:|\Omega_{\mathfrak{n}}|,|\Omega_{\mathfrak{n}'}|\lesssim T^{-1}_{\text{max}},\ S_{\mathfrak{n}}=S_{\mathfrak{n}'}=0\ \forall \mathfrak{n},\mathfrak{n}'\in T.\ k_{\mathfrak{l}}=-k_{\mathfrak{l}}'=k\}.
\end{equation}

The pairing $p$ in Wick theorem introduces new equations $k_{i}=-k'_{j}$ (defined in \eqref{eq.deltapairing}) and the coefficients $H^T_{k_1\cdots k_{l(T)+1}} H^{T}_{k'_1\cdots k'_{l(T)+1}} \delta_{p}$ concentrate near the subset 
\begin{equation}\label{eq.diophantineeqpaired}
\begin{split}
    \{k_{\mathfrak{e}}, k_{\mathfrak{e}'}\in \mathbb{Z}^d_L,\ &|k_{\mathfrak{e}}|, |k_{\mathfrak{e}'}|\lesssim 1,\ \forall \mathfrak{e},\mathfrak{e}'\in T: |\Omega_{\mathfrak{n}}|,|\Omega_{\mathfrak{n}'}|\lesssim T^{-1}_{\text{max}},\ S_{\mathfrak{n}}=S_{\mathfrak{n}'}=0\ \forall \mathfrak{n},\mathfrak{n}'\in T.
    \\
    k_{\mathfrak{l}}=-k_{\mathfrak{l}}'=k.\  
    &\textit{$k_{i}=-k'_{j}$ (and $k_{i}=-k_{j}$, $k'_{i}=-k'_{j}$) for all $k_{i}\sim_{p}k'_{j}$ (and $k_{i}\sim_{p}k_{j}$, $k'_{i}\sim_{p}k'_{j}$)}\}.
\end{split}
\end{equation}

As in the case of \eqref{eq.diophantineeq'}, there is a graphical representation of \eqref{eq.diophantineeqpaired}. To explain this, we need the concept of couples.

%It turns out we can find a very compact formula for them. To do this, we need to introduce the concept of couples.

\begin{defn}[Construction of couples]\label{def.conple} Given two trees $T$ and $T'$, we flip the orientation of all edges in $T'$ (as in the two left trees in Figure \ref{fig.treepairing}). We also label their leaves by $1, 2, \cdots, l(T)+1$ and $1, 2, \cdots, l(T')+1$ so that the corresponding variables of these leaves are $k_1, k_2, \cdots, k_{l(T)+1}$ and $k_1, k_2, \cdots, k_{l(T')+1}$. Assume that we have a pairing $p$ of the set $\{k_1, k_2, \cdots, k_{l(T)+1}, k_1, k_2, \cdots, k_{l(T')+1}\}$, then this pairing induces a pairing between leaves (if $k_i\sim_p k_j$ then define $\textit{the $i$-th leaf}\sim_p \textit{the $j$-th leaf}$). Given this pairing of leaves, we define the following procedure which glues two trees $T$ and $T'$ into a couple  $\mathcal{C}(T,T,p)$. Some example of pairing can be find in Figure \ref{fig.treepairing}. 

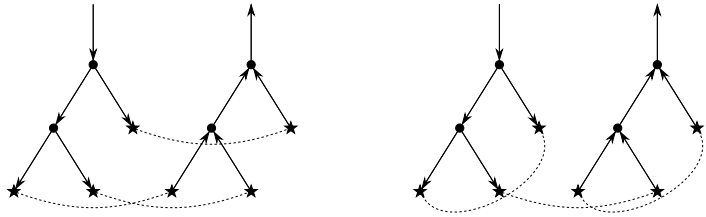
\begin{figure}[H]
    \centering
    \scalebox{0.3}{
    \begin{tikzpicture}[level distance=80pt, sibling distance=100pt]
        \node[] at (15,1.5) (1) {} 
            child {node[fillcirc] (2) {} 
                child {node[fillcirc] (3) {}
                    child {node[fillstar] (4) {}}
                    child {node[fillstar] (5) {}}
                }
                child {node[fillstar] (6) {}}
            };
        \draw[-{Stealth[length=5mm, width=3mm]}] (1) -- (2);
        \draw[-{Stealth[length=5mm, width=3mm]}] (2) -- (3);
        \draw[-{Stealth[length=5mm, width=3mm]}] (2) -- (6);
        \draw[-{Stealth[length=5mm, width=3mm]}] (3) -- (4);
        \draw[-{Stealth[length=5mm, width=3mm]}] (3) -- (5);
        \node[] at (22,1.5) (11) {} 
            child {node[fillcirc] (12) {} 
                child {node[fillcirc] (13) {}
                    child {node[fillstar] (14) {}}
                    child {node[fillstar] (15) {}}
                }
                child {node[fillstar] (16) {}}
            };
        \draw[{Stealth[length=5mm, width=3mm]}-] (11) -- (12);
        \draw[{Stealth[length=5mm, width=3mm]}-] (12) -- (13);
        \draw[{Stealth[length=5mm, width=3mm]}-] (12) -- (16);
        \draw[{Stealth[length=5mm, width=3mm]}-] (13) -- (14);
        \draw[{Stealth[length=5mm, width=3mm]}-] (13) -- (15);
        \draw[bend right =20, dashed] (4) edge (14);
        \draw[bend right =20, dashed] (5) edge (15);
        \draw[bend right =20, dashed] (6) edge (16);

        \node[] at (33,1.5) (new1) {} 
            child {node[fillcirc] (new2) {} 
                child {node[fillcirc] (new3) {}
                    child {node[fillstar] (new4) {}}
                    child {node[fillstar] (new5) {}}
                }
                child {node[fillstar] (new6) {}}
            };
        \draw[-{Stealth[length=5mm, width=3mm]}] (new1) -- (new2);
        \draw[-{Stealth[length=5mm, width=3mm]}] (new2) -- (new3);
        \draw[-{Stealth[length=5mm, width=3mm]}] (new2) -- (new6);
        \draw[-{Stealth[length=5mm, width=3mm]}] (new3) -- (new4);
        \draw[-{Stealth[length=5mm, width=3mm]}] (new3) -- (new5);
        \node[] at (40,1.5) (new11) {} 
            child {node[fillcirc] (new12) {} 
                child {node[fillcirc] (new13) {}
                    child {node[fillstar] (new14) {}}
                    child {node[fillstar] (new15) {}}
                }
                child {node[fillstar] (new16) {}}
            };
        \draw[{Stealth[length=5mm, width=3mm]}-] (new11) -- (new12);
        \draw[{Stealth[length=5mm, width=3mm]}-] (new12) -- (new13);
        \draw[{Stealth[length=5mm, width=3mm]}-] (new12) -- (new16);
        \draw[{Stealth[length=5mm, width=3mm]}-] (new13) -- (new14);
        \draw[{Stealth[length=5mm, width=3mm]}-] (new13) -- (new15);
        \draw[bend right =90, dashed] (new4) edge (new6);
        \draw[bend right =20, dashed] (new5) edge (new15);
        \draw[bend right =90, dashed] (new14) edge (new16);
    \end{tikzpicture}
    }
        \caption{Example of pairings between trees.}
        \label{fig.treepairing}
    \end{figure}

\begin{enumerate}
    \item \textbf{Merging edges connected to leaves:} Given two edges with opposite orientation connected to two paired leaves, these two edges can be \underline{merged} into one edge as in Figure \ref{fig.pairingleaves}.
    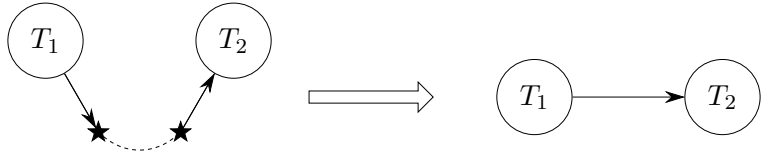
\begin{figure}[H]
    \centering
    \scalebox{0.5}{
    \begin{tikzpicture}[level distance=80pt, sibling distance=100pt]
        \node[draw, circle, minimum size=1cm, scale=2] at (0,0) (1) {$T_1$} [grow =300] 
            child {node[fillstar] (2) {}};
        \node[draw, circle, minimum size=1cm, scale=2] at (5,0) (3) {$T_2$} [grow =240] 
            child {node[fillstar] (4) {}};
        \draw[-{Stealth[length=5mm, width=3mm]}] (1) -- (2);
        \draw[{Stealth[length=5mm, width=3mm]}-] (3) -- (4);
        \draw[bend right =40, dashed] (2) edge (4);
        
        \node[draw, single arrow,
              minimum height=33mm, minimum width=8mm,
              single arrow head extend=2mm,
              anchor=west, rotate=0] at (7,-1.5) {}; 
              
        \node[draw, circle, minimum size=1cm, scale=2] at (13,-1.5) (5) {$T_1$}; 
        \node[draw, circle, minimum size=1cm, scale=2] at (18,-1.5) (6) {$T_2$};
        \draw[-{Stealth[length=5mm, width=3mm]}] (5) -- (6);
            
    \end{tikzpicture}
    }
        \caption{Pairing and merging of two edges}
        \label{fig.pairingleaves}
    \end{figure}
    We know that two edges connected to leaves correspond to two indices $k_i$, $k_j$. Merging two such edges is a graphical interpretation that $k_i=-k_j$. 
    \item \textbf{Pairing of trees and couples:} Given a pairing $p$ of the set of leaves in $T$, $T'$ we merge all edges paired by $p$ as in Figure \ref{fig.couple} and the resulting combinatorial structure is called a \underline{couple}. 
     \begin{figure}[H]
    \centering
    \scalebox{0.3}{
    \begin{tikzpicture}[level distance=80pt, sibling distance=100pt]
        \node[] at (0,0) (1) {} 
            child {node[fillcirc] (2) {} 
                child {node[fillstar] (3) {}}
                child {node[fillstar] (4) {}}
            };
        \draw[-{Stealth[length=5mm, width=3mm]}] (1) -- (2);
        \draw[-{Stealth[length=5mm, width=3mm]}] (2) -- (3);
        \draw[-{Stealth[length=5mm, width=3mm]}] (2) -- (4);
        
        \node[] at (0,-14.5) (11) {} [grow =90] 
            child {node[fillcirc] (12) {} 
                child {node[fillstar] (13) {}}
                child {node[fillstar] (14) {}}
            };    
        \draw[{Stealth[length=5mm, width=3mm]}-] (11) -- (12);
        \draw[{Stealth[length=5mm, width=3mm]}-] (12) -- (13);
        \draw[{Stealth[length=5mm, width=3mm]}-] (12) -- (14);    
            
        \draw[dashed] (3) edge (14);
        \draw[dashed] (4) edge (13);
        
        \node[draw, single arrow,
              minimum height=66mm, minimum width=16mm,
              single arrow head extend=4mm,
              anchor=west, rotate=0] at (7,-7.5) {};
        
        \node[] at (20,0) (21) {}; 
        \node[fillcirc] at (20,-4) (22) {};
        \node[fillcirc] at (20,-11) (23) {};
        \node[] (24) at (20,-15) {};
        \draw[-{Stealth[length=5mm, width=3mm]}] (21) edge (22);
        \draw[-{Stealth[length=5mm, width=3mm]}, bend right =40] (22) edge (23);
        \draw[-{Stealth[length=5mm, width=3mm]}, bend left =40] (22) edge (23);
        \draw[-{Stealth[length=5mm, width=3mm]}] (23) edge (24);
    \end{tikzpicture}
    }
        \caption{The construction of a couple}
        \label{fig.couple}
    \end{figure}
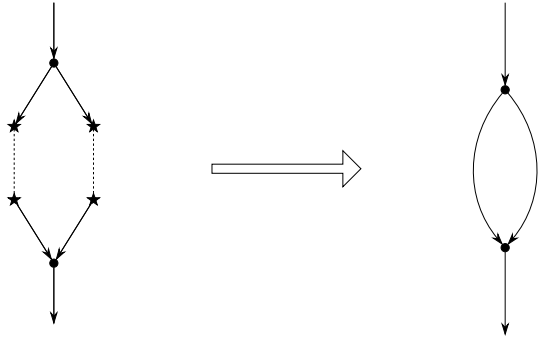
    
    We know that each edge connected to leaf corresponds to a variable $k_i$. A pairing $p$ of $\{k_1,k_2,\cdots,k_{2m}\}$ in \eqref{eq.diophantineeqpaired} induces a pairing of edges connected to leaves. Merging paired edges corresponds to $k_{i}=-k'_{j}$ for all $k_{i}\sim_{p}k'_{j}$ in \eqref{eq.diophantineeqpaired}. 
\end{enumerate}
\end{defn}

The following proposition introduce the graphical representation of \eqref{eq.diophantineeqpaired}.

\begin{prop}\label{prop.couple}
\eqref{eq.diophantineeqpaired} can be read from a couple diagram $\mathcal{C}(T,T,p)$. Each edge corresponds to a variable $k_{\mathfrak{e}}$. The leg $\mathfrak{l}$ corresponds to equation $k_{\mathfrak{l}}=k$. Each node corresponds to a momentum conservation equation
\begin{equation}
    \iota_{\mathfrak{e}_1}k_{\mathfrak{e}_1}+\iota_{\mathfrak{e}_2}k_{\mathfrak{e}_2}+\iota_{\mathfrak{e}}k_{\mathfrak{e}}=0,
\end{equation} 
and a energy conservation equation \begin{equation}
    \iota_{\mathfrak{e}_1}\Lambda_{k_{\mathfrak{e}_1}}+\iota_{\mathfrak{e}_2}\Lambda_{k_{\mathfrak{e}_2}}+\iota_{\mathfrak{e}}\Lambda_{k_{\mathfrak{e}}} = O(T^{-1}_{\text{max}}).
\end{equation}  
\end{prop}
\begin{rem}
In a couple diagram, we only have nodes decorated by $\bullet$. Nodes decorated by $\star$ have been removed in (1), (2) of Definition \ref{def.conple}.
\end{rem}
\begin{rem}
Through the process of (1), (2) in Definition \ref{def.conple}, a couple diagram can automatically encode the equation $k_{i}=-k'_{j}$ for all $k_{i}\sim_{p}k'_{j}$. Therefore, they do not appear in Proposition \ref{prop.couple}.
\end{rem}

\begin{proof}
This directly follows from the definition of couples. 
\end{proof}

The calculations of this section are summarized in the following proposition.  

\begin{prop}\label{prop.termcouple} (1) Define $Term(T,p)$ in the same way as in \eqref{eq.termexp'},
\begin{equation}\label{eq.termTp}
\begin{split}
    Term(T, p)_k=&\sum_{k_1,\, k_2,\, \cdots,\, k_{l(T)+1}}\sum_{k'_1,\, k'_2,\, \cdots,\, k'_{l(T)+1}}
    \\
    &H^T_{k_1\cdots k_{l(T)+1}} \overline{H^{T}_{k'_1\cdots k'_{l(T)+1}}} \delta_{p}(k_1,\cdots, k_{l(T)+1}, k'_1,\cdots, k'_{l(T)+1})\sqrt{n_{\textrm{in}}(k_1)}\cdots\sqrt{n_{\textrm{in}}(k'_1)}\cdots
\end{split}
\end{equation}
then $\mathbb{E}|\mathcal{J}_{T,k}|^2$ is a sum of $Term(T,p)_k$ for all $p\in \mathcal{P}$, (in \eqref{eq.termexp'} the sum is over set of all possible pairing $\mathcal{P}$)
\begin{equation}\label{eq.termexp}
\begin{split}
    \mathbb{E}|\mathcal{J}_{T,k}|^2=\left(\frac{\lambda}{L^{d}}\right)^{2l(T)}
    \sum_{p\in \mathcal{P}(\{k_1,\cdots, k_{l(T)+1}, k'_1,\cdots, k'_{l(T)+1}\})} Term(T, p).
\end{split}
\end{equation}

(2) $Term(T,p)$ concentrates near the subset \eqref{eq.diophantineeqpaired} which has a simple graphical representation given by Proposition \ref{prop.couple}. 
\end{prop}

\begin{proof} The proof of (1), (2) is easy and thus skipped. 
\end{proof}
% \begin{lem}
% Formula for $\mathbb{E}(\mathcal{J}_{T,k}\overline{\mathcal{J}_{T',k'}})$ in terms of couple $C$

% \begin{equation}\label{eq.couples}
%     1
% \end{equation}
% \end{lem}

\subsection{Counting lattice points}\label{sec.numbertheory} In this section, we use the connection between couples and concentration subsets \eqref{eq.diophantineeqpaired} to count the number of solutions of a generalized version of \eqref{eq.diophantineeqpaired},

\begin{equation}\label{eq.diophantineeqpairedsigma}
\begin{split}
    &\{k_{\mathfrak{e}}, k_{\mathfrak{e}'}\in \mathbb{Z}^d_L,\ |k_{\mathfrak{e}}|, |k_{\mathfrak{e}}'|\lesssim 1,\ \forall \mathfrak{e},\mathfrak{e}'\in T:\  |k_{\mathfrak{e}x}|\sim \kappa_{\mathfrak{e}}, |k'_{\mathfrak{e}x}|\sim \kappa_{\mathfrak{e}'},\ \forall \mathfrak{e},\mathfrak{e}'\in T, 
    \\
    &|\Omega_{\mathfrak{n}}-\sigma_{\mathfrak{n}}|,|\Omega'_{\mathfrak{n}}-\sigma'_{\mathfrak{n}}|\lesssim T^{-1}_{\text{max}},\ S_{\mathfrak{n}}=S_{\mathfrak{n}'}=0,\ \forall \mathfrak{n}.\ k_{\mathfrak{l}}=-k_{\mathfrak{l}}'=k.
    \\
    &\textit{$k_{i}=-k'_{j}$ (and $k_{i}=-k_{j}$, $k'_{i}=-k'_{j}$) for all $k_{i}\sim_{p}k'_{j}$ (and $k_{i}\sim_{p}k_{j}$, $k'_{i}\sim_{p}k'_{j}$)}\}.
\end{split}
\end{equation}

In \eqref{eq.diophantineeqpairedsigma}, $\kappa_{\mathfrak{e}}\in \{0\}\cup  \mathcal{D}(\alpha,1)$, where $\mathcal{D}(\alpha,1)\coloneqq\{2^{-K_{\mathfrak{e}}}:K_{\mathfrak{e}}\in  \mathbb{Z}\cap [0,ln\ \alpha^{-1}]\}$. The relation $|k_{\mathfrak{e}x}|\sim \kappa_{\mathfrak{e}}$ is defined by 
\begin{equation}\label{eq.kappa}
    |k_{\mathfrak{e}x}|\sim \kappa_{\mathfrak{e}}\text{ if and only if }\left\{\begin{aligned}
        & \frac{1}{2}\kappa_{\mathfrak{e}}\le  |k_{\mathfrak{e}x}|\le 2\kappa_{\mathfrak{e}} \qquad && \text{ if  $\kappa_{\mathfrak{e}}\ne 0$}
        \\[1em]
        &  |k_{\mathfrak{e}x}|\lesssim \alpha^2,\ k_{\mathfrak{e}x}\ne 0   \qquad && \text{ if  $\kappa_{\mathfrak{e}}= 0$}
    \end{aligned}
    \right.
\end{equation}

\eqref{eq.diophantineeqpairedsigma} is obtained by replacing $\Omega_{\mathfrak{n}}$, $\Omega'_{\mathfrak{n}}$ by $\Omega_{\mathfrak{n}}-\sigma_{\mathfrak{n}}$, $\Omega'_{\mathfrak{n}}-\sigma'_{\mathfrak{n}}$ in \eqref{eq.diophantineeqpaired} and adding conditions $|k_{\mathfrak{e}}|\sim \kappa_{\mathfrak{e}}$, where  $\sigma_{\mathfrak{n}}$, $\sigma'_{\mathfrak{n}}$ and $\kappa_{\mathfrak{e}}$ are some given constants. The counterpart of Proposition \ref{prop.couple} in this case is

\begin{prop}\label{prop.couple'}
\eqref{eq.diophantineeqpairedsigma} can be read from a couple diagram $\mathcal{C}=\mathcal{C}(T,T,p)$. Each edge corresponds to a variable $k_{\mathfrak{e}}$. The leg $\mathfrak{l}$ corresponds to equation $k_{\mathfrak{l}}=k$. Each node corresponds to a momentum conservation equation
\begin{equation}\label{eq.momentumconservationunit}
    \iota_{\mathfrak{e}_1}k_{\mathfrak{e}_1}+\iota_{\mathfrak{e}_2}k_{\mathfrak{e}_2}+\iota_{\mathfrak{e}}k_{\mathfrak{e}}=0,
\end{equation} 
and a energy conservation equation 
\begin{equation}\label{eq.energyconservationunit}
    \iota_{\mathfrak{e}_1}\Lambda_{k_{\mathfrak{e}_1}}+\iota_{\mathfrak{e}_2}\Lambda_{k_{\mathfrak{e}_2}}+\iota_{\mathfrak{e}}\Lambda_{k_{\mathfrak{e}}} = \sigma_{\mathfrak{n}} + O(T^{-1}_{\text{max}}).
\end{equation}  
Denote the momentum and energy conservation equations by $MC_{\mathfrak{n}}$ and $EC_{\mathfrak{n}}$ respectively, then \eqref{eq.diophantineeqpairedsigma} can be rewritten as 
\begin{equation}\label{eq.diophantineeqpairedsigma'}
    \text{\eqref{eq.diophantineeqpairedsigma}}=\{k_{\mathfrak{e}}\in \mathbb{Z}^d_L,\ |k_{\mathfrak{e}}|\lesssim 1\ \forall \mathfrak{e}\in \mathcal{C}:\  |k_{\mathfrak{e}x}| \sim \kappa_{\mathfrak{e}},\ \forall \mathfrak{e}\in \mathcal{C}_{\text{norm}}.\ MC_{\mathfrak{n}},\  EC_{\mathfrak{n}},\ \forall \mathfrak{n}\in \mathcal{C}.\ k_{\mathfrak{l}} = - k_{\mathfrak{l}}'= k.\}
\end{equation}
\end{prop}
\begin{proof}
This directly follows from Proposition \ref{prop.couple}. 
\end{proof}

To explain the counting argument in this paper, we need the following definitions related to couples.

\begin{defn}\label{def.morecouple}
\begin{enumerate}
    \item \textbf{Connected couples:} A couple $\mathcal{C}$ is a \underline{connected couple} if it is connected as a graph.  
    \item \textbf{Equations of a couple $Eq(\mathcal{C})$:} Given a couple $\mathcal{C}$ and constants $k$, $\sigma_{\mathfrak{n}}$, let $Eq(\mathcal{C},\{\sigma_{\mathfrak{n}}\}_{\mathfrak{n}}, k)$ (or simply $Eq(\mathcal{C})$) be the system of equation \eqref{eq.diophantineeqpairedsigma'} constructed in Proposition \ref{prop.couple'}. For any system of equations $Eq$, let $\#(Eq)$ be its number of solutions.
    \item \textbf{Normal edges and leaf edges:} Remember that any couple $\mathcal{C}$ is constructed from a pairing of two trees $T$, $T'$ and therefore all edges in $\mathcal{C}$ comes from $T$, $T'$. We define edges coming from $T_{\text{in}}$, $T'_{\text{in}}$ to be \underline{normal edges} and edges from those connected to leaves in $T$, $T'$ to be \underline{leaf edges}. The set of all normal edges is denoted by $\mathcal{C}_{\text{norm}}$. A leg in $\mathcal{C}$ which is a normal edge is called a \underline{normal leg}.
\end{enumerate}
\end{defn}

The main goal of this section is to prove an upper bound of $\#Eq(\mathcal{C})$. The main idea of proving this is to decompose a large couple $\mathcal{C}$ into smaller pieces and then prove this for smaller piece using induction hypothesis. To explain the idea, let us first focus on an example. Let $\mathcal{C}$ be the left couple in the following picture. (The corresponding variables of each edge are labelled near these edges.)
\begin{figure}[H]
    \centering
    \scalebox{0.4}{
    \begin{tikzpicture}[level distance=80pt, sibling distance=100pt]
        \node[] at (0,0) (1) {}; 
        \node[fillcirc] at (3,0) (2) {}; 
        \node[fillcirc] at (6,-2) (3) {}; 
        \node[fillcirc] at (9,-2) (4) {}; 
        \node[fillcirc] at (12,0) (5) {}; 
        \node[] at (15,0) (6) {}; 
        \draw[-{Stealth[length=5mm, width=3mm]}] (1) edge (2);
        \draw[-{Stealth[length=5mm, width=3mm]}] (2) edge (3);
        \draw[-{Stealth[length=5mm, width=3mm]}, bend left =40] (3) edge (4);
        \draw[-{Stealth[length=5mm, width=3mm]}, bend right =40] (3) edge (4);
        \draw[-{Stealth[length=5mm, width=3mm]}] (4) edge (5);
        \draw[-{Stealth[length=5mm, width=3mm]}] (5) edge (6);
        \draw[-{Stealth[length=5mm, width=3mm]}, bend left =40] (2) edge (5);
         
         \node[scale=2.0] at (3,-0.7) {$\mathfrak{n}_{1}$};
         \node[scale=2.0] at (6,-2.7) {$\mathfrak{n}_{2}$};
         \node[scale=2.0] at (9,-2.7) {$\mathfrak{n}_{3}$};
         \node[scale=2.0] at (12,-0.7) {$\mathfrak{n}_{4}$};
         \node[scale=2.0] at (1.5,-0.5) {$k$};
         \node[scale=2.0] at (4.3,-1.4) {$a$};
         \node[scale=2.0] at (7.5,-3.1) {$b$};
         \node[scale=2.0] at (7.5,-1) {$c$};
         \node[scale=2.0] at (10.7,-1.4) {$d$};
         \node[scale=2.0] at (7.5,2.2) {$e$};
         \node[scale=2.0] at (13.5,-0.5) {$-k$};

        \node[draw, single arrow,
              minimum height=33mm, minimum width=8mm,
              single arrow head extend=2mm,
              anchor=west, rotate=0] at (16,0) {};

        \node[] at (20,0) (11) {}; 
        \node[fillcirc] at (23,0) (12) {}; 
        \node[fillstar] at (26,-2) (13) {};
        \node[fillstar] at (26,2) (14) {};
        \draw[-{Stealth[length=5mm, width=3mm]}] (11) edge (12);
        \draw[-{Stealth[length=5mm, width=3mm]}] (12) edge (13);
        \draw[-{Stealth[length=5mm, width=3mm]}] (12) edge (14);
        
        \node[scale=2.0] at (23,-0.7) {$\mathfrak{n}_{1}$};
        \node[scale=2.0] at (21.5,-0.5) {$k$};
        \node[scale=2.0] at (24.3,-1.4) {$a$};
        \node[scale=2.0] at (24.3,1.4) {$e$};
        \node[scale=2.0] at (23,-1.8) {$A$};

        \node[] at (28,-2) (32) {}; 
        \node[fillcirc] at (31,-2) (33) {}; 
        \node[fillcirc] at (34,-2) (34) {}; 
        \node[fillcirc] at (37,0) (35) {}; 
        \node[] at (40,0) (36) {}; 
        \node[] at (34,2) (37) {};
        \draw[-{Stealth[length=5mm, width=3mm]}] (32) edge (33);
        \draw[-{Stealth[length=5mm, width=3mm]}, bend left =40] (33) edge (34);
        \draw[-{Stealth[length=5mm, width=3mm]}, bend right =40] (33) edge (34);
        \draw[-{Stealth[length=5mm, width=3mm]}] (34) edge (35);
        \draw[-{Stealth[length=5mm, width=3mm]}] (35) edge (36);
        \draw[-{Stealth[length=5mm, width=3mm]}] (37) edge (35);
        
        \node[scale=2.0] at (31,-2.7) {$\mathfrak{n}_{2}$};
        \node[scale=2.0] at (34,-2.7) {$\mathfrak{n}_{3}$};
        \node[scale=2.0] at (37,-0.7) {$\mathfrak{n}_{4}$};
        \node[scale=2.0] at (29.5,-2.5) {$a$};
        \node[scale=2.0] at (32.5,-3.1) {$b$};
        \node[scale=2.0] at (32.5,-1) {$c$};
        \node[scale=2.0] at (35.7,-1.4) {$d$};
        \node[scale=2.0] at (35.7,1.4) {$e$};
        \node[scale=2.0] at (38.5,-0.5) {$-k$};
        \node[scale=2.0] at (37,-2.2) {$B_{a,e}$};
    \end{tikzpicture}
    }
        \caption{An example of decomposing a couple}
        \label{fig.exampleofcuttingidea}
    \end{figure}
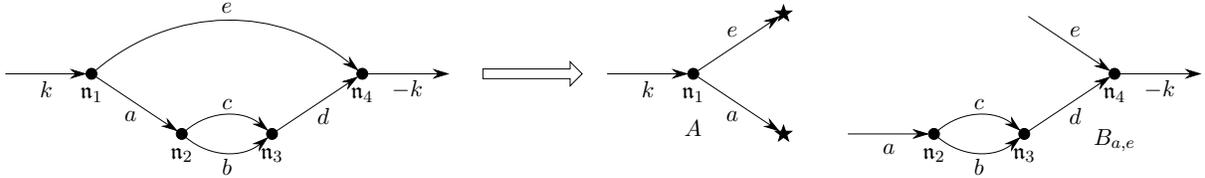

By \eqref{eq.diophantineeqpairedsigma'}, we know that the couple $\mathcal{C}$ corresponds to the following equations.
\begin{equation}\label{eq.cuttingexmaple}
    \begin{split}
        \{a, b, c, d, e:\ &(|a|\text{ to }|e|)\lesssim 1,\ (|a_x|\text{ to }|e_x|)\sim (\kappa_{a}\text{ to }\kappa_{e})
        \\
        &a+e=k,\ \Lambda(a) + \Lambda(e) - \Lambda(k) =\sigma_{1} + O(T^{-1}_{\text{max}})
        \\
        &a+c=b,\ \Lambda(a) + \Lambda(c) - \Lambda(b) =\sigma_{2} + O(T^{-1}_{\text{max}})
        \\
        &b+c=d,\ \Lambda(b) + \Lambda(c) - \Lambda(d) =\sigma_{3} + O(T^{-1}_{\text{max}})
        \\
        &d+e+k=0,\ \Lambda(d) + \Lambda(e) + \Lambda(k) =\sigma_{4} + O(T^{-1}_{\text{max}})\}
    \end{split}
\end{equation}

We know that \eqref{eq.cuttingexmaple} can be rewritten into the form $\bigcup_{a,e\in A} B_{a,e}$, where
\begin{equation}\label{eq.couplequationA}
    A=\{a, e:\ |a|,|e|\lesssim 1,\ |a_x|\sim \kappa_{a},|e_x|\sim \kappa_{e}, a+e=k,\ \Lambda(a) + \Lambda(e) - \Lambda(k) =\sigma_{1} + O(T^{-1}_{\text{max}})\}
\end{equation}
\begin{equation}\label{eq.couplequationB}
    \begin{split}
        B_{a,e}=\{b, c, d:\ &|b|,|c|,|d|\lesssim 1,\ |b_x|\sim \kappa_{b},|c_x|\sim \kappa_{c},|d_x|\sim \kappa_{d}
        \\
        &a+c=b,\ \Lambda(a) + \Lambda(c) - \Lambda(b) =\sigma_{2} + O(T^{-1}_{\text{max}})
        \\
        &b+c=d,\ \Lambda(b) + \Lambda(c) - \Lambda(d) =\sigma_{3} + O(T^{-1}_{\text{max}})
        \\
        &d+e+k=0,\ \Lambda(d) + \Lambda(e) + \Lambda(k) =\sigma_{4} + O(T^{-1}_{\text{max}})\}
    \end{split}    
\end{equation}

Since an upper bound of $\# Eq(\mathcal{C})$ can be derived from upper bounds of $\# A$, $\# B_{a,e}$, we just need to consider $A$, $B_{a,e}$ which are systems of equations of smaller size. We can reduce the size of systems of equations in this way and prove upper bounds by induction.

One problem of applying induction argument is that $A$, $B_{a,e}$ cannot be represented by couple defined by Definition \ref{def.conple} that can contain at most two legs (an edge just connected to one node). In Definition \ref{def.conple}, a leg is used to represent a variable which is fixed, as in the condition $k_{\mathfrak{l}} = - k_{\mathfrak{l}}'= k$ in \eqref{eq.diophantineeqpairedsigma'}. The definition of $\# B_{a,e}$ contains three fixed variables $a$, $e$, $k$ which cannot be represented by just two legs. Therefore, we has to define a new type of couple that allows multiple legs.

Except for the lack of legs, we also have the problem of representing free variables. We know that the couple representation of $A$ should contain one node and three edges if we insist on the rule that a node corresponds to an equation and the variables in the equation correspond to edges connected to this node. All these edges are legs, but two of three edges correspond to variable $a$, $b$ which are not fixed. Therefore, we have to define a type of legs that can correspond to unfixed variables.

To solve the above problems, we introduce the following definition.

\begin{defn}\label{def.couplemultileg}
\begin{enumerate}
    \item \textbf{Couples with multiple legs:} A graph in which all nodes have degree $1$ or $3$ is called a \underline{couples with multiple legs}. The graph $A$ and $B_{a,e}$ in Figure \ref{fig.exampleofcuttingidea} are examples of this definition.
    \item \textbf{Legs:} In a couple with multiple leg, an edge connected to a degree one node is called a \underline{leg}. Remember that we have encounter this concept in the second paragraph of section \ref{sec.refexp} and in what follows we call the leg defined there the \underline{root leg} of a tree.   
    \item \textbf{Free legs and fixed legs:} In a couple with multiple leg, we use two types of node decoration for degree $1$ nodes as in Figure \ref{fig.decorationdegreeone}. One is $\star$ and the other one is \underline{invisible}. 
    \begin{figure}[H]
    \centering
    \scalebox{0.5}{
    \begin{tikzpicture}[level distance=80pt, sibling distance=100pt]
        \node[draw, circle, minimum size=1cm, scale=2] at (0,0) (1) {$\mathcal{C}_1$} 
            child {node[fillstar] (2) {}};
        \node[draw, circle, minimum size=1cm, scale=2] at (5,0) (3) {$\mathcal{C}_2$}  
            child {node[] (4) {}};
        \draw[{Stealth[length=5mm, width=3mm]}-] (1) -- (2);
        \draw[{Stealth[length=5mm, width=3mm]}-] (3) -- (4);
        %\draw[bend right =40, dashed] (2) edge (4);
            
    \end{tikzpicture}
    }
        \caption{Node decoration of degree one nodes}
        \label{fig.decorationdegreeone}
    \end{figure}
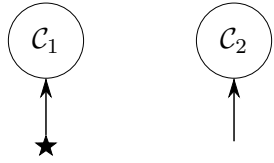
    An edge connected to a $\star$ or invisible nodes is called a \underline{free leg} or \underline{fixed leg} respectively.

    \item \textbf{Equations of a couple $Eq(\mathcal{C},\{c_{\mathfrak{l}}\}_{\mathfrak{l}})$:} We define the corresponding equations for a couples with multiple legs.
    \begin{equation}\label{eq.Eq(C,c)}
    \begin{split}
        &Eq(\mathcal{C},\{c_{\mathfrak{l}}\}_{\mathfrak{l}})
        \\
        =&\{k_{\mathfrak{e}}\in \mathbb{Z}^d_L,\ |k_{\mathfrak{e}}|\lesssim 1\ \forall \mathfrak{e}\in \mathcal{C}:\  |k_{\mathfrak{e}x}| \sim \kappa_{\mathfrak{e}},\ \forall \mathfrak{e}\in \mathcal{C}_{\text{norm}}.\ MC_{\mathfrak{n}},\  EC_{\mathfrak{n}},\ \forall \mathfrak{n}\in \mathcal{C}.\ k_{\mathfrak{l}}=c_{\mathfrak{l}},\ \forall \mathfrak{l}.\}   
    \end{split}
    \end{equation}
    In this representation, the corresponding variable of a fixed leg $\mathfrak{l}$ is fixed to be the constant $c_{\mathfrak{l}}$ and the corresponding variable of a free leg $\mathfrak{l}$ is not fixed.
\end{enumerate}
\end{defn}

With the above definition, it's easy to show that the couple $A$ and $B_{a,e}$ in Figure \ref{fig.exampleofcuttingidea} correspond to the system of equations \eqref{eq.couplequationA} and \eqref{eq.couplequationB} respectively.

Using the above argument, we can prove the following proposition which gives an upper bound of number of solutions of \eqref{eq.diophantineeqpairedsigma} (or \eqref{eq.diophantineeqpairedsigma'}).

\begin{prop}\label{prop.counting}
Let $\mathcal{C}=\mathcal{C}(T,T',p)$ be an connected couple with exactly one free and one fixed leg, $n$ be the total number of nodes in $\mathcal{C}$ and $Q=L^{d}T^{-1}_{\text{max}}$. We fix $k\in \mathbb{R}$ for the legs $\mathfrak{l}$, $\mathfrak{l}'$ and $\sigma_{\mathfrak{n}}\in\mathbb{R}$ for each $\mathfrak{n}\in \mathcal{C}$. Assume that $\alpha$ satisfies \eqref{eq.conditionalpha}. Then the number of solutions $M$ of \eqref{eq.diophantineeqpairedsigma} (or \eqref{eq.diophantineeqpairedsigma'}) is bounded by 
\begin{equation}\label{eq.countingbd0}
M\leq L^{O(n\theta)} Q^{\frac{n}{2}}\ \prod_{\mathfrak{e}\in \mathcal{C}_{\text{norm}}} \kappa^{-1}_{\mathfrak{e}}.
\end{equation}
\end{prop}
\begin{proof} 
The proof is lengthy and therefore divided into several steps. The main idea of the proof is to use the operation of edge cutting to decompose the couple $\mathcal{C}$ into smaller ones $\mathcal{C}_1$, $\mathcal{C}_2$, then apply Lemma \ref{lem.Eq(C)cutting} which relates $\#Eq(\mathcal{C})$ and $\#Eq(\mathcal{C}_i)$. The desire upper bounds of $\#Eq(\mathcal{C})$ can be obtained from that of $\#Eq(\mathcal{C}_i)$ inductively.

\textbf{Step 1.} In this step, we explain the cutting edge argument and prove the Lemma \ref{lem.Eq(C)cutting} which relates $\#Eq(\mathcal{C}_1)$, $\#Eq(\mathcal{C}_2)$ and $\#Eq(\mathcal{C})$.

Here is the formal definition of cutting 

\begin{defn}
\begin{enumerate}
    \item \textbf{Cutting an edge:} Given an edge $\mathfrak{e}$, we can cut it into two edges (a fixed and a free leg) as in Figure \ref{fig.cutedge}.
    
   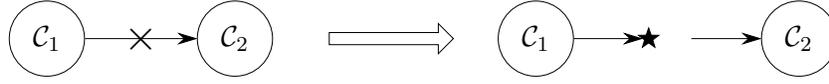
\begin{figure}[H]
    \centering
    \scalebox{0.5}{
    \begin{tikzpicture}[level distance=80pt, sibling distance=100pt]
        \node[draw, circle, minimum size=1cm, scale=2] at (0,0) (1) {$\mathcal{C}_1$}; 
        \node[draw, circle, minimum size=1cm, scale=2] at (5,0) (2) {$\mathcal{C}_2$};
        \draw[-{Stealth[length=5mm, width=3mm]}] (1) -- (2);
        \node[scale =3] at (2.5,0) {$\times$};

        \node[draw, single arrow,
              minimum height=33mm, minimum width=8mm,
              single arrow head extend=2mm,
              anchor=west, rotate=0] at (7.5,0) {};

        \node[draw, circle, minimum size=1cm, scale=2] at (13,0) (11) {$\mathcal{C}_1$}; 
        \node[fillstar] at (16,0) (12) {};
        \draw[-{Stealth[length=5mm, width=3mm]}] (11) -- (12);
         
        \node[] at (17,0) (13) {}; 
        \node[draw, circle, minimum size=1cm, scale=2] at (20,0) (14) {$\mathcal{C}_2$};
        \draw[-{Stealth[length=5mm, width=3mm]}] (13) -- (14);    
    \end{tikzpicture}
    }
        \caption{An example of cutting an edge}
        \label{fig.cutedge}
    \end{figure}

    \item \textbf{Cut:} A \underline{cut} $c$ of a couple $\mathcal{C}$ is a set of edges such that $\mathcal{C}$ is disconnected after cutting all edges in $c$. A \underline{refined cut} is a cut together with a map $\text{rc}:c\rightarrow \{\text{left}, \text{right}\}$. For each $\mathfrak{e}\in c$, if $\text{rc}(\mathfrak{e})=\text{left}$ (resp. right), then as in Figure \ref{fig.cutedge} the left node (resp. right node) produced by cutting $\mathfrak{e}$ is a $\star$ node (resp. invisible node). The map $\text{rc}$ describes which one should be the free or fixed leg in the two legs produced by cutting an edge.
    \item \textbf{$c(\mathfrak{e})$, $c(\mathfrak{n})$ and $c(\mathfrak{l})$:} Given an edge $\mathfrak{e}$ that is not a leg, define  $c(\mathfrak{e})$ to be the cut that contains only one edge $\mathfrak{e}$. Given a node $\mathfrak{n}\in \mathcal{C}$, let $\{\mathfrak{e}_{i}\}$ be edges that are connected to $\mathfrak{n}$, then define $c(\mathfrak{n})$ to be the cut that consists of edges $\{\mathfrak{e}_{i}\}$. Given an leg $\mathfrak{l}$, let $\mathfrak{n}$ be the unique node connected to it, then define  $c(\mathfrak{e})$ to be the cut $c(\mathfrak{n})$. An example of cutting $c(\mathfrak{e})$ is give by Figure \ref{fig.cutedge}. The following picture gives an example of cutting $c(\mathfrak{n})$ or $c(\mathfrak{l})$ (in this picture $\mathfrak{n}=\mathfrak{n}_1$ and $\mathfrak{l}$ is the leg labelled by $k$.)
    
    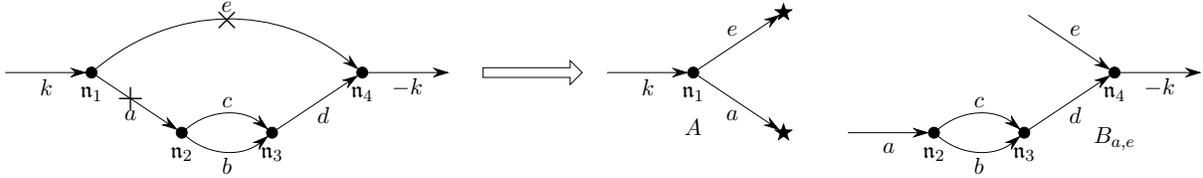
\begin{figure}[H]
    \centering
    \scalebox{0.4}{
    \begin{tikzpicture}[level distance=80pt, sibling distance=100pt]
        \node[] at (0,0) (1) {}; 
        \node[fillcirc] at (3,0) (2) {}; 
        \node[fillcirc] at (6,-2) (3) {}; 
        \node[fillcirc] at (9,-2) (4) {}; 
        \node[fillcirc] at (12,0) (5) {}; 
        \node[] at (15,0) (6) {}; 
        \draw[-{Stealth[length=5mm, width=3mm]}] (1) edge (2);
        \draw[-{Stealth[length=5mm, width=3mm]}] (2) edge (3);
        \draw[-{Stealth[length=5mm, width=3mm]}, bend left =40] (3) edge (4);
        \draw[-{Stealth[length=5mm, width=3mm]}, bend right =40] (3) edge (4);
        \draw[-{Stealth[length=5mm, width=3mm]}] (4) edge (5);
        \draw[-{Stealth[length=5mm, width=3mm]}] (5) edge (6);
        \draw[-{Stealth[length=5mm, width=3mm]}, bend left =40] (2) edge (5);
         
         \node[scale=2.0] at (3,-0.7) {$\mathfrak{n}_{1}$};
         \node[scale=2.0] at (6,-2.7) {$\mathfrak{n}_{2}$};
         \node[scale=2.0] at (9,-2.7) {$\mathfrak{n}_{3}$};
         \node[scale=2.0] at (12,-0.7) {$\mathfrak{n}_{4}$};
         \node[scale=2.0] at (1.5,-0.5) {$k$};
         \node[scale=2.0] at (4.3,-1.4) {$a$};
         \node[scale=2.0] at (7.5,-3.1) {$b$};
         \node[scale=2.0] at (7.5,-1) {$c$};
         \node[scale=2.0] at (10.7,-1.4) {$d$};
         \node[scale=2.0] at (7.5,2.2) {$e$};
         \node[scale=2.0] at (13.5,-0.5) {$-k$};
         \node[scale=3.0, rotate =45] at (4.3,-0.85) {$\times$};
         \node[scale=3.0, rotate = 0] at (7.5,1.75) {$\times$};

        \node[draw, single arrow,
              minimum height=33mm, minimum width=8mm,
              single arrow head extend=2mm,
              anchor=west, rotate=0] at (16,0) {};

        \node[] at (20,0) (11) {}; 
        \node[fillcirc] at (23,0) (12) {}; 
        \node[fillstar] at (26,-2) (13) {};
        \node[fillstar] at (26,2) (14) {}; 
        \draw[-{Stealth[length=5mm, width=3mm]}] (11) edge (12);
        \draw[-{Stealth[length=5mm, width=3mm]}] (12) edge (13);
        \draw[-{Stealth[length=5mm, width=3mm]}] (12) edge (14);
        
        \node[scale=2.0] at (23,-0.7) {$\mathfrak{n}_{1}$};
        \node[scale=2.0] at (21.5,-0.5) {$k$};
        \node[scale=2.0] at (24.3,-1.4) {$a$};
        \node[scale=2.0] at (24.3,1.4) {$e$};
        \node[scale=2.0] at (23,-1.8) {$A$};

        \node[] at (28,-2) (32) {}; 
        \node[fillcirc] at (31,-2) (33) {}; 
        \node[fillcirc] at (34,-2) (34) {}; 
        \node[fillcirc] at (37,0) (35) {}; 
        \node[] at (40,0) (36) {}; 
        \node[] at (34,2) (37) {};
        \draw[-{Stealth[length=5mm, width=3mm]}] (32) edge (33);
        \draw[-{Stealth[length=5mm, width=3mm]}, bend left =40] (33) edge (34);
        \draw[-{Stealth[length=5mm, width=3mm]}, bend right =40] (33) edge (34);
        \draw[-{Stealth[length=5mm, width=3mm]}] (34) edge (35);
        \draw[-{Stealth[length=5mm, width=3mm]}] (35) edge (36);
        \draw[-{Stealth[length=5mm, width=3mm]}] (37) edge (35);
        
        \node[scale=2.0] at (31,-2.7) {$\mathfrak{n}_{2}$};
        \node[scale=2.0] at (34,-2.7) {$\mathfrak{n}_{3}$};
        \node[scale=2.0] at (37,-0.7) {$\mathfrak{n}_{4}$};
        \node[scale=2.0] at (29.5,-2.5) {$a$};
        \node[scale=2.0] at (32.5,-3.1) {$b$};
        \node[scale=2.0] at (32.5,-1) {$c$};
        \node[scale=2.0] at (35.7,-1.4) {$d$};
        \node[scale=2.0] at (35.7,1.4) {$e$};
        \node[scale=2.0] at (38.5,-0.5) {$-k$};
        \node[scale=2.0] at (37,-2.2) {$B_{a,e}$};
    \end{tikzpicture}
    }
        \caption{An example of cuts, $c(\mathfrak{n})$ and $c(\mathfrak{l})$}
        \label{fig.c(n)c(e)}
    \end{figure}
    \item \textbf{Normal edges in couples with multiple legs:} In this paper, all couples with multiple legs are produced by cutting a couple defined in Definition \ref{def.conple}. If a normal edge $\mathfrak{e}$ is cut into $\mathfrak{e}_1$ and $\mathfrak{e}_2$, then $\mathfrak{e}_1$ and $\mathfrak{e}_2$ are defined to be normal in the resulting couples with multiple legs.
\end{enumerate}
\end{defn}
\begin{rem}
Explicitly writing down the full definition of $\text{rc}$ is often complicated, so in what follows, when defining $\text{rc}$, we will only describe which one should be the free or fixed leg in the two legs produced by cutting an edge.
%Explicitly writing down the full definition of $\text{rc}$ is often complicated, so I what follows, when defining $\text{rc}$, we will just describe which node is $\star$ nodes or invisible nodes in the two nodes after cutting one edge.
\end{rem}

The couples in Proposition \ref{prop.counting} contains just $2$ fixed legs, but after cutting, these couples may contain more fixed or free legs. 

By Definition \ref{def.couplemultileg}, for a couple $\mathcal{C}$ with multiple legs, given constants $c_{\mathfrak{l}}$ for each fixed leg $\mathfrak{l}$, the corresponding equation of $\mathcal{C}$ is denoted by $Eq(\mathcal{C},\{c_{\mathfrak{l}}\}_{\mathfrak{l}})$.
In $Eq(\mathcal{C},\{c_{\mathfrak{l}}\}_{\mathfrak{l}})$ each edge $\mathfrak{e}$ is associated with a variable $k_{\mathfrak{e}}$ and each node $\mathfrak{n}$ is still associated with equations $MC_{\mathfrak{n}}$, $EC_{\mathfrak{n}}$. The corresponding variables of free (resp. fixed) legs are free (resp. fixed to be a constant $c_{\mathfrak{l}}$).

Let us explain how does $Eq(\mathcal{C})$ and $\#Eq(\mathcal{C})$ changes after cutting. The result is summarized in the following lemma.
\begin{lem}\label{lem.Eq(C)cutting}
Let $c$ be a cut of $\mathcal{C}$ that consists of edges $\{\mathfrak{e}_{i}\}$ and $\mathcal{C}_1$, $\mathcal{C}_2$ be two components after cutting. Let $\mathfrak{e}_{i}^{(1)}\in \mathcal{C}_1$, $\mathfrak{e}_{i}^{(2)}\in \mathcal{C}_2$ be two edges obtained by cutting $\mathfrak{e}_{i}$. The $\text{rc}$ map is defined by assigning $\{\mathfrak{e}_{i}^{(1)}\}$ to be free legs and $\{\mathfrak{e}_{i}^{(2)}\}$ to be fixed legs. Then we have 
\begin{equation}\label{eq.Eq(C)cutting}
    Eq(\mathcal{C},\{c_{\mathfrak{l}}\}_{\mathfrak{l}})=\left\{(k_{\mathfrak{e}_1},k_{\mathfrak{e}_{2}}):\ k_{\mathfrak{e}_1}\in Eq(\mathcal{C}_1,\{c_{\mathfrak{l}_1}\}),\  k_{\mathfrak{e}_{2}}\in Eq\left(\mathcal{C}_{2}, \{c_{\mathfrak{l}_2}\}, \left\{k_{\mathfrak{e}_{i}^{(1)}}\right\}_{i}\right)\right\}.
\end{equation}
and
\begin{equation}\label{eq.Eq(C)cuttingcounting}
    \sup_{\{c_{\mathfrak{l}}\}_{\mathfrak{l}}}\#Eq(\mathcal{C},\{c_{\mathfrak{l}}\}_{\mathfrak{l}})\le
    \sup_{\{c_{\mathfrak{l}_1}\}_{\mathfrak{l}_1\in \text{leg}(\mathcal{C}_1)} } \# Eq(\mathcal{C}_1,\{c_{\mathfrak{l}_1}\}) \sup_{\{c_{\mathfrak{l}_2}\}_{\mathfrak{l}_2\in \text{leg}(\mathcal{C}_2)} }\# Eq(\mathcal{C}_{2}, \{c_{\mathfrak{l}_2}\}).
\end{equation}
Here $\text{leg}(\mathcal{C})$ is the set of fixed legs in $\mathcal{C}$ (not the set of all legs!).
\end{lem}
\begin{proof}
%Let $c$ be a cut of $\mathcal{C}$ that consists of edges $\{\mathfrak{e}_{i}\}$. Let $\mathcal{C}_1$ and $\mathcal{C}_2$ be two components after cutting and $\mathfrak{e}_{i}^{(1)}\in \mathcal{C}_1$, $\mathfrak{e}_{i}^{(2)}\in \mathcal{C}_2$ be two edges obtained by cutting $\mathfrak{e}_{i}$. We assume that $\{\mathfrak{e}_{i}^{(1)}\}$ are free legs and $\{\mathfrak{e}_{i}^{(2)}\}$ are legs. 
By definition \eqref{eq.Eq(C,c)} we have
\begin{equation}
\begin{split}
    Eq(\mathcal{C},\{c_{\mathfrak{l}}\}_{\mathfrak{l}})=&\{k_{\mathfrak{e}}\in \mathbb{Z}^d_L,\ |k_{\mathfrak{e}}|\lesssim 1:\  |k_{\mathfrak{e}x}| \sim \kappa_{\mathfrak{e}},\ \forall \mathfrak{e}\in \mathcal{C}_{\text{norm}}.\ MC_{\mathfrak{n}},\  EC_{\mathfrak{n}},\ \forall \mathfrak{n}.\  k_{\mathfrak{l}}=c_{\mathfrak{l}},\ \forall \mathfrak{l}\in \text{leg}(\mathcal{C}).\} 
    \\
    =&\{(k_{\mathfrak{e}_1},k_{\mathfrak{e}_2}):\ |k_{\mathfrak{e}_1}| \lesssim 1,\ MC_{\mathfrak{n}_1},\  EC_{\mathfrak{n}_1}.\ \forall \mathfrak{e}_1, \mathfrak{n}_1\in\mathcal{C}_1.\ |k_{\mathfrak{e}x}| \sim \kappa_{\mathfrak{e}},\ \forall \mathfrak{e}\in \mathcal{C}_{\text{norm}}.
    \\
    &k_{\mathfrak{l}_1}=c_{\mathfrak{l}_1},\ \forall \mathfrak{l}_1\in \text{leg}(\mathcal{C})\cap \text{leg}(\mathcal{C}_1)
    \\
    &|k_{\mathfrak{e}_2}| \lesssim 1,\ MC_{\mathfrak{n}_2},\  EC_{\mathfrak{n}_2}.\ \forall \mathfrak{e}_2, \mathfrak{n}_2\in\mathcal{C}_2.\ |k_{\mathfrak{e}x}| \sim \kappa_{\mathfrak{e}},\ \forall \mathfrak{e}\in \mathcal{C}_{\text{norm}}.
    \\
    &k_{\mathfrak{l}_2}=c_{\mathfrak{l}_2},\ \forall \mathfrak{l}_2\in \text{leg}(\mathcal{C})\cap \text{leg}(\mathcal{C}_2),\ k_{\mathfrak{e}_{i}^{(2)}}=k_{\mathfrak{e}_{i}^{(1)}},\ \forall\mathfrak{e}_{i}\in c\}
    \\
    =&\left\{(k_{\mathfrak{e}_1},k_{\mathfrak{e}_{2}}):\ k_{\mathfrak{e}_1}\in Eq(\mathcal{C}_1,\{c_{\mathfrak{l}_1}\}),\  k_{\mathfrak{e}_{2}}\in Eq\left(\mathcal{C}_{2}, \{c_{\mathfrak{l}_2}\}, \left\{k_{\mathfrak{e}_{i}^{(1)}}\right\}_{i}\right)\right\}
\end{split}
\end{equation}
Here in $Eq\left(\mathcal{C}_{2}, \{c_{\mathfrak{l}_2}\}, \left\{k_{\mathfrak{e}_{i}^{(1)}}\right\}_{i}\right)$, $k_{\mathfrak{e}_{i}^{(1)}}$ are view as a constant value and $k_{\mathfrak{e}_{i}^{(2)}}$ are fixed to be this constant value.

Therefore, we have the following identity of $Eq(\mathcal{C},\{c_{\mathfrak{l}}\}_{\mathfrak{l}})$
\begin{equation}
    Eq(\mathcal{C},\{c_{\mathfrak{l}}\}_{\mathfrak{l}})=\left\{(k_{\mathfrak{e}_1},k_{\mathfrak{e}_{2}}):\ k_{\mathfrak{e}_1}\in Eq(\mathcal{C}_1,\{c_{\mathfrak{l}_1}\}),\  k_{\mathfrak{e}_{2}}\in Eq\left(\mathcal{C}_{2}, \{c_{\mathfrak{l}_2}\}, \left\{k_{\mathfrak{e}_{i}^{(1)}}\right\}_{i}\right)\right\}.
\end{equation}
which proves \eqref{eq.Eq(C)cutting}.

We can also find the relation between $\#Eq(\mathcal{C}_1)$, $\#Eq(\mathcal{C}_2)$ and $\#Eq(\mathcal{C})$. Applying \eqref{eq.Eq(C)cutting},
\begin{equation}
\begin{split}
    \#Eq(\mathcal{C},\{c_{\mathfrak{l}}\}_{\mathfrak{l}})=&\sum_{(k_{\mathfrak{e}_1},k_{\mathfrak{e}_{2}})\in \#Eq(\mathcal{C},\{c_{\mathfrak{l}}\}_{\mathfrak{l}})} 1
    \\
    =&\sum_{\left\{(k_{\mathfrak{e}_1},k_{\mathfrak{e}_{2}}):\ k_{\mathfrak{e}_1}\in Eq(\mathcal{C}_1,\{c_{\mathfrak{l}_1}\}),\  k_{\mathfrak{e}_{2}}\in Eq\left(\mathcal{C}_{2}, \{c_{\mathfrak{l}_2}\}, \left\{k_{\mathfrak{e}_{i}^{(1)}}\right\}_{i}\right)\right\}} 1
    \\
    =&\sum_{k_{\mathfrak{e}_1}\in Eq(\mathcal{C}_1,\{c_{\mathfrak{l}_1}\})} \sum_{k_{\mathfrak{e}_{2}}\in Eq\left(\mathcal{C}_{2}, \{c_{\mathfrak{l}_2}\}, \left\{k_{\mathfrak{e}_{i}^{(1)}}\right\}_{i}\right)} 1
    \\
    =&\sum_{k_{\mathfrak{e}_1}\in Eq(\mathcal{C}_1,\{c_{\mathfrak{l}_1}\})} \# Eq\left(\mathcal{C}_{2}, \{c_{\mathfrak{l}_2}\}, \left\{k_{\mathfrak{e}_{i}^{(1)}}\right\}_{i}\right)
\end{split}
\end{equation}
Take $\sup$ in the above equation
\begin{equation}
\begin{split}
    &\sup_{\{c_{\mathfrak{l}}\}_{\mathfrak{l}}}\#Eq(\mathcal{C},\{c_{\mathfrak{l}}\}_{\mathfrak{l}})
    =\sup_{\{c_{\mathfrak{l}}\}_{\mathfrak{l}}}\sum_{k_{\mathfrak{e}_1}\in Eq(\mathcal{C}_1,\{c_{\mathfrak{l}_1}\})} \# Eq\left(\mathcal{C}_{2}, \{c_{\mathfrak{l}_2}\}, \left\{k_{\mathfrak{e}_{i}^{(1)}}\right\}_{i}\right)
    \\
    \le &\sup_{\{c_{\mathfrak{l}_1}\}_{\mathfrak{l}_1\in \text{leg}(\mathcal{C})\cap \text{leg}(\mathcal{C}_1)} }\sum_{k_{\mathfrak{e}_1}\in Eq(\mathcal{C}_1,\{c_{\mathfrak{l}_1}\})} \sup_{\{c_{\mathfrak{l}_2}\}_{\mathfrak{l}_2\in \text{leg}(\mathcal{C})\cap \text{leg}(\mathcal{C}_2)} }\# Eq\left(\mathcal{C}_{2}, \{c_{\mathfrak{l}_2}\}, \left\{k_{\mathfrak{e}_{i}^{(1)}}\right\}_{i}\right)
    \\
    \le &\sup_{\{c_{\mathfrak{l}_1}\}_{\mathfrak{l}_1\in \text{leg}(\mathcal{C}_1)} }\sum_{k_{\mathfrak{e}_1}\in Eq(\mathcal{C}_1,\{c_{\mathfrak{l}_1}\})} \sup_{\{c_{\mathfrak{l}_2}\}_{\mathfrak{l}_2\in \text{leg}(\mathcal{C}_2)} }\# Eq(\mathcal{C}_{2}, \{c_{\mathfrak{l}_2}\})
    \\
    = &\sup_{\{c_{\mathfrak{l}_1}\}_{\mathfrak{l}_1\in \text{leg}(\mathcal{C}_1)} } \# Eq(\mathcal{C}_1,\{c_{\mathfrak{l}_1}\}) \sup_{\{c_{\mathfrak{l}_2}\}_{\mathfrak{l}_2\in \text{leg}(\mathcal{C}_2)} }\# Eq(\mathcal{C}_{2}, \{c_{\mathfrak{l}_2}\})
\end{split}
\end{equation}

This proves \eqref{eq.Eq(C)cuttingcounting}.
\end{proof}

\textbf{Step 2.} In this step, we specify the cutting procedure. 

Notice that Proposition \ref{prop.countingind}, the multiple leg analog of Proposition \ref{prop.counting}, is only true for couples satisfying the "property P". When designing the cutting procedure, we must make sure that all couples generated during the execution of this procedure satisfy the "property P". The following proposition guarantees the existence of such procedure.

\begin{prop}\label{prop.cuttingalgorithm}
There exists a recursive algorithm that repeatedly decomposes $\mathcal{C}$ into smaller pieces and satisfies the following requirements. In the rest of this paper, we will call this algorithm "the cutting algorithm".

(1) The input of the $0$-th step of this algorithm is $\mathcal{C}(0)=\mathcal{C}$. The inputs of other steps are the outputs of previous steps of the algorithm itself.

(2) In step $k$, $\mathcal{C}(k)$ is decomposed into 2 or 3 connected components by cutting edges and all components with more than one node are outputted. For $\mathcal{C}(1)$, $\# Eq(\mathcal{C}(1))=\# Eq(\mathcal{C})$ and $\mathcal{C}_{\text{norm}}(1)=\mathcal{C}_{\text{norm}}$.

(3) One of the connected components in (2) contains exactly one node $\mathfrak{n}$ and one fixed normal leg $\mathfrak{l}$. We call this component $\mathcal{C}(k)_{
\mathfrak{l}}$. There are only two possibilities of $\mathcal{C}(k)_{\mathfrak{l}}$ as in Figure \ref{fig.2possibilities}. We label them by $\mathcal{C}_{I}$, $\mathcal{C}_{II}$.

\begin{figure}[H]
    \centering
    \scalebox{0.3}{
    \begin{tikzpicture}[level distance=80pt, sibling distance=100pt]
        \node[] at (0,0) (1) {} 
            child {node[fillcirc] (2) {} 
                child {node[fillstar] (3) {}}
                child {node[fillstar] (4) {}}
            };
        \draw[-{Stealth[length=5mm, width=3mm]}] (1) -- (2);
        \draw[-{Stealth[length=5mm, width=3mm]}] (2) -- (3);
        \draw[-{Stealth[length=5mm, width=3mm]}] (2) -- (4);
        
        \node[] at (12,0) (1) {} 
            child {node[fillcirc] (2) {} 
                child {node[fillstar] (3) {}}
                child {node[xshift = 5pt, yshift = -10pt] (4) {}}
            };
        \draw[-{Stealth[length=5mm, width=3mm]}] (1) -- (2);
        \draw[-{Stealth[length=5mm, width=3mm]}] (2) -- (3);
        \draw[-{Stealth[length=5mm, width=3mm]}] (2) -- (4);
    \end{tikzpicture}
    }
        \caption{Two possibilities of $\mathcal{C}_\mathfrak{l}$.}
        \label{fig.2possibilities}
    \end{figure}
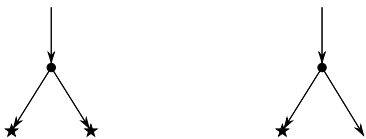
    
(4) The cutting algorithm satisfies the requirement that the all connected components in (2) generated in each step satisfy property P, where property P is defined below

\underline{Property P of a couple $\widetilde{\mathcal{C}}$}: $\widetilde{\mathcal{C}}$ is connected and contains exactly one free leg and at least one fixed normal leg.

% We need the following definition.

% \underline{Property P of a couple $\widetilde{\mathcal{C}}$}: $\widetilde{\mathcal{C}}$ is connected and contains exactly one free leg and at least one fixed normal leg.

% With the above definition, the cutting algorithm should satisfy the requirement that the output components of each steps of the algorithm satisfies property P.
\end{prop}

\begin{rem}
Although the definition of cut is rather general, only three special types of cuts, $c(\mathfrak{e})$, $c(\mathfrak{n})$ and $c(\mathfrak{l})$, are used in the cutting algorithm.
\end{rem}

\begin{rem}
The couple $\mathcal{C}$ does not satisfy the property P because it does not have any free leg.
\end{rem}

\begin{proof}[Proof of Proposition \ref{prop.cuttingalgorithm}.] Consider the following algorithm.

\medskip

\begin{mdframed}

\centerline{\textbf{The cutting algorithm}}

\medskip 

\ \ \ \textbf{Step 0.} The input $\mathcal{C}(0)$ of this step is $\mathcal{C}$. In this step, we replace one of the two fixed legs of $\mathcal{C}$ by a free leg to obtain a new couple $\widehat{\mathcal{C}}$. By Lemma \ref{lem.freeleg} (3), $\#Eq(\mathcal{C})=\#Eq(\widehat{\mathcal{C}})$. The output $\mathcal{C}(1)$ of this step is $\widehat{\mathcal{C}}$. An example of step 0 can be found in the following picture.
\begin{figure}[H]
    \centering
    \scalebox{0.4}{
    \begin{tikzpicture}[level distance=80pt, sibling distance=100pt]
        \node[] at (0,0) (1) {}; 
        \node[fillcirc] at (3,0) (2) {}; 
        \node[fillcirc] at (7,0) (3) {}; 
        \node[] at (10,0) (4) {}; 
        \draw[-{Stealth[length=5mm, width=3mm]}] (1) edge (2);
        \draw[-{Stealth[length=5mm, width=3mm]}, bend left =40] (2) edge (3);
        \draw[-{Stealth[length=5mm, width=3mm]}, bend right =40] (2) edge (3);
        \draw[-{Stealth[length=5mm, width=3mm]}] (3) edge (4);
        \node[scale=2.0] at (5.1,-2) {$\mathcal{C}(0)=\mathcal{C}$};
 
        \node[draw, single arrow,
              minimum height=33mm, minimum width=8mm,
              single arrow head extend=2mm,
              anchor=west, rotate=0] at (11,0) {};

        \node[] at (15,0) (11) {}; 
        \node[fillcirc] at (18,0) (12) {}; 
        \node[fillcirc] at (22,0) (13) {}; 
        \node[fillstar] at (25,0) (14) {}; 
        \draw[-{Stealth[length=5mm, width=3mm]}] (11) edge (12);
        \draw[-{Stealth[length=5mm, width=3mm]}, bend left =40] (12) edge (13);
        \draw[-{Stealth[length=5mm, width=3mm]}, bend right =40] (12) edge (13);
        \draw[-{Stealth[length=5mm, width=3mm]}] (13) edge (14);      
        \node[scale=2.0] at (20.1,-2) {$\mathcal{C}(1)=\widehat{\mathcal{C}}$};
    \end{tikzpicture}
    }
        \caption{An example of step $0$}
        \label{fig.step0}
    \end{figure}

\textbf{Step $k$.} Assume that the step $k-1$ have been finished. The input $\mathcal{C}(k)$ of step $k$ is the output of step $k-1$. (If there are two output couples from step $k-1$, apply step $k$ to these two couples separately.)

By property P, there exists a fixed normal leg in $\mathcal{C}(k)$. Choose one such leg $\mathfrak{l}$ and define $\mathcal{C}(k)_{\mathfrak{l}}$ to be the component which contains $\mathfrak{l}$ after cutting $c(\mathfrak{l})$ and define $\mathcal{C}(k)'=\mathcal{C}(k)\backslash \mathcal{C}(k)_{\mathfrak{l}}$. Check how many components does $\mathcal{C}(k)'$ have. Jump to case 1 if number of components equals to 1, otherwise jump to case 2. Examples of case 1 and case 2 can be found in the following picture.
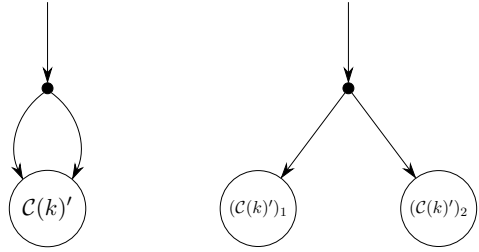
\begin{figure}[H]
    \centering
    \scalebox{0.4}{
    \begin{tikzpicture}[level distance=80pt, sibling distance=100pt]
        \node[] at (0,0) (1) {};
        \node[fillcirc] at (0, -3) (2) {};
        \node[draw, circle, minimum size=1cm, scale=2] at (0,-7) (3) {$\mathcal{C}(k)'$}; 
        \draw[-{Stealth[length=5mm, width=3mm]}] (1) edge (2);
        \draw[-{Stealth[length=5mm, width=3mm]}, bend left =40] (2) edge (3);
        \draw[-{Stealth[length=5mm, width=3mm]}, bend right =40] (2) edge (3);
        
        \node[] at (10,0) (1) {};
        \node[fillcirc] at (10, -3) (2) {};
        \node[draw, circle, minimum size=1cm, scale=1.5] at (7,-7) (3) {$(\mathcal{C}(k)')_1$};
        \node[draw, circle, minimum size=1cm, scale=1.5] at (13,-7) (4) {$(\mathcal{C}(k)')_2$};
        \draw[-{Stealth[length=5mm, width=3mm]}] (1) edge (2);
        \draw[-{Stealth[length=5mm, width=3mm]}] (2) edge (3);
        \draw[-{Stealth[length=5mm, width=3mm]}] (2) edge (4);
    \end{tikzpicture}
    }
        \caption{Examples of case 1 and case 2 (Left is case 1 and right is case two. $(\mathcal{C}(k)')_1$ and $(\mathcal{C}(k)')_2$ are two components of $\mathcal{C}(k)'$)}
        \label{fig.step1case}
    \end{figure}

\textbf{Case 1 of step $k$.} In this case $\mathcal{C}(k)'$ has one components. We rename  $\mathcal{C}(k)'$ to $\mathcal{C}(k)_1$. By property P, there exists a unique free leg $\mathfrak{l}_{fr}$ in $\mathcal{C}(k)$. Check if $\mathfrak{l}_{fr}$ and $\mathfrak{l}$ are connected to the same node. If yes, jump to case 1.1, otherwise jump to case 1.2.

\textbf{Case 1.1.} Cut edges in $c(\mathfrak{l})$ into $\{\mathfrak{e}_{i}^{(1)}\}_{i=1,2}$ and $\{\mathfrak{e}_{i}^{(2)}\}_{i=1,2}$, then $\mathcal{C}(k)$ is decomposed into $\mathcal{C}(k)_{\mathfrak{l}}$, $\mathcal{C}(k)_1=\mathcal{C}\backslash \mathcal{C}_{\mathfrak{l}}$. As in Lemma \ref{lem.Eq(C)cutting}, define $\{\mathfrak{e}_{i}^{(1)}\}\subseteq \mathcal{C}(k)_{\mathfrak{l}}$ to be free legs and $\{\mathfrak{e}_{i}^{(2)}\}\subseteq \mathcal{C}(k)_1$ to be fixed legs. 

If $\mathcal{C}(k)_1$ satisfies the property P, define $\mathcal{C}(k+1)=\mathcal{C}(k)_1$ to be the output of step $k$ and apply step $k+1$ to $\mathcal{C}(k+1)$. 

Otherwise by Lemma \ref{lem.normleg} (2) there exist exactly one free normal leg and at least one fixed leg. By Lemma \ref{lem.freeleg}, we can define a new couple $\widehat{\mathcal{C}}$ such that the free normal leg becomes fixed and the fixed leg becomes free.  Finally, define $\mathcal{C}(k+1)=\widehat{\mathcal{C}}$ to be the output. 

Examples of cutting in case 1.1 can be found in the following picture.
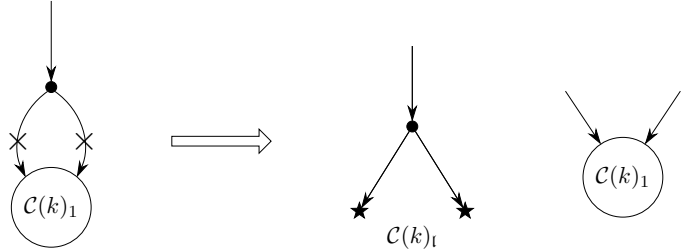
\begin{figure}[H]
    \centering
    \scalebox{0.4}{
    \begin{tikzpicture}[level distance=80pt, sibling distance=100pt]
        \node[] at (0,0) (1) {};
        \node[fillcirc] at (0, -3) (2) {};
        \node[draw, circle, minimum size=1cm, scale=2] at (0,-7) (3) {$\mathcal{C}(k)_1$}; 
        \draw[-{Stealth[length=5mm, width=3mm]}] (1) edge (2);
        \draw[-{Stealth[length=5mm, width=3mm]}, bend left =40] (2) edge (3);
        \draw[-{Stealth[length=5mm, width=3mm]}, bend right =40] (2) edge (3);
        \node[scale =3] at (-1.1,-4.8) {$\times$};
        \node[scale =3] at (1.1,-4.8) {$\times$};
        
        \node[draw, single arrow,
              minimum height=33mm, minimum width=8mm,
              single arrow head extend=2mm,
              anchor=west, rotate=0] at (4,-4.8) {};  
        
        \node[] at (12,-1.5) (11) {} 
            child {node[fillcirc] (12) {} 
                child {node[fillstar] (13) {}}
                child {node[fillstar] (14) {}}
            };
        \draw[-{Stealth[length=5mm, width=3mm]}] (11) -- (12);
        \draw[-{Stealth[length=5mm, width=3mm]}] (12) -- (13);
        \draw[-{Stealth[length=5mm, width=3mm]}] (12) -- (14);
        \node[scale =2] at (12,-8) {$\mathcal{C}(k)_{\mathfrak{l}}$};
        
        \node[] at (17,-3) (21) {};
        \node[] at (21,-3) (22) {};
        \node[draw, circle, minimum size=1cm, scale=2] at (19,-6) (23) {$\mathcal{C}(k)_1$}; 
        \draw[-{Stealth[length=5mm, width=3mm]}] (21) edge (23);
        \draw[-{Stealth[length=5mm, width=3mm]}] (22) edge (23);
    \end{tikzpicture}
    }
        \caption{Examples of cutting in case 1.1}
        \label{fig.step1case1.1}
    \end{figure}

\textbf{Case 1.2.} In this case, $\mathfrak{l}_{fr}$ and $\mathfrak{l}$ are connected to the same node $\mathfrak{n}$. Let $\mathfrak{e}$ be the edge connecting $\mathfrak{n}$ and another interior node. Cut $\mathfrak{e}$ into $\{\mathfrak{e}^{(1)},\mathfrak{e}^{(2)}\}$ and then $\mathcal{C}(k)$ is decomposed into $\mathcal{C}(k)_{\mathfrak{l}}$, $\mathcal{C}(k)_1=\mathcal{C}(k)\backslash \mathcal{C}(k)_{\mathfrak{l}}$. Define $\mathfrak{e}^{(1)}\in \mathcal{C}(k)_{\mathfrak{l}}$ to be fixed legs and $\mathfrak{e}^{(2)}\in \mathcal{C}(k)_1$ to be free legs. 

If $\mathcal{C}(k)_1$ satisfies the property P, define $\mathcal{C}(k+1)=\mathcal{C}(k)_1$.
%to be the output of step $k$ and apply step $k+1$ to $\mathcal{C}(k+1)$

Otherwise by Lemma \ref{lem.normleg} (2), %there exist exactly one free normal leg and at least one fixed leg.
%By Lemma \ref{lem.freeleg}, we can define a new couple $\widehat{\mathcal{C}}$ such that the free normal leg becomes fixed and the fixed leg become free and we define $\mathcal{C}(k+1)=\widehat{\mathcal{C}}$ in this case.
%According to Lemma \ref{lem.normleg} (2), $\mathcal{C}(k)_1$ contains at least one normal legs and two legs in $\mathcal{C}(k)_1$. 
$\mathfrak{e}^{(2)}$ is the only normal legs. Using Lemma \ref{lem.freeleg} (2), we may construct a new couple $\mathcal{C}(k+1)$ by assigning $\mathfrak{e}^{(2)}$ to be fixed and another leg to be free. 
%In any other case $\mathcal{C}(k)_1$ contains a fixed normal leg and we define $\mathcal{C}(k+1)=\mathcal{C}(k)_1$. 
Finally define $\mathcal{C}(k+1)$ to be the output of step $k$ and apply step $k+1$ to $\mathcal{C}(k+1)$. 

Examples of cutting in case 1.2 can be found in the following picture.
\begin{figure}[H]
    \centering
    \scalebox{0.4}{
    \begin{tikzpicture}[level distance=80pt, sibling distance=100pt]
        \node[] at (-2,0) (0) {};
        \node[fillstar] at (1.8,-0.2) (1) {};
        \node[fillcirc] at (0, -3) (2) {};
        \node[draw, circle, minimum size=1cm, scale=2] at (0,-7) (3) {$\mathcal{C}(k)_1$}; 
        \draw[-{Stealth[length=5mm, width=3mm]}] (0) edge (2);
        \draw[-{Stealth[length=5mm, width=3mm]}] (1) edge (2);
        \draw[-{Stealth[length=5mm, width=3mm]}] (2) edge (3);
        \node[scale =3] at (0,-4.4) {$\times$};
        
        \node[draw, single arrow,
              minimum height=33mm, minimum width=8mm,
              single arrow head extend=2mm,
              anchor=west, rotate=0] at (4,-4.8) {};  
        
        \node at (12,-7) (11) {} [grow =90]
            child {node[fillcirc] (12) {} 
                child {node[fillstar, xshift = -0.2cm, yshift = -0.2cm] (13) {}}
                child {node[] (14) {}}
            };
        \draw[{Stealth[length=5mm, width=3mm]}-] (11) -- (12);
        \draw[{Stealth[length=5mm, width=3mm]}-] (12) -- (13);
        \draw[{Stealth[length=5mm, width=3mm]}-] (12) -- (14);
        \node[scale =2] at (12,-8) {$\mathcal{C}(k)_{\mathfrak{l}}$};
        
        \node[fillstar] at (19,-2) (22) {};
        \node[draw, circle, minimum size=1cm, scale=2] at (19,-6) (23) {$\mathcal{C}(k)_1$}; 
        \draw[-{Stealth[length=5mm, width=3mm]}] (22) edge (23);
    \end{tikzpicture}
    }
        \caption{Examples of cutting in case 1.2}
        \label{fig.step1case1.2}
    \end{figure}
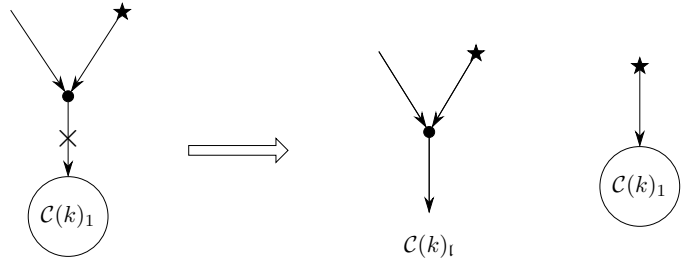

\textbf{Case 2 of step $k$.} Let the two connected components of $\mathcal{C}(k)\backslash \mathcal{C}(k)_{\mathfrak{l}}$ be $\mathcal{C}(k)_2$ and $\mathcal{C}(k)_3$. Let $\mathfrak{e}_{2}$, $\mathfrak{e}_{3}$ be the two edges that connect $\mathfrak{l}$ and  $\mathcal{C}(k)_2$, $\mathcal{C}(k)_3$ respectively. Cut
$\mathfrak{e}_{2}$, $\mathfrak{e}_{3}$ into $\{\mathfrak{e}_{2}^{(1)},\mathfrak{e}_{3}^{(1)}\}\subseteq \mathcal{C}(k)_{\mathfrak{l}}$ and $\mathfrak{e}_{2}^{(2)}\in \mathcal{C}(k)_2$, $\mathfrak{e}_{3}^{(2)}\in \mathcal{C}(k)_3$. 
By Lemma \ref{lem.normleg}, $\mathcal{C}(k)_2$, $\mathcal{C}(k)_3$ contain at least one normal leg and two legs. By symmetry, we can just consider $\mathcal{C}(k)_2$. 

If $\mathcal{C}(k)_2$ contains free legs, define $\mathfrak{e}_{2}^{(2)}\in \mathcal{C}(k)_2$ to be fixed, otherwise define $\mathfrak{e}_{2}^{(2)}$ to be free. 

In the case that $\mathcal{C}(k)_2$ contains free legs, define the output $\mathcal{C}(k+1)$ to be $\mathcal{C}(k)_2$ or $\mathcal{C}(k)_3$ and apply step $k+1$ to them separately. 

In the case that $\mathcal{C}(k)_2$ contains no free legs, $\mathfrak{e}_{2}^{(2)}$ is the only normal legs and it is defined to be free. Use Lemma \ref{lem.freeleg} (2) to construct a new couple $\widehat{\mathcal{C}}_2$ by assigning $\mathfrak{e}_2^{(2)}$ to be fixed and another leg to be free. Then define $\mathcal{C}(k+1)$ to be $\widehat{\mathcal{C}}_2$ or $\widehat{\mathcal{C}}_3$ and apply step $k+1$ to them separately. 

Examples of cutting in case 2 can be found in the following picture.
    \begin{figure}[H]
    \centering
    \scalebox{0.4}{
    \begin{tikzpicture}[level distance=80pt, sibling distance=100pt]
    
    \node[] at (0,0) (1) {};
        \node[fillcirc] at (0, -3) (2) {};
        \node[draw, circle, minimum size=1cm, scale=2] at (-3,-7) (3) {$\mathcal{C}(k)_2$};
        \node[draw, circle, minimum size=1cm, scale=2] at (3,-7) (4) {$\mathcal{C}(k)_3$};
        \draw[-{Stealth[length=5mm, width=3mm]}] (1) edge (2);
        \draw[-{Stealth[length=5mm, width=3mm]}] (2) edge (3);
        \draw[-{Stealth[length=5mm, width=3mm]}] (2) edge (4);
        \node[scale =3, rotate = 44] at (-1.18,-4.5) {$\times$};
        \node[scale =3, rotate = 44] at (1.18,-4.5) {$\times$};
        
        \node[draw, single arrow,
              minimum height=33mm, minimum width=8mm,
              single arrow head extend=2mm,
              anchor=west, rotate=0] at (5,-4.5) {};  
        
        \node[] at (12,-1.5) (11) {} 
            child {node[fillcirc] (12) {} 
                child {node[fillstar] (13) {}}
                child {node[fillstar] (14) {}}
            };
        \draw[-{Stealth[length=5mm, width=3mm]}] (11) -- (12);
        \draw[-{Stealth[length=5mm, width=3mm]}] (12) -- (13);
        \draw[-{Stealth[length=5mm, width=3mm]}] (12) -- (14);
        \node[scale =2] at (12,-8) {$\mathcal{C}(k)_{\mathfrak{l}}$};
        
        \node[] at (18,-1.5) (11) {}; 
        \node[draw, circle, minimum size=1cm, scale=2] at (18,-6.5) (12) {$\mathcal{C}(k)_2$}; 
        \node[] at (23,-1.5) (13) {}; 
        \node[draw, circle, minimum size=1cm, scale=2] at (23,-6.5) (14) {$\mathcal{C}(k)_3$};
        \draw[-{Stealth[length=5mm, width=3mm]}] (11) -- (12);
        \draw[-{Stealth[length=5mm, width=3mm]}] (13) -- (14);
    \end{tikzpicture}
    }
        \caption{Examples of cutting in case 2}
        \label{fig.step1case2}
    \end{figure}
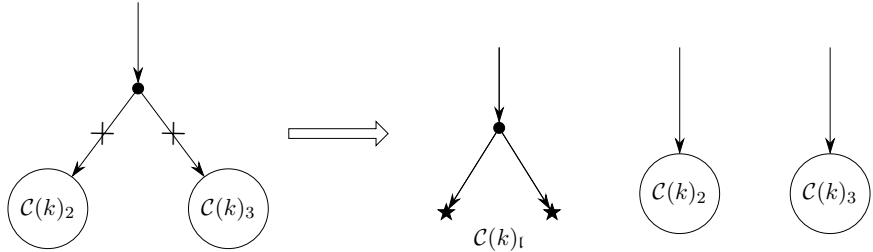
\end{mdframed}

(1), (3) are true by definition. (2) is true because
$\#Eq(\mathcal{C})=\#Eq(\widehat{\mathcal{C}})$ as explained in step $0$ of the algorithm. Since in step $0$, we just replace a fixed leg by a free leg, $\mathcal{C}_{\text{norm}}$ should not change, so we get $\mathcal{C}_{\text{norm}}(1)=\mathcal{C}_{\text{norm}}$. 

The only non-trivial part is (4), which is a corollary of Lemma \ref{lem.normleg} below.

Therefore, we complete the proof.
\end{proof}

\begin{lem}\label{lem.freeleg} %Free leg moving lemma
Given a connected couple $\mathcal{C}$ with multiple legs, then we have the following conclusions.

(1) Let $\{k_{\mathfrak{l}_i}\}_{i=1,\cdots,n_{\text{leg}}}$ be the variables corresponding to legs in couple $\mathcal{C}$. Let $\iota_{\mathfrak{e}}$ be the same as \eqref{eq.iotadef}, then $Eq(\mathcal{C})$ implies the following momentum conservation equation
\begin{equation}\label{eq.momentumconservation}
    \sum_{i=1}^{n_{\text{leg}}} \iota_{\mathfrak{l}_i}k_{\mathfrak{l}_i}=0,
\end{equation}

(2) Assume that there is exactly one free leg $\mathfrak{l}_{i_0}$ in $\mathcal{C}$ and all other variables $\{k_{\mathfrak{l}_{i}}\}_{i\ne i_0}$ corresponding to fix legs are fixed to be constants $\{c_{\mathfrak{l}_{i}}\}_{i\ne i_0}$. For any $i_{1}=1,\cdots,n_{\text{leg}}$, we can construct a new couple $\widehat{\mathcal{C}}$ by replacing the $i_0$ leg by a fix leg and $i_1$ leg by a free leg. If $i\ne i_0, i_1$, fix $k_{\mathfrak{l}_{i}}$ to be the constant $c_{\mathfrak{l}_{i}}$, if $i=i_0$, fix $k_{\mathfrak{l}_{i_0}}$ to be the constant $-\iota_{\mathfrak{l}_{i_0}}\sum_{i\ne i_0} \iota_{\mathfrak{l}_i}k_{\mathfrak{l}_i}$. Under the above assumptions, we have
\begin{equation}
    Eq(\mathcal{C}, \{c_{\mathfrak{l}_{i}}\}_{i\ne i_0})=Eq\left(\widehat{\mathcal{C}}, \{c_{\mathfrak{l}_{i}}\}_{i\ne i_0, i_1}\cup \{-\iota_{\mathfrak{l}_{i_0}}\sum_{i\ne i_0} \iota_{\mathfrak{l}_i}k_{\mathfrak{l}_i}\}\right).
\end{equation}

(3) Assume that there is no free leg in $\mathcal{C}$ and all $\{k_{\mathfrak{l}_{i}}\}_{i\ne i_0}$ are fixed to be constants $\{c_{\mathfrak{l}_{i}}\}_{i\ne i_0}$. For any $i_{1}=1,\cdots,n_{\text{leg}}$, we can construct a new couple $\widehat{\mathcal{C}}$ by replacing the $i_0$ leg by a free leg. Then we have
\begin{equation}
    Eq(\mathcal{C}, \{c_{\mathfrak{l}_{i}}\}_{i})=Eq(\widehat{\mathcal{C}}, \{c_{\mathfrak{l}_{i}}\}_{i\ne i_0}).
\end{equation}

(4) If the couple $\mathcal{C}$ contains any leg, then it contains at least two legs.

\end{lem}
\begin{proof}
We first prove (1). Given a node $\mathfrak{n}$ and an edge $\mathfrak{e}$ connected to it, we define $\iota_{\mathfrak{e}}(\mathfrak{n})$ by the following rule
\begin{equation}
    \iota_{\mathfrak{e}}(\mathfrak{n})=\begin{cases}
        +1 \qquad \textit{if $\mathfrak{e}$ pointing towards $\mathfrak{n}$}
        \\
        -1 \qquad  \textit{if $\mathfrak{e}$ pointing outwards from $\mathfrak{n}$}
    \end{cases}
\end{equation}
For a leg $\mathfrak{l}$, since it is connected to just one node, we may omit the $(\mathfrak{n})$ and just write $\iota_{\mathfrak{l}}$ as in the statement of the lemma.

For each node $\mathfrak{n}$, let $\mathfrak{e}_1(\mathfrak{n})$, $\mathfrak{e}_2(\mathfrak{n})$, $\mathfrak{e}(\mathfrak{n})$ be the three edges connected to it. For each edge $\mathfrak{e}$, let $\mathfrak{n}_1(\mathfrak{e})$, $\mathfrak{n}_2(\mathfrak{e})$ be the two nodes connected to it. Then we know that $\iota_{\mathfrak{e}}(\mathfrak{n}_1(\mathfrak{e}))+\iota_{\mathfrak{e}}(\mathfrak{n}_2(\mathfrak{e}))$, since $\mathfrak{n}_1(\mathfrak{e})$ and $\mathfrak{n}_2(\mathfrak{e})$ have the opposite direction. 

Since $k_{\mathfrak{e}}$ satisfy $Eq(\mathcal{C})$, by \eqref{eq.momentumconservationunit}, we get $\iota_{\mathfrak{e}_1(\mathfrak{n})}(\mathfrak{n})k_{\mathfrak{e}_1(\mathfrak{n})}+\iota_{\mathfrak{e}_2(\mathfrak{n})}(\mathfrak{n})k_{\mathfrak{e}_2(\mathfrak{n})}+\iota_{\mathfrak{e}(\mathfrak{n})}(\mathfrak{n})k_{\mathfrak{e}(\mathfrak{n})}=0$. Summing over $\mathfrak{n}$ gives 
\begin{equation}
\begin{split}
    0=&\sum_{\mathfrak{n}\in \mathcal{C}}\iota_{\mathfrak{e}_1(\mathfrak{n})}(\mathfrak{n})k_{\mathfrak{e}_1(\mathfrak{n})}+\iota_{\mathfrak{e}_2(\mathfrak{n})}(\mathfrak{n})k_{\mathfrak{e}_2(\mathfrak{n})}+\iota_{\mathfrak{e}(\mathfrak{n})}(\mathfrak{n})k_{\mathfrak{e}(\mathfrak{n})}
    \\
    =& \sum_{\mathfrak{e}\text{ is not a leg}} 
    (\iota_{\mathfrak{e}}(\mathfrak{n}_1(\mathfrak{e}))+\iota_{\mathfrak{e}}(\mathfrak{n}_2(\mathfrak{e}))) k_{\mathfrak{e}}+ \sum_{\mathfrak{l}\text{ is a leg}} 
    \iota_{\mathfrak{l}} k_{\mathfrak{l}}
    \\
    =& \sum_{i=1}^{n_{\text{leg}}} \iota_{\mathfrak{l}_i}k_{\mathfrak{l}_i}
\end{split}
\end{equation}
This proves \eqref{eq.momentumconservation} and thus proves (1).

Now we prove (2). Since in $Eq\left(\widehat{\mathcal{C}}, \{c_{\mathfrak{l}_{i}}\}_{i\ne i_0, i_1}\cup \{-\iota_{\mathfrak{l}_{i_0}}\sum_{i\ne i_0} \iota_{\mathfrak{l}_i}k_{\mathfrak{l}_i}\}\right)$, $\{k_{\mathfrak{l}_{i}}\}_{i\ne i_0, i_1}$ are fixed to be constants $\{c_{\mathfrak{l}_{i}}\}_{i\ne i_0, i_1}$ and $k_{\mathfrak{l}_{i_0}}$ is fixed to be the constant $-\iota_{\mathfrak{l}_{i_0}}\sum_{i\ne i_0} \iota_{\mathfrak{l}_i}k_{\mathfrak{l}_i}$, by \eqref{eq.momentumconservation}, we know that 
\begin{equation}
    \iota_{\mathfrak{l}_{i_0}}\left(-\iota_{\mathfrak{l}_{i_0}}\sum_{i\ne i_0} \iota_{\mathfrak{l}_i}k_{\mathfrak{l}_i}\right)+ \iota_{\mathfrak{l}_{i_1}}k_{\mathfrak{l}_{i_1}}+\sum_{i\ne i_0, i_1} \iota_{\mathfrak{l}_i}c_{\mathfrak{l}_i}=0.
\end{equation}

This implies that $k_{\mathfrak{l}_{i_1}}=c_{\mathfrak{l}_{i_1}}$ in $Eq(\widehat{\mathcal{C}})$. Therefore, equations in $Eq(\widehat{\mathcal{C}})$ automatically imply $k_{\mathfrak{l}_{i_1}}=c_{\mathfrak{l}_{i_1}}$. Notice that whether or not containing $k_{\mathfrak{l}_{i_1}}=c_{\mathfrak{l}_{i_1}}$ is the only difference between $Eq(\mathcal{C})$ and $Eq(\widehat{\mathcal{C}})$. We conclude that $Eq(\mathcal{C})=Eq(\widehat{\mathcal{C}})$. We thus complete the proof of (2).

The proof of (3) is similar to (2). Whether or not containing $k_{\mathfrak{l}_{i_0}}=c_{\mathfrak{l}_{i_0}}$ is the only difference between $Eq(\mathcal{C})$ and $Eq(\widehat{\mathcal{C}})$. But if $\{k_{\mathfrak{l}_{i}}\}_{i\ne i_0}$ are fixed to be constants $\{c_{\mathfrak{l}_{i}}\}_{i\ne i_0}$, by momentum conservation we know that 
\begin{equation}
     k_{\mathfrak{l}_{i_0}}=-\iota_{\mathfrak{l}_{i_0}}\sum_{i\ne i_0} \iota_{\mathfrak{l}_i}c_{\mathfrak{l}_i}.
\end{equation}
Therefore, $k_{\mathfrak{l}_{i_0}}$ is fixed to be the constant $-\iota_{\mathfrak{l}_{i_0}}\sum_{i\ne i_0} \iota_{\mathfrak{l}_i}c_{\mathfrak{l}_i}$ in $Eq(\widehat{\mathcal{C}})$ and we conclude that $Eq(\mathcal{C})=Eq(\widehat{\mathcal{C}})$. We thus complete the proof of (3).

If (4) is wrong, then $\mathcal{C}$ just has one leg $\mathfrak{l}$. By \eqref{eq.momentumconservation}, $k_{\mathfrak{l}}=0$. This contradicts with $k_{\mathfrak{l},x}\ne 0$ in \eqref{eq.kappa}.
\end{proof}

% \begin{lem}\label{lem.freeleg} %Free leg moving lemma
% Given a connected couple $\mathcal{C}$ with multiple legs, 

% (1) Moving

% (2) adding

% (3) momentum conservation

% \end{lem}

\begin{lem}\label{lem.normleg} %Normal leg lemma \textbf{need connectivity}
(1) The output $\mathcal{C}(k+1)$ of step $k$ of the cutting algorithm satisfies the property P.

(2) All the intermediate results $\mathcal{C}(k)_1$, $\mathcal{C}(k)_2$, $\mathcal{C}(k)_3$ satisfy the weak property P: either they satisfy the property P or they contain exactly one free normal leg and at least one fixed leg. Notice that the weak property implies that these couples contain at least one normal legs and two legs.
\end{lem}
\begin{proof}
We prove the following stronger result by induction.

\medskip

\textit{Claim.} Assume that the couple $\mathcal{C}=\mathcal{C}(T,T,p)$ is the input the cutting algorithm. Then for any $k$, there exist a finite number of trees disjoint subtrees $T^{(k)}_1$, $T^{(k)}_2$, $\cdots$, $T^{(k)}_{m^{(k)}}$ of the two copies of $T$ which satisfy the following property.

(1) Let $\text{leaf}(k)$ be the set of all leaves of these subtrees. Assume that $p$ induce a pairing $p|_{\text{leaf}_1(k)}$ of a subset $\text{leaf}_1(k)\subseteq\text{leaf}(k)$. Apply a similar construction to Definition \ref{def.conple} we can construct a couple from $T^{(k)}_1$, $T^{(k)}_2$, $\cdots$, $T^{(k)}_{m^{(k)}}$ from $p$ and this couple equals exactly to $\mathcal{C}(k)$.

(2) The root legs of $T^{(k)}_1$, $T^{(k)}_2$, $\cdots$, $T^{(k)}_{m^{(k)}}$ are exactly the normal legs of $\mathcal{C}(k)$. The edges in $\text{leaf}_2(k)=\text{leaf}(k)\backslash \text{leaf}_1(k)$ are exactly the leaf legs (legs that are leaf edges) of $\mathcal{C}(k)$.

(3) $\mathcal{C}(k)$ satisfies the property P.

(4) All the intermediate results  $\mathcal{C}(k)_1$, $\mathcal{C}(k)_2$, $\mathcal{C}(k)_3$ satisfy the weak property P.

\medskip

We show that $\mathcal{C}(0)$ satisfies (1) -- (4) of the above claim.

Let $\mathcal{C}$ be the input couple of step $0$ obtained by pairing two copies of $T$. In step $0$, $\mathcal{C}(0)$ is obtained by replacing a fixed leg by a free leg. Therefore, if we define $T^{(0)}_{1}=T$ and $T^{(0)}_{2}$ to be the tree obtained by replacing the fixed root leg in $T$ by a free leg, then $\mathcal{C}(0)$ is the couple obtained by pairing $T^{(0)}_{1}$ and $T^{(0)}_{2}$. Here the pairing is $p$ and $\text{leaf}_1(0)=\text{leaf}(0)$ and $\text{leaf}_2(0)=\emptyset$. Therefore, (1) is true for $\mathcal{C}(0)$.

The two fixed legs in $\mathcal{C}$ are all normal edges, because they come from the two root legs which belong to $T_{\text{in}}$. Therefore, the two legs in $\mathcal{C}(1)$ are also normal. Since the two legs in $\mathcal{C}(1)$ come from the root legs in $T^{(0)}_{1}$ and $T^{(0)}_{2}$, (2) is also true.

Since in step $0$, we replace a fixed leg by a free leg, $\mathcal{C}(0)$ contains exactly one free and one fixed leg which are all normal. Therefore, the output $\mathcal{C}(0)$ of step $0$ satisfies property P and (3) is proved.

In step $0$, (4) does not need any proof.

\medskip

Assume that $\mathcal{C}(k)$ satisfies (1) -- (4) of the above claim, then we prove the same for $\mathcal{C}(k+1)$.

Remember that in step $k$ of the algorithm, the input is $\mathcal{C}(k)$. The induction assumption implies that $\mathcal{C}(k)$ is obtained by pairing $T^{(k)}_1$, $T^{(k)}_2$, $\cdots$, $T^{(k)}_{m^{(k)}}$. There are several different cases in step $k$ and we treat them separately.

\underline{In case 1.1 of the cutting algorithm}, we cut $c(\mathfrak{l})$ into $\{\mathfrak{e}_{i}^{(1)}\}_{i=1,2}$ and $\{\mathfrak{e}_{i}^{(2)}\}_{i=1,2}$. By (2) $\mathfrak{l}$ is the root leg of some subtree $T^{(k)}_{j_0}$. Assume that in $T^{(k)}_{j_0}$ the two subtrees of the root are $T^{(k)}_{j_0,1}$, $T^{(k)}_{j_0,2}$. After cutting $c(\mathfrak{l})$, $T^{(k)}_{j_0}$ becomes two trees $T^{(k)}_{j_0,1}$, $T^{(k)}_{j_0,2}$ and $T^{(k)}_{j}$ $(j\ne j_0)$ do not change. We treat three different case separately.

\underline{Case 1.1 (i).} (Both $T^{(k)}_{j_0,1}$ and $T^{(k)}_{j_0,2}$ are one node trees.) In this case all edges in $\{\mathfrak{e}_{1}^{(1)}, \mathfrak{e}_{2}^{(1)}, \mathfrak{e}_{1}^{(2)}, \mathfrak{e}_{2}^{(2)}\}$ are leaf edges of some trees in $\{T_{j}^{(k)}\}$. Define $T^{(k+1)}_{j}=T^{(k)}_{j}$ for $j< j_0$ and $T^{(k+1)}_{j}=T^{(k)}_{j+1}$ for $j> j_0$, then $\mathcal{C}(k)_1$ can be constructed from these trees with $\text{leaf}_1(k+1)=\text{leaf}_1(k)\backslash\{\mathfrak{e}_{1}^{(1)}, \mathfrak{e}_{2}^{(1)}, \mathfrak{e}_{1}^{(2)}, \mathfrak{e}_{2}^{(2)}\}$. Remember that in the case 1.1 of the algorithm, the final result $\mathcal{C}(k+1)$ is obtained by replace (or not) some free normal leg by fixed normal leg in $\mathcal{C}(k)_1$. Therefore, if we change some free root leg of $\{T^{(k)}_{j}\}$ to be fixed, then $\mathcal{C}(k+1)$ can also be obtained from pairing these trees. Therefore, (1) is true for $\mathcal{C}(k+1)$ with $\text{leaf}_1(k+1)=\text{leaf}_1(k)\backslash\{\mathfrak{e}_{1}^{(1)}, \mathfrak{e}_{2}^{(1)}, \mathfrak{e}_{1}^{(2)}, \mathfrak{e}_{2}^{(2)}\}$.

Since both the sets of root legs and normal legs are deleted by one element $\mathfrak{l}$ and in step $k-1$ they are the same, they continue to be the same in step $k$. In case 1.1 (i), $\text{leaf}_2(k+1)=\text{leaf}_2(k)\cup \{\mathfrak{e}_{i}^{(2)}\}_{i=1,2}$ and the set of leaf legs in $\mathcal{C}(k)$ also changes in this way, so they continue to be the same in step $k$. Therefore, (2) is true.

We prove (4) by contradiction. The weak property P is equivalent to that $\mathcal{C}(k)_i$ contains exactly one free leg and at least one normal leg. In case 1.1 (i), $i=1$, assume that the weak property P is not true, then either $\mathcal{C}(k)_1$ does not any contain normal leg or the number of free leg is not 1. If $\mathcal{C}(k)_1\ne \emptyset$, then the number of $\{T^{(k)}_{j}\}$ is not zero. Because the root of $\{T^{(k)}_{j}\}$ are normal legs, there exists at least one normal leg in $\mathcal{C}(k)_1$. By our hypothesis that weak property P is wrong, the number of free leg is not 1. Since $\mathcal{C}(k)_1$ is obtained from $\mathcal{C}(k)$ by cutting a fixed leg, the number of free legs in $\mathcal{C}(k)_1$ equals to  $\mathcal{C}(k)$ which is one. Therefore, we find a contradiction.

Notice that in the case 1.1 of the algorithm we change the free normal leg in $\mathcal{C}(k)_1$ to be fixed if $\mathcal{C}(k)_1$ does not satisfy the property P. The result $\mathcal{C}(k+1)$ must satisfy the property P, so we proves (3). 

\underline{Case 1.1 (ii).} (One of $T^{(k)}_{j_0,1}$ or $T^{(k)}_{j_0,2}$ is an one node tree.) Without loss of generality assume that $T^{(k)}_{j_0,2}$ is an one node tree. In this case, $\mathfrak{e}_{2}^{(1)}$, $\mathfrak{e}_{2}^{(2)}$ are leaf edges of some trees in $\{T_{j}^{(k)}\}$. Define $T^{(k+1)}_{j}=T^{(k)}_{j}$ for $j\ne j_0$ and $T^{(k+1)}_{j_0}=T^{(k)}_{j_0,1}$, then $\mathcal{C}(k)_1$ can be constructed from these trees with  $\text{leaf}_1(k+1)=\text{leaf}_1(k)\backslash\{\mathfrak{e}_{2}^{(1)}, \mathfrak{e}_{2}^{(2)}\}$. By the same reason as in case 1.1 (i), (1) is true for $\mathcal{C}(k+1)$ with $\text{leaf}_1(k+1)=\text{leaf}_1(k)\backslash\{\mathfrak{e}_{2}^{(1)}, \mathfrak{e}_{2}^{(2)}\}$.

By the same reason as in case 1.1 (i), the set of root legs and the set of normal legs continue to be the same in step $k$. In case 1.1 (ii), $\text{leaf}_2(k+1)=\text{leaf}_2(k)\cup \{ \mathfrak{e}_{2}^{(2)}\}$ and the set of leaf legs also changes in this way, so they continue to be the same in step $k$. Therefore, (2) is true.

The proof of (3), (4) in case 1.1 (ii) is the same as that in case 1.1 (i).

\underline{Case 1.1 (iii).} (Neither $T^{(k)}_{j_0,1}$ or $T^{(k)}_{j_0,2}$ is an one node tree.) In this case, no edge in $\{\mathfrak{e}_{1}^{(1)}, \mathfrak{e}_{2}^{(1)}, \mathfrak{e}_{1}^{(2)}, \mathfrak{e}_{2}^{(2)}\}$ is leaf edge of $T_{j}^{(k)}$. Define $T^{(k+1)}_{j}=T^{(k)}_{j}$ for $j\le j_0$, $T^{(k+1)}_{j}=T^{(k)}_{j-1}$ for $j\ge j_0+2$, $T^{(k+1)}_{j_0}=T^{(k)}_{j_0,1}$ and $T^{(k+1)}_{j_0+1}=T^{(k)}_{j_0,2}$, then $\mathcal{C}(k)_1$ can be constructed from these trees with  $\text{leaf}_1(k+1)=\text{leaf}_1(k)$. By the same reason as in case 1.1 (i), (1) is true for $\mathcal{C}(k+1)$ with $\text{leaf}_1(k+1)=\text{leaf}_1(k)$.

By the same reason as in case 1.1 (i), the set of root legs and the set of normal legs continue to be the same in step $k$. In case 1.1 (iii), $\text{leaf}_2(k+1)=\text{leaf}_2(k)$ and the set of leaf legs also changes in this way, so they continue to be the same in step $k$. Therefore, (2) is true.

The proof of (3), (4) in case 1.1 (iii) is the same as that in case 1.1 (i).

\underline{In case 1.2 of the cutting algorithm}, we cut $\mathfrak{e}$ into $\{\mathfrak{e}^{(1)},\mathfrak{e}^{(2)}\}$. By (2), $\mathfrak{l}$ and $\mathfrak{l}_{fr}$ are the root leg and leaf of some subtrees $T^{(k)}_{j_0}$. Assume that in $T^{(k)}_{j_0}$ one subtree of the root is $T^{(k)}_{j_0,1}$ and the other one is a one node tree with edge $\mathfrak{l}_{fr}$. After cutting $c(\mathfrak{l})$, $T^{(k)}_{j_0}$ becomes $T^{(k)}_{j_0,1}$ and $T^{(k)}_{j}$ $(j\ne j_0)$ do not change. 
We treat two different case separately.

\underline{Case 1.2 (i).} ($T^{(k)}_{j_0,1}$ is an one node trees.) In this case $\mathfrak{e}^{(1)}$, $\mathfrak{e}^{(2)}$ are leaf edges of some trees in $\{T_{j}^{(k)}\}$. Define $T^{(k+1)}_{j}=T^{(k)}_{j}$ for $j< j_0$ and $T^{(k+1)}_{j}=T^{(k)}_{j+1}$ for $j> j_0$, then $\mathcal{C}(k)_1$ can be constructed from these trees with  $\text{leaf}_1(k+1)=\text{leaf}_1(k)\backslash\{\mathfrak{e}^{(1)}, \mathfrak{e}^{(2)}\}$. By the same reason as in case 1.1 (i), (1) is true for $\mathcal{C}(k+1)$ with $\text{leaf}_1(k+1)=\text{leaf}_1(k)\backslash\{\mathfrak{e}^{(1)}, \mathfrak{e}^{(2)}\}$.

By the same reason as in case 1.1 (i), the set of root legs and the set of normal legs continue to be the same in step $k$. In case 1.2 (i), $\text{leaf}_2(k+1)=\text{leaf}_2(k)\cup \{ \mathfrak{e}^{(2)}\}$ and the set of leaf legs also changes in this way, so they continue to be the same in step $k$. Therefore, (2) is true.

The proof of (3), (4) in case 1.2 (i) is the same as that in case 1.1 (i).

\underline{Case 1.2 (ii).} ($T^{(k)}_{j_0,1}$ is not an one node trees.) In this case $\mathfrak{e}^{(1)}$, $\mathfrak{e}^{(2)}$ are not leaf edges of any trees in $\{T_{j}^{(k)}\}$. Define $T^{(k+1)}_{j}=T^{(k)}_{j}$ for $j\ne j_0$ and $T^{(k+1)}_{j_0}=T^{(k)}_{j_0,1}$, then $\mathcal{C}(k)_1$ can be constructed from these trees with  $\text{leaf}_1(k+1)=\text{leaf}_1(k)$. By the same reason as in case 1.1 (i), (1) is true for $\mathcal{C}(k+1)$ with $\text{leaf}_1(k+1)=\text{leaf}_1(k)$.

By the same reason as in case 1.1 (i), the set of root legs and the set of normal legs continue to be the same in step $k$. In case 1.2 (ii), $\text{leaf}_2(k+1)=\text{leaf}_2(k)$ and the set of leaf legs also changes in this way, so they continue to be the same in step $k$. Therefore, (2) is true.

The proof of (3), (4) in case 1.2 (ii) is the same as that in case 1.1 (i).

\underline{In case 2 of the cutting algorithm}, we cut $c(\mathfrak{l})=\{\mathfrak{e}_2,\mathfrak{e}_3\}$ into $\{\mathfrak{e}_{2}^{(1)}, \mathfrak{e}_{3}^{(1)}\}$ and $\{\mathfrak{e}_{2}^{(2)}, \mathfrak{e}_{3}^{(2)}\}$. By (2), $\mathfrak{l}$ is the root leg of some subtree $T^{(k)}_{j_0}$. Assume that in $T^{(k)}_{j_0}$ the two subtrees of the root are $T^{(k)}_{j_0,1}$, $T^{(k)}_{j_0,2}$. After cutting $c(\mathfrak{l})$, $T^{(k)}_{j_0}$ becomes two trees $T^{(k)}_{j_0,1}$, $T^{(k)}_{j_0,2}$ and $T^{(k)}_{j}$ $(j\ne j_0)$ do not change. Since $\mathcal{C}(k)_{2}$ and $\mathcal{C}(k)_{3}$ are disjoint, for each $j\ne j_0$, $T^{(k)}_{j}$ should be a subset of one of these couples. Without loss of generality we assume that $j_0=1$ and $T^{(k)}_2, \cdots, T^{(k)}_{m^{(k)'}}\subseteq \mathcal{C}(k)_{2}$ and $T^{(k)}_{m^{(k)'}+1}, \cdots, T^{(k)}_{m^{(k)}}\subseteq \mathcal{C}(k)_{3}$. Since the output $\mathcal{C}(k+1)$ either comes from $\mathcal{C}(k)_{2}$ or $\mathcal{C}(k)_{3}$, without loss of generality we assume that the output comes from $\mathcal{C}(k)_{2}$. We treat two different case separately.

\underline{Case 2 (i).} ($T^{(k)}_{j_0,1}$ is an one node trees.) In this case, all edges in $\{\mathfrak{e}_{2}^{(1)}, \mathfrak{e}_{2}^{(2)}\}$ are leaf edges of some trees in $\{T_{j}^{(k)}\}$. Define $T^{(k+1)}_{j}=T^{(k)}_{j+1}$, then $\mathcal{C}(k)_2$ can be constructed from $T^{(k)}_2, \cdots, T^{(k)}_{m^{(k)'}}$ with $\text{leaf}_1(k+1)=\{\text{all leaves in }T^{(k)}_2, \cdots,$ $T^{(k)}_{m^{(k)'}}\}$. Remember that in the case 2 of the algorithm, the final result $\mathcal{C}(k+1)$ is obtained by replace (or not) some free normal leg by fixed normal leg in $\mathcal{C}(k)_2$, so by the same argument as in case 1.1 (i), (1) is true for $\mathcal{C}(k+1)$.

By the same reason as in case 1.1 (i), the set of root legs and the set of normal legs continue to be the same in step $k$. In case 2 (i), in the output $\mathcal{C}(k+1)$ coming from $\mathcal{C}(k)_2$, $\text{leaf}_2(k+1)=(\text{leaf}_1(k+1)\cap \mathcal{C}(k)_2)\cup \{ \mathfrak{e}_{2}^{(2)}\}$ and the set of leaf legs also changes in this way, so they continue to be the same in step $k$. Therefore, (2) is true.

The proof of (3), (4) in case 2 (i) is the same as that in case 1.1 (i).

\underline{Case 2 (ii).} ($T^{(k)}_{j_0,1}$ is not an one node trees.) In this case $\{\mathfrak{e}_{2}^{(1)}$, $\mathfrak{e}_{2}^{(2)}\}$ are not leaf edges of any trees in $\{T_{j}^{(k)}\}$. Define $T^{(k+1)}_{j}=T^{(k)}_{j}$ for $j>1$ and $T^{(k+1)}_{1}=T^{(k)}_{j_0,1}$, then $\mathcal{C}(k)_2$ can be constructed from $T^{(k+1)}_1, \cdots, T^{(k+1)}_{m^{(k)'}}$ with $\text{leaf}_1(k+1)=\{\text{all leaves in }T^{(k+1)}_1, \cdots, T^{(k+1)}_{m^{(k)'}}\}$. By the same reason as in case 1.1 (i), (1) is true for $\mathcal{C}(k+1)$.

By the same reason as in case 1.1 (i), the set of root legs and the set of normal legs continue to be the same in step $k$. In case 2 (ii), in the output $\mathcal{C}(k+1)$ coming from $\mathcal{C}(k)_2$, $\text{leaf}_2(k+1)=(\text{leaf}_1(k+1)\cap \mathcal{C}(k)_2)$ and the set of leaf legs also changes in this way, so they continue to be the same in step $k$. Therefore, (2) is true.

The proof of (3), (4) in case 2 (i) is the same as that in case 1.1 (i).
\end{proof}

\textbf{Step 3.} In this step, we state Proposition \ref{prop.countingind} which is a stronger version of Proposition \ref{prop.counting} and derive Proposition \ref{prop.counting} from it. In the end, we prove part (1), (2) and the 1 node case of part (3) of Proposition \ref{prop.countingind}.

\begin{prop}\label{prop.countingind}
Let $\mathcal{C}(k)$ be a couple which is the output of step $k$ of the cutting algorithm Proposition \ref{prop.cuttingalgorithm}. For any couple $\mathcal{C}$, let $n(\mathcal{C})$ be the total number of nodes in $\mathcal{C}$ and $n_e(\mathcal{C})$ (resp. $n_{fx}(\mathcal{C})$, $n_{\textit{fr}}(\mathcal{C})$) be the total number of non-leg edges (resp. fixed legs, free legs). We fix $\sigma_{\mathfrak{n}}\in\mathbb{R}$ for each $\mathfrak{n}\in \mathcal{C}(k)$ and $c_{\mathfrak{l}}\in \mathbb{R}$ for each fixed leg $\mathfrak{l}$. Assume that $\alpha$ satisfies \eqref{eq.conditionalpha}. Then we have

(1) We have following relation of $n(\mathcal{C})$, $n_e(\mathcal{C})$, $n_{fx}(\mathcal{C})$ and $n_{\textit{fr}}(\mathcal{C})$
\begin{equation}
    2n_e(\mathcal{C})+n_{fx}(\mathcal{C})+n_{\textit{fr}}(\mathcal{C})=3n(\mathcal{C})
\end{equation}

(2) For any couple $\mathcal{C}$, define $\chi(\mathcal{C})=n_e(\mathcal{C})+n_{\textit{fr}}(\mathcal{C})-n(\mathcal{C})$. Let $c$, $\mathcal{C}_1$, $\mathcal{C}_2$ be the same as in Lemma \ref{lem.Eq(C)cutting} and also assume that $\{\mathfrak{e}_{i}^{(1)}\}$ are free legs and $\{\mathfrak{e}_{i}^{(2)}\}$ are fixed legs. Then 
\begin{equation}
    \chi(\mathcal{C})=\chi(\mathcal{C}_1)+\chi(\mathcal{C}_2).
\end{equation}

(3) If we assume that $\mathcal{C}(k)$ satisfies the property P, then (recall that $Q=L^{d}T^{-1}_{\text{max}}$)
\begin{equation}\label{eq.countingbd3}
\sup_{\{c_{\mathfrak{l}}\}_{\mathfrak{l}}}\#Eq(\mathcal{C}(k),\{c_{\mathfrak{l}}\}_{\mathfrak{l}})\leq L^{O\left(\chi(\mathcal{C}(k))\theta\right)} Q^{\chi(\mathcal{C}(k))}\prod_{\mathfrak{e}\in \mathcal{C}_{\text{norm}}(k)} \kappa^{-1}_{\mathfrak{e}} .
\end{equation}

\end{prop}

\underline{Derivation of Proposition \ref{prop.counting} from Proposition \ref{prop.countingind}:} 
By Proposition \ref{prop.cuttingalgorithm} (2), $\mathcal{C}_{\text{norm}}(1)=\mathcal{C}_{\text{norm}}$ and $\# Eq(\mathcal{C})=\# Eq(\mathcal{C}(1))$. Then by Proposition \ref{prop.countingind} (3), 
\begin{equation}\label{eq.countinglemstep3}
    \# Eq(\mathcal{C})=\# Eq(\mathcal{C}(1))\leq L^{O(\chi(\mathcal{C}(1))\theta)} Q^{\chi(\mathcal{C}(1))}\prod_{\mathfrak{e}\in \mathcal{C}_{\text{norm}}(1)} \kappa^{-1}_{\mathfrak{e}}
    %=L^{O(\theta)} Q^{\chi(\mathcal{C})}\prod_{\mathfrak{e}\in \mathcal{C}_{\text{norm}}} \kappa^{-1}_{\mathfrak{e}}
\end{equation}
Since in $\mathcal{C}(1)$, $n_{fx}(\mathcal{C}(1))=n_{\textit{fr}}(\mathcal{C}(1))=1$, by Proposition \ref{prop.countingind} (1), $2n_e(\mathcal{C}(1))+2=3n(\mathcal{C}(1))$. Using this fact and the definition of $\chi$, we get $\chi(\mathcal{C}(1))=n_e(\mathcal{C}(1))+1-n(\mathcal{C}(1))=n(\mathcal{C}(1))/2$. Substituting this expression of $\chi(\mathcal{C}(1))$ into \eqref{eq.countinglemstep3} proves the conclusion of Proposition \ref{prop.counting}. 

The proof of Proposition \ref{prop.countingind} is the main goal of the rest of the proof.

%In the cutting algorithm, we keep cutting nodes $\mathfrak{n}$ that is connected to normal fixed legs $\mathfrak{l}$ using cut $c(\mathfrak{l})$. Suppose that $\mathcal{C}$ has $n$ nodes, let $\mathcal{C}_{\mathfrak{l}}$ and $\mathcal{C}'$ be the two components after cutting (as in (2) of the cutting algorithm), then  $\mathcal{C}'$ contains $n-1$ nodes and $\mathcal{C}_{\mathfrak{l}}$ contains $1$ nodes. We prove Proposition \ref{prop.countingind} for $\mathcal{C}_{\mathfrak{l}}$, i.e. prove the 1 node case of it.

\underline{Proof Proposition \ref{prop.countingind} (1):} Consider the set $\mathcal{S}=\{(\mathfrak{n}, \mathfrak{e})\in \mathcal{C}: \mathfrak{n} \textit{ is an end point of }\mathfrak{e}\}$, then 
\begin{equation}
\#\mathcal{S}=\sum_{\substack{(\mathfrak{n}, \mathfrak{e})\in \mathcal{C}\\ \mathfrak{n} \textit{ is an end point of }\mathfrak{e}}} 1
\end{equation}

First sum over $\mathfrak{e}$ and then over $\mathfrak{n}$, we get 
\begin{equation}
\#\mathcal{S}=\sum_{\mathfrak{n}}\sum_{\substack{\mathfrak{e}\in \mathcal{C}\\ \mathfrak{n} \textit{ is an end point of }\mathfrak{e}}} 1=\sum_{\mathfrak{n}}\ 3=3n(\mathcal{C}).
\end{equation}
In the second equality, $\sum_{\mathfrak{e}\in \mathcal{C}: \mathfrak{n} \textit{ is an end point of }\mathfrak{e}} 1 =3$ because for each node $\mathfrak{n}$ there are $3$ edges connected to it.

Switch the order of summation, we get 
\begin{equation}
\begin{split}
\#\mathcal{S}=&\sum_{\mathfrak{e} \textit{ is a non-leg edges}}\sum_{\substack{\mathfrak{n}\in \mathcal{C}\\ \mathfrak{n} \textit{ is an end point of }\mathfrak{e}}} 1+\sum_{\mathfrak{e} \textit{ is a leg}}\sum_{\substack{\mathfrak{n}\in \mathcal{C}\\ \mathfrak{n} \textit{ is an end point of }\mathfrak{e}}} 1
\\
=&\sum_{\mathfrak{e} \textit{ is a non-leg edges}} 2+\sum_{\mathfrak{e} \textit{ is a leg}} 1
\\
=& 2n_e(\mathcal{C})+n_{fx}(\mathcal{C})+n_{\textit{fr}}(\mathcal{C})
\end{split}
\end{equation}
In the second equality, $\sum_{\substack{\mathfrak{n}\in \mathcal{C}\\ \mathfrak{n} \textit{ is an end point of }\mathfrak{e}}} 1$ equals to $1$ or $2$ because for each non-leg edge (resp. leg) there are $2$ (resp. 1) nodes connected to it.

Because the value of $\#\mathcal{S}$ does not depend on the order of summation, we conclude that $2n_e(\mathcal{C})+n_{fx}(\mathcal{C})+n_{\textit{fr}}(\mathcal{C})=3n(\mathcal{C})$, which proves Proposition \ref{prop.countingind} (1).

\underline{Proof Proposition \ref{prop.countingind} (2):} Since cutting does change the number of nodes, we have $n(\mathcal{C})=n(\mathcal{C}_1)+n(\mathcal{C}_2)$. Let $c$ be the cut that consists of edges $\{\mathfrak{e}_{i}\}$ and $n(c)$ be the number of edges in $c$. 
%Given a couple $\mathcal{C}$, let $n(\mathcal{C})$, $n_e(\mathcal{C})$, $n_{\textit{fx}}(\mathcal{C})$ and $n_{\textit{fr}}(\mathcal{C})$ be the total number of the number of nodes, non-leg edges, fixed legs and free legs respectively. 
When cutting $\mathcal{C}$ into $\mathcal{C}_1$ and $\mathcal{C}_2$, $n(c)$ non-leg edges are cut into pairs of free and fixed legs, so $n_e(\mathcal{C})=n(\mathcal{C}_1)+n(\mathcal{C}_2)+n(c)$. Because we have $n(c)$ additional free legs after cutting, so $n_{\textit{fr}}(\mathcal{C})=n_{\textit{fr}}(\mathcal{C}_1)+n_{\textit{fr}}(\mathcal{C}_2)-n(c)$. Therefore, we get
\begin{equation}
\begin{split}
    \chi(\mathcal{C})=&n_e(\mathcal{C})+n_{\textit{fr}}(\mathcal{C})-n(\mathcal{C})
    \\
    =&n(\mathcal{C}_1)+n(\mathcal{C}_2)+n(c)+n_{\textit{fr}}(\mathcal{C}_1)+n_{\textit{fr}}(\mathcal{C}_2)-n(c)-(n(\mathcal{C}_1)+n(\mathcal{C}_2))
    \\
    =&(n(\mathcal{C}_1)+n_{\textit{fr}}(\mathcal{C}_1)-n(\mathcal{C}_1))+(n(\mathcal{C}_2)+n_{\textit{fr}}(\mathcal{C}_2)-n(\mathcal{C}_2))
    \\
    =&\chi(\mathcal{C}_1)+\chi(\mathcal{C}_2).
\end{split}
\end{equation}

We complete the proof of (2) of Proposition \ref{prop.countingind}.

\underline{Proof of 1 node case of Proposition \ref{prop.countingind} (3):} We prove the following Lemma \ref{lem.countingbdunit} which is the 1 node case of Proposition \ref{prop.countingind} (3). Recall that $\mathcal{C}_{I}$ and $\mathcal{C}_{II}$ are the two possibilities of the 1 node couple $\mathcal{C}_{\mathfrak{l}}$ in Proposition \ref{prop.cuttingalgorithm} (3).

\begin{lem}\label{lem.countingbdunit}
$\mathcal{C}_{I}$, $\mathcal{C}_{II}$ satisfy the bound \eqref{eq.countingbd3} in Proposition \ref{prop.countingind}. In other words, fix $c_1$ (resp. $c_1$, $c_2$) for the fixed legs of $\mathcal{C}_{I}$ (resp. $\mathcal{C}_{II}$), then we have 
\begin{equation}\label{eq.countingbdunit}
    \# Eq(\mathcal{C}_{I})\leq L^\theta Q\kappa^{-1}_{\mathfrak{e}},\qquad \# Eq(\mathcal{C}_{II})\leq Q^0=1.
\end{equation}
\end{lem}
\begin{proof} Given $c_1$, $c_2$, the equation of $\mathcal{C}_{II}$ is 
\begin{equation}
    \begin{cases}
    k_{\mathfrak{e}_1}+k_{\mathfrak{e}_2}-k_{\mathfrak{e}}=0,\ k_{\mathfrak{e}_1}=c_1,\ k_{\mathfrak{e}}=c_2,\ |k_{\mathfrak{e}x}|\sim \kappa_{\mathfrak{e}}
    \\
    \Lambda_{k_{\mathfrak{e}_1}}+\Lambda_{k_{\mathfrak{e}_2}}-\Lambda_{k_{\mathfrak{e}}}=\sigma_{\mathfrak{n}}+O(T^{-1}_{\text{max}})
    \end{cases}
\end{equation}
It's obvious that there is at most one solution to this system of equations.

Given $c_1$, the equation of $\mathcal{C}_{I}$ is 
\begin{equation}
    \begin{cases}
    k_{\mathfrak{e}_1}+k_{\mathfrak{e}_2}-k_{\mathfrak{e}}=0,\ k_{\mathfrak{e}}=c_1,\ |k_{\mathfrak{e}x}|\sim \kappa_{\mathfrak{e}}
    \\
    \Lambda_{k_{\mathfrak{e}_1}}+\Lambda_{k_{\mathfrak{e}_2}}-\Lambda_{k_{\mathfrak{e}}}=\sigma_{\mathfrak{n}}+O(T^{-1}_{\text{max}})
    \end{cases}
\end{equation}
By Theorem \ref{th.numbertheory1}, the number of solutions of the above system of equations can be bounded by $L^\theta L^dT^{-1}_{\text{max}}|k_{\mathfrak{e}x}|^{-1}\lesssim L^\theta Q\kappa^{-1}_{\mathfrak{e}}$

Therefore, we complete the proof of this lemma.
\end{proof}

\textbf{Step 4.} In this step, we apply the edge cutting algorithm Proposition \ref{prop.cuttingalgorithm} to prove Proposition \ref{prop.countingind} (3) by induction.

If $\mathcal{C}$ has only one node ($n=1$), then $\mathcal{C}$ equals to $\mathcal{C}_{I}$ or $\mathcal{C}_{II}$ and Proposition \ref{prop.countingind} (3) in this case follows from Lemma \ref{lem.countingbdunit}.

Suppose that Proposition \ref{prop.countingind} (3) holds true for couples with number of nodes $\le n-1$. We prove it for couples with number of nodes $n$. 

Given the couple $\mathcal{C}(k)$ which is the output of the $k-1$-th step, apply the cutting algorithm to $\mathcal{C}(k)$. Then according to Proposition \ref{prop.cuttingalgorithm}, $\mathcal{C}(k)$ (2) is decomposed into 2 or 3 components.

In the first case, denote by $\mathcal{C}(k)_{
\mathfrak{l}}$ and $\mathcal{C}(k)_1$ the two components after cutting. In the second case, denote by $\mathcal{C}(k)_{
\mathfrak{l}}$, $\mathcal{C}(k)_2$ and $\mathcal{C}(k)_3$ the three components after cutting. By Proposition \ref{prop.cuttingalgorithm} (4), $\mathcal{C}(k)_1$, $\mathcal{C}(k)_2$, $\mathcal{C}(k)_3$ are couples of nodes $\le n-1$ that satisfy the property P and thus satisfy the assumption of Proposition \ref{prop.countingind} (3). Therefore, the induction assumption is applicable and Proposition \ref{prop.countingind} (3) is true for these three couples.

\textbf{Case 1.} Assume that there are two components after cutting. One example of this case is the left couple in Figure \ref{fig.step1case}.
    
Then applying Lemma \ref{lem.Eq(C)cutting} gives

\begin{equation}
\begin{split}
    \sup_{\{c_{\mathfrak{l}}\}_{\mathfrak{l}}}\#Eq(\mathcal{C}(k),\{c_{\mathfrak{l}}\}_{\mathfrak{l}})\le&
    \sup_{\{c_{\mathfrak{l}_1}\}_{\mathfrak{l}_1\in \text{leg}(\mathcal{C}(k)_{
\mathfrak{l}})} } \# Eq(\mathcal{C}(k)_{
\mathfrak{l}},\{c_{\mathfrak{l}_1}\}) \sup_{\{c_{\mathfrak{l}_2}\}_{\mathfrak{l}_2\in \text{leg}(\mathcal{C}(k)_1)} }\# Eq(\mathcal{C}(k)_1, \{c_{\mathfrak{l}_2}\})
    \\
    \lesssim&  L^{\theta} Q^{\chi(\mathcal{C}_1)} L^{O(\chi(\mathcal{C}_2)\theta)} Q^{\chi(\mathcal{C}_2)} = L^{O((\chi(\mathcal{C}_1)+\chi(\mathcal{C}_2))\theta)} Q^{\chi(\mathcal{C}_1)+\chi(\mathcal{C}_2)} \\
    =& L^{O(\chi(\mathcal{C})\theta)} Q^{\chi(\mathcal{C})}.
\end{split}
\end{equation}

Here the second inequality follows from induction assumption and Lemma \ref{lem.countingbdunit}. The last equality follows from Proposition \ref{prop.countingind} (2).

\textbf{Case 2.} Assume that there are three components after cutting. One example of this case is the right couple in Figure \ref{fig.step1case}.

Then applying Lemma \ref{lem.Eq(C)cutting} gives

\begin{equation}\label{eq.case2expand}
\begin{split}
    \sup_{\{c_{\mathfrak{l}}\}_{\mathfrak{l}}}\#Eq(\mathcal{C}(k),\{c_{\mathfrak{l}}\}_{\mathfrak{l}})\le&
    \sup_{\{c_{\mathfrak{l}_1}\}_{\mathfrak{l}_1\in \text{leg}(\mathcal{C}(k)\backslash\mathcal{C}(k)_2)} } \# Eq(\mathcal{C}(k)\backslash\mathcal{C}(k)_2,\{c_{\mathfrak{l}_1}\}) \sup_{\{c_{\mathfrak{l}_2}\}_{\mathfrak{l}_2\in \text{leg}(\mathcal{C}(k)_2)} }\# Eq(\mathcal{C}(k)_2, \{c_{\mathfrak{l}_2}\})
    \\
    \lesssim&  L^{O(\chi(\mathcal{C}(k)_2)\theta)} Q^{\chi(\mathcal{C}(k)_2)}\sup_{\{c_{\mathfrak{l}_1}\}_{\mathfrak{l}_1\in \text{leg}(\mathcal{C}(k)\backslash\mathcal{C}(k)_2)} } \# Eq(\mathcal{C}(k)\backslash\mathcal{C}(k)_2,\{c_{\mathfrak{l}_1}\}).
\end{split}
\end{equation}
Here in the second inequality, we apply the induction assumption to $\mathcal{C}(k)_2$.

Applying Lemma \ref{lem.Eq(C)cutting} and \eqref{eq.case2expand} gives
\begin{equation}\label{eq.case2expand'}
\begin{split}
    &\sup_{\{c_{\mathfrak{l}}\}_{\mathfrak{l}}}\#Eq(\mathcal{C}(k),\{c_{\mathfrak{l}}\}_{\mathfrak{l}})
    \lesssim  L^{O(\mathcal{C}(k)_2\theta)} Q^{\chi(\mathcal{C}(k)_2)}\sup_{\{c_{\mathfrak{l}_1}\}_{\mathfrak{l}_1\in \text{leg}(\mathcal{C}(k)\backslash\mathcal{C}(k)_2)} } \# Eq(\mathcal{C}(k)\backslash\mathcal{C}(k)_2,\{c_{\mathfrak{l}_1}\})
    \\
    \lesssim& L^{O(\chi(\mathcal{C}(k)_2)\theta)} Q^{\chi(\mathcal{C}(k)_2)}\sup_{\{c_{\mathfrak{l}_1}\}_{\mathfrak{l}_1\in \text{leg}(\mathcal{C}(k)_{\mathfrak{l}})} } \# Eq(\mathcal{C}(k)_{\mathfrak{l}},\{c_{\mathfrak{l}_1}\}) \sup_{\{c_{\mathfrak{l}_3}\}_{\mathfrak{l}_3\in \text{leg}(\mathcal{C}(k)_{3})} }\# Eq(\mathcal{C}(k)_{3}, \{c_{\mathfrak{l}_3}\})
    \\
    \lesssim& L^{O((\chi(\mathcal{C}(k)_2)+\chi(\mathcal{C}(k)_3))\theta)} Q^{\chi(\mathcal{C}(k)_2)+\chi(\mathcal{C}(k)_3)}\sup_{\{c_{\mathfrak{l}_1}\}_{\mathfrak{l}_1\in \text{leg}(\mathcal{C}(k)_{\mathfrak{l}})} } \# Eq(\mathcal{C}(k)_{\mathfrak{l}},\{c_{\mathfrak{l}_1}\}) 
    \\
    \lesssim& (L^{O(\theta)} Q)^{\chi(\mathcal{C}(k)_2)+\chi(\mathcal{C}(k)_3)+\chi(\mathcal{C}(k)_{\mathfrak{l}})}=L^{O(\chi(\mathcal{C})\theta)} Q^{\chi(\mathcal{C})}.
\end{split}
\end{equation}
Here in the third inequality, we apply the induction assumption to $\mathcal{C}(k)_3$. The last equality follows from Proposition \ref{prop.countingind} (2).

Therefore, we complete the proof of Proposition \ref{prop.countingind}
and thus the proof of Proposition \ref{prop.counting}.
\end{proof}

\subsection{An upper bound of tree terms}\label{sec.treetermsupperbound} In this section, we first prove the following Proposition which gives an upper bound of the variance of $\mathcal{J}_{T,k}$ and then prove Proposition \ref{prop.treetermsupperbound} as its corollary.

\begin{prop}\label{prop.treetermsvariance}
Assume that $\alpha$ satisfies \eqref{eq.conditionalpha} and $\rho=\alpha\, T^{\frac{1}{2}}_{\text{max}}$. For any $\theta>0$, we have
\begin{equation}
    \sup_k\, \mathbb{E}|(\mathcal{J}_T)_k|^2\lesssim L^{O(l(T)\theta)} \rho^{2l(T)}.
\end{equation}
and $\mathbb{E}|(\mathcal{J}_T)_k|^2=0$ if $|k|\gtrsim 1$.
\end{prop}
\begin{proof}
By \eqref{eq.termTp} and \eqref{eq.termexp}, we know that $\mathcal{J}_{T,k}$ is a linear combination of $Term(T,p)_k$.
\begin{equation}
\begin{split}
    \mathbb{E}|\mathcal{J}_{T,k}|^2=\left(\frac{\lambda}{L^{d}}\right)^{2l(T)}
    \sum_{p\in \mathcal{P}(\{k_1,\cdots, k_{l(T)+1}, k'_1,\cdots, k'_{l(T)+1}\})} Term(T, p)_k.
\end{split}
\end{equation}

Since $\alpha=\frac{\lambda}{L^{\frac{d}{2}}}$, $\frac{\lambda}{L^{d}}=\alpha L^{-\frac{d}{2}}$. Since the number of elements in $\mathcal{P}$ can be bounded by a constant, by Lemma \ref{lem.Tpvariance} proved below, we get
\begin{equation}\label{eq.proptreetermsvariance1}
\begin{split}
    \mathbb{E}|\mathcal{J}_{T,k}|^2\lesssim (\alpha L^{-\frac{d}{2}})^{2l(T)}
    L^{O(n\theta)} Q^{\frac{n}{2}}  T^{n}_{\text{max}}  .
\end{split}
\end{equation}

By definition, $n$ is the total number of $\bullet$ nodes in the couple constructed from tree $T$ and pairing $p$. Therefore, $n$ equals to $2l(T)$.  Replacing $n$ by $2l(T)$ and $Q$ by $L^dT^{-1}_{\text{max}}$ in \eqref{eq.proptreetermsvariance1}, we get
\begin{equation}
\begin{split}
    \mathbb{E}|\mathcal{J}_{T,k}|^2\lesssim& (\alpha L^{-\frac{d}{2}})^{2l(T)}
    L^{O(l(T)\theta)} (L^dT^{-1}_{\text{max}})^{l(T)}  T_{\text{max}}^{2l(T)}  
    \\
    =& L^{O(l(T)\theta)} (\alpha^2 T_{\text{max}})^{l(T)}  
    \\
    =&  L^{O(l(T)\theta)} \rho^{2l(T)}.  
\end{split}
\end{equation}

By Lemma \ref{lem.Tpvariance} below $Term(T,p)_k=0$ if $|k|\gtrsim 1$, we know that the same is true for $\mathbb{E}|(\mathcal{J}_T)_k|^2=0$. 

Therefore, we complete the proof of this proposition.
\end{proof}

\begin{lem}\label{lem.Tpvariance} Let $n$ and $Q=L^dT^{-1}_{\text{max}}$ be the same as in Proposition \ref{prop.counting}. Assume that $\alpha$ satisfies \eqref{eq.conditionalpha} and $n_{\mathrm{in}} \in C^\infty_0(\mathbb{R}^d)$ is compactly supported. Let $\mathcal{C}$ be the couple constructed from tree $T$ and pairing $p$. Then for any $\theta>0$, we have
\begin{equation}
    \sup_k\, |Term(T,p)_k|\le L^{O(n\theta)} Q^{\frac{n}{2}} T_{\text{max}}^{n}.
\end{equation}
and $Term(T,p)_k=0$ if $|k|\gtrsim 1$.
\end{lem}
\begin{proof} By \eqref{eq.termTp}, we get

\begin{equation}
\begin{split}
    Term(T, p)_k=&\sum_{k_1,\, k_2,\, \cdots,\, k_{l(T)+1}}\sum_{k'_1,\, k'_2,\, \cdots,\, k'_{l(T)+1}} H^T_{k_1\cdots k_{l(T)+1}} H^{T}_{k'_1\cdots k'_{l(T)+1}}
    \\
    & \delta_{p}(k_1,\cdots, k_{l(T)+1}, k'_1,\cdots, k'_{l(T)+1})\sqrt{n_{\textrm{in}}(k_1)}\cdots\sqrt{n_{\textrm{in}}(k'_1)}\cdots
\end{split}
\end{equation}

Since $n_{\mathrm{in}}$ are compactly supported and there are bounded many of them in $Term(T, p)_k$, by $k_1 + k_2 + \cdots + k_{l(T)+1}=k$, we know that $Term(T, p)_k=0$ if $|k|\gtrsim 1$.

By \eqref{eq.boundcoef}, we get  %$\frac{1}{\sqrt{(|c(\Omega)_{\mathfrak{n}}|+T^{-2}_{\text{max}})^2+|r_{\mathfrak{n}}|^2}}\lesssim \frac{1}{|c(\Omega)_{\mathfrak{n}}|+T^{-2}_{\text{max}}}$
\begin{equation}\label{eq.termlemmaeq1}
    |H^T_{k_1\cdots k_{l+1}}|\lesssim \sum_{\{d_{\mathfrak{n}}\}_{\mathfrak{n}\in T_{\text{in}}}\in\{0,1\}^{l(T)}}\prod_{\mathfrak{n}\in T_{\text{in}}}\frac{1}{|q_{\mathfrak{n}}|+T^{-1}_{\text{max}}}\ \prod_{\mathfrak{e}\in T_{\text{in}}}|k_{\mathfrak{e},x}|\ \delta_{\cap_{\mathfrak{n}\in T_{\text{in}}} \{S_{\mathfrak{n}}=0\}}.
\end{equation}

% If $k_1,\cdots,k_{l(T)+1}$ are bounded, then the above argument proves the boundedness of $|k_{\mathfrak{e},x}|$, so by \eqref{eq.termlemmaeq1} we have
% \begin{equation}\label{eq.termlemmaeq2}
%     H^T_{k_1\cdots k_{l+1}}\lesssim \sum_{\{d_{\mathfrak{n}}\}_{\mathfrak{n}\in T_{\text{in}}}\in\{0,1\}^{l(T)}}\prod_{\mathfrak{n}\in T_{\text{in}}}\frac{1}{\sqrt{(|q_{\mathfrak{n}}|+T^{-1}_{\text{max}})^2+|r_{\mathfrak{n}}|^2}}\ \delta_{\cap_{\mathfrak{n}\in T_{\text{in}}} \{S_{\mathfrak{n}}=0\}}.
% \end{equation}

By \eqref{eq.q_n}, $q_{\mathfrak{n}}$ is a linear combination of $\Omega_{\mathfrak{n}}$, so there exist constants $c_{\mathfrak{n},\widetilde{\mathfrak{n}}}$ such that $q_{\mathfrak{n}}=\sum_{\widetilde{\mathfrak{n}}}c_{\mathfrak{n},\widetilde{\mathfrak{n}}}\Omega_{\widetilde{\mathfrak{n}}}$. Let $c$ be the matrix $[c_{\mathfrak{n},\widetilde{\mathfrak{n}}}]$ and $\mathscr{M}(T)$ be the set of all possible such matrices, then the number of elements in $\mathscr{M}(T)$ can be bounded by a constant. Let $c(\Omega)$ be the vector $\{\sum_{\widetilde{\mathfrak{n}}}c_{\mathfrak{n},\widetilde{\mathfrak{n}}}\Omega_{\widetilde{\mathfrak{n}}}\}_{\mathfrak{n}}$ and $c(\Omega)_{\mathfrak{n}}=\sum_{\widetilde{\mathfrak{n}}}c_{\mathfrak{n},\widetilde{\mathfrak{n}}}\Omega_{\widetilde{\mathfrak{n}}}$ be the components of $c(\Omega)$.

With this notation, we know that $q_{\mathfrak{n}}=c(\Omega)_{\mathfrak{n}}$ and the right hand side of \eqref{eq.termlemmaeq1} becomes $\sum_{c\in \mathscr{M}(T) }\prod_{\mathfrak{n}\in T_{\text{in}}} \frac{1}{|c(\Omega)_{\mathfrak{n}}|+T^{-1}_{\text{max}}}$. Therefore, using the fact that $n_{\textrm{in}}$ are compactly supported, we have
\begin{equation}\label{eq.termlemmaeq3}
\begin{split}
    &|Term(T, p)_k|\lesssim \sum_{\substack{k_1,\, \cdots,\, k_{l(T)+1},\, k'_1,\, \cdots,\, k'_{l(T)+1}\\ |k_{j}|, |k'_j|\lesssim 1, \forall j}} \sum_{c\in \mathscr{M}(T) }\prod_{\mathfrak{n}\in T_{\text{in}}}\frac{1}{|c(\Omega)_{\mathfrak{n}}|+T^{-1}_{\text{max}}} \prod_{\mathfrak{e}\in T_{\text{in}}} |k_{\mathfrak{e},x}|\ \delta_{\cap_{\mathfrak{n}\in T_{\text{in}}} \{S_{\mathfrak{n}}=0\}} 
    \\
    &\sum_{c'\in \mathscr{M}(T)}\prod_{\mathfrak{n}'\in T_{\text{in}}}\frac{1}{|c'(\Omega)_{\mathfrak{n}'}|+T^{-1}_{\text{max}}}\prod_{\mathfrak{e}'\in T_{\text{in}}}|k_{\mathfrak{e}',x}|\ \delta_{\cap_{\mathfrak{n}'\in T_{\text{in}}} \{S_{\mathfrak{n}'}=0\}} \delta_{p}(k_1,\cdots, k_{l(T)+1}, k'_1,\cdots, k'_{l(T)+1})
\end{split}
\end{equation}

Switch the order of summations and products in \eqref{eq.termlemmaeq3}, then we get
\begin{equation}\label{eq.termlemmaeq2}
\begin{split}
    &|Term(T, p)_k|\lesssim \sum_{\substack{k_1,\, \cdots,\, k_{l(T)+1},\, k'_1,\, \cdots,\, k'_{l(T)+1}\\ |k_{j}|, |k'_j|\lesssim 1, \forall j}} \sum_{c, c'\in \mathscr{M}(T) }\prod_{\mathfrak{n}, \mathfrak{n}'\in T_{\text{in}}}\frac{1}{|c(\Omega)_{\mathfrak{n}}|+T^{-1}_{\text{max}}}
    \\
    &\frac{1}{|c'(\Omega)_{\mathfrak{n}'}|+T^{-1}_{\text{max}}}\prod_{\mathfrak{e},\mathfrak{e}'\in T_{\text{in}}}(|k_{\mathfrak{e},x}||k_{\mathfrak{e}',x}|)\ \delta_{\cap_{\mathfrak{n},\mathfrak{n}'\in T_{\text{in}}} \{S_{\mathfrak{n}}=0, S_{\mathfrak{n}'}=0\}} \delta_{p}(k_1,\cdots, k_{l(T)+1}, k'_1,\cdots, k'_{l(T)+1}).
\end{split}
\end{equation}

Given a tree $T$ and pairing $p$, we can construct a couple $\mathcal{C}$. We show that 
\begin{equation}\label{eq.termlemmaeq4}
\sum_{c, c'\in \mathscr{M}(T) }=\sum_{c\in \mathscr{M}(\mathcal{C}) },\qquad \prod_{\mathfrak{n}, \mathfrak{n}'\in T_{\text{in}}}=\prod_{\mathfrak{n}\in \mathcal{C}}, \qquad \prod_{\mathfrak{e},\mathfrak{e}'\in T_{\text{in}}}=\prod_{\mathfrak{e}\in \mathcal{C}_{\text{norm}}},\qquad \cap_{\mathfrak{n},\mathfrak{n}'\in T_{\text{in}}}=\cap_{\mathfrak{n}\in \mathcal{C}}.    
\end{equation}

Remember that $\mathcal{C}$ is constructed by glueing two copies of $T$ by $p$. In \eqref{eq.termlemmaeq2}, $\mathfrak{n}$, $\mathfrak{e}$, $\mathfrak{n}'$, $\mathfrak{e}'$ denote nodes and edges in the first or second copy respectively. Since all nodes in $\mathcal{C}$ come from the two copies of $T$, we get $\prod_{\mathfrak{n}, \mathfrak{n}'\in T_{\text{in}}}=\prod_{\mathfrak{n}\in \mathcal{C}}$ and $\cap_{\mathfrak{n},\mathfrak{n}'\in T_{\text{in}}}=\cap_{\mathfrak{n}\in \mathcal{C}}$. Remember that $T_{\text{in}}$ is the tree formed by all non-leaf nodes, so edges of $\mathcal{C}_{\text{norm}}$ all come from the two copies of $T_{\text{in}}$. Therefore $\prod_{\mathfrak{e},\mathfrak{e}'\in T_{\text{in}}}=\prod_{\mathfrak{e}\in \mathcal{C}_{\text{norm}}}$. Given two matrix $c$, $c'$, we can construct a new matrix $c\oplus c'$ as the following. Consider two vectors $\Omega=\{\Omega_{\mathfrak{n}}\}_{\mathfrak{n}'\in T}$, $\Omega'=\{\Omega_{\mathfrak{n}'}\}_{\mathfrak{n}'\in T}$, then define $\Omega\oplus\Omega'\coloneqq \{\Omega_{\mathfrak{n}},\Omega_{\mathfrak{n}'}\}_{\mathfrak{n},\mathfrak{n}'\in T}$. We know that $c(\Omega)=\sum_{\widetilde{\mathfrak{n}}}c_{\mathfrak{n},\widetilde{\mathfrak{n}}}\Omega_{\widetilde{\mathfrak{n}}}$ and $c(\Omega')=\sum_{\widetilde{\mathfrak{n}}'}c_{\mathfrak{n}',\widetilde{\mathfrak{n}}'}\Omega_{\widetilde{\mathfrak{n}}'}$. Define $c\oplus c'$ to be the linear map whose domain is all vector of the form $\Omega\oplus\Omega'$ and whose action is $c\oplus c'(\Omega\oplus\Omega')=\{\sum_{\widetilde{\mathfrak{n}}}c_{\mathfrak{n},\widetilde{\mathfrak{n}}}\Omega_{\widetilde{\mathfrak{n}}},\sum_{\widetilde{\mathfrak{n}}'}c_{\mathfrak{n}',\widetilde{\mathfrak{n}}'}\Omega_{\widetilde{\mathfrak{n}}'}\}_{\mathfrak{n},\mathfrak{n}'\in T}$. Define $\mathscr{M}(\mathcal{C})=\{c\oplus c':c, c'\in \mathscr{M}(T)\}$, then we get $\sum_{c, c'\in \mathscr{M}(T) }=\sum_{c\in \mathscr{M}(\mathcal{C}) }$.

By \eqref{eq.termlemmaeq4} and the fact that leaves corresponding to $k_j, k_j'$ are merged in $\mathcal{C}$, \eqref{eq.termlemmaeq2} is equivalent to 
\begin{equation}\label{eq.termlemmaeq5}
|Term(T, p)_k|\lesssim \sum_{\substack{k_1,\, \cdots,\, k_{l(T)+1},\, k'_1,\, \cdots,\, k'_{l(T)+1}\\ |k_{j}|, |k'_j|\lesssim 1, \forall j}} \sum_{c\in \mathscr{M}(\mathcal{C}) }\prod_{\mathfrak{n}\in \mathcal{C}}\frac{1}{|c(\Omega)_{\mathfrak{n}}|+T^{-1}_{\text{max}}} \prod_{\mathfrak{e}\in \mathcal{C}_{\text{norm}}}|k_{\mathfrak{e},x}|\ \delta_{\cap_{\mathfrak{n}\in \mathcal{C}} \{S_{\mathfrak{n}}=0\}}
\end{equation}

Assigning a number $\sigma_{\mathfrak{n}}\in \mathbb{Z}_{T_{\text{max}}}$ for each node $\mathfrak{n}\in \mathcal{C}$, a number $\kappa_{\mathfrak{e}}\in \mathcal{D}(\alpha,1)\coloneqq\{2^{-K_{\mathfrak{e}}}:K_{\mathfrak{e}}\in  \mathbb{Z}\cap [0,ln\ \alpha^{-1}]\}$ for each edge $\mathfrak{e}$ and a number $k\in \mathbb{Z}^d_{L}$ for each fixed leg, we can define the associated equation $Eq(\mathcal{C}, \{\sigma_{\mathfrak{n}}\}_{\mathfrak{n}}, \{\kappa_{\mathfrak{e}}\}_{\mathfrak{e}},k)=Eq(\mathcal{C})$ as in \eqref{eq.diophantineeqpairedsigma'}. Then we have  
\begin{equation}\label{eq.termlemmaeq8}
    \sum_{\substack{k_1,\, \cdots,\, k_{l(T)+1},\, k'_1,\, \cdots,\, k'_{l(T)+1}\\\cap_{\mathfrak{n}\in \mathcal{C}} \{S_{\mathfrak{n}}=0\}}}=\sum_{\kappa_{\mathfrak{e}}\in \mathcal{D}(\alpha,1)}\sum_{\sigma_{\mathfrak{n}}\in \mathbb{Z}_{T_{\text{max}}}}\sum_{Eq(\mathcal{C}, \{\sigma_{\mathfrak{n}}\}_{\mathfrak{n}}, \{\kappa_{\mathfrak{e}}\}_{\mathfrak{e}},k)},
\end{equation}
which implies that
\begin{equation}\label{eq.termlemmaeq6}
\begin{split}
    |Term(T, p)_k|\lesssim \sum_{\kappa_{\mathfrak{e}}\in \mathcal{D}(\alpha,1)}\sum_{\sigma_{\mathfrak{n}}\in \mathbb{Z}_{T_{\text{max}}}}\sum_{Eq(\mathcal{C}, \{\sigma_{\mathfrak{n}}\}_{\mathfrak{n}}, \{\kappa_{\mathfrak{e}}\}_{\mathfrak{e}},k)} \sum_{c\in \mathscr{M}(\mathcal{C}) }\prod_{\mathfrak{n}\in \mathcal{C}}\frac{1}{|c(\Omega)_{\mathfrak{n}}|+T^{-1}_{\text{max}}} \prod_{\mathfrak{e}\in \mathcal{C}_{\text{norm}}}|k_{\mathfrak{e},x}|
\end{split}
\end{equation}
% \begin{equation}
% \begin{split}
%     |Term(T, p)_k|\lesssim& \sum_{\kappa_{\mathfrak{e}}\in \mathcal{D}(\alpha,1)}\sum_{\sigma_{\mathfrak{n}}\in \mathbb{Z}_{T_{\text{max}}}}\sum_{Eq(\mathcal{C}, \{\sigma_{\mathfrak{n}}\}_{\mathfrak{n}}, \{\kappa_{\mathfrak{e}}\}_{\mathfrak{e}},k)}
%     \\
%     &\sum_{c\in \mathscr{M}(T) }\prod_{\mathfrak{n}\in T_{\text{in}}}\frac{1}{\sqrt{(|c(\Omega)_{\mathfrak{n}}|+T^{-1}_{\text{max}})^2+|r_{\mathfrak{n}}|^2}} \sum_{c'\in \mathscr{M}(T)}\prod_{\mathfrak{n}'\in T_{\text{in}}}\frac{1}{\sqrt{(|c(\Omega)_{\mathfrak{n}'}|+T^{-1}_{\text{max}})^2+|r_{\mathfrak{n}'}|^2}} .
% \end{split}
% \end{equation}
%Since $Eq(\mathcal{C}, \{\sigma_{\mathfrak{n}}\}_{\mathfrak{n}}, \{\kappa_{\mathfrak{e}}\}_{\mathfrak{e}},k)$ contains equations in $\delta_{\cap_{\mathfrak{n}\in T_{\text{in}}} \{S_{\mathfrak{n}}=0\}}$, $\delta_{\cap_{\mathfrak{n}'\in T_{\text{in}}} \{S_{\mathfrak{n}'}=0\}}$ and $\delta_{p}$, so we may remove the three indicator functions from the above sum. 

Remember in the definition of $Eq(\mathcal{C})$, we have conditions that $|k_{\mathfrak{e}}|\lesssim 1$. These conditions come from the fact that on the support of $\delta_{\cap_{\mathfrak{n}\in T_{\text{in}}} \{S_{\mathfrak{n}}=0\}}$, if $k_1,\cdots,k_{l(T)+1}$ are bounded, then $|k_{\mathfrak{e}}|$ are also bounded. Recall that by \eqref{eq.defnS_n}, $S_{\mathfrak{n}}=\iota_{\mathfrak{e}_1}k_{\mathfrak{e}_1}+\iota_{\mathfrak{e}_2}k_{\mathfrak{e}_2}+\iota_{\mathfrak{e}}k_{\mathfrak{e}}$. The conditions that $S_{\mathfrak{n}}=0$ imply that the variables $k_{\mathfrak{e}}$ of the parent edge $\mathfrak{e}$ are linear combinations of the variables $k_{\mathfrak{e}_1}$, $k_{\mathfrak{e}_2}$ of children edges $\mathfrak{e}_1$, $\mathfrak{e}_2$. The variables of children edges $\mathfrak{e}_j$ are again a linear combinations of the variables of their children. We may iterate this argument to show that all variables $k_{\mathfrak{e}}$ are linear combinations of variables of leaves $k_1,\cdots,k_{l(T)+1}$. Because $k_1,\cdots,k_{l(T)+1}$ are bounded, then their linear combinations $k_{\mathfrak{e}}$ are also bounded.

In the definition of $Eq(\mathcal{C}, \{\sigma_{\mathfrak{n}}\}_{\mathfrak{n}}, \{\kappa_{\mathfrak{e}}\}_{\mathfrak{e}},k)$,
$|\Omega_{\mathfrak{n}}-\sigma_{\mathfrak{n}}|=O(T^{-1}_{\text{max}})$ and $|k_{\mathfrak{e}x}| \sim \kappa_{\mathfrak{e}}$. Denote the constant in $O(T^{-1}_{\text{max}})$ by $\delta$ and then we have $|\Omega_{\mathfrak{n}}-\sigma_{\mathfrak{n}}|\le \delta T^{-1}_{\text{max}}$. Therefore, we have $|c(\Omega)_{\mathfrak{n}}-c(\{\sigma_{\mathfrak{n}}\})_{\mathfrak{n}}|\lesssim \delta T^{-1}_{\text{max}}$. We have the freedom of choosing $\delta$ in the definition and we take it sufficiently small so that $|c(\Omega)_{\mathfrak{n}}-c(\{\sigma_{\mathfrak{n}}\})_{\mathfrak{n}}|\le \frac{1}{2}T^{-1}_{\text{max}}$. This implies that $|c(\Omega)_{\mathfrak{n}}|+T^{-1}_{\text{max}}\gtrsim |c(\{\sigma_{\mathfrak{n}}\})_{\mathfrak{n}}|+T^{-1}_{\text{max}}$. Since we also have $|k_{\mathfrak{e}x}| \sim \kappa_{\mathfrak{e}}$, by \eqref{eq.termlemmaeq6} we get
\begin{equation}\label{eq.lemboundtermTp}
\begin{split}
    &Term(T, p)_k
    \\
    \lesssim& \sum_{\kappa_{\mathfrak{e}}\in \mathcal{D}(\alpha,1)}\sum_{\sigma_{\mathfrak{n}}\in \mathbb{Z}_{T_{\text{max}}}}\sum_{Eq(\mathcal{C}, \{\sigma_{\mathfrak{n}}\}_{\mathfrak{n}}, \{\kappa_{\mathfrak{e}}\}_{\mathfrak{e}},k)} \sum_{c\in \mathscr{M}(\mathcal{C}) }\prod_{\mathfrak{n}\in \mathcal{C}}\frac{1}{|c(\Omega)_{\mathfrak{n}}|+T^{-1}_{\text{max}}} \prod_{\mathfrak{e}\in \mathcal{C}_{\text{norm}}}|k_{\mathfrak{e},x}|
    \\
    \lesssim &\sum_{c\in \mathscr{M}(\mathcal{C}) }\sum_{\kappa_{\mathfrak{e}}\in \mathcal{D}(\alpha,1)}\sum_{\substack{\sigma_{\mathfrak{n}}\in \mathbb{Z}_{T_{\text{max}}}\\ |\sigma_{\mathfrak{n}}|\lesssim 1}}\prod_{\mathfrak{n}\in \mathcal{C}}\frac{1}{|c(\{\sigma_{\mathfrak{n}}\})_{\mathfrak{n}}|+T^{-1}_{\text{max}}} \sum_{Eq(\mathcal{C}, \{\sigma_{\mathfrak{n}}\}_{\mathfrak{n}}, \{\kappa_{\mathfrak{e}}\}_{\mathfrak{e}},k)} \prod_{\mathfrak{e}\in \mathcal{C}_{\text{norm}}}|k_{\mathfrak{e},x}|
    \\
    \lesssim &\sum_{c\in \mathscr{M}(\mathcal{C}) }\sum_{\kappa_{\mathfrak{e}}\in \mathcal{D}(\alpha,1)}\sum_{\substack{\sigma_{\mathfrak{n}}\in \mathbb{Z}_{T_{\text{max}}}\\ |\sigma_{\mathfrak{n}}|\lesssim 1}}\prod_{\mathfrak{n}\in \mathcal{C}}\frac{1}{|c(\{\sigma_{\mathfrak{n}}\})_{\mathfrak{n}}|+T^{-1}_{\text{max}}} \left(\prod_{\mathfrak{e}\in \mathcal{C}_{\text{norm}}}\kappa_{\mathfrak{e}}\right)\#Eq(\mathcal{C}, \{\sigma_{\mathfrak{n}}\}_{\mathfrak{n}}, \{\kappa_{\mathfrak{e}}\}_{\mathfrak{e}},k) 
    \\
    \lesssim &\sum_{c\in \mathscr{M}(\mathcal{C}) }\sum_{\kappa_{\mathfrak{e}}\in \mathcal{D}(\alpha,1)}\sum_{\substack{\sigma_{\mathfrak{n}}\in \mathbb{Z}_{T_{\text{max}}}\\ |\sigma_{\mathfrak{n}}|\lesssim 1}}\prod_{\mathfrak{n}\in \mathcal{C}}\frac{1}{|c(\{\sigma_{\mathfrak{n}}\})_{\mathfrak{n}}|+T^{-1}_{\text{max}}} \left(\prod_{\mathfrak{e}\in \mathcal{C}_{\text{norm}}}\kappa_{\mathfrak{e}}\right)L^{O(n\theta)} Q^{\frac{n}{2}}\ \prod_{\mathfrak{e}\in \mathcal{C}_{\text{norm}}} \kappa^{-1}_{\mathfrak{e}}
\end{split}
\end{equation}
Here in the last inequality we applied \eqref{eq.countingbd0} in Proposition \ref{prop.counting}.

After simplification, \eqref{eq.lemboundtermTp} gives us 
\begin{equation}\label{eq.lemboundtermTpsimplify}
\begin{split}
    |Term(T, p)_k|\lesssim &L^{O(n\theta)} Q^{\frac{n}{2}}\sum_{c\in \mathscr{M}(\mathcal{C}) }\sum_{\kappa_{\mathfrak{e}}\in \mathcal{D}(\alpha,1)}\sum_{\substack{\sigma_{\mathfrak{n}}\in \mathbb{Z}_{T_{\text{max}}}\\ |\sigma_{\mathfrak{n}}|\lesssim 1}} \prod_{\mathfrak{n}\in \mathcal{C}}\frac{1}{|c(\{\sigma_{\mathfrak{n}}\})_{\mathfrak{n}}|+T^{-1}_{\text{max}}}
    \\
    \lesssim &L^{O(n\theta)} Q^{\frac{n}{2}} \left(\sum_{\kappa_{\mathfrak{e}}\in \mathcal{D}(\alpha,1)} 1\right) \sum_{c\in \mathscr{M}(\mathcal{C})}\sum_{\substack{\sigma_{\mathfrak{n}}\in \mathbb{Z}_{T_{\text{max}}}\\ |\sigma_{\mathfrak{n}}|\lesssim 1}} \prod_{\mathfrak{n}\in \mathcal{C}}\frac{1}{|c(\{\sigma_{\mathfrak{n}}\})_{\mathfrak{n}}|+T^{-1}_{\text{max}}}
    \\
    \lesssim & L^{O(n\theta)} Q^{\frac{n}{2}} \sum_{c\in \mathscr{M}(\mathcal{C})}\sum_{\substack{\sigma_{\mathfrak{n}}\in \mathbb{Z}_{T_{\text{max}}}\\ |\sigma_{\mathfrak{n}}|\lesssim 1}} \prod_{\mathfrak{n}\in \mathcal{C}}\frac{1}{|c(\{\sigma_{\mathfrak{n}}\})_{\mathfrak{n}}|+T^{-1}_{\text{max}}}
\end{split}
\end{equation}
% \begin{equation}
% \begin{split}
%     &Term(T, p)_k
%     \\
%     \lesssim&\sum_{c,c'\in \mathscr{M}(T) } \sum_{\sigma_{\mathfrak{n}}\in \mathbb{Z}_{T_{\text{max}}}} \prod_{\mathfrak{n}\in T_{\text{in}}}\frac{t\alpha}{|c(\{\sigma_{\mathfrak{n}}\})_{\mathfrak{n}}|+\alpha} \prod_{\mathfrak{n}\in T_{\text{in}}}\frac{t\alpha}{|c'(\{\sigma_{\mathfrak{n}}\})_{\mathfrak{n}}|+\alpha} (\#Eq(\mathcal{C}, \{\sigma_{\mathfrak{n}}\}_{\mathfrak{n}},k)+O(L^{-8d\, l(T)-8d}))
%     \\
%     \lesssim & L^{O(n\theta)} Q^{n} L^{\frac{1}{2} dn_d}\sum_{c,c'\in \mathscr{M}(T) } \sum_{\substack{\sigma_{\mathfrak{n}}\in \mathbb{Z}_{T_{\text{max}}}\\ |\sigma_{\mathfrak{n}}|\lesssim 1}} \prod_{\mathfrak{n}\in T_{\text{in}}}\frac{t\alpha}{|c(\{\sigma_{\mathfrak{n}}\})_{\mathfrak{n}}|+\alpha} \prod_{\mathfrak{n}\in T_{\text{in}}}\frac{t\alpha}{|c'(\{\sigma_{\mathfrak{n}}\})_{\mathfrak{n}}|+\alpha}  + O(L^{-6d\, l(T)-6d})
% \end{split}
% \end{equation}
Here in the last step we use the fact that $\#\mathcal{D}(\alpha,1)\lesssim ln(\alpha^{-1})$. $|\sigma_{\mathfrak{n}}|\lesssim 1$ in the sum of the first two inequalities comes from the fact that $\#Eq(\mathcal{C}, \{\sigma_{\mathfrak{n}}\}_{\mathfrak{n}},k)=0$ if some $|\sigma_{\mathfrak{n}}|\gtrsim 1$. This fact is true because in $Eq(\mathcal{C}, \{\sigma_{\mathfrak{n}}\}_{\mathfrak{n}},k)$, all $|k_{\mathfrak{e}}|\lesssim 1$, which implies that $|\Omega_{\mathfrak{n}}|\lesssim 1$ and therefore $|\sigma_{\mathfrak{n}}|\gtrsim 1$ (notice that $|\Omega_{\mathfrak{n}}-\sigma_{\mathfrak{n}}|=O(T^{-1}_{\text{max}})$).

We claim that 
\begin{equation}\label{eq.lemTpvarianceclaim}
     \sup_{c}\sum_{\substack{\sigma_{\mathfrak{n}}\in \mathbb{Z}_{T_{\text{max}}}\\ |\sigma_{\mathfrak{n}}|\lesssim 1}} \prod_{\mathfrak{n}\in \mathcal{C}}\frac{1}{|c(\{\sigma_{\mathfrak{n}}\})_{\mathfrak{n}}|+T^{-1}_{\text{max}}}\lesssim L^{O(n\theta)} T^{n}_{\text{max}}
\end{equation}

Since there are only bounded many matrices in $\mathscr{M}(\mathcal{C})$.  Given the above claim, we know that 
\begin{equation}
    |Term(T, p)_k|\lesssim L^{O(n\theta)} Q^{\frac{n}{2}} T^{n}_{\text{max}},
\end{equation}
which proves the lemma.

Now prove the claim. Notice that there are $n$ nodes in $\mathcal{C}$. We label these nodes by $h=1,\cdots,n$ and denote $\sigma_{\mathfrak{n}}$ by $\sigma_{h}$ if $\mathfrak{n}$ is labelled by $h$. Since $\sigma_{h}\in \mathbb{Z}_{T_{\text{max}}}$, there exists $m_{h}\in \mathbb{Z}$ such that $\sigma_{h}=\frac{m_{h}}{T_{\text{max}}}$. \eqref{eq.lemTpvarianceclaim} is thus equivalent to 
\begin{equation}\label{eq.lemTpvarianceclaim1}
    \sum_{\substack{m_{h}\in \mathbb{Z}\\ |m_{h}|\lesssim  T_{\text{max}}}} \prod_{h=1}^{n}\frac{T_{\text{max}}}{|c(\{m_{h}\})_{h}|+1}\lesssim L^{O(n\theta)}T^{n}_{\text{max}}
\end{equation}

Before proving \eqref{eq.lemTpvarianceclaim1}, let's first look at one of its special case. If $c=Id$, then $c(\{m_{h}\})_{h}=m_h$. The right hand side of \eqref{eq.lemTpvarianceclaim1} becomes
\begin{equation}
\begin{split}
    T^{n}_{\text{max}}\sum_{\substack{m_{h}\in \mathbb{Z}\\ |m_{h}|\lesssim  T_{\text{max}}}} \prod_{h=1}^{n}\frac{1}{|m_{h}|+1} = T^{n}_{\text{max}}\prod_{h=1}^{n}\left(\sum_{\substack{m_{h}\in \mathbb{Z}\\ |m_{h}|\lesssim  T_{\text{max}}}} \frac{1}{|m_{h}|+1}\right)
    \lesssim T^{n}_{\text{max}} (ln(\alpha^{-1}))^{n}= L^{O(n\theta)}T^{n}_{\text{max}}.
\end{split}
\end{equation}
Here we use the fact that $\sum_{\substack{j\in \mathbb{Z}\\ |j|\lesssim  T_{\text{max}}}} \frac{1}{|j|+1}=O(ln(\alpha^{-1}))$.

Now we prove \eqref{eq.lemTpvarianceclaim1}. We just need to show that 
\begin{equation}\label{eq.lemTpvarianceEulerMac}
    \sum_{\substack{m_{h}\in \mathbb{Z}\\ |m_{h}|\lesssim  T_{\text{max}}}} \prod_{h=1}^{n}\frac{1}{|c(\{m_{h}\})_{h}|+1}\lesssim L^{O(n\theta)}
\end{equation}

By Euler-Maclaurin formula \eqref{eq.EulerMaclaurin} and change of variable formula, we get
\begin{equation}
\begin{split}
    \sum_{\substack{m_{h}\in \mathbb{Z}\\ |m_{h}|\lesssim  T_{\text{max}}}} \prod_{h=1}^{n}\frac{1}{|c(\{m_{h}\})_{h}|+1}\le& \int_{|m_{h}|\lesssim  T_{\text{max}}}  \prod_{h=1}^{n}\frac{1}{|c(\{m_{h}\})_{h}|+1}\prod_{h=1}^{n} dm_{h}
    \\
    =& \int_{c(\{|m_{h}|\lesssim  T_{\text{max}}\})}  \prod_{h=1}^{n}\frac{1}{|m_{h}|+1}|\text{det}\ c|\prod_{h=1}^{n}  dm_{h}
    \\
    \lesssim &\prod_{h=1}^{n}\int_{|m_{h}|\lesssim  T_{\text{max}}}  \frac{1}{|m_{h}|+1}  dm_{h}
    \\
    \lesssim & (ln(\alpha^{-1}))^{n}\lesssim L^{O(n\theta)}
\end{split}
\end{equation}

We complete the proof of the claim and thus the proof of the lemma.
\end{proof}

Now we prove Proposition \ref{prop.treetermsupperbound}. To start with, recall the large deviation estimate for Gaussian polynomial.

\begin{lem}[Large deviation for Gaussian polynomial]\label{lem.largedev}
Let $\{\eta_k(\omega)\}$ be i.i.d. complex Gaussian variables with mean $0$ and variance $1$. Let $F=F(\omega)$ be an degree $n$ polynomial of $\{\eta_k(\omega)\}$ defined by \begin{equation}\label{indp}
F(\omega)=\sum_{k_1,\cdots,k_n}a_{k_1\cdots k_n}\prod_{j=1}^n\eta_{k_j}^{\iota_j},
\end{equation} 
where $a_{k_1\cdots k_n}$ are constants, then we have 
\begin{equation}\label{largedevest}\mathbb{P}\left(|F(\omega)|\geq A\cdot \left(\mathbb{E}|F(\omega)|^2\right)^{\frac{1}{2}}\right)\leq Ce^{-cA^{\frac{2}{n}}}
\end{equation} 
\end{lem}
\begin{proof} This is a corollary of the hypercontractivity of Ornstein-Uhlenbeck semigroup. A good reference of this topic is \cite{Oh}. \eqref{largedevest} is equivalent to (B.9) in \cite{Oh}.
\end{proof}

Proposition \ref{prop.treetermsupperbound} is a corollary of the above large deviation estimate and Proposition \ref{prop.treetermsvariance}.

\begin{proof}[Proof of Proposition \ref{prop.treetermsupperbound}] Because $(\mathcal{J}_{T})_{k}$ are Gaussian polynomials, we can take $F(\omega)=(\mathcal{J}_{T})_{k}$ in Lemma \ref{lem.largedev}. Then we obtained
\begin{equation}
|(\mathcal{J}_{T})_{k}(t)|\lesssim L^{\frac{n}{2}\theta} \sqrt{\mathbb{E}|(\mathcal{J}_T)_k|^2}
\end{equation} 
with probability less than $e^{-c(L^{\frac{n}{2}\theta})^{\frac{2}{n}}}=e^{-cL^{\theta}}$. By definition \ref{def.Lcertainly}, the above inequality holds true $L$-certainly.

Since by Proposition \ref{prop.treetermsvariance}, $\mathbb{E}|(\mathcal{J}_T)_k|^2\lesssim L^{O(l(T)\theta)} \rho^{2l(T)}$, then we get 
\begin{equation}\label{eq.proptreetermsupperbound1}
|(\mathcal{J}_{T})_{k}(t)|\lesssim L^{O(l(T)\theta)} \rho^{l(T)},\qquad \textit{L-certainly}
\end{equation} 

\eqref{eq.proptreetermsupperbound1} is very similar to the final goal \eqref{eq.treetermsupperbound}, except for the $\sup_t$ and $\sup_k$ in front. In what follows we apply the standard epsilon net and union bound method to remove these two $\sup$.

Assume that $|t-t'|\lesssim \rho^{l(T)}L^{-M}$, it's not hard to show that $|(\mathcal{J}_{T})_{k}(t)-(\mathcal{J}_{T})_{k}(t')|\lesssim \rho^{l(T)}$. Therefore, if $\sup_{i} \sup_{k} |(\mathcal{J}_{T})_{k}(i\rho^{l(T)}L^{-M})|\lesssim L^{O(l(T)\theta)} \rho^{l(T)}$, then $\sup_{t} \sup_{k} |(\mathcal{J}_{T})_{k}(t)|\lesssim L^{O(l(T)\theta)} \rho^{l(T)}$ and the proof is completed.
%Consider the decomposition $\left[0, T_{\text{max}}\right]=\bigcup_{i=0}^{L^{M} T_{\text{max}}}\left[\frac{i}{L^M},\frac{i+1}{L^M}\right]$.

Notice that 
\begin{equation}\label{eq.unionbound}
\begin{split}
    &\mathbb{P}\left(\sup_{i\in \mathbb{Z}} \sup_{k} |(\mathcal{J}_{T})_{k}(i\rho^{l(T)}L^{-M})|\gtrsim L^{O(l(T)\theta)} \rho^{l(T)}\right)
    \\
    =&\mathbb{P}\left(\bigcup_{i\in \mathbb{Z}\cap [0, T_{\text{max}}\rho^{-l(T)}L^{M}]}\bigcup_{k\in \mathbb{Z}_{L}\cap [0,1]}\{ |(\mathcal{J}_{T})_{k}(i\rho^{l(T)}L^{-M})|\gtrsim L^{O(l(T)\theta)} \rho^{l(T)}\}\right)
    \\
    \lesssim & \sum_{i\in \mathbb{Z}\cap [0, T_{\text{max}}\rho^{-l(T)}L^{M}]}\sum_{k\in \mathbb{Z}_{L}\cap [0,1]}\mathbb{P}\left( |(\mathcal{J}_{T})_{k}(i\rho^{l(T)}L^{-M})|\gtrsim L^{O(l(T)\theta)} \rho^{l(T)}\right)
    \\
    \lesssim & \sum_{i\in \mathbb{Z}\cap [0, T_{\text{max}}\rho^{-l(T)}L^{M}]}\sum_{k\in \mathbb{Z}_{L}\cap [0,1]} e^{-O(L^{\theta})}= L^{2M}e^{-O(L^{\theta})}=e^{-O(L^{\theta})}
\end{split}
\end{equation}
Here in the second inequality the two ranges $[0, T_{\text{max}}\rho^{-l(T)}L^{M}]$ and $[0,1]$ of $i$ and $k$ come from the fact that $t=i\rho^{l(T)}L^{-M}\lesssim  T_{\text{max}}$ and $(\mathcal{J}_{T})_{k}=0$ for $|k|\gtrsim 1$. In the last line we can replace the probability by $e^{-O(L^{\theta})}$ because the estimate $|(\mathcal{J}_{T})_{k}(t)|\lesssim L^{O(l(T)\theta)} \rho^{l(T)}$ holds true $L$-certainly. 

Now we complete the proof of Proposition \ref{prop.treetermsupperbound}.
\end{proof}

\subsection{Norm estimate of random matrices} \label{sec.randommatrices} In this section, we prove Proposition \ref{prop.operatorupperbound}. 

Remember that by definition of $\mathcal{P}_{T}$ and $\mathcal{T}$, 
\begin{equation}\label{eq.formulaP_T}
\mathcal{P}_{T}(w)=\mathcal{T}(\mathcal{J}_{T},w)=\frac{i\lambda}{L^{d}} \sum\limits_{S(k_1,k_2,k)=0}\int^{t}_0k_{x}\mathcal{J}_{T,k_1} w_{k_2}e^{i s\Omega(k_1,k_2,k)- \nu|k|^2(t-s)} ds.
\end{equation}    

\subsubsection{Dyadic decomposition of $\mathcal{P}_{T}$}
Remember that we have the dyadic decomposition $[0,1]= \bigcup_{\tau=0}^{\infty}[2^{-\tau},2^{-\tau-1}]$.

We can then construct a dyadic decomposition  $\mathcal{P}_{T}=\sum_{l=0}^{\infty} \mathcal{P}^l_{T}$. Here $\mathcal{P}_T^l$ is defined by the following formula. 
\begin{equation}
\mathcal{P}_T^l(w)=\frac{i\lambda}{L^{d}} \sum\limits_{S(k_1,k_2,k)=0}\int_{(t-s)/T_{\text{max}}\in [2^{-\tau},2^{-\tau-1}]}k_{x}\mathcal{J}_{T,k_1} w_{k_2}e^{i s\Omega(k_1,k_2,k)- \nu|k|^2(t-s)} ds
\end{equation}

We also introduce the bilinear operator $\mathcal{T}^l(\phi,\phi)_k$
\begin{equation}
\mathcal{T}^l(\phi,\phi)_k=\frac{i\lambda}{L^{d}} \sum\limits_{S(k_1,k_2,k)=0}\int_{(t-s)/T_{\text{max}}\in [2^{-\tau},2^{-\tau-1}]}k_{x}\phi_{k_1} \phi_{k_2}e^{i s\Omega(k_1,k_2,k)- \nu|k|^2(t-s)} ds
\end{equation}

Proposition \ref{prop.operatorupperbound} is a corollary of the following proposition.

\begin{prop}\label{prop.operatorupperbound'}
Let $\rho=\alpha\, T^{\frac{1}{2}}_{\text{max}}$ and $\mathcal{P}^l_{T}$ be the operator defined above. Define the space $X^{p}_{L^{2M}}=\{w\in X^p: w_k=0\text{ if }|k|\gtrsim L^{2M}\}$ and the norm $||w||_{X^{p}_{L^{2M}}}=\sup_{|k|\lesssim L^{2M}} \langle k\rangle^{p} |w_k|$. Then for any sequence of trees and numbers $\{T_1,\cdots,T_K\}$ and $\{\tau_1,\cdots,\tau_K\}$, we have $L$-certainly the operator bound
\begin{equation}\label{eq.operatornormtau}
    \left|\left|\sum_{\tau_1,\cdots,\tau_K}\prod_{j=1}^K\mathcal{P}^{\tau_j}_{T_j}\right|\right|_{L_t^{\infty}X^{p}_{L^{2M}}\rightarrow L_t^{\infty}X^{p}}\le L^{O\left(1+\theta \sum_{j=1}^K l(T_j)\right)} \rho^{\sum_{j=1}^K l(T_j)}.
\end{equation}
for any $T_j$ with $l(T_j)\le N$. 
\end{prop}

\begin{proof}[Proof of Proposition \ref{prop.operatorupperbound}:] By Proposition \ref{prop.operatorupperbound'}, we have
\begin{equation}
    \left|\left|\prod_{j=1}^K\mathcal{P}_{T_j}\right|\right|_{L_t^{\infty}X^p_{L^{2M}}\rightarrow L_t^{\infty}X^p}=\left|\left|\sum_{\tau_1,\cdots,\tau_K}\prod_{j=1}^K\mathcal{P}^{\tau_j}_{T}\right|\right|_{L_t^{\infty}X^{p}_{L^{2M}}\rightarrow L_t^{\infty}X^{p}}\le L^{O\left(1+\theta \sum_{j=1}^K l(T_j)\right)} \rho^{\sum_{j=1}^K l(T_j)}.
\end{equation}

Define $\left(X^{p}_{L^{2M}}\right)^{\perp}=\{w\in X^p: w_k=0\text{ if }|k|\lesssim L^{2M}\}$. To prove Proposition \ref{prop.operatorupperbound}, it suffices to show that for all $w\in \left(X^{p}_{L^{2M}}\right)^{\perp}$,
\begin{equation}\label{eq.prop2.8eq1}
    \left|\left|\prod_{j=1}^K\mathcal{P}_{T_j}w\right|\right|_{L_t^{\infty}X^p}\lesssim L^{O\left(1+\theta \sum_{j=1}^K l(T_j)\right)} \rho^{\sum_{j=1}^K l(T_j)} \left|\left|w\right|\right|_{L_t^{\infty}X^p}.
\end{equation}

By \eqref{eq.formulaP_T}, if $w\in \left(X^{p}_{L^{2M}}\right)^{\perp}$, then $\mathcal{P}_{T_j} w\in \left(X^{p}_{L^{2M}}\right)^{\perp}$ if we enlarge the constant in $|k|\lesssim L^{2M}$ in the definition of $\left(X^{p}_{L^{2M}}\right)^{\perp}$. Therefore, to prove \eqref{eq.prop2.8eq1}, it suffices to prove
\begin{equation}\label{eq.prop2.8eq2}
    \left|\left|\mathcal{P}_{T_j}w\right|\right|_{L_t^{\infty}X^p}\lesssim L^{O(1+l(T_j)\theta)} \rho^{l(T_j)} \left|\left|w\right|\right|_{L_t^{\infty}X^p},
\end{equation}
for all $w\in \left(X^{p}_{L^{2M}}\right)^{\perp}$.

By \eqref{eq.formulaP_T},
\begin{equation}\label{eq.prop2.8eq3}
\begin{split}
    |\mathcal{P}_{T_j}(w)_k|\le &\frac{\lambda}{L^{d}} \sum\limits_{S(k_1,k_2,k)=0}\int^{t}_0|k_{x}||\mathcal{J}_{T_j,k_1}| |w_{k_2}|e^{- \nu|k|^2(t-s)} ds
    \\
    \lesssim& L^{O(l(T_j)\theta)} \rho^{l(T_j)}\frac{\lambda}{L^{d}}|k_{x}| \int^{t}_0e^{- \nu|k|^2(t-s)} ds \sum_{k_2:|k_2-k|\lesssim 1} \langle k_2\rangle^{-p} 
    \\
    \lesssim& L^{O(l(T_j)\theta)} \rho^{l(T_j)}\frac{\lambda}{L^{d}} |k_{x}| \nu^{-1} \langle k\rangle^{-2} \sum_{k_2:|k_2-k|\lesssim 1} \langle k_2\rangle^{-p} 
    \\
    \lesssim& L^{O(l(T_j)\theta)} \rho^{l(T_j)}\frac{\lambda}{L^{d}} \nu^{-1} \langle k\rangle^{-1}  \langle k\rangle^{-p} \lesssim L^{O(l(T_j)\theta)} \rho^{l(T_j)}\frac{\lambda}{L^{d}} \nu^{-1} L^{-2M}  \langle k\rangle^{-p} 
    \\
    \lesssim& L^{-M} \rho^{l(T_j)} \langle k\rangle^{-p}
\end{split}
\end{equation} 
Here in the second inequality, we apply Proposition \ref{prop.treetermsupperbound}. In the third line we use the fact that $\int^{t}_0e^{- \nu|k|^2(t-s)} ds\le \nu^{-1} \langle k\rangle^{-2}$. In the fourth line we use the fact that $\sum_{k_2:|k_2-k|\lesssim 1} \langle k_2\rangle^{-p}\le \langle k\rangle^{-p}$ and $|k|\gtrsim L^{2M}$ (since if $|k|\lesssim L^{2M}$ then $|\mathcal{P}_{T_j}(w)_k|$ vanishes and there is nothing to prove).

\eqref{eq.prop2.8eq3} implies that $ \left|\left|\mathcal{P}_{T_j}w\right|\right|_{L_t^{\infty}X^p}\lesssim L^{-M} \rho^{l(T_j)} \left|\left|w\right|\right|_{L_t^{\infty}X^p}\lesssim L^{O(1+l(T_j)\theta)} \rho^{l(T_j)} \left|\left|w\right|\right|_{L_t^{\infty}X^p}$, so we have proved \eqref{eq.prop2.8eq2} and thus \eqref{eq.operatornorm'}.

Now we prove \eqref{eq.operatornorm}. Because $L=\sum_{1\le l(T)\le N} \mathcal{P}_{T}$ and $L^K=\sum_{1\le l(T_1),\cdots l(T_K)\le N} \mathcal{P}_{T_1}\cdots\mathcal{P}_{T_K}$, by \eqref{eq.operatornorm'} 
\begin{equation}
     \left|\left|L^K\right|\right|_{L_t^{\infty}X^p\rightarrow L_t^{\infty}X^p}\lesssim L^{O(1)} \sum_{1\le l(T_1),\cdots l(T_K)\le N} (L^{O(\theta)}\rho)^{\sum_{j=1}^K l(T_j)}\le L^{O\left(1+\theta \sum_{j=1}^K l(T_j)\right)} \rho^{K}
\end{equation}
Here in the last step we use the fact that $l(T_j)\ge 1$ for all $j$ and there are bounded many trees satisfy $1\le l(T_1),\cdots l(T_K)\le N$.

Therefore, we complete the proof of Proposition \ref{prop.operatorupperbound}.
\end{proof}

\subsubsection{Formulas for product of random matrices $\mathcal{P}^l_{T_j}$} In order to prove Proposition \ref{prop.operatorupperbound'}, it is very helpful to find a good formula of  $\prod_{j=1}^K\mathcal{P}^{\tau_j}_{T_j}$, which is the main goal of this section.

$\mathcal{P}^{l}_{T}$ is almost the same as $\mathcal{P}_{T}$ except for the limits in the time integral, so in the rest part of this section we do not stress their difference. Now let's find a tree representation for $\mathcal{P}^{l}_{T}$ or $\mathcal{P}_{T}$.

%as in section \ref{sec.connection}.

By \eqref{eq.treeterm}, we know that $\mathcal{J}_{T}=\mathcal{T}(\mathcal{J}_{T_{\mathfrak{n}_1}}, \mathcal{J}_{T_{\mathfrak{n}_2}})$ corresponds to the tree $T$ in which the two subtrees of the root nodes are $T_{\mathfrak{n}_1}$ and $T_{\mathfrak{n}_2}$. Taking $T_{\mathfrak{n}_2}$ to be an one node tree, as the left tree in Figure \ref{fig.T(J,xi)andT(J,w)}, then this graph represents a term $\mathcal{T}(\mathcal{J}_{T_{\mathfrak{n}_1}}, \xi)$. The right tree in Figure \ref{fig.T(J,xi)andT(J,w)} represents the term $\mathcal{T}(\mathcal{J}_{T_{\mathfrak{n}_1}}, w)$, because as in section \ref{sec.connection}, the $\Box$ node represents a function $w$ (in section \ref{sec.connection} the function is $\phi$).

 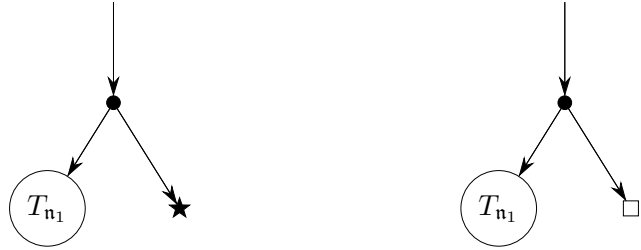
\begin{figure}[H]
    \centering
    \scalebox{0.5}{
    \begin{tikzpicture}[level distance=80pt, sibling distance=100pt]
        \node[] at (0,0) (1) {} 
            child {node[fillcirc] (2) {} 
                child {node[draw, circle, minimum size=1cm, scale=2] (3) {$T_{\mathfrak{n}_1}$}}
                child {node[fillstar] (4) {}}
            };
        \draw[-{Stealth[length=5mm, width=3mm]}] (1) -- (2);
        \draw[-{Stealth[length=5mm, width=3mm]}] (2) -- (3);
        \draw[-{Stealth[length=5mm, width=3mm]}] (2) -- (4);
        
        \node[] at (12,0) (1) {} 
            child {node[fillcirc] (2) {} 
                child {node[draw, circle, minimum size=1cm, scale=2] (3) {$T_{\mathfrak{n}_1}$}}
                child {node[draw, minimum size=0.4cm] (4) {}}
            };
        \draw[-{Stealth[length=5mm, width=3mm]}] (1) -- (2);
        \draw[-{Stealth[length=5mm, width=3mm]}] (2) -- (3);
        \draw[-{Stealth[length=5mm, width=3mm]}] (2) -- (4);
    \end{tikzpicture}
    }
        \caption{Graphical representations of $\mathcal{T}(\mathcal{J}_{T_{\mathfrak{n}_1}}, \xi)$ and $\mathcal{T}(\mathcal{J}_{T_{\mathfrak{n}_1}}, w)$.}
        \label{fig.T(J,xi)andT(J,w)}
    \end{figure}

Let's find the tree representation for $\prod_{j=1}^K\mathcal{P}^{\tau_j}_{T_j}$ or $\prod_{j=1}^K\mathcal{P}_{T_j}$. Recall the expansion process in section \ref{sec.connection}: the replacement of $\Box$ by a branching node indicates the substitution of $\phi$ by $\mathcal{T}(\xi, \xi)$. The action of composition $\mathcal{P}_{T_1}\circ \mathcal{P}_{T_2}(w)=\mathcal{T}(\mathcal{J}_{T_{1}}, \mathcal{T}(\mathcal{J}_{T_{2}}, w))$ is the substitution of $\cdot$ by $\mathcal{T}(\mathcal{J}_{T_{2}}, w)$ in $\mathcal{T}(\mathcal{J}_{T_{1}},\cdot)$. As in Figure \ref{fig.substitution}, if $\mathcal{P}_{T_1}=\mathcal{T}(\mathcal{J}_{T_{1}},\cdot)$ is represented by the left tree, then as an analog that a $\Box$ node is replaced by a branching node, the substitution of $\cdot$ by $\mathcal{T}(\mathcal{J}_{T_{2}}, w)$ should correspond to the operation that the $\Box$ node in the left tree is replaced by the middle tree corresponding to $\mathcal{T}(\mathcal{J}_{T_{2}}, w)$, and the finally resulting right tree should correspond to $\mathcal{P}_{T_1}\circ \mathcal{P}_{T_2}(w)$.

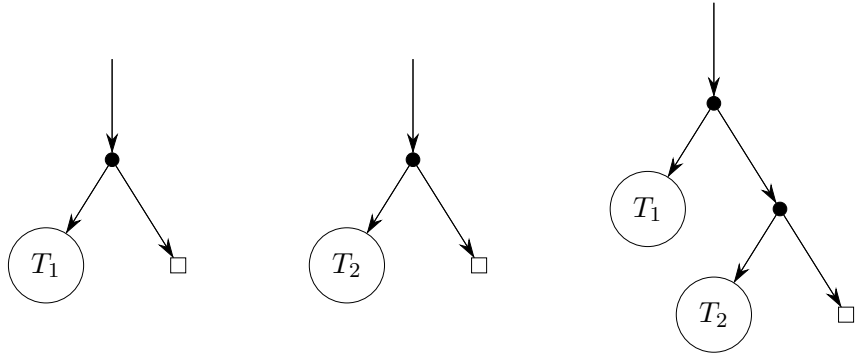
\begin{figure}[H]
    \centering
    \scalebox{0.5}{
    \begin{tikzpicture}[level distance=80pt, sibling distance=100pt]
        \node[] at (0,0) (1) {} 
            child {node[fillcirc] (2) {} 
                child {node[draw, circle, minimum size=1cm, scale=2] (3) {$T_{1}$}}
                child {node[draw, minimum size=0.4cm] (4) {}}
            };
        \draw[-{Stealth[length=5mm, width=3mm]}] (1) -- (2);
        \draw[-{Stealth[length=5mm, width=3mm]}] (2) -- (3);
        \draw[-{Stealth[length=5mm, width=3mm]}] (2) -- (4);
        
        \node[] at (8,0) (1) {} 
            child {node[fillcirc] (2) {} 
                child {node[draw, circle, minimum size=1cm, scale=2] (3) {$T_{2}$}}
                child {node[draw, minimum size=0.4cm] (4) {}}
            };
        \draw[-{Stealth[length=5mm, width=3mm]}] (1) -- (2);
        \draw[-{Stealth[length=5mm, width=3mm]}] (2) -- (3);
        \draw[-{Stealth[length=5mm, width=3mm]}] (2) -- (4);
        
        \node[] at (16,1.5) (1) {} 
            child {node[fillcirc] (2) {} 
                child {node[draw, circle, minimum size=1cm, scale=2] (3) {$T_{1}$}}
                child {node[fillcirc] (4) {} 
                child {node[draw, circle, minimum size=1cm, scale=2] (5) {$T_{2}$}}
                child {node[draw, minimum size=0.4cm] (6) {}}
            }
            };
        \draw[-{Stealth[length=5mm, width=3mm]}] (1) -- (2);
        \draw[-{Stealth[length=5mm, width=3mm]}] (2) -- (3);
        \draw[-{Stealth[length=5mm, width=3mm]}] (2) -- (4);
        \draw[-{Stealth[length=5mm, width=3mm]}] (4) -- (5);
        \draw[-{Stealth[length=5mm, width=3mm]}] (4) -- (6);
    \end{tikzpicture}
    }
        \caption{The substitution process.}
        \label{fig.substitution}
    \end{figure}

In conclusion, $\mathcal{P}_{T_1}\circ \mathcal{P}_{T_2}(w)$ corresponds to the right tree in Figure \ref{fig.substitution} and more generally, $\prod_{j=1}^K\mathcal{P}_{T_j}(w)$ (or $\prod_{j=1}^K\mathcal{P}^{\tau_j}_{T_j}(w)$) corresponds to the tree in in Figure \ref{fig.productformula}.

\begin{figure}[H]
    \centering
    \scalebox{0.4}{
    \begin{tikzpicture}[level distance=80pt, sibling distance=100pt]
        
        \node[] at (0,0) (1) {} 
            child {node[fillcirc] (2) {} 
                child {node[draw, circle, minimum size=1cm, scale=2] (3) {$T_{1}$}}
                child {node[fillcirc] (4) {} 
                child {node[draw, circle, minimum size=1cm, scale=2] (5) {$T_{2}$}}
                child {node[] (6) {}}
            }
            };
            
        \node[scale =3] at (3,-10.5) {$\cdots$};
        \draw[-{Stealth[length=5mm, width=3mm]}] (1) -- (2);
        \draw[-{Stealth[length=5mm, width=3mm]}] (2) -- (3);
        \draw[-{Stealth[length=5mm, width=3mm]}] (2) -- (4);
        \draw[-{Stealth[length=5mm, width=3mm]}] (4) -- (5);
        \draw[-{Stealth[length=5mm, width=3mm]}] (4) -- (6);

        \node[fillcirc] at (5,-12) (11) {} 
            child {node[draw, circle, minimum size=1cm, scale=2] (12) {$T_{K}$}}
            child {node[draw, minimum size=0.4cm] (13) {}
            };
        \node[scale =2] at (0.7,-2.8) {$s_1$};
        \node[scale =2] at (2.5,-5.6) {$s_2$};
        \node[scale =2] at (5.8,-12) {$s_K$};
    \end{tikzpicture}
    }
        \caption{Picture of $T_1\circ \cdots \circ T_{K}$}
        \label{fig.productformula}
    \end{figure}
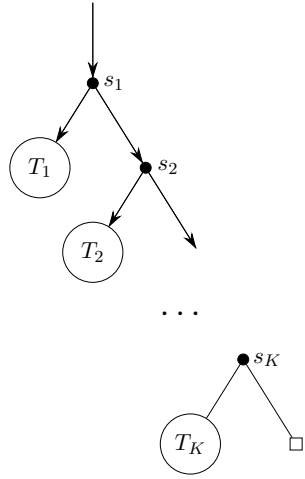

We introduce the following definition.

\begin{defn}
\begin{enumerate}
    \item \textbf{Definition of $T_1\circ \cdots \circ T_{K}$:} We define  $T_1\circ \cdots \circ T_{K}$ to be the tree in Figure \ref{fig.productformula}.
    \item \textbf{Substitution nodes:} A node in a tree $T$ with $\Box$ nodes is defined to be a \underline{substitution node} if it is an ancestor of some $\Box$ nodes. For a substitution node, we assign a number $\tau$ to it, called its index. In Figure \ref{fig.productformula}, $s_1,\cdots,s_{K}$ are all the substitution nodes in $T_1\circ \cdots \circ T_{K}$ and we assign index $\tau_1,\cdots,\tau_{K}$ to them. Notice that $s_1$ is the root $\mathfrak{r}$.
\end{enumerate}
\end{defn}

Because the tree in this section contains $\Box$ nodes and substitution nodes, we propose the following generalization of Definition \ref{def.treeterms} of tree terms.

\begin{defn}\label{def.treetermsoperator} Given a binary tree $T$ with $\Box$ nodes and substitution nodes, we inductively define the quantity $\mathcal{J}_T$ by:
\begin{equation}\label{eq.treetermoperator}
    \mathcal{J}_T=
    \begin{cases}
    \xi, \qquad\qquad\quad\ \   \textit{ if $T$ has only one node $\star$.}
    \\
    w, \qquad\qquad\quad\ \textit{ if $T$ has only one node $\Box$.}
    \\
    \mathcal{T}^l(\mathcal{J}_{T_{\mathfrak{n}_1}}, \mathcal{J}_{T_{\mathfrak{n}_2}}), \textit{ if the root of $T$ is a substitution node with index $\tau$.}
    \\
    \mathcal{T}(\mathcal{J}_{T_{\mathfrak{n}_1}}, \mathcal{J}_{T_{\mathfrak{n}_2}}),\  \textit{ otherwise.}
    \end{cases}
\end{equation}
Here $\mathfrak{n}_1$, $\mathfrak{n}_2$ are two children of the root node $\mathfrak{r}$ and $T_{\mathfrak{n}_1}$, $T_{\mathfrak{n}_2}$ are the subtrees of $T$ rooted at above nodes.
\end{defn}

Then we have the following formula of $\prod_{j=1}^K\mathcal{P}^{\tau_j}_{T_j}(w)$

\begin{lem}
With above definition of $T_1\circ \cdots \circ T_{K}$ and $\mathcal{J}_T$, we have 
\begin{equation}\label{eq.operatoreqsimple}
    \prod_{j=1}^K\mathcal{P}^{\tau_j}_{T_j}(w)=\mathcal{J}_{T_1\circ \cdots \circ T_{K}}
\end{equation}
\end{lem}
\begin{proof}
This lemma follows from the above explanation .
\end{proof}

To get a good upper bound, we need a better formula of $\prod_{j=1}^K\mathcal{P}^{\tau_j}_{T_j}(w)$ which is an analog of Lemma \ref{lem.treeterms}. 

\begin{lem}\label{lem.treetermsoperator} (1) Using the same notation as Lemma \ref{lem.treeterms}. Given a tree $T$ with $\Box$ nodes and substitution nodes. Assume that the root $\mathfrak{r}$ is a substitution nodes of index $\tau_{1}$. Let $\mathcal{J}_T$ be terms defined in Definition \ref{def.treetermsoperator}, then their Fourier coefficients $\mathcal{J}_{T,k}$ are degree $l$ polynomials of $\xi$ and $w$ given by the following formula

\begin{equation}\label{eq.coeftermoperator}
\mathcal{J}_{T,k}=\left(\frac{i\lambda}{L^{d}}\right)^l\sum_{k_1,\, k_2,\, \cdots,\, k_{l+1}} \int_{\cup_{\mathfrak{n}\in T_{\text{in}}} A_{\mathfrak{n}}} e^{\sum_{\mathfrak{n}\in T_{\text{in}}} it_{\mathfrak{n}}\Omega_{\mathfrak{n}}-\nu(t_{\widehat{\mathfrak{n}}}-t_{\mathfrak{n}})|k_{\mathfrak{e}}|^2} \prod_{j=1}^{l+1} (\xi|w)_{k_j} %(\xi|w)_{k_1}(\xi|w)_{k_2}\cdots(\xi|w)_{k_{l+1}}    
\prod_{\mathfrak{n}\in T_{\text{in}}} dt_{\mathfrak{n}} 
\ \delta_{\cap_{\mathfrak{n}\in T_{\text{in}}} \{S_{\mathfrak{n}}=0\}}\ \prod_{\mathfrak{e}\in T_{\text{in}}}\iota_{\mathfrak{e}}k_{\mathfrak{e},x}
\end{equation}
Here $\iota$, $(\xi|w)_{k_j}$, $A_{\mathfrak{n}}$, $S_{\mathfrak{n}}$, $\Omega_{\mathfrak{n}}$ are defined by 
\begin{equation}
    \iota_{\mathfrak{e}}=\begin{cases}
        +1 \qquad \textit{if $\mathfrak{e}$ pointing inwards to $\mathfrak{n}$}
        \\
        -1 \qquad  \textit{if $\mathfrak{e}$ pointing outwards from $\mathfrak{n}$}
    \end{cases}
\end{equation}
\begin{equation}
    (\xi|w)_{k_j}=\begin{cases}
        \xi \qquad\  \textit{if $j$-th leaf node is a $\star$ node}
        \\
        w \qquad  \textit{if $j$-th leaf node is a $\Box$ node}
    \end{cases}
\end{equation}
\begin{equation}
    A_{\mathfrak{n}}=\left\{
    \begin{aligned}
        &\{t_{\mathfrak{r}}\le t, (t-t_{\mathfrak{r}})/T_{\text{max}}\in [2^{-\tau_{1}},2^{-\tau_{1}-1}]\} && \textit{if $\mathfrak{n}$ is the root $\mathfrak{r}$ }
        \\
        &\{t_{\mathfrak{n}_1},\, t_{\mathfrak{n}_2},\, t_{\mathfrak{n}_3}\le t_{\mathfrak{n}}\} && \textit{if $\mathfrak{n}\ne \mathfrak{r}$ and is not a substitution node}
        \\
        &\{t_{\mathfrak{n}_1},\, t_{\mathfrak{n}_2},\, t_{\mathfrak{n}_3}\le t_{\mathfrak{n}}, (t_{\widehat{\mathfrak{n}}}-t_{\mathfrak{n}})/T_{\text{max}}\in [2^{-\tau},2^{-\tau-1}]\}  &&\textit{if $\mathfrak{n}\ne \mathfrak{r}$ is an index $\tau$ substitution node}
    \end{aligned}\right.
\end{equation}

\begin{equation}\label{eq.defnS_noperator}
    S_{\mathfrak{n}}=\iota_{\mathfrak{e}_1}k_{\mathfrak{e}_1}+\iota_{\mathfrak{e}_2}k_{\mathfrak{e}_2}+\iota_{\mathfrak{e}}k_{\mathfrak{e}}
\end{equation}

\begin{equation}
    \Omega_{\mathfrak{n}}=\iota_{\mathfrak{e}_1}\Lambda_{k_{\mathfrak{e}_1}}+\iota_{\mathfrak{e}_2}\Lambda_{k_{\mathfrak{e}_2}}+\iota_{\mathfrak{e}}\Lambda_{k_{\mathfrak{e}}}
\end{equation}
For root node $\mathfrak{r}$, we impose the constrain that $k_{\mathfrak{r}}=k$. We also define $t_{\widehat{\mathfrak{r}}}=t$ to fix the problem that $\widehat{\mathfrak{r}}$ is not well defined because $\mathfrak{r}$ does not have a parent. 

(2) Define $T=T_1\circ \cdots \circ T_{K}$ and by definition there are only one $\Box$ leaf. Assume that there are $l+1$ leaves in $T$ and label all $\star$ leaves by $1$, $\cdots$, $l$, then $l=\sum_{j=1}^K l(T_j)$. As a corollary of (1), we have the following formula for Fourier coefficients of $\prod_{j=1}^K\mathcal{P}^{\tau_j}_{T_j}$.
\begin{equation}
    \left(\prod_{j=1}^K\mathcal{P}^{\tau_j}_{T_j}(w)\right)_{k}(t)=\sum_{k'}\int_0^t H^{\tau_1\cdots \tau_{K}}_{Tkk'}(t,s) w_{k'}(s) ds
\end{equation}
and the kernel $H^{\tau_1\cdots \tau_{K}}_{Tkk'}$ is given by
\begin{equation}
\begin{split}
H^{\tau_1\cdots \tau_{K}}_{Tkk'}(t,s)=\left(\frac{i\lambda}{L^{d}}\right)^l\sum_{k_1,\, k_2,\, \cdots,\, k_{l}} H^{\tau_1\cdots \tau_{K}}_{Tk_1\cdots k_{l}kk'} \xi_{k_1}\cdots \xi_{k_{l}}
\end{split}
\end{equation}
% \begin{equation}
% \begin{split}
% &H^{\tau_1\cdots \tau_{K}}_{Tkk'}(t,s)=
% \\
% &\left(\frac{i\lambda}{L^{d}}\right)^l\sum_{k_1,\, k_2,\, \cdots,\, k_{l}} \int_{\cup_{\mathfrak{n}\in T_{\text{in}}} B_{\mathfrak{n}}} e^{\sum_{\mathfrak{n}\in T_{\text{in}}} it_{\mathfrak{n}}\Omega_{\mathfrak{n}}-\nu(t_{\widehat{\mathfrak{n}}}-t_{\mathfrak{n}})|k_{\mathfrak{e}}|^2}  %(\xi|w)_{k_1}(\xi|w)_{k_2}\cdots(\xi|w)_{k_{l}}    
% \prod_{\mathfrak{n}\in T_{\text{in}}} dt_{\mathfrak{n}} 
% \ \delta_{\cap_{\mathfrak{n}\in T_{\text{in}}} \{S_{\mathfrak{n}}=0\}}\ \prod_{\mathfrak{e}\in T_{\text{in}}}\iota_{\mathfrak{e}}k_{\mathfrak{e},x}\prod_{j=1}^{l} \xi_{k_j}    
% \end{split}
% \end{equation}
and the coefficients $H^{\tau_1\cdots \tau_{K}}_{Tk_1\cdots k_{l}kk'}$ of the kernel is given by 
\begin{equation}
H^{\tau_1\cdots \tau_{K}}_{Tk_1\cdots k_{l}kk'}(t,s)=\int_{\cup_{\mathfrak{n}\in T_{\text{in}}} B_{\mathfrak{n}}} e^{\sum_{\mathfrak{n}\in T_{\text{in}}} it_{\mathfrak{n}}\Omega_{\mathfrak{n}}-\nu(t_{\widehat{\mathfrak{n}}}-t_{\mathfrak{n}})|k_{\mathfrak{e}}|^2}  %(\xi|w)_{k_1}(\xi|w)_{k_2}\cdots(\xi|w)_{k_{l}}    
\prod_{\mathfrak{n}\in T_{\text{in}}} dt_{\mathfrak{n}} 
\ \delta_{\cap_{\mathfrak{n}\in T_{\text{in}}} \{S_{\mathfrak{n}}=0\}}\ \prod_{\mathfrak{e}\in T_{\text{in}}}\iota_{\mathfrak{e}}k_{\mathfrak{e},x}
\end{equation}
Here $k_1,\cdots,k_{l}$ are all variables corresponding to $\star$ leaves, $k'$ is the variable corresponding to the $\Box$ node and $B_{\mathfrak{n}}$ is defined by 
\begin{equation}
    B_{\mathfrak{n}}=\left\{
    \begin{aligned}
        &\{t_{\mathfrak{r}}\le t, (t-t_{\mathfrak{r}})/T_{\text{max}}\in [2^{-\tau_{1}},2^{-\tau_{1}-1}]\} && \textit{if $\mathfrak{n}$ is the root $\mathfrak{r}$ }
        \\
        &\{t_{\mathfrak{n}}\ge s\} && \textit{if $\mathfrak{n}$ is a parent of the $\Box$ nodes}
        \\
        &\  &&\textit{and is not a substitution nodes}
        \\
        &\{t_{\mathfrak{n}}\ge s, (t_{\widehat{\mathfrak{n}}}-t_{\mathfrak{n}})/T_{\text{max}}\in [2^{-\tau},2^{-\tau-1}]\} && \textit{if $\mathfrak{n}$ is a parent of the $\Box$ nodes}
        \\
        &\  &&\textit{and is a substitution nodes of index $\tau$}
        \\
        &\{t_{\mathfrak{n}_1},\, t_{\mathfrak{n}_2},\, t_{\mathfrak{n}_3}\le t_{\mathfrak{n}}, (t_{\widehat{\mathfrak{n}}}-t_{\mathfrak{n}})/T_{\text{max}}\in [2^{-\tau},2^{-\tau-1}]\}  &&\textit{if not in the first three cases}
        \\
        &\  &&\textit{and $\mathfrak{n}$ is a substitution node of index $\tau$}
        \\
        &\{t_{\mathfrak{n}_1},\, t_{\mathfrak{n}_2},\, t_{\mathfrak{n}_3}\le t_{\mathfrak{n}}\} && \textit{otherwise}
    \end{aligned}\right.
\end{equation}
\end{lem}
\begin{proof}
Lemma \ref{lem.treetermsoperator} (1) can be proved by the same method as Lemma \ref{lem.treeterms}. We can check that $\mathcal{J}_T$ defined by \eqref{eq.coefterm} and \eqref{eq.coef} satisfies the recursive formula \eqref{eq.treeterm} by a direct substitution, so they are the unique solution of that recursive formula, and this proves (1).

(2) is a corollary of (1).
\end{proof}

We can also calculate $\prod_{j=1}^K\mathcal{P}^{\tau_j}_{T_j}(\mathcal{J}_{T})$, replacing $w$ by $\mathcal{J}_{T}$ corresponds to replacing a $\Box$ node by a tree $T$, so $\prod_{j=1}^K\mathcal{P}^{\tau_j}_{T_j}(\mathcal{J}_{T})=\mathcal{J}_{T_1\circ T_2\circ \cdots\circ T_K\circ T}$. Since $T$ does not contain $\Box$ node, so does $T_1\circ T_2\circ \cdots\circ T_K\circ T$ and $\mathcal{J}_{T_1\circ T_2\circ \cdots\circ T_K\circ T}$ is a polynomial of Gaussian variables $\xi_k$. Then we have the following lemma
\begin{lem}
We have 
\begin{equation}\label{eq.operatoreqsimpleJ_T}
    \prod_{j=1}^K\mathcal{P}^{\tau_j}_{T_j}(\mathcal{J}_{T})=\mathcal{J}_{T_1\circ T_2\circ \cdots\circ T_K\circ T}
\end{equation}
\end{lem}
\begin{proof}
This lemma follows from the above explanation .
\end{proof}

\subsubsection{The upper bound for coefficients} In this section, we prove the Lemma \ref{lem.boundcoefoperator} which gives an upper bound for $H^{\tau_1\cdots \tau_{K}}_{Tk_1\cdots k_{l}kk'}$. This lemma is an analog of Lemma \ref{lem.boundcoef}.

\begin{lem}\label{lem.boundcoefoperator}
Assume that $|k_1|, \cdots, |k_{l}|\lesssim 1$, then for $t\le  T_{\text{max}}$, we have the following upper bound for $H^{\tau_1\cdots \tau_{K}}_{Tk_1\cdots k_{l}kk'}$,
\begin{equation}\label{eq.boundcoefoperator}
    |H^{\tau_1\cdots \tau_{K}}_{Tk_1\cdots k_{l}kk'}|\lesssim \sum_{\{d_{\mathfrak{n}}\}_{\mathfrak{n}\in T_{\text{in}}}\in\{0,1\}^{l(T)}}\prod_{\mathfrak{n}\in T_{\text{in}}}\frac{2^{-\frac{\tau_{\mathfrak{n}}}{2}}}{|q_{\mathfrak{n}}|+T^{-1}_{\text{max}}}\ \delta_{\cap_{\mathfrak{n}\in T_{\text{in}}} \{S_{\mathfrak{n}}=0\}} \prod_{\mathfrak{e}\in T_{\text{in}}} p_{\mathfrak{e}}.
\end{equation}
Here $\tau_{\mathfrak{n}}$ is defined by 
\begin{equation}
   \tau_{\mathfrak{n}}=\left\{
    \begin{aligned}
        &0 && \textit{if $\mathfrak{n}$ is not a substitution node}
        \\
        &\tau_j  &&\textit{if $\mathfrak{n}$ is the $j$-th substitution node}
    \end{aligned}\right.
\end{equation}
Fix a sequence $\{d_{\mathfrak{n}}\}_{\mathfrak{n}\in T_{\text{in}}}$ whose elements $d_{\mathfrak{n}}$ takes boolean values $\{0,1\}$. We define the sequences $\{p_{\mathfrak{n}}\}_{\mathfrak{n}\in T_{\text{in}}}$, $\{q_{\mathfrak{n}}\}_{\mathfrak{n}\in T_{\text{in}}}$, $\{r_{\mathfrak{n}}\}_{\mathfrak{n}\in T_{\text{in}}}$ by following formulas
\begin{equation}\label{eq.p_noperator}
    p_{\mathfrak{e}}=\frac{|k_{\mathfrak{e},x}|}{|k_{\mathfrak{e},x}|+1}
\end{equation}
\begin{equation}\label{eq.q_noperator}
    q_{\mathfrak{n}}=
    \begin{cases}
    \Omega_{\mathfrak{r}}, \qquad\qquad \textit{ if $\mathfrak{n}=$ the root $\mathfrak{r}$.}
    \\
    \Omega_{\mathfrak{n}}+d_{\mathfrak{n}}q_{\mathfrak{n}'},\ \ \textit{ if $\mathfrak{n}\neq\mathfrak{r}$ and $\mathfrak{n}'$ is the parent of $\mathfrak{n}$.}
    \end{cases}
\end{equation}
\begin{equation}\label{eq.r_noperator}
    r_{\mathfrak{n}}=
    \begin{cases}
    |k_{\mathfrak{r}}|^2, \qquad\qquad \textit{ if $\mathfrak{n}=$ the root $\mathfrak{r}$.}
    \\
    |k_{\mathfrak{n}}|^2+d_{\mathfrak{n}}q_{\mathfrak{n}'},\ \ \textit{ if $\mathfrak{n}\neq\mathfrak{r}$ and $\mathfrak{n}'$ is the parent of $\mathfrak{n}$.}
    \end{cases}
\end{equation}

\end{lem}
\begin{proof} By definition
\begin{equation}
H^{\tau_1\cdots \tau_{K}}_{Tk_1\cdots k_{l}kk'}(t,s)=\int_{\cup_{\mathfrak{n}\in T_{\text{in}}} B_{\mathfrak{n}}} e^{\sum_{\mathfrak{n}\in T_{\text{in}}} it_{\mathfrak{n}}\Omega_{\mathfrak{n}}-\nu(t_{\widehat{\mathfrak{n}}}-t_{\mathfrak{n}})|k_{\mathfrak{e}}|^2}  %(\xi|w)_{k_1}(\xi|w)_{k_2}\cdots(\xi|w)_{k_{l}}    
\prod_{\mathfrak{n}\in T_{\text{in}}} dt_{\mathfrak{n}} 
\ \delta_{\cap_{\mathfrak{n}\in T_{\text{in}}} \{S_{\mathfrak{n}}=0\}}\ \prod_{\mathfrak{e}\in T_{\text{in}}}\iota_{\mathfrak{e}}k_{\mathfrak{e},x}
\end{equation}

For any edge $\mathfrak{e}$, assume that the two end points of $\mathfrak{e}$ are $\mathfrak{n}_1$ and $\mathfrak{n}_2$. If neither $\mathfrak{n}_1$ and $\mathfrak{n}_2$ is a substitution node, then we claim that $|k_{\mathfrak{e}}|\lesssim 1$. 

This is because if no end point of $\mathfrak{e}$ is a substitution node, then the subtree $T_{\mathfrak{e}}$ rooted at the upper end points of $\mathfrak{e}$ does not contain the $\Box$ node as its leaf. Therefore, the momentum conservation (Lemma \ref{lem.freeleg}) is applicable which implies that $k_{\mathfrak{e}}$ is a linear combination of $k_1,\cdots,k_{l}$. By $|k_1|, \cdots, |k_{l}|\lesssim 1$, we get $|k_{\mathfrak{e}}|\lesssim 1$. 

Since $|k_{\mathfrak{e}}|\lesssim 1$, we get $\left|\prod_{\mathfrak{e}\in T_{\text{in}}}\iota_{\mathfrak{e}}k_{\mathfrak{e},x}\right|\lesssim \prod^K_{j=1}|k_{s_j,x}|$. Here $s_1, \cdots, s_K$ are all substitution nodes and $k_{s_j}$ are corresponding variables of edges pointing towards $s_j$.

As in section \ref{sec.uppcoef}, we define 
\begin{equation}\label{eq.defF_Toperator}
F_{T}(t,\{a_{\mathfrak{n}}\}_{\mathfrak{n}\in T_{\text{in}}},\{b_{\mathfrak{n}}\}_{\mathfrak{n}\in T_{\text{in}}})=\int_{\cup_{\mathfrak{n}\in T_{\text{in}}} B_{\mathfrak{n}}} e^{\sum_{\mathfrak{n}\in T_{\text{in}}} it_{\mathfrak{n}} a_{\mathfrak{n}} - \nu(t_{\widehat{\mathfrak{n}}}-t_{\mathfrak{n}})b_{\mathfrak{n}}} \prod_{\mathfrak{n}\in T_{\text{in}}} dt_{\mathfrak{n}} \prod^K_{j=1}|k_{s_j,x}|
\end{equation}

Keeping notation the same as Lemma \ref{lem.boundcoef'}, if we can show that
\begin{equation}\label{eq.lemboundop1}
    |F_{T}(t,\{a_{\mathfrak{n}}\}_{\mathfrak{n}\in T_{\text{in}}},\{b_{\mathfrak{n}}\}_{\mathfrak{n}\in T_{\text{in}}})|\lesssim\sum_{\{d_{\mathfrak{n}}\}_{\mathfrak{n}\in T_{\text{in}}}\in\{0,1\}^{l(T)}}\prod_{\mathfrak{n}\in T_{\text{in}}}\frac{2^{-\frac{\tau_{\mathfrak{n}}}{2}}}{|q_{\mathfrak{n}}|+T^{-1}_{\text{max}}}\prod_{\mathfrak{e}\in T_{\text{in}}} p_{\mathfrak{e}}
,
\end{equation}
then this lemma can be proved by taking $a_{\mathfrak{n}}=\Omega_{\mathfrak{n}}$, $b_{\mathfrak{n}}=|k_{\mathfrak{e}}|^2$ in \eqref{eq.lemboundop1}.

Therefore, it suffices to prove \eqref{eq.lemboundop1}.

We run a similar inductive integration by parts argument of Lemma \ref{lem.boundcoef'}. If the roots $\mathfrak{r}$ is not a substitution node, then the same argument of Lemma \ref{lem.boundcoef'} works. Therefore, we just consider the case that the roots $\mathfrak{r}$ is a substitution node. 

Using the same calculation as \eqref{eq.lemboundcoef'1}, we get 
\begin{equation}
\begin{split}
    F_{T}(t)=\int_{\cup_{\mathfrak{n}\in T_{\text{in}}} B_{\mathfrak{n}}}&e^{it_{\mathfrak{r}}(a_{\mathfrak{r}}+T^{-1}_{\text{max}}\, \text{sgn}(a_{\mathfrak{r}}))- \nu(t-t_{\mathfrak{r}})b_{\mathfrak{r}}} e^{-iT^{-1}_{\text{max}}t_{\mathfrak{r}} \text{sgn}(a_{\mathfrak{r}})} 
    \\
    &e^{\sum_{\mathfrak{n}\in T_{\text{in},1}\cup T_{\text{in},2}} it_{\mathfrak{n}} a_{\mathfrak{n}} - \nu(t_{\widehat{\mathfrak{n}}}-t_{\mathfrak{n}})b_{\mathfrak{n}}}  \left(dt_{\mathfrak{r}}\prod_{j=1}^2\prod_{\mathfrak{n}\in T_{\text{in},j}}dt_{\mathfrak{n}}  \right)|k_{s_1,x}|\prod^K_{j=2}|k_{s_j,x}|
\end{split}
\end{equation}

We do integration by parts in above integrals using Stokes formula. Notice that for $t_{\mathfrak{r}}$, there are four inequality constrains, $t_{\mathfrak{r}}\le t-2^{-\tau_{1}}T_{\text{max}}$, $t_{\mathfrak{r}}\ge t-2^{-\tau_{1}-1}T_{\text{max}}$ and $t_{\mathfrak{r}}\ge t_{\mathfrak{n}_1},t_{\mathfrak{n}_2}$. Notice that the first two come from $(t-t_{\mathfrak{r}})/T_{\text{max}}\in [2^{-\tau_{1}},2^{-\tau_{1}-1}]$.

\begin{equation}\label{eq.lemboundcoefexpandop}
\begin{split}
    F_{T}(t)=&\frac{|k_{s_1,x}|}{ia_{\mathfrak{r}}+iT^{-1}_{\text{max}} \text{sgn}(a_{\mathfrak{r}})+\nu b_{\mathfrak{r}} }\int_{\cup_{\mathfrak{n}\in T_{\text{in}}} B_{\mathfrak{n}}} \frac{d}{dt_{\mathfrak{r}}}e^{it_{\mathfrak{r}}(a_{\mathfrak{r}}+T^{-1}_{\text{max}}\, \text{sgn}(a_{\mathfrak{r}}))- \nu(t-t_{\mathfrak{r}})b_{\mathfrak{r}}}  
    \\
    &\qquad\qquad\qquad\ \  e^{-iT^{-1}_{\text{max}}t_{\mathfrak{r}} \text{sgn}(a_{\mathfrak{r}})} e^{\sum_{\mathfrak{n}\in T_{\text{in},1}\cup T_{\text{in},2}} it_{\mathfrak{n}} a_{\mathfrak{n}} - \nu(t_{\widehat{\mathfrak{n}}}-t_{\mathfrak{n}})b_{\mathfrak{n}}}  \left(dt_{\mathfrak{r}}\prod_{j=1}^2\prod_{\mathfrak{n}\in T_{\text{in},j}}dt_{\mathfrak{n}}  \right)\prod^K_{j=2}|k_{s_j,x}|
\end{split}
\end{equation}
\begin{flalign*}
\hspace{1.3cm}
=&\frac{|k_{s_1,x}|}{ia_{\mathfrak{r}}+iT^{-1}_{\text{max}} \text{sgn}(a_{\mathfrak{r}})+\nu b_{\mathfrak{r}} }&&
\\
&\left(\int_{\cup_{\mathfrak{n}\in T_{\text{in}}} B_{\mathfrak{n}},\ t_{\mathfrak{r}}=t-2^{-\tau_{1}}T_{\text{max}}}
-\int_{\cup_{\mathfrak{n}\in T_{\text{in}}} B_{\mathfrak{n}},\ t_{\mathfrak{r}}=t-2^{-\tau_{1}-1}T_{\text{max}}}
-\int_{\cup_{\mathfrak{n}\in T_{\text{in}}} B_{\mathfrak{n}},\ t_{\mathfrak{r}}=t_{\mathfrak{n}_1}}
-\int_{\cup_{\mathfrak{n}\in T_{\text{in}}} B_{\mathfrak{n}},\ t_{\mathfrak{r}}=t_{\mathfrak{n}_2}}\right) &&
\\
& e^{it_{\mathfrak{r}}(a_{\mathfrak{r}}+T^{-1}_{\text{max}}\, \text{sgn}(a_{\mathfrak{r}}))- \nu(t-t_{\mathfrak{r}})b_{\mathfrak{r}}} e^{-iT^{-1}_{\text{max}}t_{\mathfrak{r}} \text{sgn}(a_{\mathfrak{r}})} e^{\sum_{\mathfrak{n}\in T_{\text{in},1}\cup T_{\text{in},2}} it_{\mathfrak{n}} a_{\mathfrak{n}} - \nu(t_{\widehat{\mathfrak{n}}}-t_{\mathfrak{n}})b_{\mathfrak{n}}} \left(dt_{\mathfrak{r}}\prod_{j=1}^2\prod_{\mathfrak{n}\in T_{\text{in},j}}dt_{\mathfrak{n}}  \right)\prod^K_{j=2}|k_{s_j,x}| &&
\end{flalign*}
\begin{flalign*}
\hspace{1.3cm}
-&\frac{|k_{s_1,x}|}{ia_{\mathfrak{r}}+iT^{-1}_{\text{max}} \text{sgn}(a_{\mathfrak{r}})+\nu b_{\mathfrak{r}} }\int_{\cup_{\mathfrak{n}\in T_{\text{in}}} B_{\mathfrak{n}}}e^{it_{\mathfrak{r}}(a_{\mathfrak{r}}+T^{-1}_{\text{max}}\, \text{sgn}(a_{\mathfrak{r}}))- \nu(t-t_{\mathfrak{r}})b_{\mathfrak{r}}} &&
    \\
    &\qquad\qquad \frac{d}{dt_{\mathfrak{r}}}(e^{-iT^{-1}_{\text{max}}t_{\mathfrak{r}} \text{sgn}(a_{\mathfrak{r}})}) e^{\sum_{\mathfrak{n}\in T_{\text{in},1}\cup T_{\text{in},2}} it_{\mathfrak{n}} a_{\mathfrak{n}} - \nu(t_{\widehat{\mathfrak{n}}}-t_{\mathfrak{n}})b_{\mathfrak{n}}}  \left(dt_{\mathfrak{r}}\prod_{j=1}^2\prod_{\mathfrak{n}\in T_{\text{in},j}}dt_{\mathfrak{n}}  \right)\prod^K_{j=2}|k_{s_j,x}| &&
\end{flalign*}
\begin{flalign*}
\hspace{1.3cm}
= \frac{|k_{s_1,x}|}{ia_{\mathfrak{r}}+iT^{-1}_{\text{max}} \text{sgn}(a_{\mathfrak{r}})+\nu b_{\mathfrak{r}} }(F_{I}-F_{I'}-\widetilde{F}_{T^{(1)}}-\widetilde{F}_{T^{(2)}}-F_{II}) &&
\end{flalign*}

We now derive upper bounds for $F_{I}$, $F_{I'}$, $\widetilde{F}_{T^{(1)}}$, $\widetilde{F}_{T^{(2)}}$, $F_{II}$.

The argument of $F_{I}$ and $F_{I'}$ is very similar, so we just consider $F_{I}$. By a direct calculation, we know that $\frac{|k_{s_1,x}|}{ia_{\mathfrak{r}}+iT^{-1}_{\text{max}} \text{sgn}(a_{\mathfrak{r}})+\nu b_{\mathfrak{r}} }F_{I}(t)$ equals to
\begin{equation}
    \begin{split}
        &\frac{|k_{s_1,x}|}{ia_{\mathfrak{r}}+iT^{-1}_{\text{max}} \text{sgn}(a_{\mathfrak{r}})+\nu b_{\mathfrak{r}} } \int_{\cup_{\mathfrak{n}\in T_{\text{in}}} B_{\mathfrak{n}},\ t_{\mathfrak{r}}=t-2^{-\tau_{1}}T_{\text{max}}} e^{it_{\mathfrak{r}}(a_{\mathfrak{r}}+T^{-1}_{\text{max}}\, \text{sgn}(a_{\mathfrak{r}}))- \nu(t-t_{\mathfrak{r}})b_{\mathfrak{r}}} e^{-iT^{-1}_{\text{max}}t_{\mathfrak{r}} \text{sgn}(a_{\mathfrak{r}})} 
        \\
        &e^{\sum_{\mathfrak{n}\in T_{\text{in},1}\cup T_{\text{in},2}} it_{\mathfrak{n}} a_{\mathfrak{n}} - \nu(t_{\widehat{\mathfrak{n}}}-t_{\mathfrak{n}})b_{\mathfrak{n}}} \left(dt_{\mathfrak{r}}\prod_{j=1}^2\prod_{\mathfrak{n}\in T_{\text{in},j}}dt_{\mathfrak{n}}  \right)\prod^K_{j=2}|k_{s_j,x}|
        \\
        =&\frac{|k_{s_1,x}|}{ia_{\mathfrak{r}}+iT^{-1}_{\text{max}} \text{sgn}(a_{\mathfrak{r}})+\nu b_{\mathfrak{r}} }  e^{i(t-2^{-\tau_{1}}T_{\text{max}})a_{\mathfrak{r}}- \nu T_{\text{max}} 2^{-\tau_{1}}b_{\mathfrak{r}}}\int_{\cup_{\mathfrak{n}\in T_{\text{in},1}\cup T_{\text{in},2}} B_{\mathfrak{n}},\ t_{\mathfrak{r}_1}, t_{\mathfrak{r}_2}\lesssim t-2^{-\tau_{1}}T_{\text{max}}}
        \\
        &  e^{\sum_{\mathfrak{n}\in T_{\text{in},1}\cup T_{\text{in},2}} it_{\mathfrak{n}} a_{\mathfrak{n}} - \nu(t_{\widehat{\mathfrak{n}}}-t_{\mathfrak{n}})b_{\mathfrak{n}}}\left(dt_{\mathfrak{r}}\prod_{j=1}^2\prod_{\mathfrak{n}\in T_{\text{in},j}}dt_{\mathfrak{n}}  \right)\prod^K_{j=2}|k_{s_j,x}|
        \\
        =&O\left(\left|\frac{|k_{s_1,x}|}{ia_{\mathfrak{r}}+iT^{-1}_{\text{max}} \text{sgn}(a_{\mathfrak{r}})+\nu b_{\mathfrak{r}} }  e^{i(t-2^{-\tau_{1}}T_{\text{max}})a_{\mathfrak{r}}- \nu T_{\text{max}} 2^{-\tau_{1}}b_{\mathfrak{r}}}\right||F_{T_1}(t)| |F_{T_2}(t)|\right)
        \\
        =&O\left(\frac{(|k_{s_1,x}|+1)e^{- 2^{-\tau_{1}}|k_{s_1}|^2} }{|q_{\mathfrak{r}}|+T^{-1}_{\text{max}}} |F_{T_1}(t)| |F_{T_2}(t)|\frac{|k_{s_1,x}|}{|k_{s_1,x}|+1}\right)=O\left(\frac{2^{-\frac{\tau_{1}}{2}}|F_{T_1}(t)| |F_{T_2}(t)|}{|q_{\mathfrak{r}}|+T^{-1}_{\text{max}}}p_{\mathfrak{e}}\right)
    \end{split}
\end{equation}
Here in the last line we use the fact that $b_{\mathfrak{r}}=|k_{s_1}|^2$, $\nu T_{\text{max}}\gtrsim 1$ \eqref{eq.conditionnu} and $(|k_{s_1,x}|+1)e^{- 2^{-\tau_{1}}|k_{s_1}|^2} \lesssim 2^{-\frac{\tau_{1}}{2}}$. By above equation and the induction assumption, we know that $\frac{|k_{s_1,x}|}{ia_{\mathfrak{r}}+iT^{-1}_{\text{max}} \text{sgn}(a_{\mathfrak{r}})+\nu b_{\mathfrak{r}} }F_{I}(t)$ can be bounded by the right hand side of \eqref{eq.boundcoefoperator}.

Now we find upper bounds of $\widetilde{F}_{T^{(1)}}$ and $\widetilde{F}_{T^{(2)}}$. Let $T^{(j)}$, $j=1,2$ are trees that is obtained by deleting the root $\mathfrak{r}$, adding edges connecting $\mathfrak{n}_j$ with another node and defining $\mathfrak{n}_j$ to be the new root. For $T^{(j)}$, we can define the term $F_{T^{(j)}}$ by \eqref{eq.defF_Toperator}. By a direct calculation, we know that $\frac{|k_{s_1,x}|}{ia_{\mathfrak{r}}+iT^{-1}_{\text{max}} \text{sgn}(a_{\mathfrak{r}})+\nu b_{\mathfrak{r}} }\widetilde{F}_{T^{(1)}}(t)$ equals to
\begin{equation}
    \begin{split}
        &\frac{|k_{s_1,x}|}{ia_{\mathfrak{r}}+iT^{-1}_{\text{max}} \text{sgn}(a_{\mathfrak{r}})+\nu b_{\mathfrak{r}} } \int_{\cup_{\mathfrak{n}\in T_{\text{in}}} B_{\mathfrak{n}},\ t_{\mathfrak{r}}=t_{\mathfrak{n}_j}} e^{it_{\mathfrak{r}}(a_{\mathfrak{r}}+T^{-1}_{\text{max}}\, \text{sgn}(a_{\mathfrak{r}}))- \nu(t-t_{\mathfrak{r}})b_{\mathfrak{r}}} e^{-iT^{-1}_{\text{max}}t_{\mathfrak{r}} \text{sgn}(a_{\mathfrak{r}})} 
        \\
        &e^{\sum_{\mathfrak{n}\in T_{\text{in},1}\cup T_{\text{in},2}} it_{\mathfrak{n}} a_{\mathfrak{n}} - \nu(t_{\widehat{\mathfrak{n}}}-t_{\mathfrak{n}})b_{\mathfrak{n}}} \left(dt_{\mathfrak{r}}\prod_{j=1}^2\prod_{\mathfrak{n}\in T_{\text{in},j}}dt_{\mathfrak{n}}  \right)\prod^K_{j=2}|k_{s_j,x}|
        \\
        =&\frac{|k_{s_1,x}|}{ia_{\mathfrak{r}}+iT^{-1}_{\text{max}} \text{sgn}(a_{\mathfrak{r}})+\nu b_{\mathfrak{r}} }  e^{- \nu T_{\text{max}} 2^{-\tau_{1}}b_{\mathfrak{r}}}\int_{\cup_{\mathfrak{n}\in T_{\text{in}}} B_{\mathfrak{n}}} e^{it_{\mathfrak{n}_j}(a_{\mathfrak{r}}+T^{-1}_{\text{max}}\, \text{sgn}(a_{\mathfrak{r}}))- \nu((t-T_{\text{max}} 2^{-\tau_{1}})-t_{\mathfrak{n}_j})b_{\mathfrak{r}}}
        \\
        &  e^{\sum_{\mathfrak{n}\in T_{\text{in},1}\cup T_{\text{in},2}} it_{\mathfrak{n}} a_{\mathfrak{n}} - \nu(t_{\widehat{\mathfrak{n}}}-t_{\mathfrak{n}})b_{\mathfrak{n}}}\left(dt_{\mathfrak{r}}\prod_{j=1}^2\prod_{\mathfrak{n}\in T_{\text{in},j}}dt_{\mathfrak{n}}  \right)\prod^K_{j=2}|k_{s_j,x}|
        \\
        =&O\left(\frac{|k_{s_1,x}|e^{- 2^{-\tau_{1}}|k_{s_1}|^2} }{|q_{\mathfrak{r}}|+T^{-1}_{\text{max}}} |F_{T^{(j)}}(t)|\right)=O\left(\frac{2^{-\frac{\tau_{1}}{2}}|F_{T^{(j)}}(t)| }{|q_{\mathfrak{r}}|+T^{-1}_{\text{max}}}p_{\mathfrak{e}}\right)
    \end{split}
\end{equation}
We can apply the induction assumption to $F_{T^{(j)}}$ and show that $\frac{|k_{s_1,x}|}{ia_{\mathfrak{r}}+iT^{-1}_{\text{max}} \text{sgn}(a_{\mathfrak{r}})+\nu b_{\mathfrak{r}} } F_{T^{(j)}}$ can be bounded by the right hand side of \eqref{eq.boundcoefoperator}.

Another direct calculation gives that 
\begin{equation}
    F_{II}(t)=\int_{(t-t_{\mathfrak{r}})/T_{\text{max}}\in [2^{-\tau_{1}},2^{-\tau_{1}-1}]}  e^{it_{\mathfrak{r}}(a_{\mathfrak{r}}+T^{-1}_{\text{max}}\, \text{sgn}(a_{\mathfrak{r}}))- \nu(t-t_{\mathfrak{r}})b_{\mathfrak{r}}} \frac{d}{dt_{\mathfrak{r}}}(e^{-iT^{-1}_{\text{max}}t_{\mathfrak{r}} \text{sgn}(a_{\mathfrak{r}})})  F_{T_1}(t_{\mathfrak{r}})F_{T_2}(t_{\mathfrak{r}}) dt_{\mathfrak{r}}.
\end{equation}
Apply the induction assumption
\begin{equation}
\begin{split}
    &\left| \frac{|k_{s_1,x}|}{ia_{\mathfrak{r}}+iT^{-1}_{\text{max}} \text{sgn}(a_{\mathfrak{r}})+\nu b_{\mathfrak{r}} } F_{II}(t)\right|
    \\
    \le &\frac{T^{-1}_{\text{max}}|k_{s_1,x}|}{|q_{\mathfrak{r}}|+T^{-1}_{\text{max}}}\int_{(t-t_{\mathfrak{r}})/T_{\text{max}}\in [2^{-\tau_{1}},2^{-\tau_{1}-1}]}  e^{- \nu(t-t_{\mathfrak{r}})b_{\mathfrak{r}}}   |F_{T_1}(t_{\mathfrak{r}})| |F_{T_2}(t_{\mathfrak{r}})| dt_{\mathfrak{r}}
    \\
    \le& \frac{T^{-1}_{\text{max}} |k_{s_1,x}| e^{- \nu T_{\text{max}} 2^{-\tau_{1}}|k_{s_1}|^2}}{|q_{\mathfrak{r}}|+T^{-1}_{\text{max}}}\prod_{j=1}^2\left(\sum_{\{d_{\mathfrak{n}}\}_{\mathfrak{n}\in T_{\text{in},j}}\in\{0,1\}^{l(T_j)}}\prod_{\mathfrak{n}\in T_{\text{in},j}}\frac{2^{-\frac{\tau_{\mathfrak{n}}}{2}}}{|q_{\mathfrak{n}}|+T^{-1}_{\text{max}}}\prod_{\mathfrak{e}\in T_{\text{in},j}} p_{\mathfrak{e}}\right)
    \\
    \le& \sum_{\{d_{\mathfrak{n}}\}_{\mathfrak{n}\in T_{\text{in}}}\in\{0,1\}^{l(T)}}\prod_{\mathfrak{n}\in T_{\text{in}}}\frac{2^{-\frac{\tau_{\mathfrak{n}}}{2}}}{|q_{\mathfrak{n}}|+T^{-1}_{\text{max}}}\prod_{\mathfrak{e}\in T_{\text{in}}} p_{\mathfrak{e}}.
\end{split}
\end{equation}

Therefore, we get an upper bound for $F_{II}$.

Combining the bounds of $F_{I}$, $F_{I'}$, $\widetilde{F}_{T^{(1)}}$, $\widetilde{F}_{T^{(2)}}$, $F_{II}$, we conclude that $F_T$ can be bounded by the right hand side of \eqref{eq.boundcoefoperator} and thus complete the proof of Lemma \ref{lem.boundcoef'}.
\end{proof}

\subsubsection{The upper bound for expectation of entries} In this section, we prove the Proposition \ref{prop.treetermsvarianceoperator} which gives an upper bound for $\mathbb{E}|H^{\tau_1\cdots \tau_{K}}_{Tkk'}|^2$. This lemma is an analog of Proposition \ref{prop.treetermsvariance}.

\begin{prop}\label{prop.treetermsvarianceoperator}
Assume that $\alpha$ satisfies \eqref{eq.conditionalpha}. For any $\theta>0$, we have
\begin{equation}
    \sup_k\, \mathbb{E}|H^{\tau_1\cdots \tau_{K}}_{Tkk'}|^2\lesssim L^{O(l(T)\theta)} 2^{-\frac{1}{2}\sum_{j=1}^K \tau_{j}} \rho^{2l(T)}.
\end{equation}
and $\mathbb{E}|H^{\tau_1\cdots \tau_{K}}_{Tkk'}|^2=0$ if $|k-k'|\gtrsim 1$.
\end{prop}

\begin{proof} We first find a formula of $\mathbb{E}|H^{\tau_1\cdots \tau_{K}}_{Tkk'}|^2$ which is similar to \eqref{eq.termTp} and \eqref{eq.termexp}.

A direct calculation gives
\begin{equation}\label{eq.termexp1op}
\begin{split}
    \mathbb{E}|H^{\tau_1\cdots \tau_{K}}_{Tkk'}|^2=&\mathbb{E}(H^{\tau_1\cdots \tau_{K}}_{Tkk'}\overline{H^{\tau_1\cdots \tau_{K}}_{Tkk'}})=\left(\frac{\lambda}{L^{d}}\right)^{2l(T)}
    \sum_{k_1,\, k_2,\, \cdots,\, k_{l(T)}}\sum_{k'_1,\, k'_2,\, \cdots,\, k'_{l(T)}}
    \\[0.5em]
    & H^{\tau_1\cdots \tau_{K}}_{Tk_1\cdots k_{l(T)}kk'} \overline{H^{\tau_1\cdots \tau_{K}}_{Tk_1'\cdots k_{l(T)}'kk'}}  \mathbb{E}\Big(\xi_{k_1}\xi_{k_2}\cdots\xi_{k_{l(T)}}\xi_{k'_1}\xi_{k'_2}\cdots\xi_{k'_{l(T)}}\Big)
\end{split}
\end{equation}

Applying Wick theorem (Lemma \ref{th.wick}) to \eqref{eq.termexp1op}, we get
\begin{equation}\label{eq.termexpop}
\mathbb{E}|H^{\tau_1\cdots \tau_{K}}_{Tkk'}|^2=\left(\frac{\lambda}{L^{d}}\right)^{2l(T)}
    \sum_{p\in \mathcal{P}(\{k_1,\cdots, k_{l(T)}, k'_1,\cdots, k'_{l(T)}\})} Term(T, p)_{op,k,k
    '}.
\end{equation}
and
\begin{equation}\label{eq.termTpop}
\begin{split}
    &Term(T, p)_{op,k,k
    '}
    \\
    =&\sum_{k_1,\, k_2,\, \cdots,\, k_{l(T)}}\sum_{k'_1,\, k'_2,\, \cdots,\, k'_{l(T)}} H^{\tau_1\cdots \tau_{K}}_{Tk_1\cdots k_{l(T)}kk'} \overline{H^{\tau_1\cdots \tau_{K}}_{Tk_1'\cdots k_{l(T)}'kk'}} \delta_{p}(k_1,\cdots, k_{l(T)}, k'_1,\cdots, k'_{l(T)})\sqrt{n_{\textrm{in}}(k_1)}\cdots.
\end{split}
\end{equation}

Since $\alpha=\frac{\lambda}{L^{\frac{d}{2}}}$, $\frac{\lambda}{L^{d}}=\alpha L^{-\frac{d}{2}}$. Since the number of elements in $\mathcal{P}$ can be bounded by a constant, by Lemma \ref{lem.Tpvarianceop} proved below, we get
\begin{equation}
\begin{split}
    \mathbb{E}|H^{\tau_1\cdots \tau_{K}}_{Tkk'}|^2\lesssim& (\alpha L^{-\frac{d}{2}})^{2l(T)}
    L^{O(l(T)\theta)}2^{-\frac{1}{2}\sum_{j=1}^K \tau_{j}} (L^dT^{-1}_{\text{max}})^{l(T)} T_{\text{max}}^{2l(T)}
    \\
    =& L^{O(l(T)\theta)} 2^{-\frac{1}{2}\sum_{j=1}^K \tau_{j}} \rho^{2l(T)}
\end{split}
\end{equation}

By Lemma \ref{lem.Tpvariance} below $Term(T, p)_{op,k,k
    '}=0$ if $|k-k'|\gtrsim 1$, we know that the same is true for $\mathbb{E}|H^{\tau_1\cdots \tau_{K}}_{Tkk'}|^2=0$. 

Therefore, we complete the proof of this proposition.
\end{proof}

\begin{lem}\label{lem.Tpvarianceop} Let $Q=L^dT^{-1}_{\text{max}}$ be the same as in Proposition \ref{prop.counting}. Assume that $\alpha$ satisfies \eqref{eq.conditionalpha} and $n_{\mathrm{in}} \in C^\infty_0(\mathbb{R}^d)$ is compactly supported. Then for any $\theta>0$, we have
\begin{equation}
    \sup_k\, |Term(T, p)_{op,k,k
    '}|\le L^{O(l(T)\theta)}2^{-\frac{1}{2}\sum_{j=1}^K \tau_{j}} Q^{l(T)} T_{\text{max}}^{2l(T)}.
\end{equation}
and $Term(T, p)_{op,k,k
    '}=0$ if $|k-k'|\gtrsim 1$.
\end{lem}
\begin{proof} By \eqref{eq.termTpop}, we get
\begin{equation}
\begin{split}
    &Term(T, p)_{op,k,k
    '}
    \\
    =&\sum_{k_1,\, k_2,\, \cdots,\, k_{l(T)}}\sum_{k'_1,\, k'_2,\, \cdots,\, k'_{l(T)}} H^{\tau_1\cdots \tau_{K}}_{Tk_1\cdots k_{l(T)}kk'} \overline{H^{\tau_1\cdots \tau_{K}}_{Tk_1'\cdots k_{l(T)}'kk'}} \delta_{p}(k_1,\cdots, k_{l(T)}, k'_1,\cdots, k'_{l(T)})\sqrt{n_{\textrm{in}}(k_1)}\cdots.
\end{split}
\end{equation}

Since $n_{\mathrm{in}}$ are compactly supported and there are bounded many of them in $Term(T, p)_{op,k,k'}$, by $k_1 + k_2 + \cdots + k_{l(T)}=k-k'$, we know that $Term(T, p)_{op,k,k'}=0$ if $|k-k'|\gtrsim 1$.

By \eqref{eq.boundcoefoperator}, we get  %$\frac{1}{\sqrt{(|c(\Omega)_{\mathfrak{n}}|+T^{-1}_{\text{max}})^2+|r_{\mathfrak{n}}|^2}}\lesssim \frac{1}{|c(\Omega)_{\mathfrak{n}}|+T^{-1}_{\text{max}}}$
\begin{equation}\label{eq.termlemmaeq1op}
    |H^{\tau_1\cdots \tau_{K}}_{Tk_1\cdots k_{l}kk'}|\lesssim \sum_{\{d_{\mathfrak{n}}\}_{\mathfrak{n}\in T_{\text{in}}}\in\{0,1\}^{l(T)}}\prod_{\mathfrak{n}\in T_{\text{in}}}\frac{2^{-\frac{\tau_{\mathfrak{n}}}{2}}}{|q_{\mathfrak{n}}|+T^{-1}_{\text{max}}}\ \delta_{\cap_{\mathfrak{n}\in T_{\text{in}}} \{S_{\mathfrak{n}}=0\}}\prod_{\mathfrak{e}\in T_{\text{in}}} p_{\mathfrak{e}}.
\end{equation}

Define $[c_{\mathfrak{n},\widetilde{\mathfrak{n}}}]$, $\mathscr{M}(T)$, $c(\Omega)$ in the same way as the proof of Lemma \ref{lem.Tpvariance}. We can apply the same derivation of \eqref{eq.termlemmaeq3} to obtain
\begin{equation}\label{eq.termlemmaeq3op}
\begin{split}
    &|Term(T, p)_{op,k,k'}|\lesssim \sum_{\substack{k_1,\, \cdots,\, k_{l(T)},\, k'_1,\, \cdots,\, k'_{l(T)}\\ |k_{j}|, |k'_j|\lesssim 1, \forall j}} \sum_{c\in \mathscr{M}(T) }\prod_{\mathfrak{n}\in T_{\text{in}}}\frac{2^{-\frac{\tau_{\mathfrak{n}}}{2}}}{|c(\Omega)_{\mathfrak{n}}|+T^{-1}_{\text{max}}}    \ \delta_{\cap_{\mathfrak{n}\in T_{\text{in}}} \{S_{\mathfrak{n}}=0\}}\prod_{\mathfrak{e}\in T_{\text{in}}} p_{\mathfrak{e}} 
    \\
    &\sum_{c'\in \mathscr{M}(T)}\prod_{\mathfrak{n}'\in T_{\text{in}}}\frac{2^{-\frac{\tau_{\mathfrak{n}}}{2}}}{|c'(\Omega)_{\mathfrak{n}'}|+T^{-1}_{\text{max}}} \prod_{\mathfrak{e}'\in T_{\text{in}}} p_{\mathfrak{e}'}  \ \delta_{\cap_{\mathfrak{n}'\in T_{\text{in}}} \{S_{\mathfrak{n}'}=0\}} \delta_{p}(k_1,\cdots, k_{l(T)}, k'_1,\cdots, k'_{l(T)})
\end{split}
\end{equation}

We obviously have the following inequality
\begin{equation}\label{eq.termlemmaeq4op}
\begin{split}
    &|Term(T, p)_{op,k,k'}|\lesssim \sum_{k'}|Term(T, p)_{op,k,k'}|
    \\
    \lesssim& \sum_{\substack{k_1,\, \cdots,\, k_{l(T)},\, k'_1,\, \cdots,\, k'_{l(T)}, k'\\ |k_{j}|, |k'_j|\lesssim 1, \forall j,\ |k'-k|\lesssim 1}} \sum_{c\in \mathscr{M}(T) }\prod_{\mathfrak{n}\in T_{\text{in}}}\frac{2^{-\frac{\tau_{\mathfrak{n}}}{2}}}{|c(\Omega)_{\mathfrak{n}}|+T^{-1}_{\text{max}}}    \ \delta_{\cap_{\mathfrak{n}\in T_{\text{in}}} \{S_{\mathfrak{n}}=0\}} \prod_{\mathfrak{e}\in T_{\text{in}}} p_{\mathfrak{e}}
    \\
    &\sum_{c'\in \mathscr{M}(T)}\prod_{\mathfrak{n}'\in T_{\text{in}}}\frac{2^{-\frac{\tau_{\mathfrak{n}}}{2}}}{|c'(\Omega)_{\mathfrak{n}'}|+T^{-1}_{\text{max}}}  \prod_{\mathfrak{e}'\in T_{\text{in}}} p_{\mathfrak{e}'}  \ \delta_{\cap_{\mathfrak{n}'\in T_{\text{in}}} \{S_{\mathfrak{n}'}=0\}} \delta_{p}(k_1,\cdots, k_{l(T)}, k'_1,\cdots, k'_{l(T)})
\end{split}
\end{equation}

Switch the order of summations and products in \eqref{eq.termlemmaeq3}, then we get
\begin{equation}\label{eq.termlemmaeq2op}
\begin{split}
    |Term(T, p)_{op,k,k'}|\lesssim& \sum_{\substack{k_1,\, \cdots,\, k_{l(T)},\, k'_1,\, \cdots,\, k'_{l(T)}, k'\\ |k_{j}|, |k'_j|\lesssim 1, \forall j,\ |k'-k|\lesssim 1}} \sum_{c, c'\in \mathscr{M}(T) }\prod_{\mathfrak{n}, \mathfrak{n}'\in T_{\text{in}}}\frac{2^{-\frac{\tau_{\mathfrak{n}}}{2}}}{|c(\Omega)_{\mathfrak{n}}|+T^{-1}_{\text{max}}}\frac{2^{-\frac{\tau_{\mathfrak{n}}}{2}}}{|c'(\Omega)_{\mathfrak{n}'}|+T^{-1}_{\text{max}}}
    \\
    & \prod_{\mathfrak{e},\mathfrak{e}'\in T_{\text{in}}} (p_{\mathfrak{e}}p_{\mathfrak{e}'})\ \delta_{\cap_{\mathfrak{n},\mathfrak{n}'\in T_{\text{in}}} \{S_{\mathfrak{n}}=0, S_{\mathfrak{n}'}=0\}} \delta_{p}(k_1,\cdots, k_{l(T)}, k'_1,\cdots, k'_{l(T)}).
\end{split}
\end{equation}

Consider a tree $T$ with a $\Box$ nodes and a pairing $p\in \mathcal{P}(\{k_1,\cdots, k_{l(T)}, k'_1,\cdots, k'_{l(T)}\})$. $p$ can be viewed as a pairing of all star nodes of two copies of $T$. A couple $\mathcal{C}$ can be constructed by merging all paired star nodes according $p$ and merging the two $\Box$ nodes. As in \eqref{eq.termlemmaeq4}, we can show that
\begin{equation}\label{eq.termlemmaeq4op'}
\sum_{c, c'\in \mathscr{M}(T) }=\sum_{c\in \mathscr{M}(\mathcal{C}) },\qquad \prod_{\mathfrak{n}, \mathfrak{n}'\in T_{\text{in}}}=\prod_{\mathfrak{n}\in \mathcal{C}}, \qquad \prod_{\mathfrak{e},\mathfrak{e}'\in T_{\text{in}}}=\prod_{\mathfrak{e}\in \mathcal{C}_{\text{norm}}},\qquad \cap_{\mathfrak{n},\mathfrak{n}'\in T_{\text{in}}}=\cap_{\mathfrak{n}\in \mathcal{C}}.    
\end{equation}

The analog of \eqref{eq.termlemmaeq5} is 
\begin{equation}\label{eq.termlemmaeq5op}
|Term(T, p)_{op,k,k'}|\lesssim \sum_{\substack{k_1,\, \cdots,\, k_{l(T)},\, k'_1,\, \cdots,\, k'_{l(T)}, k'\\ |k_{j}|, |k'_j|\lesssim 1, \forall j,\ |k'-k|\lesssim 1}} \sum_{c\in \mathscr{M}(\mathcal{C}) }\prod_{\mathfrak{n}\in \mathcal{C}}\frac{2^{-\frac{\tau_{\mathfrak{n}}}{2}}}{|c(\Omega)_{\mathfrak{n}}|+T^{-1}_{\text{max}}} \prod_{\mathfrak{e}\in \mathcal{C}_{\text{norm}}} p_{\mathfrak{e}}\  \delta_{\cap_{\mathfrak{n}\in \mathcal{C}} \{S_{\mathfrak{n}}=0\}}
\end{equation}

The rest part of the proof is exactly the same as the proof of Lemma \ref{lem.Tpvariance} after \eqref{eq.termlemmaeq5}. For completeness, we include a sketch.

As \eqref{eq.termlemmaeq8}, we have  
\begin{equation}
    \sum_{\substack{k_1,\, \cdots,\, k_{l(T)+1},\, k'_1,\, \cdots,\, k'_{l(T)+1}\\\cap_{\mathfrak{n}\in \mathcal{C}} \{S_{\mathfrak{n}}=0\}}}=\sum_{\kappa_{\mathfrak{e}}\in \mathcal{D}(\alpha,1)}\sum_{\sigma_{\mathfrak{n}}\in \mathbb{Z}_{T_{\text{max}}}}\sum_{Eq(\mathcal{C}, \{\sigma_{\mathfrak{n}}\}_{\mathfrak{n}}, \{\kappa_{\mathfrak{e}}\}_{\mathfrak{e}},k)},
\end{equation}
which implies that
\begin{equation}\label{eq.termlemmaeq6op}
\begin{split}
    |Term(T, p)_{op,k,k'}|\lesssim \sum_{\kappa_{\mathfrak{e}}\in \mathcal{D}(\alpha,1)}\sum_{\sigma_{\mathfrak{n}}\in \mathbb{Z}_{T_{\text{max}}}}\sum_{Eq(\mathcal{C}, \{\sigma_{\mathfrak{n}}\}_{\mathfrak{n}}, \{\kappa_{\mathfrak{e}}\}_{\mathfrak{e}},k)} \sum_{c\in \mathscr{M}(\mathcal{C}) }\prod_{\mathfrak{n}\in \mathcal{C}}\frac{2^{-\frac{\tau_{\mathfrak{n}}}{2}}}{|c(\Omega)_{\mathfrak{n}}|+T^{-1}_{\text{max}}} \prod_{\mathfrak{e}\in \mathcal{C}_{\text{norm}}} p_{\mathfrak{e}}
\end{split}
\end{equation}

As \eqref{eq.lemboundtermTp}, we get
\begin{equation}\label{eq.lemboundtermTpop}
\begin{split}
    &Term(T, p)_k
    \\
    \lesssim& \sum_{\kappa_{\mathfrak{e}}\in \mathcal{D}(\alpha,1)}\sum_{\sigma_{\mathfrak{n}}\in \mathbb{Z}_{T_{\text{max}}}}\sum_{Eq(\mathcal{C}, \{\sigma_{\mathfrak{n}}\}_{\mathfrak{n}}, \{\kappa_{\mathfrak{e}}\}_{\mathfrak{e}},k)} \sum_{c\in \mathscr{M}(\mathcal{C}) }\prod_{\mathfrak{n}\in \mathcal{C}}\frac{2^{-\frac{\tau_{\mathfrak{n}}}{2}}}{|c(\Omega)_{\mathfrak{n}}|+T^{-1}_{\text{max}}} \prod_{\mathfrak{e}\in \mathcal{C}_{\text{norm}}} p_{\mathfrak{e}}
    \\
    \lesssim &\sum_{c\in \mathscr{M}(\mathcal{C}) }\sum_{\kappa_{\mathfrak{e}}\in \mathcal{D}(\alpha,1)}\sum_{\substack{\sigma_{\mathfrak{n}}\in \mathbb{Z}_{T_{\text{max}}}\\ |\sigma_{\mathfrak{n}}|\lesssim 1}}\prod_{\mathfrak{n}\in \mathcal{C}}\frac{2^{-\frac{\tau_{\mathfrak{n}}}{2}}}{|c(\{\sigma_{\mathfrak{n}}\})_{\mathfrak{n}}|+T^{-1}_{\text{max}}} \sum_{Eq(\mathcal{C}, \{\sigma_{\mathfrak{n}}\}_{\mathfrak{n}}, \{\kappa_{\mathfrak{e}}\}_{\mathfrak{e}},k)} 1 \prod_{\mathfrak{e}\in \mathcal{C}_{\text{norm}}} \frac{\kappa_{\mathfrak{e}}}{\kappa_{\mathfrak{e}}+1}
    \\
    \lesssim &\sum_{c\in \mathscr{M}(\mathcal{C}) }\sum_{\kappa_{\mathfrak{e}}\in \mathcal{D}(\alpha,1)}\sum_{\substack{\sigma_{\mathfrak{n}}\in \mathbb{Z}_{T_{\text{max}}}\\ |\sigma_{\mathfrak{n}}|\lesssim 1}}\prod_{\mathfrak{n}\in \mathcal{C}}\frac{2^{-\frac{\tau_{\mathfrak{n}}}{2}}}{|c(\{\sigma_{\mathfrak{n}}\})_{\mathfrak{n}}|+T^{-1}_{\text{max}}} \#Eq(\mathcal{C}, \{\sigma_{\mathfrak{n}}\}_{\mathfrak{n}}, \{\kappa_{\mathfrak{e}}\}_{\mathfrak{e}},k) \prod_{\mathfrak{e}\in \mathcal{C}_{\text{norm}}} \frac{\kappa_{\mathfrak{e}}}{\kappa_{\mathfrak{e}}+1}
    \\
    \lesssim &\sum_{c\in \mathscr{M}(\mathcal{C}) }\sum_{\kappa_{\mathfrak{e}}\in \mathcal{D}(\alpha,1)}\sum_{\substack{\sigma_{\mathfrak{n}}\in \mathbb{Z}_{T_{\text{max}}}\\ |\sigma_{\mathfrak{n}}|\lesssim 1}}\prod_{\mathfrak{n}\in \mathcal{C}}\frac{2^{-\frac{\tau_{\mathfrak{n}}}{2}}}{|c(\{\sigma_{\mathfrak{n}}\})_{\mathfrak{n}}|+T^{-1}_{\text{max}}} L^{O(n\theta)} Q^{\frac{n}{2}} \prod_{\mathfrak{e}\in \mathcal{C}_{\text{norm}}} \frac{\kappa_{\mathfrak{e}}}{\kappa_{\mathfrak{e}}+1} \prod_{\mathfrak{e}\in \mathcal{C}_{\text{norm}}} \kappa_{\mathfrak{e}}^{-1}
\end{split}
\end{equation}
Here in the last inequality we applied \eqref{eq.countingbd0} in Proposition \ref{prop.counting}.

After simplification, \eqref{eq.lemboundtermTpop} gives us 
\begin{equation}
\begin{split}
    |Term(T, p)_{op,k,k'}|\lesssim &L^{O(n\theta)} Q^{\frac{n}{2}}\sum_{c\in \mathscr{M}(\mathcal{C}) }\sum_{\kappa_{\mathfrak{e}}\in \mathcal{D}(\alpha,1)}\sum_{\substack{\sigma_{\mathfrak{n}}\in \mathbb{Z}_{T_{\text{max}}}\\ |\sigma_{\mathfrak{n}}|\lesssim 1}} \prod_{\mathfrak{n}\in \mathcal{C}}\frac{2^{-\frac{\tau_{\mathfrak{n}}}{2}}}{|c(\{\sigma_{\mathfrak{n}}\})_{\mathfrak{n}}|+T^{-1}_{\text{max}}}
    \\
    \lesssim &L^{O(n\theta)} Q^{\frac{n}{2}} \left(\sum_{\kappa_{\mathfrak{e}}\in \mathcal{D}(\alpha,1)} 1\right) \sum_{c\in \mathscr{M}(\mathcal{C})}\sum_{\substack{\sigma_{\mathfrak{n}}\in \mathbb{Z}_{T_{\text{max}}}\\ |\sigma_{\mathfrak{n}}|\lesssim 1}} \prod_{\mathfrak{n}\in \mathcal{C}}\frac{2^{-\frac{\tau_{\mathfrak{n}}}{2}}}{|c(\{\sigma_{\mathfrak{n}}\})_{\mathfrak{n}}|+T^{-1}_{\text{max}}}
    \\
    \lesssim & L^{O(n\theta)} Q^{\frac{n}{2}} \sum_{c\in \mathscr{M}(\mathcal{C})}\sum_{\substack{\sigma_{\mathfrak{n}}\in \mathbb{Z}_{T_{\text{max}}}\\ |\sigma_{\mathfrak{n}}|\lesssim 1}} \prod_{\mathfrak{n}\in \mathcal{C}}\frac{2^{-\frac{\tau_{\mathfrak{n}}}{2}}}{|c(\{\sigma_{\mathfrak{n}}\})_{\mathfrak{n}}|+T^{-1}_{\text{max}}}
\end{split}
\end{equation}
% \begin{equation}
% \begin{split}
%     &|Term(T, p)_{op,k,k'}|
%     \\
%     \lesssim&\sum_{c,c'\in \mathscr{M}(T) } \sum_{\sigma_{\mathfrak{n}}\in \mathbb{Z}_{T_{\text{max}}}} \prod_{\mathfrak{n}\in T_{\text{in}}}\frac{t\alpha}{|c(\{\sigma_{\mathfrak{n}}\})_{\mathfrak{n}}|+\alpha} \prod_{\mathfrak{n}\in T_{\text{in}}}\frac{t\alpha}{|c'(\{\sigma_{\mathfrak{n}}\})_{\mathfrak{n}}|+\alpha} (\#Eq(\mathcal{C}, \{\sigma_{\mathfrak{n}}\}_{\mathfrak{n}},k)+O(L^{-8d\, l(T)-8d}))
%     \\
%     \lesssim & L^{O(n\theta)} Q^{n} L^{\frac{1}{2} dn_d}\sum_{c,c'\in \mathscr{M}(T) } \sum_{\substack{\sigma_{\mathfrak{n}}\in \mathbb{Z}_{T_{\text{max}}}\\ |\sigma_{\mathfrak{n}}|\lesssim 1}} \prod_{\mathfrak{n}\in T_{\text{in}}}\frac{t\alpha}{|c(\{\sigma_{\mathfrak{n}}\})_{\mathfrak{n}}|+\alpha} \prod_{\mathfrak{n}\in T_{\text{in}}}\frac{t\alpha}{|c'(\{\sigma_{\mathfrak{n}}\})_{\mathfrak{n}}|+\alpha}  + O(L^{-6d\, l(T)-6d})
% \end{split}
% \end{equation}
Here in the first step we use the fact that $\prod_{\mathfrak{e}\in \mathcal{C}_{\text{norm}}} \frac{\kappa_{\mathfrak{e}}}{\kappa_{\mathfrak{e}}+1} \prod_{\mathfrak{e}\in \mathcal{C}_{\text{norm}}} \kappa_{\mathfrak{e}}^{-1}=\prod_{\mathfrak{e}\in \mathcal{C}_{\text{norm}}} \frac{1}{\kappa_{\mathfrak{e}}+1}\le 1$. The reason for other steps can be find in the derivation of \eqref{eq.lemboundtermTpsimplify}.

We claim that 
\begin{equation}\label{eq.lemTpvarianceclaimop}
     \sup_{c}\sum_{\substack{\sigma_{\mathfrak{n}}\in \mathbb{Z}_{T_{\text{max}}}\\ |\sigma_{\mathfrak{n}}|\lesssim 1}} \prod_{\mathfrak{n}\in \mathcal{C}}\frac{2^{-\frac{\tau_{\mathfrak{n}}}{2}}}{|c(\{\sigma_{\mathfrak{n}}\})_{\mathfrak{n}}|+T^{-1}_{\text{max}}}\lesssim L^{O(l(T)\theta)}2^{-\frac{1}{2}\sum_{j=1}^K \tau_{j}} T^{2l(T)}_{\text{max}}
\end{equation}

Since there are only bounded many matrices in $\mathscr{M}(\mathcal{C})$.  Given above claim, we know that 
\begin{equation}
    |Term(T, p)_{op,k,k'}|\lesssim L^{O(l(T)\theta)} 2^{-\frac{1}{2}\sum_{j=1}^K \tau_{j}} Q^{\frac{n}{2}} T^{2l(T)}_{\text{max}},
\end{equation}
which proves the lemma since $n=2l(T)$.

Now prove the claim. In a tree $T$, there are $l(T)$ branching nodes, so there are $l(T)$ nodes in $T_{\text{in}}$. Since all nodes of $\mathcal{C}$ comes the two copies of $T_{\text{in}}$, so there are $2l(T)$ nodes in $\mathcal{C}$. Label these nodes by $h=1,\cdots,2l(T)$ and denote $\sigma_{\mathfrak{n}}$ by $\sigma_{h}$ if $\mathfrak{n}$ is labelled by $h$. Since $\sigma_{h}\in \mathbb{Z}_{T_{\text{max}}}$, there exists $m_{h}\in \mathbb{Z}$ such that $\sigma_{h}=T^{-1}_{\text{max}} m_{h}$. \eqref{eq.lemTpvarianceclaimop} is thus equivalent to 
\begin{equation}\label{eq.lemTpvarianceclaim1op}
    T^{2l(T)}_{\text{max}}\sum_{\substack{m_{h}\in \mathbb{Z}\\ |m_{h}|\lesssim T_{\text{max}}}} \prod_{h=1}^{2l(T)}\frac{2^{-\frac{\tau_{\mathfrak{n}}}{2}}}{|c(\{m_{h}\})_{h}|+1}\lesssim L^{O(l(T)\theta)}2^{-\frac{1}{2}\sum_{j=1}^K \tau_{j}} T^{2l(T)}_{\text{max}}
\end{equation}

To we prove \eqref{eq.lemTpvarianceclaim1op}. We just need to show that 
\begin{equation}
    \sum_{\substack{m_{h}\in \mathbb{Z}\\ |m_{h}|\lesssim T_{\text{max}}}} \prod_{h=1}^{2l(T)}\frac{1}{|c(\{m_{h}\})_{h}|+1}\lesssim L^{O(l(T)\theta)}
\end{equation}
This can be proved by Euler-Maclaurin formula \eqref{eq.EulerMaclaurin} as \eqref{eq.lemTpvarianceEulerMac}.

Now we complete the proof of the claim and thus the proof of the lemma.
\end{proof}

\subsubsection{Proof of the operator norm bound} In this subsection, we finish the proof of Proposition \ref{prop.operatorupperbound'}.

\begin{proof}[Proof of Proposition \ref{prop.operatorupperbound'}]
By Lemma \ref{lem.treetermsoperator}, we have 

\begin{equation}
    \left(\prod_{j=1}^K\mathcal{P}^{\tau_j}_{T_j}(w)\right)_{k}(t)=\sum_{k'}\int_0^t H^{\tau_1\cdots \tau_{K}}_{Tkk'}(t,s) w_{k'}(s) ds
\end{equation}
and the kernel $H^{\tau_1\cdots \tau_{K}}_{Tkk'}$ is a polynomial of Gaussian variables given by
\begin{equation}
\begin{split}
H^{\tau_1\cdots \tau_{K}}_{Tkk'}(t,s)=\left(\frac{i\lambda}{L^{d}}\right)^l\sum_{k_1,\, k_2,\, \cdots,\, k_{l}} H^{\tau_1\cdots \tau_{K}}_{Tk_1\cdots k_{l}kk'} \xi_{k_1}\cdots \xi_{k_{l}}.
\end{split}
\end{equation}

By Proposition \ref{prop.treetermsvarianceoperator}, we have
\begin{equation}
    \sup_k\, \mathbb{E}|H^{\tau_1\cdots \tau_{K}}_{Tkk'}|^2\lesssim L^{O(l(T)\theta)}2^{-\frac{1}{2}\sum_{j=1}^K \tau_{j}} \rho^{2l(T)}.
\end{equation}

Then the large deviation estimate Lemma \ref{lem.largedev} gives 
\begin{equation}
|H^{\tau_1\cdots \tau_{K}}_{Tkk'}(t,s)|\lesssim L^{\frac{n}{2}\theta} \sqrt{\mathbb{E}|H^{\tau_1\cdots \tau_{K}}_{Tkk'}|^2}\lesssim L^{O(l(T)\theta)} 2^{-\frac{1}{4}\sum_{j=1}^K \tau_{j}} \rho^{l(T)},\qquad \textit{L-certainly}.
\end{equation} 

Summing over all $\tau_1,\cdots, \tau_{K}$ gives
\begin{equation}
\sum_{\tau_1,\cdots, \tau_{K}}|H^{\tau_1\cdots \tau_{K}}_{Tkk'}(t,s)|\lesssim L^{O(l(T)\theta)} \rho^{l(T)},\qquad \textit{L-certainly}.
\end{equation} 

Applying the epsilon net and union bound method as in \eqref{eq.unionbound}, we obtain
\begin{equation}
    \sup_{t,s}\sup_{|k|,|k'|\lesssim L^{2M}}\sum_{\tau_1,\cdots, \tau_{K}}|H^{\tau_1\cdots \tau_{K}}_{Tkk'}(t,s)|\lesssim L^{O(l(T)\theta)}  \rho^{l(T)},\qquad \textit{L-certainly}.
\end{equation}

For $w\in X^{p}_{L^{2M}}$, we have $|w_k(t)|\le \sup_{t}||w(t)||_{X^{p}_{L^{2M}}} \langle k\rangle^{-p}$. Since $w_{k'}=0$ if $|k'|\gtrsim L^{2M}$ and $H^{\tau_1\cdots \tau_{K}}_{Tkk'}=0$ if $|k-k'|\gtrsim 1$, we know that
\begin{equation}
\begin{split}
    \left|\left(\sum_{\tau_1,\cdots, \tau_{K}}\prod_{j=1}^K\mathcal{P}^{\tau_j}_{T_j}(w)\right)_{k}(t)\right|\le&\sum_{|k'|\lesssim L^{2M}}\int_0^t \sum_{\tau_1,\cdots, \tau_{K}}|H^{\tau_1\cdots \tau_{K}}_{Tkk'}(t,s)| |w_{k'}(s)| ds
    \\
    \lesssim&  L^{O(l(T)\theta)}  \rho^{l(T)}t \sup_{t}||w(t)||_{X^{p}_{L^{2M}}} \sum_{\substack{|k'|\lesssim L^{2M}\\ |k'-k|\lesssim 1}} \langle k'\rangle^{-p}
    \\
    \lesssim& L^{O(1+l(T)\theta)}  \rho^{l(T)} \sup_{t}||w(t)||_{X^{p}_{L^{2M}}}  \langle k\rangle^{-p}.
\end{split}
\end{equation}

Since $l(T)=\sum_{j=1}^K l(T_j)$, $L$-certainly we have 
\begin{equation}
    \left|\left|\sum_{\tau_1,\cdots,\tau_K}\prod_{j=1}^K\mathcal{P}^{\tau_j}_{T_j}\right|\right|_{L_t^{\infty}X^{p}}\le L^{O\left(1+\theta\sum_{j=1}^K l(T_j)\right)} \rho^{\sum_{j=1}^K l(T_j)}||w||_{L_t^{\infty}X^{p}_{L^{2M}}}.
\end{equation}

Therefore, we complete the proof of Proposition \ref{prop.operatorupperbound'}.
\end{proof}

\subsection{Asymptotics of the main terms} In this section, we prove \eqref{eq.n1} in Theorem \ref{th.main} which characterize the asymptotic behavior of $n^{(1)}(k)$.

\begin{prop}\label{prop.mainterms} Using the same notation as Theorem \ref{th.main} (2), then we have 
\begin{equation}
        n^{(1)}(k)=\left\{
\begin{aligned}
    &\frac{t}{T_{\mathrm{kin}}}\mathcal K(n_{\mathrm{in}})(k)+O_{\ell^\infty_k}\left(L^{-\theta}\frac{T_{\text{max}}}{T_{\mathrm {kin}}}\right)+\widetilde{O}_{\ell^\infty_k}\left(\epsilon_1\text{Err}_{D}(k_x)\frac{T_{\text{max}}}{T_{\mathrm {kin}}}\right)
    && \text{for any } |k|\le \epsilon_1 l_{d}^{-1},
    \\
    &0, && \text{for any } |k|\ge 2C_{2}  l_{d}^{-1}
\end{aligned}\right.
    \end{equation}
    Here 
    \begin{equation}
        \text{Err}_{D}(k_x)=\left\{\begin{aligned}
    &D^{d+1}, && \text{if } |k_x|\le D,
    \\
    &D^{d-1}(|k_x|^2+D|k_x|), && \text{if } |k_x|\ge D.
\end{aligned}
    \right.
    \end{equation}
\end{prop}
\begin{proof} The second case of \eqref{eq.n1} is obvious. We divide the proof of the first case into several steps.

\textbf{Step 1.} (Calculation of $\psi_{app,k}$ and $n^{(1)}(k)$) The first three terms in the tree expansion \eqref{eq.approxsol} can be calculated explicitly. 
\begin{equation}
\begin{split}
    \psi_{app,k}=\psi^{(0)}_{app,k}+\psi^{(1)}_{app,k}+\psi^{(2)}_{app,k}+\cdots
\end{split}
\end{equation}
% \begin{equation}
% \begin{split}
%     \psi_{app,k}=\xi_k+\frac{i\lambda}{L^{d}} \sum\limits_{k_1+k_2=k} k_{x}\xi_{k_1} \xi_{k_2} \int^{t}_0e^{i s\Omega(k_1,k_2,k)- \nu|k|^2(t-s)} ds
%     \\
%     -&2\left(\frac{\lambda}{L^{d}}\right)^2 \sum\limits_{k_1+k_2+k_3=k} k_{x}(k_{2x}+k_{3x})\xi_{k_1} \xi_{k_2}\xi_{k_3} \times
%     \\
%     &\int_{0\le r<s\le t}e^{i s\Omega(k_1,k_2+k_3,k)- \nu|k|^2(t-s)} e^{i r\Omega(k_2,k_3,k_2+k_3)- \nu|k_2+k_3|^2(s-r)} dsdr
%     \\
%     +&\cdots
% \end{split}
% \end{equation}
where $\psi^{(0)}_{app,k}$,$\psi^{(1)}_{app,k}$, $\psi^{(2)}_{app,k}$ are given by
\begin{equation}
    \psi^{(0)}_{app,k}=\xi_k
\end{equation}
\begin{equation}
    \psi^{(1)}_{app,k}=\frac{i\lambda}{L^{d}} \sum\limits_{k_1+k_2=k} k_{x}\xi_{k_1} \xi_{k_2} \int^{t}_0e^{i s\Omega(k_1,k_2,k)- \nu|k|^2(t-s)} ds
\end{equation}
\begin{equation}
\begin{split}
    \psi^{(2)}_{app,k}=&-2\left(\frac{\lambda}{L^{d}}\right)^2 \sum\limits_{k_1+k_2+k_3=k} k_{x}(k_{2x}+k_{3x})\xi_{k_1} \xi_{k_2}\xi_{k_3}\times
    \\
    &\int_{0\le r<s\le t}e^{i s\Omega(k_1,k_2+k_3,k)- \nu|k|^2(t-s)} e^{i r\Omega(k_2,k_3,k_2+k_3)- \nu|k_2+k_3|^2(s-r)} dsdr
\end{split}
\end{equation}

By \eqref{eq.n(j)}, we know that
\begin{equation}\label{eq.n(1)}
    n^{(1)}(k)=\mathbb E \left|\psi^{(1)}_{app,k}\right|^2+ 2\text{Re}\  \mathbb E \left(\psi^{(2)}_{app,k}\overline{\xi_k}\right).
\end{equation}

\textbf{Step 2.} (Decomposition of $\psi_{app,k}$) $e^{- \nu|k|^2(t-s)}$ is supposed to be close to $1$, so we have the decomposition $e^{- \nu|k|^2(t-s)}=1+(e^{- \nu|k|^2(t-s)}-1)=1-\left(\int_{0}^1 e^{-a\nu|k|^2(t-s)} da\right) \nu|k|^2(t-s)$. We can also decompose $\psi^{(1)}_{app,k}=\psi^{(11)}_{app,k}+\psi^{(12)}_{app,k}$ and $\psi^{(2)}_{app,k}=\psi^{(21)}_{app,k}+\psi^{(22)}_{app,k}$ into $\psi^{(1)}_{app,k}$ accordingly.
\begin{equation}
    \psi^{(11)}_{app,k}=\frac{i\lambda}{L^{d}} \sum\limits_{k_1+k_2=k} k_{x}\xi_{k_1} \xi_{k_2} \int^{t}_0e^{i s\Omega(k_1,k_2,k)} ds
\end{equation}
\begin{equation}
    \psi^{(12)}_{app,k}=-\nu|k|^2\frac{i\lambda}{L^{d}} \int_{0}^1\Bigg(\sum\limits_{k_1+k_2=k} k_{x}\xi_{k_1} \xi_{k_2} \int^{t}_0e^{i s\Omega(k_1,k_2,k)- a\nu|k|^2(t-s)} (t-s)ds\Bigg) da
\end{equation}
\begin{equation}
\psi^{(21)}_{app,k}=-2\left(\frac{\lambda}{L^{d}}\right)^2 \sum\limits_{k_1+k_2+k_3=k} k_{x}(k_{2x}+k_{3x})\xi_{k_1} \xi_{k_2}\xi_{k_3}\int_{0\le r<s\le t}e^{i s\Omega(k_1,k_2+k_3,k)} e^{i r\Omega(k_2,k_3,k_2+k_3)} dsdr
\end{equation}
\begin{equation}
\begin{split}
    \psi^{(22)}_{app,k}=&2\left(\frac{\lambda}{L^{d}}\right)^2 \int_{0}^1\Bigg(\sum\limits_{k_1+k_2+k_3=k} k_{x}(k_{2x}+k_{3x})\int_{0\le r<s\le t}(\nu|k|^2(t-s)+\nu|k_2+k_3|^2(s-r))
    \\
    & e^{i s\Omega(k_1,k_2+k_3,k)- a\nu|k|^2(t-s)} e^{i r\Omega(k_2,k_3,k_2+k_3)- a\nu|k_2+k_3|^2(s-r)} dsdr\xi_{k_1}\xi_{k_2}\xi_{k_3}\Bigg)da
\end{split}
\end{equation}

$n^{(1)}(k)$ is also decomposed into $n^{(1)}(k)=n^{(11)}(k)+n^{(12)}(k)$
\begin{equation}\label{eq.n(11)}
    n^{(11)}(k)=\mathbb E \left|\psi^{(11)}_{app,k}\right|^2+ 2\text{Re}\  \mathbb E \left(\psi^{(21)}_{app,k}\overline{\xi_k}\right).
\end{equation}
\begin{equation}
    n^{(12)}(k)=\mathbb E \left|\psi^{(12)}_{app,k}\right|^2+2\text{Re}\ \mathbb E \left(\psi^{(11)}_{app,k}\overline{\psi^{(12)}_{app,k}}\right)+ 2\text{Re}\  \mathbb E \left(\psi^{(22)}_{app,k}\overline{\xi_k}\right).
\end{equation}

\textbf{Step 3.} (Estimate of $n^{(12)}(k)$ for $|k|\le \epsilon_1 l_{d}^{-1}$) In this step, all constants in $\lesssim$ depend only on the dimension $d$. Define 
\begin{equation}\label{eq.H(11)}
    H^{(11)}_{k_1k_2k}=k_{x} \int^{t}_0e^{i s\Omega(k_1,k_2,k)} ds
\end{equation}
\begin{equation}\label{eq.H(12)}
    H^{(12)}_{k_1k_2k}=-\nu|k|^2k_{x} \int^{t}_0e^{i s\Omega(k_1,k_2,k)- a\nu|k|^2(t-s)} (t-s)ds
\end{equation}
\begin{equation}\label{eq.H(21)}
    H^{(21)}_{k_1k_2k_3k}=2k_{x}(k_{2x}+k_{3x})\int_{0\le r<s\le t}e^{i s\Omega(k_1,k_2+k_3,k)} e^{i r\Omega(k_2,k_3,k_2+k_3)} dsdr
\end{equation}
\begin{equation}\label{eq.H(22)}
\begin{split}
    H^{(22)}_{k_1k_2k_3k}=&-2k_{x}(k_{2x}+k_{3x})\int_{0\le r<s\le t}(\nu|k|^2(t-s)+\nu|k_2+k_3|^2(s-r))
    \\
    & e^{i s\Omega(k_1,k_2+k_3,k)- a\nu|k|^2(t-s)} e^{i r\Omega(k_2,k_3,k_2+k_3)- a\nu|k_2+k_3|^2(s-r)} dsdr
\end{split}
\end{equation}

Then we have
\begin{equation}\label{eq.psi(11)app}
    \psi^{(11)}_{app,k}=\frac{i\lambda}{L^{d}} \sum\limits_{k_1+k_2=k} H^{(11)}_{k_1k_2k}\xi_{k_1} \xi_{k_2} 
\end{equation}
\begin{equation}\label{eq.psi(12)app}
    \psi^{(12)}_{app,k}=\int^1_{0}\frac{i\lambda}{L^{d}} \sum\limits_{k_1+k_2=k} H^{(12)}_{k_1k_2k}\xi_{k_1} \xi_{k_2} da 
\end{equation}
\begin{equation}\label{eq.psi(21)app}
    \psi^{(21)}_{app,k}=\left(\frac{i\lambda}{L^{d}}\right)^2 \sum\limits_{k_1+k_2+k_3=k} H^{(11)}_{k_1k_2k_3k}\xi_{k_1} \xi_{k_2}\xi_{k_3} 
\end{equation}
\begin{equation}\label{eq.psi(22)app}
    \psi^{(22)}_{app,k}=\int^1_{0}\left(\frac{i\lambda}{L^{d}}\right)^2 \sum\limits_{k_1+k_2+k_3=k} H^{(12)}_{k_1k_2k_3k}\xi_{k_1} \xi_{k_2}\xi_{k_3} da 
\end{equation}

To derive an upper bound for $n^{(12)}(k)$, it suffices to consider $\mathbb E \left|\psi^{(12)}_{app,k}\right|^2$, $\text{Re}\ \mathbb E \left(\psi^{(11)}_{app,k}\overline{\psi^{(12)}_{app,k}}\right)$ and $\text{Re}\  \mathbb E \left(\psi^{(22)}_{app,k}\overline{\xi_k}\right)$ separately.

\textbf{Step 3.1.} (Upper bounds of $\mathbb E \left|\psi^{(12)}_{app,k}\right|^2$ and $\text{Re}\ \mathbb E \left(\psi^{(11)}_{app,k}\overline{\psi^{(12)}_{app,k}}\right)$) We first derive upper bounds for $H^{(1j)}$. For $H^{(11)}$, we have
\begin{equation}
\begin{split}
    |H^{(11)}_{k_1k_2k}|\lesssim& \left|\int^{t}_0\frac{k_x}{i(\Omega+T^{-1}_{\text{max}}\text{sgn}(\Omega))}e^{-isT^{-1}_{\text{max}}\text{sgn}(\Omega)}\frac{d}{ds}e^{i s\Omega+is\, \text{sgn}(\Omega)/T_{\text{max}}} ds\right|
\end{split}
\end{equation}

Integration by parts we get 
\begin{equation}\label{eq.H(11)bound}
\begin{split}
    |H^{(11)}_{k_1k_2k}|\lesssim&\left|\int^{t}_0\frac{T_{\text{max}}^{-1}k_x}{\Omega+T^{-1}_{\text{max}}\text{sgn}(\Omega)}e^{i s\Omega} ds\right|+\left|\left[\frac{k_x}{i(\Omega+T^{-1}_{\text{max}}\text{sgn}(\Omega)}e^{i s\Omega}\right]_{0}^t\right|
    \\
    \lesssim& \frac{|k_x|}{|\Omega|+T^{-1}_{\text{max}}}
\end{split}
\end{equation}

By a similar integration by parts method, we get
\begin{equation}\label{eq.H(12)bound}
    |H^{(12)}_{k_1k_2k}|\lesssim \frac{|k|^2|k_x|\nu T_{\text{max}}}{|\Omega|+T^{-1}_{\text{max}}}.
\end{equation}

By \eqref{eq.psi(12)app}, we know that 
\begin{equation}\label{eq.psi(12)bound}
\begin{split}
    \mathbb E \left|\psi^{(12)}_{app,k}\right|^2\le& \int^1_{0}\frac{\lambda^2}{L^{2d}} \mathbb E\left|\sum\limits_{k_1+k_2=k} H^{(12)}_{k_1k_2k}\xi_{k_1} \xi_{k_2}\right|^2 da 
    \\
    \le& \frac{2\lambda^2}{L^{2d}} \sup_{a\in[0,1]}\sum\limits_{k_1+k_2=k} \left|H^{(12)}_{k_1k_2k}\right|^2n(k_1) n(k_2)
    \\
    \le& \frac{2\lambda^2}{L^{2d}} (|k|^2|k_x|\nu T_{\text{max}})^2 \sum\limits_{k_1} \frac{n(k_1) n(k-k_1)}{(|\Omega(k_1,k-k_1,k)|+T^{-1}_{\text{max}})^2}
    \\
    \le& \frac{2\epsilon_1^2\lambda^2}{L^{2d}} \max(|k_x|^2,D^2) T_{\text{max}} L^dD^{d-1}=2\epsilon_1^2 \max(|k_x|^2,D^2) T_{\text{max}} T^{-1}_{\text{kin}}D^{d-1} 
\end{split}   
\end{equation}
Here in the second inequality we apply the Wick theorem to calculate the expectation. In the third inequality we apply \eqref{eq.H(12)bound}. In the last line we apply \eqref{eq.asymptoticsbound} in Theorem \ref{th.numbertheory} by taking $t=T_{\text{max}}$, $g(x)=\frac{1}{(1+|x|)^2}$ and $F(k_1)=n(k_1) n(k-k_1)$. In the last line we also use the facts that $|k|\le \epsilon_1 l_{d}^{-1}=\epsilon_1 (\nu T_{\text{max}})^{-\frac{1}{2}}$, $T_{\text{kin}}=\frac{1}{8\pi\alpha^2}=\frac{L^{d}}{8\pi\lambda^2}$ and $|k_x|\le |k_{x1}|+|k_{x2}|\lesssim D$.

By \eqref{eq.H(11)bound} and \eqref{eq.H(12)bound}, we know that 
\begin{equation}\label{eq.psi(11)(12)bound}
\begin{split}
    \left|\text{Re}\ \mathbb E \left(\psi^{(11)}_{app,k}\overline{\psi^{(12)}_{app,k}}\right)\right|
    \le& \int^1_{0}\frac{2\lambda^2}{L^{2d}} \text{Re}\,\mathbb E\left(\sum\limits_{k_1+k_2=k} H^{(11)}_{k_1k_2k}\xi_{k_1} \xi_{k_2}  \sum\limits_{k_1+k_2=k} \overline{H^{(12)}_{k_1k_2k}\xi_{k_1} \xi_{k_2}}\right) da 
    \\
    \le& \frac{2\lambda^2}{L^{2d}} \sup_{a\in[0,1]}\sum\limits_{k_1+k_2=k} \left|H^{(11)}_{k_1k_2k}\right|\left|H^{(12)}_{k_1k_2k}\right|n(k_1) n(k_2)
    \\
    \le& \frac{2\lambda^2}{L^{2d}} |k|^2|k_x|^2\nu T_{\text{max}} \sum\limits_{k_1} \frac{n(k_1) n(k-k_1)}{(|\Omega(k_1,k-k_1,k)|+T^{-1}_{\text{max}})^2}D^{d-1}
    \\
    \le& \frac{2\epsilon_1\lambda^2}{L^{2d}} \max(|k_x|^2,D^2) T_{\text{max}} L^d=\epsilon_1 \max(|k_x|^2,D^2) T_{\text{max}} T^{-1}_{\text{kin}} D^{d-1}
\end{split}   
\end{equation}
Here in the second inequality we apply the Wick theorem to calculate the expectation. In the third inequality we apply \eqref{eq.H(11)bound} and \eqref{eq.H(12)bound}. In the last line we apply \eqref{eq.asymptoticsbound} in Theorem \ref{th.numbertheory} by taking $t=T_{\text{max}}$, $g(x)=\frac{1}{(1+|x|)^2}$ and $F(k_1)=n(k_1) n(k-k_1)$. In the last line we also use the facts that $|k|\le \epsilon_1 l_{d}^{-1}=\epsilon_1 (\nu T_{\text{max}})^{-\frac{1}{2}}$, $T_{\text{kin}}=\frac{1}{8\pi\alpha^2}=\frac{L^{d}}{8\pi\lambda^2}$ and $|k_x|\le |k_{x1}|+|k_{x2}|\lesssim D$.

\eqref{eq.psi(12)bound} and \eqref{eq.psi(11)(12)bound} give desire upper bounds of $\mathbb E \left|\psi^{(12)}_{app,k}\right|^2$ and $\text{Re}\ \mathbb E \left(\psi^{(11)}_{app,k}\overline{\psi^{(12)}_{app,k}}\right)$.

\textbf{Step 3.2.} (Upper bound of $\text{Re}\  \mathbb E \left(\psi^{(22)}_{app,k}\overline{\xi_k}\right)$) By \eqref{eq.psi(22)app}, we have 
\begin{equation}
\text{Re}\  \mathbb E \left(\psi^{(22)}_{app,k}\overline{\xi_k}\right)=\int^1_{0}\left(\frac{i\lambda}{L^{d}}\right)^2 \text{Re}\,\sum\limits_{k_1+k_2+k_3=k} H^{(22)}_{k_1k_2k_3k}\mathbb E\left(\xi_{k_1} \xi_{k_2}\xi_{k_3}\overline{\xi_k}\right) da 
\end{equation}

By Wick theorem, $\mathbb E\left(\xi_{k_1} \xi_{k_2}\xi_{k_3}\overline{\xi_k}\right)=\delta_{k_1=-k_2}\delta_{k_3=k}+\delta_{k_1=-k_3}\delta_{k_2=k}+\delta_{k_1=k}\delta_{k_2=-k_3}$. Therefore we get 
\begin{equation}\label{eq.H(22)wick}
\begin{split}
    \text{Re}\  \mathbb E \left(\psi^{(22)}_{app,k}\overline{\xi_k}\right)=&\int^1_{0}\left(\frac{i\lambda}{L^{d}}\right)^2 \text{Re}\,\sum\limits_{k_1+k_2+k_3=k} H^{(22)}_{k_1k_2k_3k}(\delta_{k_1=-k_2}\delta_{k_3=k}+\delta_{k_1=-k_3}\delta_{k_2=k}+0) da
    \\
    =&\int^1_{0}2\left(\frac{i\lambda}{L^{d}}\right)^2 \sum\limits_{k_1+k_2+k_3=k} \text{Re}\,\left(H^{(22)}_{k_1,-k_1,k,k}\right) da
\end{split}
\end{equation}
Here in the first equality the term corresponding to $\delta_{k_1=k}\delta_{k_2=-k_3}$ vanishes because $H^{(22)}_{k,k_2,-k_2,k}=0$ and the two terms corresponding to $\delta_{k_1=-k_2}\delta_{k_3=k}$, $\delta_{k_1=-k_3}\delta_{k_2=k}$ are equal.

By \eqref{eq.H(22)}, we get
\begin{equation}
\begin{split}
    H^{(22)}_{k_1,-k_1,k,k}=-2k_{x}(k_{x}-k_{1x})\int_{0\le r<s\le t}&(\nu|k|^2(t-s)+\nu|k-k_1|^2(s-r))
    \\
    & e^{i (s-r)\Omega(k_1,k-k_1,k)- a\nu|k|^2(t-s)-a\nu|k-k_1|^2(s-r)}  dsdr.
\end{split}
\end{equation}

We find upper bound of $\text{Re}\,\left(H^{(22)}_{k_1,-k_1,k,k}\right)$ using integration by parts.
\begin{equation}\label{eq.ReH(22)bound}
\begin{split}
    &\text{Re}\,\left(H^{(22)}_{k_1,-k_1,k,k}\right)=2k_{x}(k_{x}-k_{1x})\int_{0\le r<s\le t}(\nu|k|^2(t-s)+\nu|k-k_1|^2(s-r))
    \\
    & \frac{e^{irT^{-1}_{\text{max}}sgn\, \Omega}}{i(\Omega+T^{-1}_{\text{max}}sgn\, \Omega)}\frac{d}{dr}e^{i (s-r)\Omega-irT^{-1}_{\text{max}}sgn\, \Omega}e^{- a\nu|k|^2(t-s)-a\nu|k-k_1|^2(s-r)}  dsdr
    \\
    =&\text{Re}\,\frac{2k_{x}(k_{x}-k_{1x})}{i(\Omega+T^{-1}_{\text{max}}sgn\, \Omega)}\int_{0\le s\le t}\nu|k|^2(t-s) e^{- a\nu|k|^2(t-s)}  ds
    \\
    -&\text{Re}\,\frac{2k_{x}(k_{x}-k_{1x})}{i(\Omega+T^{-1}_{\text{max}}sgn\, \Omega)}\int_{0\le s\le t}(\nu|k|^2(t-s)+\nu|k-k_1|^2s) e^{- a\nu|k|^2(t-s)-a\nu|k-k_1|^2s} e^{is\Omega} ds
    \\
    -&\text{Re}\,\frac{2k_{x}(k_{x}-k_{1x})}{i(\Omega+T^{-1}_{\text{max}}sgn\, \Omega)}\int_{0\le r<s\le t}\big[\nu|k-k_1|^2(\nu|k|^2(t-s)+\nu|k-k_1|^2(s-r)-1)
    \\
    & -iT^{-1}_{\text{max}}sgn\, \Omega\big]e^{i (s-r)\Omega- a\nu|k|^2(t-s)-a\nu|k-k_1|^2(s-r)}  dsdr.
\end{split}
\end{equation}

The first term on the right hands side equals to $0$ after taking the real part. Using the same integration by parts argument as in \eqref{eq.H(11)bound}, the second term can be bounded by 
\begin{equation}\label{eq.H(22)secondterm}
    \frac{|k|^2|k_x|(|k_{1x}|+|k_x|)\nu T_{\text{max}}}{(|\Omega(k_1,k-k_1,k)|+T^{-1}_{\text{max}})^2}.
\end{equation}

The last term can be bounded by integration by parts in the following integral
\begin{equation}\label{eq.asymstep3.2}
\begin{split}
    \int_{0\le r<s\le t}\big[&\nu|k-k_1|^2(\nu|k|^2(t-s)+\nu|k-k_1|^2(s-r)-1)
    \\
    & -iT^{-1}_{\text{max}}sgn\, \Omega\big]\frac{e^{irT^{-1}_{\text{max}}sgn\, \Omega}}{i(\Omega+T^{-1}_{\text{max}}sgn\, \Omega)}\frac{d}{dr}e^{i (s-r)\Omega}e^{- a\nu|k|^2(t-s)-a\nu|k-k_1|^2(s-r)}  dsdr
\end{split}
\end{equation}

Integration by parts and bound the three resulting integral by taking the absolute value of the integrand, then we get 
\begin{equation}
|\eqref{eq.asymstep3.2}|\le  \frac{|k|^2}{|\Omega|+T^{-1}_{\text{max}}}
\end{equation}

Therefore, the last term in \eqref{eq.ReH(22)bound} can also be bounded by \eqref{eq.H(22)secondterm}.

Then we get 
\begin{equation}
    \text{Re}\,\left(H^{(22)}_{k_1,-k_1,k,k}\right)\lesssim \nu T_{\text{max}}|k|^2|k_x|\frac{|k_{1x}|+|k_x|}{(|\Omega(k_1,k-k_1,k)|+T^{-1}_{\text{max}})^2}.
\end{equation}

Substitute into \eqref{eq.H(22)wick}, then we have
\begin{equation}\label{eq.psi(22)bound}
\begin{split}
    \text{Re}\  \mathbb E \left(\psi^{(22)}_{app,k}\overline{\xi_k}\right)&\lesssim\frac{\lambda^2}{L^{2d}} |k|^2|k_x|\nu T_{\text{max}}n(k)\sum\limits_{k_1} \frac{|k_{1x}|+|k_x|}{(|\Omega(k_1,k-k_1,k)|+T^{-1}_{\text{max}})^2}n(k_1)
    \\
    &\lesssim \frac{\epsilon_1\lambda^2}{L^{2d}}|k_x|\sum\limits_{k_1} \frac{D}{(|\Omega(k_1,k-k_1,k)|+T^{-1}_{\text{max}})^2}n(k_1)
    \\
    &+\frac{\epsilon_1\lambda^2}{L^{2d}}|k_x|^2\sum\limits_{k_1} \frac{1}{(|\Omega(k_1,k-k_1,k)|+T^{-1}_{\text{max}})^2}n(k_1)
    \\
    &\lesssim \epsilon_1\frac{T_{\text{max}}}{T_{\text{kin}}}\max(D^{d+1},|k_x|^2D^{d-1}+|k_x|D^{d})
\end{split}
\end{equation}
In the second inequality, we also use the facts that $|k|\le \epsilon_1 l_{d}^{-1}=\epsilon_1 (\nu T_{\text{max}})^{-\frac{1}{2}}$ and $T_{\text{kin}}=\frac{1}{8\pi\alpha^2}=\frac{L^{d}}{8\pi\lambda^2}$. In the last inequality we apply \eqref{eq.asymptoticsbound} in Theorem \ref{th.numbertheory} by taking $t=T_{\text{max}}$, $g(x)=\frac{1}{(1+|x|)^2}$, $F(k_1)=n(k_1) n(k-k_1)$ and $|k_x|\le |k_{x1}|+|k_{x2}|\lesssim D$.

Combining \eqref{eq.psi(12)bound}, \eqref{eq.psi(11)(12)bound} and \eqref{eq.psi(22)bound}, we get the following upper bound
\begin{equation}\label{eq.n(12)final}
    n^{(12)}(k)\lesssim \epsilon_1\frac{T_{\text{max}}}{T_{\text{kin}}}\underbrace{\max(D^{d+1},|k_x|^2D^{d-1}+|k_x|D^{d})}_{\text{Err}_{D}(k_x)}.
\end{equation}

\textbf{Step 4.} (Asymptotics of $n^{(11)}(k)$) By \eqref{eq.n(11)} and Wick theorem, we get 
\begin{equation}\label{eq.n(11)asym}
\begin{split}
    n^{(11)}(k)=&\frac{2\lambda^2}{L^{2d}} |k_x|^2\sum\limits_{k_1+k_2=k}n(k_1) n(k_2) \left|\int^{t}_0e^{i s\Omega(k_1,k_2,k)} ds\right|^2
    \\
    -&\frac{8\lambda^2}{L^{2d}}\sum_{k_1}k_x(k_x-k_{1x})n(k_1) n(k)\text{Re}\left(\int_{0\le r<s\le t} e^{i (s-r)\Omega(k_1,k-k_1,k)}  dsdr\right)
    \\
    =&\frac{2\lambda^2}{L^{2d}} |k_x|^2\sum\limits_{k_1+k_2=k}n(k_1) n(k_2) \frac{4\sin^2 \left(\frac{t}{2}\Omega(k_1,k_2,k\right))}{\Omega^2(k_1,k_2,k)}
    \\
    -&\frac{8\lambda^2}{L^{2d}}\sum_{k_1}k_x(k_x-k_{1x})n(k_1) n(k) \frac{2\sin^2 \left(\frac{t}{2}\Omega(k_1,k-k_1,k\right))}{\Omega^2(k_1,k-k_1,k)}
    \\
    =&8\pi \alpha^2t|k_x|^2\int_{\substack{(k_1, k_2)\in \R^{2d}\\k_1+k_2=k}}n(k_1) n(k_2)\delta(|k_1|^2k_{1x}+|k_2|^2k_{2x}-|k|^2k_{x})\, dk_1 dk_2
    \\
    -& 16\pi \alpha^2t\, n(k)\int_{\mathbb{R}^d}k_x(k_x-k_{1x})n(k_1) \delta(|k_1|^2k_{1x}+|k_2|^2k_{2x}-|k|^2k_{x})\, dk_1+O\left(L^{-\theta}\frac{T_{\text{max}}}{T_{\mathrm {kin}}}\right)
    \\
    =&\frac{t}{T_{\text{kin}}}\mathcal{K}(n)(k)+O\left(L^{-\theta}\frac{T_{\text{max}}}{T_{\mathrm {kin}}}\right)
\end{split}
\end{equation}
Here in the third equality we apply \eqref{eq.numbertheory1} in Theorem \ref{th.numbertheory} by taking $t\rightarrow\frac{t}{2}$, $g(x)=\frac{\sin^2(x)}{x^2}$ and $F(k_1)=n(k_1) n(k-k_1)$ or $F(k_1)=n(k_1)$. In the third equality we also use the facts that  $T_{\text{kin}}=\frac{1}{8\pi\alpha^2}=\frac{L^{d}}{8\pi\lambda^2}$.

Combining \eqref{eq.n(12)final} and \eqref{eq.n(11)asym}, we complete the prove of the first case in \eqref{eq.n1}.
\end{proof}
% \textbf{Step 5.} (Upper bound for $|k|\ge  l_{d}^{-1}$) By \eqref{eq.n(1)} and Wick theorem, we get 
% \begin{equation}
% \begin{split}
%     n^{(1)}(k)=&\frac{2\lambda^2}{L^{2d}} |k_x|^2\sum\limits_{k_1+k_2=k}n(k_1) n(k_2) \left|\int^{t}_0e^{i s\Omega(k_1,k_2,k)- \nu|k|^2(t-s)} ds\right|^2
%     \\
%     -&\frac{8\lambda^2}{L^{2d}}\sum_{k_1}k_x(k_x-k_{1x})n(k_1) n(k)\text{Re}\left(\int_{0\le r<s\le t} e^{i (s-r)\Omega(k_1,k-k_1,k)- \nu|k|^2(t-s)- \nu|k-k_1|^2(s-r)}  dsdr\right)
%     \\
%     =&\frac{2\lambda^2}{L^{2d}} |k_x|^2\sum\limits_{k_1+k_2=k}n(k_1) n(k_2) \frac{4\sin^2 \left(\frac{t}{2}\Omega(k_1,k_2,k\right))}{\Omega^2(k_1,k_2,k)}
%     \\
%     -&\frac{8\lambda^2}{L^{2d}}\sum_{k_1}k_x(k_x-k_{1x})n(k_1) n(k) \frac{2\sin^2 \left(\frac{t}{2}\Omega(k_1,k-k_1,k\right))}{\Omega^2(k_1,k-k_1,k)}
%     \\
%     =&8\pi \alpha^2t|k_x|^2\int_{\substack{(k_1, k_2)\in \R^{2d}\\k_1+k_2=k}}n(k_1) n(k_2)\delta(|k_1|^2k_{1x}+|k_2|^2k_{2x}-|k|^2k_{x})\, dk_1 dk_2
%     \\
%     -& 16\pi \alpha^2t\, n(k)\int_{\mathbb{R}^d}k_x(k_x-k_{1x})n(k_1) \delta(|k_1|^2k_{1x}+|k_2|^2k_{2x}-|k|^2k_{x})\, dk_1+O\left(L^{-\theta}\frac{T_{\text{max}}}{T_{\mathrm {kin}}}\right)
%     \\
%     =&\frac{t}{T_{\text{kin}}}\mathcal{K}(n)(k)+O\left(L^{-\theta}\frac{T_{\text{max}}}{T_{\mathrm {kin}}}\right)
% \end{split}
% \end{equation}

\medskip

\appendix

\section{Number theoretic results}\label{sec.numbertheoryA}
The mains results of this appendix is to prove Theorem \ref{th.numbertheory1} and Theorem \ref{th.numbertheory} 

\begin{thm}\label{th.numbertheory1}
Let $\Lambda(k)=(\beta_x k_x^2+\beta_2 k_2^2+\cdots\beta_d k_d^2)k_x$, $d\ge 3$, then for all $\beta\in [1,2]^d$, the following number theory estimate is true

\begin{equation}\label{eq.numbertheory1}
    \sup_{\substack{k,\sigma\in\mathbb{Z}_L^d\\k\ne 0,\ T\le L}} |k_x|T\#\left
    \{\begin{matrix}
k_1,k_2\in\mathbb{Z}_L^d \\
|k_1|\lesssim 1
\end{matrix}
:
\begin{matrix}
k_1+k_2=k \\
\Lambda(k_1)+\Lambda(k_2)=\Lambda(k)+\sigma+O(T^{-1})
\end{matrix}
\right\}\le L^{2d}.
\end{equation}

\end{thm}

\begin{rem}
The restriction $T\le L$ is not optimal. The optimal result is expected to be $T\le L^2$ for general $\beta$ and $T\le L^{d}$ for generic $\beta$. It is possible to apply circle method and probabilistic method in number theory method to prove the optimal result.
\end{rem}

\begin{rem}\label{rem.nottrue2d}
\ref{th.numbertheory1} and \ref{th.numbertheory} is unlikely to be true when $d=2$. Because in this case, the quadratic term of $k_1$ in \eqref{eq.A11} becomes $3k_xk_{1x}^2+k_x|k_{1y}|^2+2k_{y}\cdot k_{1x} k_{1y}$ ($k_{1\perp}$ becomes $k_{1\perp}=k_y$), which is degenerate when $k_{y}^2=3k_x^2$. 
\end{rem}

\begin{thm}\label{th.numbertheory} Let $\Omega_k(k_1)\defeq \Lambda(k_1)+\Lambda(k-k_1)-\Lambda(k)$ and $t$ be a large number, then given any smooth compactly supported $F(k)$ and smooth $g(s)$ satisfying $|g|(s)+|g'|(s)\lesssim 1/(1+s^2)$ and $\int_{\mathbb{R}} g(s) ds=c$, we have
\begin{equation}\label{eq.asymptotics}
    \sum_{k_1\in \mathbb{Z}_L^d} g(t\Omega_k(k_1)) F(k_1) = cL^dt^{-1} \int F(k_1) \delta(\Omega_k(k_1)) dk_1+ O(L^{d-1}).
\end{equation}
in particular
\begin{equation}\label{eq.asymptoticsbound}
    \sum_{k_1\in \mathbb{Z}_L^d} g(t\Omega_k(k_1)) F(k_1) \le  2ct^{-1}L^d D^{d-1}.
\end{equation}
\end{thm}

% k_1,k_2,k_3\in\mathbb{Z}_L^d \\
% |k_1|,|k_2|\le L^\theta
% \end{matrix}
% :
% \begin{matrix}
% k_1-k_2+k_3=k \\
% \Lambda(k_1)-\Lambda(k_2)+\Lambda(k_3)=\Lambda(k)+n+O(T^{-1})
% \end{matrix}
% \right\}\le L^{2d+O(\theta)}
% \end{equation}
% More generally, the proof of this theorem works for many other choices of $\Lambda(k)$.
% \end{thm}

% \begin{rem}
% Although the proof of Theorem \ref{th.numbertheory} works for much more general $\Lambda(k)$, it's quite hard to formulate a general theorem that covers all physically interesting cases. Therefore, we only illustrate the ideas by proving some special cases.
% \end{rem}

\begin{proof}[Proof of Theorem \ref{th.numbertheory1}]\eqref{eq.numbertheory1} is a corollary of the following lemma. 

\begin{lem}\label{lem.rationallemma} For any $d\ge 4$ and $\beta$
\begin{equation}
    \sup_{\substack{k,\sigma\in\mathbb{Z}_L^d\\k\ne 0}} |k_x| \#\{k_1\in\mathbb{Z}^d_L,\ |k_1|\lesssim 1:\Lambda(k_1)+\Lambda(k-k_1)-\Lambda(k)=\sigma+O(L^{-1})\}\lesssim L^{d-1} .
\end{equation}
\end{lem}
\begin{proof}
Define $\mathcal{D}_{k,\sigma}=\#\{k_1\in\mathbb{Z}^d_L,\ |k_1|\lesssim 1:\Lambda(k_1)+\Lambda(k-k_1)-\Lambda(k)=\sigma+O(L^{-1})\}$, then we just need to show that 
\begin{equation}\label{eq.goalofrationallemma}
    \sup_{\substack{\substack{k,\sigma\in\mathbb{Z}_L^d\\k\ne 0}}} \#\mathcal{D}_{k,\sigma}\lesssim L^{d-1} .
\end{equation}

We prove \eqref{eq.goalofrationallemma} using volume bound. The proof is divided into three steps.

% \textbf{Step 1.} In this step, we reduce Lemma \ref{lem.rationallemma} to a simpler lattice point counting problem.

% A simple calculation suggests that
% \begin{equation}
% &\Lambda(k_1)+\Lambda(k-k_1)-\Lambda(k)
% \\
% =&
% \end{equation}

\textbf{Step 1.} In this step, we show that 
\begin{equation}\label{eq.goalofrationallemmastep1}
    \#\mathcal{D}_{k,\sigma}\le L^{d} \text{vol}(\mathcal{D}^{\mathbb{R}}_{k,\sigma}),
\end{equation}
where $\mathcal{D}^{\mathbb{R}}_{k,\sigma}=\{k_1\in \mathbb{R}^d,\ |k_1|\lesssim 1:\Lambda(k_1)+\Lambda(k-k_1)-\Lambda(k)=\sigma+O(L^{-1})\}$.

\eqref{eq.goalofrationallemmastep1} can be proved from the following claim.

\textit{Claim.} If $k_1\in \mathcal{D}_{k,\sigma}$, then $D_{1/(2L)}(k_1)\subseteq \mathcal{D}^{\mathbb{R}}_{k,\sigma}$. Here  $D_{r}(k_1)=\{k_1'\in \mathbb{R}^d: \sup_{i=1,\cdots, d} |(k_1')_i-(k_1)_i|\le r\}$ ($(k_1)_i$ are the components of $k_1$). 

We prove the claim now. $x\in \mathcal{D}_{k,\sigma}$ is equivalent to $\Lambda(k_1)+\Lambda(k-k_1)-\Lambda(k)=\sigma+O(L^{-1})$. For any $k_1'\in D_{1/(2L)}(k_1)$, $|k_1'-k_1|\lesssim 1/L$, because $\Lambda$ is a Lifschitz function, we have $|\Lambda(k_1)-\Lambda(k_1')|\lesssim L^{-1}$. Therefore, 
\begin{equation}
\begin{split}
    &|\Lambda_{\beta}(k_1')+\Lambda_{\beta}(k-k_1')-\Lambda(k)-\sigma|
    \\
    \le &|\Lambda_{\beta}(k_1)+\Lambda_{\beta}(k-k_1)-\Lambda(k)-\sigma|+|\Lambda_{\beta}(k_1)-\Lambda_{\beta}(k_1')|+|\Lambda_{\beta}(k-k_1)-\Lambda_{\beta}(k-k_1')|
    \\
    \lesssim & L^{-1}.
\end{split}
\end{equation}
Therefore, we have $\Lambda(k_1')+\Lambda(k-k_1')-\Lambda(k)=\sigma+O(L^{-1})$ and thus $k_1'\in \mathcal{D}^{\mathbb{R}}_{k,\sigma}$. This is true for any $k_1'\in D_{1/(2L)}(k_1)$, so $D_{1/(2L)}(k_1)\subseteq \mathcal{D}^{\mathbb{R}}_{k,\sigma}$.

Now we prove \eqref{eq.goalofrationallemmastep1}. Since for different $k_1,k_1'\in \mathcal{D}_{k,\sigma}$, $D_{1/(2L)}(k_1)\cap D_{1/(2L)}(k_1')=\emptyset$, we have 
\begin{equation}
    \sum_{k_1\in \mathcal{D}_{k,\sigma}} \text{vol}( D_{1/(2L)}(k_1))=\text{vol}\left( \bigcup_{k_1\in \mathcal{D}_{k,\sigma}} D_{1/(2L)}(k_1)\right)\le \text{vol}(\mathcal{D}^{\mathbb{R}}_{k,\sigma}).
\end{equation}
The left hand side equals to $L^{-d}\#\mathcal{D}_{k,\sigma}$, so we get
\begin{equation}
    L^{-d}\#\mathcal{D}_{k,\sigma}\le \text{vol}(\mathcal{D}^{\mathbb{R}}_{k,\sigma}),
\end{equation}
which implies \eqref{eq.goalofrationallemmastep1}.

\textbf{Step 2.} In this step, we show that 
\begin{equation}\label{eq.goalofrationallemmastep2}
    \text{vol}(\mathcal{D}^{\mathbb{R}}_{k,\sigma})\le L^{-1} |k_x|^{-1}.
\end{equation}
Combining \eqref{eq.goalofrationallemmastep1} and \eqref{eq.goalofrationallemmastep2}, we get
\eqref{eq.goalofrationallemma}, which proves Lemma \ref{lem.rationallemma}.

Given a vector $k$, we denote the first component of $k$ by $k_x$ and the vector formed by other components by $k_{\perp}$. Then $k=(k_x, k_{\perp})$, $k_1=(k_{1x}, k_{1\perp})$ and a simple calculation suggests that
\begin{equation}\label{eq.A11}
\begin{split}
    &\Lambda(k_1)+\Lambda(k-k_1)-\Lambda(k)
 \\
 =&3k_xk_{1x}^2+k_x|k_{1\perp}|^2+2k_{\perp}\cdot k_{1\perp}k_{1x}-(3k_x^2+|k_{\perp}|^2)k_{1x}-2k_x k_{\perp}\cdot k_{1\perp}
\end{split}
\end{equation}

If we fix $k_{1x}$ to be a constant $c$ in $\Lambda(k_1)+\Lambda(k-k_1)-\Lambda(k)$ and denote the resulting function by $F_{k,\sigma, c}$, then 
\begin{equation}
    F_{k,\sigma, c}(k_{1\perp})=k_x|k_{1\perp}|^2+2(c-k_x)k_{\perp}\cdot k_{1\perp}+3k_x c^2-(3k_x^2+|k_{\perp}|^2)c.
\end{equation}

Therefore, if we define $\mathcal{D}^{\mathbb{R}}_{k,\sigma}(k_{1x}=c)$ by
\begin{equation}
\mathcal{D}^{\mathbb{R}}_{k,\sigma}(k_{1x}=c)=\{k_{1\perp}\in \mathbb{R}^{d-1},\ |k_{1\perp}|\lesssim 1:F_{k,\sigma, c}(k_{1\perp})=\sigma+O(L^{-1})\}
\end{equation}

Then 
\begin{equation}
    \mathcal{D}^{\mathbb{R}}_{k,\sigma}=\cup_{|c|\lesssim 1} \mathcal{D}^{\mathbb{R}}_{k,\sigma}(k_{1x}=c)
\end{equation}

By Fubini theorem (or coarea formula), we get
\begin{equation}\label{eq.rationallemmastep2'}
    \text{vol}(\mathcal{D}^{\mathbb{R}}_{k,\sigma})=\int_{|c|\lesssim 1} \text{vol}(\mathcal{D}^{\mathbb{R}}_{k,\sigma}(k_{1x}=c)) dc.
\end{equation}

To prove \eqref{eq.goalofrationallemmastep2}, it suffices to find an upper of $\text{vol}(\mathcal{D}^{\mathbb{R}}_{k,\sigma}(k_{1x}=c))$. Since $F_{k,\sigma, c}(k_{1\perp})$ is a quadratic function in $k_{1\perp}$ whose degree 2 term is $k_x|k_{1\perp}|^2$, a translation $k_{1\perp}\rightarrow k_{1\perp}-(c-k_x)k_{1\perp}/k_x$ transforms $F_{k,\sigma, c}(k_{1\perp})=\sigma+O(L^{-1})$ to $k_x|k_{1\perp}|^2=C_{k,\sigma,c}+O(L^{-1})$.

Therefore, 
\begin{equation}\label{eq.rationallemmastep2''}
\begin{split}
    \text{vol}(\mathcal{D}^{\mathbb{R}}_{k,\sigma}(k_{1x}=c))\le &\text{vol}(\{k_{1\perp}\in \mathbb{R}^{d-1}:k_x|k_{1\perp}|^2=C_{k,\sigma,c}+O(L^{-1})\})
    \\
    =& \text{vol}(\{k_{1\perp}\in \mathbb{R}^{d-1}:|k_{1\perp}|^2=C_{k,\sigma,c}/k_x+O(L^{-1}|k_x|^{-1})\})
    \\
    \lesssim& L^{-1}|k_x|^{-1}
\end{split}
\end{equation}
Here the last inequality follows from the elementary fact that $\text{vol}(\{x\in \mathbb{R}^n:|x|^2=R^2+O(\eta)\})\lesssim \eta$ when $n\ge 2$. Notice here the dimension of $k_{1\perp}$ is $d-1$ which is greater than $2$, so this fact is applicable.

Combining \eqref{eq.rationallemmastep2'} and \eqref{eq.rationallemmastep2'}, we proves \eqref{eq.goalofrationallemmastep2}. Therefore, we complete the proof of Lemma \ref{lem.rationallemma}
\end{proof}

We now return to the proof of Theorem \ref{th.numbertheory1}. Define 
\begin{equation}
    \mathcal{D}_{k,\sigma}([-T^{-1},T^{-1} ])=\{k_1\in\mathbb{Z}^d_L,\ |k_1|\lesssim 1:\Lambda(k_1)+\Lambda(k-k_1)-\Lambda(k)=\sigma+O(L^{-1})\}
\end{equation}

Let us prove \eqref{eq.numbertheory1}. For any $T\le L$, let $N=[L/T]+1$, then $\cup_{j=-N}^N [jL^{-1}, (j+1)L^{-1}]$ is a cover of $[-T^{-1}, T^{-1}]$. Therefore, $\cup_{j=-N}^N\mathcal{D}_{k,\sigma}([jL^{-1}, (j+1)L^{-1}])$ is a cover of $\mathcal{D}_{k,\sigma}([-T^{-1}$ $,T^{-1} ])$

By Lemma \ref{lem.rationallemma} we get 
\begin{equation}\label{eq.thrationalexpand}
    \#\mathcal{D}_{k,\sigma}([-T^{-1},T^{-1} ])\lesssim \sum_{j=-N}^N\#\mathcal{D}_{k,\sigma}([jL^{-1}, (j+1)L^{-1}])\lesssim NL^{d-1} |k_x|^{-1}\lesssim L^dT^{-1}|k_x|^{-1}.
\end{equation}

This proves \eqref{eq.numbertheory1}.
\end{proof}

\begin{proof}[Proof of Theorem \ref{th.numbertheory}] The inequality in Theorem \ref{th.numbertheory} follows from the equality because $\int F(k_1)$ $ \delta(\Omega_k(k_1)) dk_1\le D^{d-1}$, $t\le L^{-\theta}\alpha^{-2}\le L^{1-\theta}$ and $D^{d-1}\lesssim L^{\theta}$ if $L$ is large enough.

Let $k_{1}=\frac{K_1}{L}$ and $k=\frac{K}{L}$. Apply the high dimensional Euler-Maclaurin formula, we know that
\begin{equation}
\begin{split}
    \sum_{k_1\in \mathbb{Z}_L^d} g(t\Omega_k(k_1)) F(k_1)=&\int g(t\Omega_k(K_1/L)) F\left(\frac{K_1}{L}\right) dK_1 
    \\
    +& \sum_{ |J|_{\infty} = 1}\int_{\mathbb{R}^d} \{K_1\}^{J} \partial_{K_1}^{J}\left(g(t\Omega_k(K_1/L)) F\left(\frac{K_1}{L}\right)\right) dK_1
\end{split}
\end{equation}

We have the following estimates
\begin{equation}\label{eq.asymptoticlemmaeq1}
    \int g(t\Omega_k(K_1/L)) F\left(\frac{K_1}{L}\right) dK_1 =L^d\int g(t\Omega_k(k_1)) F(k_1) dk_1 
\end{equation}
and
\begin{equation}\label{eq.asymptoticlemmaeq2}
\begin{split}
    &\int_{\mathbb{R}^d} \{K_1\}^{J} \partial_{K_1}^{J}\left(g(t\Omega_k(K_1/L)) F\left(\frac{K_1}{L}\right)\right) dK_1
    \\
    = &L^{d-|J|}\int_{\mathbb{R}^d} \{Lk_1\}^{J} \partial_{k_1}^{J}\left(g(t\Omega_k(k_1)) F(k_1)\right) dk_1
    \\
    = &O\left(L^{d-|J|}\int_{\mathbb{R}^d}  |\partial_{k_1}^{J}\left(g(t\Omega_k(k_1))F(k_1)|\right) dk_1\right)
    \\
    = &O\left(\sum^{|J|}_{j=1}L^{d-j}t^{j}\int_{\mathbb{R}^d}  g^{(j)}(t\Omega_k(k_1))F^{(j)}(k_1) dk_1 \right)
\end{split}
\end{equation}
Here in the second equality we apply the fact that $\{Lk_1\}\le 1$. In the last line we define $g^{(j)}=\sum_{j'\le j} \left|\frac{d^{j'}}{d^{j'}s}g(s)\right|$ and $F^{(j)}(k_1)=\sum_{|J'|\le j}|\partial^{J'}_{k_1}F(k_1)|$ (Here $J'$ is a multi-index) and we also use the fact that $|\partial_{k_1}^{j}(g(t\Omega_k(k_1)))|\le t^{j} g^{(j)}(t\Omega_k(k_1))$. 

Now we just need to show that 
\begin{equation}\label{eq.asymptoticlemmaeq3}
    \int g(t\Omega_k(k_1)) F(k_1) dk_1 =ct^{-1} \int F(k_1) \delta(\Omega_k(k_1)) dk_1+ O(t^{-2})
\end{equation}
and 
\begin{equation}\label{eq.asymptoticlemmaeq4}
    \int_{\mathbb{R}^d}  g^{(j)}(t\Omega_k(k_1))F^{(j)}(k_1) dk_1 \le t^{-1}
\end{equation}

In fact, if we substitute \eqref{eq.asymptoticlemmaeq3} into \eqref{eq.asymptoticlemmaeq1}, we get the main term in \eqref{eq.asymptotics}. If we substitute \eqref{eq.asymptoticlemmaeq4} into \eqref{eq.asymptoticlemmaeq2}, we know that the error terms come from \eqref{eq.asymptoticlemmaeq2} can be bounded by $\sum^{|J|}_{j=1}L^{d-j}t^{j-1}\le |J| L^{d-1}$. Therefore, \eqref{eq.asymptoticlemmaeq3} and \eqref{eq.asymptoticlemmaeq4} implies the lemma.

By coarea formula we know that
\begin{equation}
    \int g(t\Omega_k(k_1)) F(k_1) dk_1 =\int g(t\omega) \underbrace{\left(\int F(k_1) \delta(\Omega_k(k_1)-\omega)dk_1\right) }_{h(\omega)} d\omega 
\end{equation}

Notice that $h(\omega)$ is differentiable, then we have
\begin{equation}
\begin{split}
    \int g(t\omega)h(\omega) d\omega=&\int g(t\omega)(h(\omega)-h(0)) d\omega + h(0)\int g(t\omega)d\omega
    \\
    = &ct^{-1}h(0) +O(t^{-2}). 
\end{split}
\end{equation}

Therefore, 
\begin{equation}
\begin{split}
    &\int g(t\Omega_k(k_1)) F(k_1) dk_1 =\int g(t\omega)h(\omega)  d\omega 
    \\
    =&ct^{-1}h(0) +O(t^{-2})=ct^{-1}\left(\int F(k_1) \delta(\Omega_k(k_1))dk_1\right)  +O(t^{-2})
\end{split}
\end{equation}
Using the same argument we can also prove \eqref{eq.asymptoticlemmaeq4}, so we complete the proof of the Theorem \ref{th.numbertheory}.
\end{proof}
% \textbf{Step 1.} (Reduction)

% To prove \eqref{eq.asymptotics}, without loss of generality, we just need to consider
% \begin{equation}
%     \sum_{k_1\in \mathbb{Z}_{L+}^d} g(t\Omega_k(k_1)) F(k_1).
% \end{equation}
% Here $\mathbb{Z}_{L+}^d\defeq \{(k^{(1)}, \cdots, k^{(d)})\in \mathbb{Z}_{L}^d: k^{(j)}\ge 0,\ \forall j\}$. To prove the general case, we can replace $F(k^{(1)}_1,\cdots,k^{(d)}_1)$ by $F(\pm k^{(1)}_1,\cdots,\pm k^{(d)}_1)$.

% Let $k=(k^{(1)}, \cdots, k^{(d)})$, $k_1=(k_1^{(1)}, \cdots, k_1^{(d)})$, $g'(s)=\frac{d}{ds} g(s)$ and $F'(k)=\partial_{k^{(1)}}\cdots\partial_{k^{(d)}}F(k)$. Then we get
% \begin{equation}
%     g(tx)=\int_{tx}^{\infty} g'(s) ds=\int \chi_{|tx|} g'(s)ds
% \end{equation}
% \begin{equation}
%     F(k)=\int_{k^{(1)}}^{\infty}\cdots\int_{k^{(d)}}^{\infty}  F'(m)dm^{(1)}\cdots dm^{(d)}
% \end{equation}

% Substituting into the left hand side of \eqref{eq.asymptotics}, we get
% \begin{equation}
%     \sum_{k_1\in \mathbb{Z}_L^d} g(t\Omega_k(k_1)) F(k_1).
% \end{equation}

\section{High Dimensional Euler-Maclaurin Formula}

\begin{thm}
Assume that $J=(j_1,\cdots,j_d)$ is a multi-index. Given a vector $K=(K^{(1)},\cdots,K^{(d)})$, define $K^J=\left(K^{(1)}\right)^{j_1}\cdots\left(K^{(d)}\right)^{j_d}$. Given a number $a$, $\{a\}=a-[a]$ is its fractional part, $\{K\}\defeq(\{K^{(1)}\},\cdots,\{K^{(d)}\})$. We also define $|J|_{\infty}=\sup_{1\le n\le d} j_{n}$, $|J|=\sum_{1\le n\le d} j_{n}$. Then we have

\begin{equation}\label{eq.EulerMaclaurin}
    \sum_{K\in\mathbb{Z}^d} f(K)=\int_{\mathbb{R}^d} f(K)dK+\sum_{ |J|_{\infty} = 1}\int_{\mathbb{R}^d} \{K\}^{J} \partial_K^{J}f(K) dK
\end{equation}

\end{thm}
\begin{proof}
This can be proved by induction.

When $d=1$, \eqref{eq.EulerMaclaurin} becomes
\begin{equation}
    \sum_{K\in\mathbb{Z}} f(K)=\int_{\mathbb{R}} f(K)dK+\int_{\mathbb{R}} \{K\} \partial_K f(K) dK
\end{equation}

This is the standard Euler-Maclaurin formula which can be proved by integration by parts in the second integral of the right hand side.

If the formula is true in dimension $d$, we prove it for dimension $d+1$. Assume that $K=(\widetilde{K},K^{(d+1)})$ is a $d+1$ dimensional vector, then apply the induction assumption 
\begin{equation}
    \sum_{\widetilde{K}\in\mathbb{Z}^d} f(\widetilde{K},K^{(d+1)})=\int_{\mathbb{R}^d} f(\widetilde{K},K^{(d+1)})d\widetilde{K}+\sum_{ |\widetilde{J}|_{\infty} = 1}\int_{\mathbb{R}^d} \{\widetilde{K}\}^{\widetilde{J}} \partial_{\widetilde{K}}^{\widetilde{J}}f(\widetilde{K},K^{(d+1)}) d\widetilde{K}.
\end{equation}

Sum over $K^{(d+1)}$ and apply $d=1$ Euler-Maclaurin formula we get
\begin{equation}
\begin{split}
    &\sum_{K^{(d+1)}\in\mathbb{Z}}\sum_{\widetilde{K}\in\mathbb{Z}^d} f(\widetilde{K},K^{(d+1)})
    \\
    =&\int_{\mathbb{R}^d} \sum_{K^{(d+1)}\in\mathbb{Z}}f(\widetilde{K},K^{(d+1)})d\widetilde{K}+\sum_{ |\widetilde{J}|_{\infty} = 1}\int_{\mathbb{R}^d} \{\widetilde{K}\}^{\widetilde{J}} \partial_{\widetilde{K}}^{\widetilde{J}}\sum_{K^{(d+1)}\in\mathbb{Z}}f(\widetilde{K},K^{(d+1)}) d\widetilde{K}.
    \\
    =&\int_{\mathbb{R}^d} \int_{\mathbb{R}}f(\widetilde{K},K^{(d+1)})dK^{(d+1)} d\widetilde{K}+\int_{\mathbb{R}^d} \int_{\mathbb{R}}\{K^{(d+1)}\}\partial_{K^{(d+1)}}f(\widetilde{K},K^{(d+1)})dK^{(d+1)} d\widetilde{K}
    \\
    +&\sum_{ |\widetilde{J}|_{\infty} = 1}\int_{\mathbb{R}^d}\int_{\mathbb{R}} \{\widetilde{K}\}^{\widetilde{J}} \partial_{\widetilde{K}}^{\widetilde{J}}f(\widetilde{K},K^{(d+1)})dK^{(d+1)} d\widetilde{K}
    \\
    +&\sum_{ |\widetilde{J}|_{\infty} = 1}\int_{\mathbb{R}^d}\int_{\mathbb{R}}\{K^{(d+1)}\}\partial_{K^{(d+1)}} \{\widetilde{K}\}^{\widetilde{J}} \partial_{\widetilde{K}}^{\widetilde{J}}f(\widetilde{K},K^{(d+1)})dK^{(d+1)} d\widetilde{K}.
    \\
    =&\int_{\mathbb{R}^{d+1}} f(K)dK + \sum_{ |J|_{\infty} = 1}\int_{\mathbb{R}^d} \{K\}^{J} \partial_K^{J}f(K) dK
\end{split}
\end{equation}
In the last step, $J=(\widetilde{J},j^{(d+1)})$. In the second equality, the second term corresponds to $\widetilde{J}=0$ and $j^{(d+1)}=1$, the third term corresponds to $\widetilde{J}\ne 0$ and $j^{(d+1)}=0$ and the fourth term corresponds to $\widetilde{J}\ne 0$ and $j^{(d+1)}=1$.
\end{proof}

\noindent {\bf Acknowledgments.}
The author would like to thank Tristan Buckmaster, Yu Deng, Zaher Hani and his advisor Alexandru Ionescu for many helpful discussions. The author also thanks Gigliola Staffilani and Minh-Binh Tran for the discussion on their deep paper \cite{ST}.

\bigskip

\Addresses

 \end{document}